\documentclass[10pt]{article}
\usepackage{mypackages}
\usepackage{makeidx}

\usepackage[T1]{fontenc}
\usepackage{titlesec}

\title{\textbf{\textsc{The $L^3$-based strong Onsager theorem}}}

\author{{\textsc{vikram giri, hyunju kwon, and matthew novack}}}

\date{}
\makeindex

\allowdisplaybreaks

\usepackage{tkz-euclide}
\definecolor{ao(english)}{rgb}{0.0, 0.5, 0.0}
\definecolor{applegreen}{rgb}{0.55, 0.71, 0.0}
\definecolor{darkpastelgreen}{rgb}{0.01, 0.75, 0.24}
\definecolor{azure}{rgb}{0.0, 0.5, 1.0}
\definecolor{burgundy}{rgb}{0.5, 0.0, 0.13}
\definecolor{battleshipgrey}{rgb}{0.52, 0.52, 0.51}
\definecolor{bluey}{rgb}{.15, 0, 5}
\definecolor{alizarin}{rgb}{0.82, 0.1, 0.26}
\definecolor{brightpink}{rgb}{1.0, 0.0, 0.5}

\titleformat{\subsection}[runin]
      {\normalfont\bfseries}
      {\thesubsection}
      {0.5em}
      {}
      [.]
\titleformat{\subsubsection}[runin]
      {\normalfont\bfseries}
      {\thesubsubsection}
      {0.5em}
      {}
      [.]

\begin{document}

\maketitle

\begin{abstract}
In this work, we prove the $L^3$-based strong Onsager conjecture for the three-dimensional Euler equations.  Our main theorem states that there exist weak solutions which dissipate the total kinetic energy, satisfy the local energy inequality, and belong to $C^0_t (W^{\sfrac 13-, 3} \cap L^{\infty-})$. More precisely, for every $\upbeta<\sfrac 13$, we can construct such solutions in the space $C^0_t ( B^{\upbeta}_{3,\infty} \cap L^{\frac{1}{1-3\upbeta}} )$.
\end{abstract}

\setcounter{tocdepth}{1}
\tableofcontents

\section{Introduction}
In this article, we consider the three-dimensional incompressible Euler equations 
\begin{equation}\label{eqn:Euler}
    \begin{cases}
    \pa_t u + (u\cdot \na) u + \na p = 0 \\
    \div \, u =0 
    \end{cases}
\end{equation}
posed on the periodic domain $\T^3=[-\pi,\pi]^3$.  These equations describe the evolution of an ideal volume-preserving fluid. The unknowns $u:\T^3\times [0,T] \to \R^3$ and $p:\T^3\times [0,T] \to \R$ represent the fluid velocity and the pressure, respectively. The first equation gives the balance of momentum and is derived from Newton's second law of motion with internal force $-\na p$. The second equation, or divergence-free condition, encodes the incompressibility of the fluid. 

Smooth solutions of the Euler equations satisfy the \emph{local energy equality}
\begin{align}\label{LEE}
    \pa_t\left(\frac12 |u|^2\right) + 
    \div\left(\left(\frac12|u|^2+p \right) u\right) =0 \,,
\end{align}
which follows from taking the inner product of the momentum equation with the velocity. This identity asserts that the rate of change of kinetic energy in an arbitrary local region is balanced by the energy flux and the work done by the pressure through the boundary of the region. Integrating \eqref{LEE} in space, we see that the total kinetic energy is conserved in time:
\begin{align}\label{cons:total:kinetic}
    \frac 12 \int_{\T^3} |u(\cdot, t)|^2 dx = \frac 12 \int_{\T^3} |u(\cdot,0)|^2 dx, \qquad \forall \, t>0 \, . 
\end{align}
General weak solutions, however, need not satisfy \eqref{LEE}. Instead, weak solutions $u\in L^3_{t,x}$ are known to satisfy the following generalized local energy balance in the sense of distributions \cite{DuchonRobert00}:
\begin{align}
    \partial_t \left( \frac 12 |u|^2 \right) + \div \left( \left( \frac{1}{2} |u|^2 + p \right) u \right) = - D[u] \, . \label{thurzday:evening:sunday}
\end{align} 
The Duchon-Robert measure $D[u]$ above is a space-time distribution measuring the dissipation due to possible singularities and is defined by the formula
\begin{equation}\label{dr:measure:euler}
    D[u](x,t) = \lim_{\ell\to 0} \frac 14 \int_{\T^3} \nabla\phi_\ell(z) \cdot \left(u(x+z,t)-u(x,t)\right) \left|u(x+z,t)-u(x,t)\right|^2 \, dz \, ,
\end{equation}
where $\phi$ is any even bump function with unit mass and $\phi_\ell(z)=\ell^{-3}\phi(\ell^{-1}z)$. In the case that $D[u]$ is a non-negative distribution, then the weak solution $u$ is said to satisfy the \emph{local energy inequality}
\begin{align}
    \partial_t \left( \frac 12 |u|^2 \right) + \div \left( \left( \frac{1}{2} |u|^2 + p \right) u \right) = - D[u]  \leq 0 \, . \label{thurzday:evening:sunday:sunday}
\end{align} 
It is straightforward to check that $D[u]\equiv 0$ if there exists $\upbeta>\sfrac 13$ such that $u\in C^0_t B^{\upbeta}_{3,\infty}(\T^3)$, where
\begin{align}\label{eq:bosov}
    \norm{v}_{B_{p,\infty}^\upbeta(\T^3)}
    = \norm{v}_{L^p(\T^3)} + \sup_{|z|>0} \frac{\norm{v(\cdot +z)-v(\cdot)}_{L^p(\T^3)}}{|z|^\upbeta}\, .
\end{align}
The following conjecture addresses the sharpness of this $\sfrac 13$ threshold.
\begin{theorem-non}[\bf $L^3$-based strong Onsager conjecture] Let $\upbeta\in (0,1)$ and $T\in (0,\infty)$. 
\begin{enumerate}[(a)]
    \item\label{heatsie:a} {\bf (Conservation and local energy equality)} For any $\upbeta>\sfrac13$, if a weak solution to the Euler equations belongs to $C^0([0,T]; B^{\upbeta}_{3,\infty}(\T^3))$, then it satisfies the local energy equality \eqref{LEE} in the sense of distributions.
    \item\label{heatsie:b} {\bf (Dissipation and local energy inequality)} For any $0<\upbeta<\sfrac13$, there exist weak solutions to the Euler equations belonging to $C^0([0,T]; B^{\upbeta}_{3,\infty}(\T^3))$ which satisfy the local energy inequality \eqref{thurzday:evening:sunday:sunday} in the sense of distributions and dissipate the total kinetic energy. 
\end{enumerate}
\end{theorem-non}
\noindent  While the proof of part~\eqref{heatsie:a} is known from \cite{DuchonRobert00}, the best result to date for part~\eqref{heatsie:b} is due to De Lellis and the second author and treats the range $0<\upbeta<\sfrac 17$. In this paper, we provide for the first time a proof of part~\eqref{heatsie:b} in the full range $0<\upbeta<\sfrac 13$, via the following theorem.
\begin{theorem}[\bf Dissipation and local energy inequality]\label{thm:main}
For any fixed $\upbeta\in (0,\sfrac{1}{3})$ there exist weak solutions to the Euler equations \eqref{eqn:Euler} which belong to $C^0([0,T];(B^{\upbeta}_{3,\infty}\cap L^{\frac{1}{(1-3\upbeta)}})(\T^3))$, satisfy the local energy inequality \eqref{thurzday:evening:sunday:sunday}, and dissipate the total kinetic energy.
\end{theorem}
\noindent In the next subsection, we present the physical and mathematical motivation and history underlying the $L^3$-based strong Onsager conjecture.

\subsection{Mathematical and physical background}

We begin with the mathematical background, recalling that the Euler equations are formally related to the Navier-Stokes equations
\begin{equation}\label{eqn:NS}
    \begin{cases}
    \pa_t u^{\nu} + (u^{\nu}\cdot \na) u^{\nu} + \na p^{\nu} = \nu \Delta u^{\nu} \\
    \div \, u^{\nu} =0  \, 
    \end{cases}
\end{equation}
through the inviscid limit $\nu\to 0$ (or the infinite Reynolds number limit after nondimensionalization).  The well-known suitable weak solutions $u^\nu$ of \eqref{eqn:NS} are distributional solutions belonging to $L^\infty_t L^2(\T^3) \cap L^2_t W^{1,2}(\T^3)$ which satisfy 
\begin{align}
    \partial_t \left( \frac 12 |u^\nu|^2 \right) + \div \left( \left( \frac{1}{2} |u^\nu|^2 + p^\nu \right) u^\nu - \nu \nabla \frac 12 |u^\nu|^2 \right)  + \nu |\nabla u^\nu |^2 \leq 0 \label{thurzday:evening}
\end{align} 
in the sense of distributions \cite{CKN}.  Since the left-hand side is also equal to $-D[u^\nu]$ \cite{DuchonRobert00}, \eqref{thurzday:evening} is equivalent to the condition $D[u^\nu]\geq 0$.  It is known that the weak solutions to \eqref{eqn:NS} constructed by Leray~\cite{Leray34} do indeed satisfy this sign condition \cite{DuchonRobert00}.  Furthermore, if $u^\nu\rightarrow u$ in the $L^3_{x,t}$ topology, then $u$ must be a weak solution to Euler satisfying the local energy inequality~\eqref{thurzday:evening:sunday:sunday}.  Thus constructing Euler flows satisfying the local energy inequality allows for the possibility that such flows may be obtained as inviscid limits of suitable solutions of Navier-Stokes.

On the other hand, the $L^3$-based strong Onsager conjecture is closely related to the famous original conjecture of Onsager. Onsager observed that any weak solution of \eqref{eqn:Euler} in the space $L^\infty_tC^\alpha_x$ with $\alpha > \sfrac 13$ must conserve the total kinetic energy as in \eqref{cons:total:kinetic}, while for $\alpha < \sfrac 13$, he conjectured that ``turbulent dissipation as described could take place just as readily without the final assistance by viscosity'' \cite{Onsager49}. Remarkably, Onsager's conjecture has been proven in the subsequent decades. The conservation statement in $L^3_t B^{\alpha}_{3,\infty}$ for $\alpha>\sfrac 13$ was partially resolved by Eyink in \cite{Eyink75}, fully proven by Constantin, E, and Titi \cite{ConstantinETiti94}, and extended by Constantin, Cheskidov, Friedlander, and Shvydkoy \cite{CCFS08} and Drivas and Eyink \cite{DE}.  Conversely, Isett \cite{Isett2018} provided the first construction of non-conservative solutions in $C^\alpha$ for $\alpha<\sfrac 13$, Buckmaster, De Lellis, Sz\'ekelyhidi, and Vicol extended Isett's result to dissipative solutions \cite{BDLSV17}, and the first author and Radu proved the two-dimensional Onsager conjecture~\cite{GR}.  The first such pathological constructions are due to Scheffer \cite{Scheffer93} and Shnirelman \cite{Shnirelman00}, while the proofs in \cite{Isett2018} and \cite{BDLSV17} build upon a series of partial results, including \cite{Is2013, BDLISZ15, DaneriSzekelyhidi17}, which are rooted in the seminal works of De Lellis and Sz\'ekelyhidi in \cite{DLS09, DeLellisSzekelyhidi13}. These works connected ideas of Nash \cite{Nash} for isometric embeddings to fluid dynamics, and the technology stemming from \cite{DLS09, DeLellisSzekelyhidi13} is known by the names ``convex integration'' or ``Nash iteration.'' We refer to the surveys \cite{BVReview, DSReview} for a more complete history of the Onsager program. 

Despite the success of the Onsager program, the $L^3$-based ``strong'' version of Onsager's conjecture, or more generally any type of Onsager conjecture incorporating \eqref{thurzday:evening:sunday:sunday}, remained open prior to this work. The previous best result is due to the second author and De Lellis \cite{DK22}, who showed the existence of H\"older continuous weak solutions to the Euler equations in $C^{\upbeta}_{t,x}$ for any $\upbeta<\sfrac17$ which also satisfy the strict local energy inequality \eqref{thurzday:evening}. We also refer to earlier results of De Lellis and Sz\'ekelyhidi \cite{DLeSz2010} and Isett \cite{Is22}, the latter of which formulated the strong $C^0$ Onsager conjecture, which posits for each $\upbeta<\sfrac 13$ (or even for $\upbeta=\sfrac 13$) the existence of weak solutions satisfying the local energy inequality which belong to $L^\infty_t C^{\upbeta}_x$. 

The $L^3$-based strong Onsager conjecture is also strongly motivated by the study of turbulent fluids in the inviscid limit $\nu\rightarrow 0$.  The starting point for this discussion is the \emph{zeroth law of turbulence}, which asserts that the mean energy dissipation rate per unit mass does not vanish in the limit $\nu\rightarrow 0$; more precisely, 
\begin{align}\label{zero:law}
    \lim_{\nu\to 0} \left\langle  \nu|\na u^{\nu}|^2 + D [u^\nu] \right\rangle =  \varepsilon >0 \, .
\end{align}
Here $\langle \cdot \rangle$ denotes either a space-time average or an ensemble average. As a consequence of \eqref{cons:total:kinetic} and the zeroth law, one therefore expects to find weak solutions satisfying \eqref{eqn:Euler} in the sense of distributions, which however do not satisfy \eqref{LEE} or \eqref{cons:total:kinetic}. While the zeroth law has been verified to a great degree in experiments (see discussion in~\cite{BVReview}), and recent progress has been made in the context of the forced Navier-Stokes equations \cite{BDL, BCCDLS}, a full mathematical validation of the zeroth law remains a major open problem in fluid dynamics.

The $L^3$-based strong Onsager conjecture is connected as well to Kolmogorov's 1941 phenomenological theory of turbulence \cite{K2, K3, K1}. Using the zeroth law and the additional assumptions that the statistics of turbulent flows are homogeneous, isotropic, and self-similar,\footnote{For a treatment of K41 based on these assumptions of restored symmetry, we refer to the book of Frisch \cite{Frisch95}.} K41 theory offers predictions for the behavior of the $p^{\rm th}$-order absolute structure functions $S_p(\ell)$, defined by
\begin{align*}
S_p(\ell) := \dashint_0^T \dashint_{\mathbb{S}^2} \dashint_{\T^3} |u^{\nu}(x+\ell z ,t) - u^{\nu}(x,t)|^p \, dx \, dz \, dt \, .
\end{align*}
Specifically, K41-style asymptotic analysis predicts that for $\nu\ll 1$, and for length scales $\ell$ within the inertial range, $S_p(\ell)$ satisfies a scaling relation of the form $S_p(\ell)\sim |\ell|^{\zeta_p}$, where $\zeta_p=\sfrac p3$.  Recalling the Besov spaces $B^{\sfrac 13}_{p,\infty}$ defined in~\eqref{eq:bosov}, this suggests the uniform boundedness of turbulent flows in the Besov space $B_{p, \infty}^{\sfrac 13}$.
The case $p=3$ is special because the longitudinal variant $S^\parallel_3(\ell)$ of the absolute structure function $S_3(\ell)$ obeys a scaling relation known as Kolmogorov's $\sfrac 45$ law, consistent with the prediction that $\zeta_3 = 1$. The $\sfrac 45$ law is strongly supported by experimental evidence \cite[Figure~8.8]{Frisch95}, \cite[Figure~5]{ChenEtAl05}, \cite[Figure~3]{IshiharaEtAl09}, \cite[Figure~1]{ISY20}, indicating that $B^{\sfrac 13}_{3,\infty}$ is a natural function space for turbulent flows. 

Our final motivating factor from the study of turbulence is the phenomenon of intermittency, which may be characterized as deviations from the K41 scaling $S_p(\ell)\sim \sfrac p3$ for $p\neq 3$. When $p<3$, one typically observes that $\sfrac{\zeta_p}p> \sfrac 13$, while for $p>3$, one typically observes that $\sfrac{\zeta_p}p<\sfrac 13$; see \cite[Figure 8.8]{Frisch95}, or \cite[Figure 6]{ISY20} for a recent numerical simulation. These observations suggest that while $B^{\sfrac 13}_{3,\infty}$ is a physically reasonable space for turbulent flows, the H\"older space $C^{\sfrac13}$ in which Onsager's theorem has been proven may not be. In this direction, the third author and Vicol recently proved an intermittent Onsager theorem \cite{NV22} for non-conservative solutions in $C^0_t(H^{\sfrac 12-}\cap L^{\infty-})\subset C^0_tB_{3,\infty}^{\sfrac 13-}$, building upon earlier work of the third author with Buckmaster, Masmoudi, and Vicol \cite{BMNV21}.

We close this subsection with a short discussion of the proof of Theorem~\ref{thm:main}, upon which we expand in the next subsection. Our construction utilizes Nash iteration (convex integration) to generate a sequence of approximate solutions to the Euler equations with errors converging to zero over the course of the iteration. Standard Nash iterations adapted to the Euler equations measure approximate solutions in either $L^2$ or $L^\infty$.  Our \emph{intermittent $L^3$ iteration}, however, measures approximate solutions directly in $L^3$. Furthermore, the approximate velocities in standard Nash iterations are essentially partial Fourier sums, and the frequency separation of the velocity increments added at each step is crucial. This article however introduces a {\it wavelet-inspired scheme} that enables frequency overlap between successive velocity increments. Finally, in contrast to the straightforward construction of pressure in existing Nash iterations for the Euler equations, we must carefully construct an {\it intermittent pressure}, which plays an integral role throughout the proof.

\subsection{Blueprint for an intermittent, wavelet-inspired scheme}
In order to motivate an iteration which produces solutions in $C^0_t B^{\upbeta}_{3,\infty}$, we must identify the main difficulties in the iterations which produce solutions in H\"older spaces $C^\beta_{x,t}$. We recall that in \cite{Is22}, Isett constructed $C_{x,t}^{\sfrac{1}{15}-}$ weak solutions of 3D Euler as a limit of subsolutions $u_q$ to the following system:
\begin{equation}\label{eqn:rel:intro}
    \begin{cases}
    \pa_t u_q + \div \left( u_q \otimes u_q \right) + \na p_q = \div R_q \\
    \partial_t \left( \frac 12 |u_q|^2 \right) + \div \left( \left( \frac 12 |u_q|^2 + p_q \right) u_q \right) \leq (\partial_t + u_q \cdot \nabla) \kappa_q + \div \varphi_q + \div \left( R_q u_q \right) \\
    \div \, u_q =0  \, , 
    \end{cases}
\end{equation}
where $(u\otimes v)_{ij} = u_i v_j$ and $[\div(u\otimes v)]_j = \pa_i (u_iv_j)$.\index{$\otimes$} The first equation is known as the Euler-Reynolds system, while the second one is called the relaxed local energy inequality. The Reynolds stress error $R_q$ is a negative definite symmetric tensor, $\kappa_q = \sfrac 12 \tr R_q$, and $\varphi_q$ is a vector field called the current error. All three terms converge to zero as $q$ goes to $\infty$ in the sense of distributions, thus producing in the limit a weak solution $u$ to 3D Euler which satisfies the local energy inequality. The functions $u_q, R_q, \varphi_q$ are assumed to oscillate at spatial frequencies no larger than $\lambda_q \approx a^{(b^q)}$, where $a$ is sufficiently large and $b>1$ is as small as possible. Abbreviating the mixed $L^\infty_t L^p_x$ norms with simply $\| \cdot \|_p$, the natural inductive estimates for a $C_{x,t}^{\beta}$ scheme are
\begin{equation}\label{airport}
 \left\| u_q \right\|_\infty \les 1 \, , \qquad  \left\| \nabla_x^N \nabla u_q \right\|_\infty \leq \lambda_q^{-\beta + N + 1} \, , \qquad \left\| \nabla_x^N R_q \right\|_\infty \leq \lambda_{q+1}^{-2\beta} \lambda_q^N \, , \qquad \left\| \nabla_x^N \varphi_q \right\|_\infty \leq \lambda_{q+1}^{-3\beta}\lambda_q^N \, .
 \end{equation} 
Note that interpolating the first two bounds shows that $\{u_q\}_{q=1}^\infty$ is uniformly bounded in the $C^\beta$ norm. Then $w_{q+1}=u_{q+1}-u_q$ is constructed to oscillate at frequency $\lambda_{q+1}\approx\lambda_q^b$ and to cancel out the previous step stress/current errors through the following equations:
\begin{align}\label{eqn:correction}
\mathbb{P}_{\leq \la_q}|w_{q+1}|^2w_{q+1} = -2\ph_q, \quad
\mathbb{P}_{\leq \la_q} (w_{q+1}\otimes w_{q+1}) =- R_q\, . 
\end{align}
Eliding for the moment the fact that $R_q$ and $\varphi_q$ are not scalar-valued functions, we define the correction $w_{q+1}=w_{q+1,R} + w_{q+1,\varphi} $, similar to \cite{Is22}, by
$$  w_{q+1,\varphi} \approx  \vartheta_{q+1,\varphi}(x_2,x_3) \vec e_1
(-2\varphi_q)^{\sfrac 13} \, , \quad  w_{q+1,R} \approx \vartheta_{q+1,R}(x_2,x_3)\vec e_1
\left((-2\varphi_q)^{\sfrac 13}\texttt{m}_{\ph,2}^{\sfrac12} + (-R_q)^{\sfrac 12} \right) \, , $$
using products of low-frequency functions and shear flows $\vartheta_{q+1,\varphi}$, $\vartheta_{q+1,R}$ with concentrated frequency support $\approx \la_{q+1}$. Here $\texttt{m}_{\ph,2}\sim 1$ is the square mean $(2\pi)^{-3} \int_{\T^3} \vartheta_{q+1,\ph}^2 $ of the high frequency function $\vartheta_{q+1,\ph}$.  We further impose that the high frequency function $\vartheta_{q+1,R}$ has unit squared mean and zero cubic mean, while $\vartheta_{q+1,\ph}$ has a unit cubic mean, which essentially decouples the equations in \eqref{eqn:correction}. We thus have that the first equation in \eqref{eqn:correction} involves only $w_{q+1,\ph}$, which gives the desired $(-2\ph_q)^{\sfrac13}$. After plugging $w_{q+1,\ph}$ into the second equation, we treat $-\mathbb{P}_{\leq \la_q} (w_{q+1,\ph}\otimes w_{q+1,\ph})$ as another error to be corrected by $w_{q+1,R}$. Thus $w_{q+1,R}$ produces a low frequency part $(-2\varphi_q)^{\sfrac 13}\texttt{m}_{\ph,2}^{\sfrac12} + (-R_q)^{\sfrac 12}$ in \eqref{eqn:correction}.  One naturally assumes that to achieve the optimal regularity for the constructed solution, both terms in the low-frequency portion of $w_{q+1,R}$ have the same size, as given by \eqref{airport}.  

Then $w_{q+1}$ satisfies
$$  \left\| \nabla_x^{N} w_{q+1} \right\|_{\infty} \les \left\| R_q \right\|_\infty^{\sfrac{1}{2}} \lambda_{q+1}^{N} + \left\| \varphi_q \right\|_\infty^{\sfrac 13} \lambda_{q+1}^{N} \les \lambda_{q+1}^{-\beta + N} \, .  $$
Interpolating the bounds for $N=0,1$, we find that $w_{q+1}$ has unit $C^\beta$ norm, as did $u_q$. The new Reynolds stress will then include the error term $R_{q+1, \rm Nash} = \div^{-1} \left( w_{q+1} \cdot \nabla u_q \right)$ (named after the analogous error term in Nash's original isometric embedding iteration \cite{Nash}), which can be estimated by
\begin{align} 
\big{|} \underbrace{\div^{-1}}_{\textnormal{gains $\lambda_{q+1}$}} ( \underbrace{w_{q+1}}_{\lambda_{q+1}^{-\beta}} \cdot \underbrace{\nabla u_q}_{\approx \lambda_q^{1-\beta}} ) \big{|} &\leq \lambda_{q+2}^{-2\beta} \notag \\
\iff \lambda_q^{b(-1-\beta) + 1-\beta + 2\beta b^2} \leq 1 
\iff \beta (2b^2 - b -1) &\leq b-1 \iff \beta \leq \frac{1}{2b+1}  \, . \notag
\end{align}
Thus as $b\rightarrow 1$, $\beta\rightarrow \sfrac 13$, as desired. However, the analogous error term in the local energy inequality, called the Nash current error, only satisfies the estimate
\begin{align} 
\big{|} \underbrace{\div^{-1}}_{\textnormal{gains $\lambda_{q+1}$}} ( \underbrace{\mathbb{P}_{=\lambda_{q+1}} ( w_{q+1} \otimes w_{q+1} )}_{\lambda_{q+1}^{-2\beta}} \, : \,  \underbrace{\nabla u_q}_{\approx \lambda_q^{1-\beta}} ) \big{|} &\leq \lambda_{q+2}^{-3\beta} \notag \\
\iff \lambda_q^{b(-1-2\beta) + 1-\beta + 3\beta b^2} \leq 1 
\iff \beta (3b^2 - 2b -1) &\leq b-1 
\iff \beta \leq \frac{1}{3b+1}  \, . \label{weeeeeeeeeee}
\end{align}
This evident $\sfrac 14$ regularity ceiling is also imposed by several similar current error terms. All the evidence from existing Nash iteration schemes indicates that the above heuristics cannot be improved.  Furthermore, the best $C^\alpha$ result to date suffers from further complications which limit the regularity to $C^{\sfrac 17-}$ \cite{DK22}, suggesting that \emph{even in the most optimistic scenario, Nash iterations are incapable of reaching the $C^{\sfrac 13-}$ threshold for the strong Onsager conjecture and can only reach the threshold $C^{\sfrac 14}$.}

\subsubsection{Intermittent $L^3$ Nash iteration}\label{sss:L3}
The first difference between an intermittent Nash iteration and a homogeneous (opposite of intermittent)\index{homogeneous} iteration is that the high-frequency homogeneous shear flows $\vartheta_{q+1,\varphi} \vec e_1$ and $\vartheta_{q+1,R} \vec e_1$ are replaced by a pair of \emph{intermittent} shear flows $\varrho_{q+1,R} \vec e_1$ and $\varrho_{q+1,\varphi} \vec e_1$ (described in detail in the next subsubsection). Here ``intermittent''\index{intermittent} means that different $L^p_x$ norms satisfy very different bounds, i.e.
\begin{equation}\label{eq:sat:after}
\left\| \nabla_x^N \varrho_{q+1,R} \right\|_p \les r_q^{\frac 2p -1} \la_{q+1}^N \, , \qquad \left\| \nabla_x^N \varrho_{q+1,\varphi} \right\|_p \les r_q^{\frac 2p - \frac 23} \la_{q+1}^N \, , \qquad \textnormal{for some $0 < r_q \ll 1$} \, .  
\end{equation}
Then one naturally defines $w_{q+1}=w_{q+1,R}+w_{q+1,\varphi}$ by
$$  w_{q+1,\varphi} \approx \varrho_{q+1,\varphi} (-2\varphi_q)^{\sfrac 13} \, , \qquad  w_{q+1,R} \approx \varrho_{q+1,R} \left( r_q^{\sfrac 13} (-2\varphi_q)^{\sfrac 13} + (-R_q)^{\sfrac 12} \right) \, .$$
Unlike the $C^0$ iteration, the low frequency in $w_{q+1,R}$ has an additional smallness factor $\texttt{m}_{\ph,2}^{\sfrac12}\sim r_q^{\sfrac13}$ in front of $(2\ph_q)^{\sfrac13}$, due to the intermittency of $\varrho_{q+1,\ph}$. We now explain below that the flexibility afforded by the extra parameter $r_q$ allows our solutions to exceed the $\sfrac 14$ threshold described above. To see this, we must first recall that the iteration in \cite{NV22} required a ``Goldilocks amount'' of intermittency $r_q=(\lambda_q \lambda_{q+1}^{-1})^{\sfrac 12}$ in order to produce a solution in $B^{\sfrac 13-}_{3,\infty}$; any larger or smaller choice of $r_q$ causes the size of $\nabla w_{q+1}$ to grow too quickly as $q\rightarrow \infty$. This value of $r_q$ is determined by the ratio between the frequency of the error $\ph_q$ and the maximum frequency of the intermittent shear flow $\varrho_{q+1,\ph}$ associated to $w_{q+1,\ph}$, raised to the power of $\sfrac12$.
Rather remarkably, we shall see below that the Goldilocks amount of intermittency is \emph{precisely} the minimum amount required in order to make the current Nash error estimate consistent with $B^{\sfrac 13-}_{3,\infty}$ regularity.

We first interpolate the $L^1$ and $L^\infty$ inductive estimates for $R_q$ and $L^2$ and $L^\infty$ inductive estimates for $\nabla u_q$ from \cite{NV22} to posit that
$$ \left\| u_q \right\|_3 \les 1 \ , \qquad  \left\| \nabla^N \nabla u_q \right\|_3 \leq \lambda_q^{-\beta + N + 1} r_{q-1}^{-\sfrac 13} \, , \qquad  \left\| \nabla_x^N R_q \right\|_{\sfrac 32} \leq \lambda_{q+1}^{-2\beta} \la_q^N \, ,  $$
for integers $N\geq 0$, where $\beta<\sfrac 13$ and $u_q\rightarrow u$ in the $B^{\beta-}_{3,\infty}$ topology. We write $\beta-$ to emphasize that $r_{q-1}^{-\sfrac 13}$ incurs a small power loss $\lambda_{q-1}^{\frac{b-1}{6}}$ which disappears as $b\rightarrow 1$. In order to make the low-frequency coefficient $\varphi_q^{\sfrac 13}$ in $w_{q+1,R}$ no larger than $(-R_q)^{\sfrac 12}$, the natural inductive bound for $\varphi_q$ is
$$  \left\| \nabla^N \varphi_q \right\|_{1} \leq \lambda_{q+1}^{-3\beta} r_q^{-1} \lambda_q^N \, .  $$ 
Combining this bound with the sharp $L^p$ decoupling estimate proved in \cite[Lemma~A.1]{GKN23} and the extra factor of $r_q^{\sfrac 13}$ in the definition of $w_{q+1,R}$ above yields the balanced estimates
\begin{align*}
\left\| \nabla^{N} w_{q+1,\varphi} \right\|_3 &\les \left\| \varphi_q \right\|_{1}^{\sfrac 13} \left\| \nabla^{N} \varrho_{q+1,\varphi} \right\|_3 \approx \lambda_{q+1}^{-\beta+N} r_q^{-\sfrac 13}  \, , \\
\left\| \nabla^{N} w_{q+1,R} \right\|_3 &\les \left(\left\| R_q \right\|_{\sfrac 32}^{\sfrac 12} + r_q^{\sfrac 13} \left\| \varphi_q \right\|_1^{\sfrac 13} \right) \left\| \nabla^{N} \varrho_{q+1,R} \right\|_3 \approx \lambda_{q+1}^{-\beta+N} r_q^{-\sfrac 13}  \, .
\end{align*}
Now recalling the structure of $w_{q+1,\varphi} \approx \varrho_{q+1,\varphi} (-2\varphi_q)^{\sfrac 13}$ and using decoupling, H\"older's inequality, and our estimates on $\varrho_{q+1,\varphi}$, $\varphi_q$, and $\nabla u_q$, we may estimate the Nash current error term corresponding to $w_{q+1,\varphi}$ by
\begin{align} 
\| \div^{-1} \mathbb{P}_{=\lambda_{q+1}} (w_{q+1,\varphi}\otimes w_{q+1,\varphi}) \, : \, \nabla u_q \|_1 &\les \big{\|} \div^{-1}{\mathbb{P}_{=\lambda_{q+1}} (\varrho_{q+1,\varphi}^2) } \big{\|}_{1} \big{\|} |\varphi_q|^{\sfrac 23} \nabla u_q \big{\|}_1 \notag \\
&\les \lambda_{q+1}^{-1} r_q^{\sfrac 23} \left\| \varphi_q \right\|_1^{\sfrac 23} \left\| \nabla u_q \right\|_3 \notag \\
&\les \lambda_{q+1}^{-1} r_q^{\sfrac 23} \lambda_{q+1}^{-2\beta} r_q^{-\sfrac 23} \lambda_q^{-\beta + 1} r_{q-1}^{-\sfrac 13} \, . \label{sunday:sunday:sunday:night}
\end{align}
In order for this estimate to meet the desired inductive bound of  $\lambda_{q+2}^{-3\beta} r_{q+1}^{-1}$, we see that we need
\begin{align*}
    \lambda_{q+1}^{-1-2\beta} \lambda_q^{1-\beta} r_{q-1}^{-\sfrac 13} \leq \lambda_{q+2}^{-3\beta} r_{q+1}^{-1}  \qquad \underbrace{\impliedby}_{r_{q-1}^{\sfrac 13} r_{q+1}^{-\sfrac 13} >1}  \qquad  \lambda_{q+1}^{-1-2\beta} \lambda_q^{1-\beta} \leq \lambda_{q+2}^{-3\beta} r_{q+1}^{-\sfrac 23}
\end{align*}
Note crucially that the inequality on the right has gained $r_{q+1}^{-\sfrac 23}$ compared with \eqref{weeeeeeeeeee}. Then using that $r_{q+1}^{-1}=\lambda_q^{\frac{b(b-1)}{2}}$, the inequality on the right is equivalent to
\begin{align*}
    \beta\left( 3b^2 - 2b - b \right) \leq (b-1)\left( 1 + \frac{b}{3} \right) \qquad \iff \qquad \beta \leq \frac{1+\frac{b}{3}}{3b+1} \, ,
\end{align*}
so that $\beta\rightarrow\sfrac 13$ as $b\rightarrow 1$. Similar estimates hold for the Nash current error from $w_{q+1,R}$, as well as for a number of current error terms which faced $C^{\sfrac 14}$ regularity limitations in the $C^0$ iteration. 

We conclude by noting that while the basic scaling considerations above indicate that an $L^3$ iteration inspired by \cite{NV22} has some hope, the techniques from \cite{NV22} would suffer from a number of significant shortcomings if one were to attempt to use them in a proof of the strong $L^3$ Onsager conjecture.  We explain the most immediate of these shortcomings in the next subsubsection.

\subsubsection{Partial wavelet sums}



In order to understand the need for partial wavelet sums in our iteration, we begin by examining the consequences of replacing high-frequency shear flows $\vartheta_{q+1,R}\vec e_1$ with high-frequency, intermittent shear flows $\varrho_{q+1,R}\vec e_1$. Intermittency in Nash iterations dates back to the work of Buckmaster and Vicol \cite{BV19} for the 3D Navier-Stokes equations; we refer to \cite{CheskidovLuo, BCV, Luo} for further developments for the Navier-Stokes equations. The intermittent Mikado flows used in \cite{NV22} were introduced by Modena and Sz\'ekelyhidi in \cite{MS} (see also the homogeneous Mikado flows due to Daneri and Sz\'ekelyhidi \cite{DaneriSzekelyhidi17}).  One should visualize the intermittent Mikado flows $\varrho_{q+1,\varphi}\vec e_1$ or $\varrho_{q+1,R}\vec e_1$ as shear flows supported in thin tubes of diameter $\lambda_{q+1}^{-1}$ around lines in the $\vec e_1$ direction, which have been periodized to scale $(\lambda_{q+1}r_q)^{-1}$. The parameter $r_q=\la_q^{\sfrac 12}\la_{q+1}^{-\sfrac 12}$ thus quantifies both the measure of the support and the $L^p$ norms, and the effective frequencies are contained in the range $[\lambda_{q+1}r_q,\lambda_{q+1}]=[(\lambda_{q}\la_{q+1})^{\sfrac 12},\lambda_{q+1}]$. Thus we see that intermittency \emph{smears out} the frequency support of $w_{q+1}$.

This smearing of frequencies greatly affects nonlinear errors such as the current oscillation error
\begin{align*}
    \div^{-1} \circ \div \left( \varphi_q + (\mathbb{P}_{\leq \la_q} + \mathbb{P}_{>\la_q}) (\sfrac 12|w_{q+1,\varphi}|^2 w_{q+1,\varphi}) \right) \approx - \div^{-1} \mathbb{P}_{\geq \la_q^{\sfrac 12}\la_{q+1}^{\sfrac 12}} \left( \sfrac 12 |\varrho_{q+1,\varphi}|^2 \varrho_{q+1,\varphi} \right) \vec e_1 \cdot \nabla\varphi_q \, .
\end{align*}
In the above approximate equality we have used the form of $w_{q+1,\varphi}=(-2\varphi_q)^{\sfrac 13}\varrho_{q+1,\varphi} \vec e_1$, the identity $\vec e_1\cdot\nabla \varrho_{q+1,\varphi}\equiv 0$, and the heuristic that the leading order behavior of the operator $\div^{-1}$ on a product of high and low frequency terms can be understood by simply applying it to the high frequency term. The maximum frequency of the error term above is $\la_{q+1}$, and the minimum frequency is $\la_q^{\sfrac 12}\la_{q+1}^{\sfrac 12}$.  Then if we attempt to absorb this error term into $\varphi_{q+1}$, we see that
\begin{align*}
    \big{\|} - \underbrace{\div^{-1}\mathbb{P}_{\geq \la_q^{\sfrac 12}\la_{q+1}^{\sfrac 12}}}_{\textnormal{gains $\la_{q}^{-\sfrac 12}\la_{q+1}^{-\sfrac 12}$}} \underbrace{\left( \sfrac 12 |\varrho_{q+1,\varphi}|^2 \varrho_{q+1,\varphi} \right) \vec e_1}_{\textnormal{unit $L^1$ norm}} \cdot \underbrace{\nabla\varphi_q}_{\substack{\textnormal{$L^1$ size} \\ \lambda_{q+1}^{-3\beta} r_q^{-1} }} \big{\|}_{L^1} &\leq \lambda_{q+2}^{-3\beta} r_{q+1}^{-1} \\
    \iff \la_q^{3\beta b^2 - 3\beta b + \frac 12 (1-b) + \frac 12 (1-b)(b-1) }  \leq 1 \iff  \beta &\leq \frac 16 \, .
\end{align*}
Thus intermittency creates errors at frequencies \emph{below} $\la_{q+1}$ which are too large to be absorbed into $\varphi_{q+1}$.

In \cite{NV22}, the analogue of this issue in the Euler-Reynolds system was rectified by performing a further frequency decomposition of $[(\lambda_{q}\la_{q+1})^{\sfrac 12},\lambda_{q+1}]$ and adding further velocity increments, which still have maximum frequency $\la_{q+1}$, to handle the errors at frequencies lower than $\la_{q+1}$.  Attempting such a strategy here leads one to define the \emph{higher order current error} $\varphi_{q,\alpha}$ at frequency $\la_q^{1-\alpha}\la_{q+1}^{\alpha}$ for $\alpha\in[\sfrac 12, 1]$ by
$$ \varphi_{q,\alpha} := -\div^{-1} \mathbb{P}_{\approx\la_q^{1-\alpha}\la_{q+1}^{\alpha}} \left( \sfrac 12 |\varrho_{q+1,\varphi}|^2 \varrho_{q+1,\varphi} \right) \vec e_1 \cdot \nabla\varphi_q  \, , $$
which would be corrected by a \emph{higher order velocity increment} $w_{q+1,\alpha,\varphi}$.  Now let $R_{q,\alpha}$ be the analogous {\it higher order stress error} as in \cite{NV22}, which will be corrected by $w_{q+1,\al,R}$. As before, we then set
\begin{align*}
w_{q+1,\alpha,\varphi}\approx\varrho_{q+1,\al,\ph}(-2\ph_{q,\al})^{\sfrac13}, \quad
w_{q+1,\alpha,R}\approx\varrho_{q+1,\al,R}\left(r_{q,\al}^{\sfrac13}(-2\ph_{q,\al})^{\sfrac13} + (-R_{q,\al})^{\sfrac12}\right)\, ,
\end{align*}
where $r_{q,\alpha}$ is the intermittency parameter for $\varrho_{q+1,\ph,\alpha}$.
To determine $r_{q,\alpha}$, we recall that $|\ph_{q,\alpha}|^{\sfrac13}r_{q,\alpha}^{\sfrac13}$ must be no larger than $|R_{q,\alpha}|^{\sfrac12}$. Assuming that all intermediate higher order errors are corrected as desired, the higher order errors $R_{q,\alpha}$ and $\ph_{q,\alpha}$ with $\al\approx1^-$ will be absorbed to $R_{q+1}$ and $\ph_{q+1}$, respectively, leading to
\begin{align*}
 \la_{q+2}^{-2\be} \approx 
\norm{R_{q+1}}_{L^{\sfrac32}}\approx \norm{R_{q,1^-}}_{L^{\sfrac32}}
 \gtrsim  \norm{\ph_{q,1^-}}_{L^1}^{\sfrac23}r_{q,1^-}^{\sfrac23} \approx \norm{\ph_{q+1}}_{L^1}^{\sfrac32}\approx \la_{q+2}^{-2\be}r_{q+1}^{-\frac23}r_{q,1^-}^{\sfrac23} 
 \implies r_{q,1^-} \lesssim r_{q+1}\ll 1\, . 
\end{align*}
This however stands in contradiction with the restriction that maximum frequency of $w_{q+1,\ph}$ is no larger than $\la_{q+1}$. This is due to the fact that the associated Mikado flows $\varrho_{q+1,\al,\ph}$ have effective frequency support $[\la_{q+1}r_{q,\al}, \la_{q+1}]$, which must remain \emph{above} the frequency $\la_{q}^{1-\al}\la_{q+1}^{\al}$ of the error $\ph_{q,\al}$, leading $r_{q,1^-}\approx 1$ as in \cite{NV22}.
The failure of the strategy from \cite{NV22} suggests lifting the restriction on the maximum frequency of the higher order velocity increments.

To adjust the frequency support of $w_{q+1,\alpha,\ph}$, we turn to the Goldilocks ratio that determines the maximum and minimum frequencies of the Mikado flow $\varrho_{q+1,\al,\ph}$, for given frequency of the corrected error $\ph_{q,\al}$. The Goldilocks ratio suggests setting the maximum frequency of $\varrho_{q+1,\ph,1^-}$ to be $\la_{q+2}$, and the minimum frequency to be $\la_{q+1}^{\sfrac 12}\la_{q+2}^{\sfrac 12}$. The same reasoning also suggests that for general $\alpha$, one should set $r_{q,\al}\approx\left(\la_{q}^{1-\alpha}\la_{q+1}^\alpha\la_{q+1}^{1-\alpha}\la_{q+2}^\alpha\right)^{\sfrac 12}$, so that the maximum frequency of $\varrho_{q+1,\al,\ph}$ is $\la_{q+1}^{1-\al}\la_{q+2}^{\al}$ and the minimum is $\la_{q+1}^{1-\al}\la_{q+2}^{\al}r_{q,\alpha}$.  Interestingly, one may view these choices as a restoration of self-similarity which had been broken by the scheme in \cite{NV22}.  Indeed the choice of $r_{q,\alpha}$ from \cite{NV22} implies that $w_{q+1,\alpha,R}$ was \emph{much} less intermittent than $w_{q+1,R}$ as $\alpha\rightarrow 1$, thus breaking the intermittent self-similarity of the different components of the velocity field.  The natural conclusion of these observations, which in some sense is validated by our analysis in this paper, is that \emph{the local energy inequality imposes intermittent self-similarity by fixing the Goldilocks parameter of intermittency throughout the iteration}. 

Before we delve into the consequences of adjusting the maximum frequency of velocity increments, let us first introduce a basic set-up for our iteration based on the discussion in the previous paragraphs. In this set-up, we treat the construction of $w_{q+1,\alpha,\varphi}$ as a distinct iterative step rather than a sub-step of $q\to q+1$, because
$w_{q+1,\alpha}$ corrects $\varphi_{q,\alpha}$ and $R_{q,\alpha}$ in a manner completely analogous to how $w_{q+1}$ corrected $\varphi_q$ and $R_q$.
We assume the existence of a velocity field $u_q = \hat u_q + (u_q - \hat u_q)$ (where the ``hat'' notation is used to encode frequency information described below), a Reynolds stress $R_q$, a current error $\varphi_q$, a pressure $p_q$, and an intermittent pressure $\pi_q$ which satisfy
\begin{equation}\label{eqn:ER:intro:new}
    \begin{cases}
    \pa_t u_q + \div \left( u_q \otimes u_q \right) + \na p_q = \div \left( R_q - \pi_q \Id \right) \\
    \partial_t \left( \frac 12 |u_q|^2 \right) + \div \left( \left( \frac 12 |u_q|^2 + p_q \right) u_q \right) \leq (\partial_t + \hat u_q \cdot \nabla) \kappa_q + \div \varphi_q + \div \left( (R_q - \pi_q \Id) \hat u_q \right) \\
    \div \, u_q =0  \, , 
    \end{cases}
\end{equation}
where $\kappa_q:=\sfrac 12\tr\left(R_q-\pi_q\Id\right)$. Note that we have substituted the negative definite symmetric tensor $R_q-\pi_q\Id$ for $R_q$ in \eqref{eqn:rel:intro}. The role of the scalar function $\pi_q$ will be discussed in subsection~\ref{ss:ip:intro}. We assume the existence of a large parameter $\bn$ such that\footnote{In a standard Nash iteration, $\bn=1$.  For us $\bn$ will be a large positive integer which however is still fixed independently of $q$.} $\hat u_q$ oscillates at spatial frequencies no larger than $\la_q$ and $u_q - \hat u_q$ oscillates at spatial frequencies in between $\la_{q+1}$ and $\la_{\qbn-1}$.  In general, the subscript $q'$ with a ``hat'' (as in $\hat u_{q'}$) denotes a velocity field with maximum frequency $\la_{q'}$, while the subscript $q'$ and no ``hat'' (as in $u_{q'}$) denotes a velocity field with maximum frequency $\la_{q'+\bn-1}$. Choosing $\beta$ close to $\sfrac 13$, we then inductively assume that
\begin{align}
    \left\| u_q \right\|_3 \les 1 \, , \qquad \left\| \nabla_x^N \nabla u_q \right\|_3 \les \la_{\qbn-1}^{-\beta+1+N} \qquad \iff \qquad \left\| \hat u_{\qbn-1} \right\|_3 \les 1 \, , \qquad \left\| \nabla_x^N \nabla \hat u_{\qbn-1} \right\|_3 \les \la_{\qbn-1}^{-\beta+1+N} \, . \notag
\end{align}
Next, the Reynolds stress $R_q$ may decomposed as $R_q= \sum_{q'=q}^{\qbn-1} R_q^{q'}$, the intermittent pressure $\pi_q$ may be decomposed as $\pi_q = \sum_{q'=q}^{\infty} \pi_q^{q'}$,\footnote{While the error terms stop at $\qbn-1$, or the highest active frequency of an error term, the pressure terms past $\qbn-1$ are required so that the pressure satisfies certain scaling laws specified in subsubsection~\ref{sec:pik:ant}.  See also subsubsection~\ref{ss:ip:intro}.} and the current error $\varphi_q$ may be decomposed as $\varphi_q = \sum_{q'}^{\qbn-1} \varphi_q^{q'}$. The parameter $q'$ encodes the frequency $\la_{q'}$ at which $R_q^{q'}$, $\varphi_q^{q'}$, and $\pi_q^{q'}$ oscillate.  We therefore assume that
\begin{align}\notag
    \left\| \nabla_x^N R_q^{q'} \right\|_{\sfrac 32} + \left\| \nabla_x^N \pi_q^{q'} \right\|_{\sfrac 32} \leq \la_{q'+\bn}^{-2\beta} \la_{q'}^N \, , \qquad \left\| \nabla_x^{N} \varphi_q^{q'} \right\|_1 \leq \la_{q'+\bn}^{-3\beta} r_{q'}^{-1} \la_{q'}^N \, ,
\end{align}
where $r_{q'} = \la_{q'+\half} \la_{q'+\bn}^{-1}\approx (\la_{q'}\la_{q'+\bn})^{\sfrac12}$. 

We then construct $w_{q+1} = \hat w_{\qbn} = u_{q+1}- u_q$\index{$\hat u_q$}\index{$\hat w_q$} using intermittent Mikado flows $\varrho_{\qbn,R}$ and $\varrho_{\qbn,\varphi}$ which have minimum frequency $\la_{q+\half}$ and maximum frequency $\la_\qbn$. Since $w_{q+1}$ is used to correct errors at frequency $\la_q$, these choices adhere to the Goldilocks ratio of intermittency. Furthermore, $w_{q+1}$ is used to correct $R_q^q-\pi_q^q \Id$ and $\varphi_q^q$ while leaving $R_q^{q'}-\pi_q^{q'}\Id$ and $\varphi_q^{q'}$ \emph{intact} for $q'>q$.  The concept of correcting only the lowest frequency error term out of a large collection of error terms at many frequencies distinguishes our scheme from all existing schemes for fluid equations. The net result of adding $w_{q+1}$ will be the creation of new stress and current errors, which will get sorted into bins between $\la_{q+1}$ and $\la_{\qbn}$ and added to $R_q^{q'}$ and $\varphi_q^{q'}$ to form $R_{q+1}$ and $\varphi_{q+1}$. We emphasize that the terms in the partial sum $u_{q+1}=w_{q+1}+w_q + w_{q-1} + \dots$ have overlap in frequency when $|q'-q''|\leq \half$, so that $u_{q+1}$ is not a partial Fourier sum of the limiting solution. Rather, we consider it as a partial wavelet sum and will address the difficulties arising from frequency overlap with spatial support separation.
See Figure~\ref{figure:frequencies} for a schematic of the active frequencies of various functions.  

\begin{figure}
\begin{center}
\begin{tikzpicture}
\begin{scope}[xshift=.75cm, yshift=-2.5cm]
    \draw[line width=1pt] (-10, -3) -- (3.5, -3);
    \draw[line width=1.5pt, black] (-10,-3.2) -- (-10,-2.8);
    \draw[line width=1pt, dotted] (-10,-3) -- (-9.3,-2.3);
    \draw[black] (-9.3,-2.3)node[above=.1cm]{\textnormal{max. freq. of $R_{q}^{q}$, $\varphi_{q}^{q}$}};
    \draw[line width = 1.5pt, black] (3.5,-3.2) -- (3.5,-2.8);
    \draw[line width=1pt, dotted] (3.5,-3) -- (2.75, -2.3);
    \draw[black] (2.3,-2.3)node[above=.1cm]{\textnormal{max. freq. of $w_{q}=\hat w_{\qbn}$}};
    \draw[black] (-10,-3)node[below=.1cm]{$\lambda_{q}$};
    \draw[black] (4,-3)node[below=.1cm]{$\lambda_{\qbn}$};
    \draw[black] (-3.25,-3)node[below=.1cm]{$\lambda_{q+\half}$};
    \draw[line width=1pt, dotted] (-3.25,-3) -- (-3.25, -2.3);
    \draw[black] (-3.25, -2.3)node[above=.1cm]{\textnormal{min. freq. of $w_{q}=\hat w_{\qbn}$}};
    \draw[line width=1.5pt, black] (-3.25,-3.2) --  (-3.25,-2.8);
    \draw[line width=1pt, ao(english)] (-2.5-0*.75,-3.1) --  (-2.5-0*.75,-2.9);
    \draw[line width=1pt, ao(english)] (-2.5-2*.75,-3.1) --  (-2.5-2*.75,-2.9);
    \draw[line width=1pt, ao(english)] (-2.5-3*.75,-3.1) --  (-2.5-3*.75,-2.9);
    \draw[line width=1pt, ao(english)] (-2.5-4*.75,-3.1) --  (-2.5-4*.75,-2.9);
    \draw[line width=1pt, ao(english)] (-2.5-5*.75,-3.1) --  (-2.5-5*.75,-2.9);
    \draw[line width=1pt, ao(english)] (-2.5-6*.75,-3.1) --  (-2.5-6*.75,-2.9);
    \draw[line width=1pt, ao(english)] (-2.5-7*.75,-3.1) --  (-2.5-7*.75,-2.9);
    \draw[line width=1pt, ao(english)] (-2.5-8*.75,-3.1) --  (-2.5-8*.75,-2.9);
    \draw[line width=1pt, ao(english)] (-2.5-9*.75,-3.1) --  (-2.5-9*.75,-2.9);
    \draw[line width=1pt, ao(english)] (-2.5+1*.75,-3.1) --  (-2.5+1*.75,-2.9);
    \draw[line width=1pt, ao(english)] (-2.5+2*.75,-3.1) --  (-2.5+2*.75,-2.9);
    \draw[line width=1pt, ao(english)] (-2.5+3*.75,-3.1) --  (-2.5+3*.75,-2.9);
    \draw[line width=1pt, ao(english)] (-2.5+4*.75,-3.1) --  (-2.5+4*.75,-2.9);
    \draw[line width=1pt, ao(english)] (-2.5+5*.75,-3.1) --  (-2.5+5*.75,-2.9);
    \draw[line width=1pt, ao(english)] (-2.5+6*.75,-3.1) --  (-2.5+6*.75,-2.9);
    \draw[line width=1pt, ao(english)] (-2.5+7*.75,-3.1) --  (-2.5+7*.75,-2.9);
\end{scope}
    \draw[line width=1pt] (-10, -3) -- (3.5, -3);
    \draw[line width=1.5pt, black] (-10,-3.2) -- (-10,-2.8);
    \draw[line width=1pt, dotted] (-10,-3) -- (-9.3,-2.3);
    \draw[black] (-9.3,-2.3)node[above=.1cm]{\textnormal{max. freq. of $R_{q-1}^{q-1}$, $\varphi_{q-1}^{q-1}$}};
    \draw[line width = 1.5pt, black] (3.5,-3.2) -- (3.5,-2.8);
    \draw[line width=1pt, dotted] (3.5,-3) -- (2.75, -2.3);
    \draw[black] (2.8,-2.3)node[above=.1cm]{\textnormal{max. freq. of $w_{q}=\hat w_{\qbn-1}$}};
    \draw[black] (-10,-3)node[below=.1cm]{$\lambda_{q-1}$};
    \draw[black] (4,-3)node[below=.1cm]{$\lambda_{\qbn-1}$};
    \draw[black] (-3.25,-3)node[below=.1cm]{$\lambda_{q+\half-1}$};
    \draw[line width=1pt, dotted] (-3.25,-3) -- (-3.25, -2.3);
    \draw[black] (-3.25, -2.3)node[above=.1cm]{\textnormal{min. freq. of $w_{q}=\hat w_{\qbn-1}$}};
    \draw[line width=1.5pt, black] (-3.25,-3.2) --  (-3.25,-2.8);
    \draw[line width=1pt, ao(english)] (-2.5-0*.75,-3.1) --  (-2.5-0*.75,-2.9);
    \draw[line width=1pt, ao(english)] (-2.5-2*.75,-3.1) --  (-2.5-2*.75,-2.9);
    \draw[line width=1pt, ao(english)] (-2.5-3*.75,-3.1) --  (-2.5-3*.75,-2.9);
    \draw[line width=1pt, ao(english)] (-2.5-4*.75,-3.1) --  (-2.5-4*.75,-2.9);
    \draw[line width=1pt, ao(english)] (-2.5-5*.75,-3.1) --  (-2.5-5*.75,-2.9);
    \draw[line width=1pt, ao(english)] (-2.5-6*.75,-3.1) --  (-2.5-6*.75,-2.9);
    \draw[line width=1pt, ao(english)] (-2.5-7*.75,-3.1) --  (-2.5-7*.75,-2.9);
    \draw[line width=1pt, ao(english)] (-2.5-8*.75,-3.1) --  (-2.5-8*.75,-2.9);
    \draw[line width=1pt, ao(english)] (-2.5-9*.75,-3.1) --  (-2.5-9*.75,-2.9);
    \draw[line width=1pt, ao(english)] (-2.5+1*.75,-3.1) --  (-2.5+1*.75,-2.9);
    \draw[line width=1pt, ao(english)] (-2.5+2*.75,-3.1) --  (-2.5+2*.75,-2.9);
    \draw[line width=1pt, ao(english)] (-2.5+3*.75,-3.1) --  (-2.5+3*.75,-2.9);
    \draw[line width=1pt, ao(english)] (-2.5+4*.75,-3.1) --  (-2.5+4*.75,-2.9);
    \draw[line width=1pt, ao(english)] (-2.5+5*.75,-3.1) --  (-2.5+5*.75,-2.9);
    \draw[line width=1pt, ao(english)] (-2.5+6*.75,-3.1) --  (-2.5+6*.75,-2.9);
    \draw[line width=1pt, ao(english)] (-2.5+7*.75,-3.1) --  (-2.5+7*.75,-2.9);
\end{tikzpicture}
\end{center}
\caption{The top segment depicts the relevant frequencies at stage $q-1$.  The most important frequencies are indicated with black line segments and correspond to the frequency of $R_{q-1}^{q-1}$ and $\varphi_{q-1}^{q-1}$ (which are being corrected), and the minimum and maximum frequency of the velocity increment $w_q=\hat w_{\qbn-1}$ (which corrects $R_{q-1}^{q-1})$, $\varphi_{q-1}^{q-1}$). Note that these three frequencies obey the Goldilocks ratio of intermittency. The green line segments correspond to frequencies of errors $R_{q-1}^{k}$, $\varphi_{q-1}^k$ for $q\leq k \leq \qbn-2$ which are being ignored for the moment.  The bottom segment depicts the same frequencies, but at stage $q$.}\label{figure:frequencies}
\end{figure}
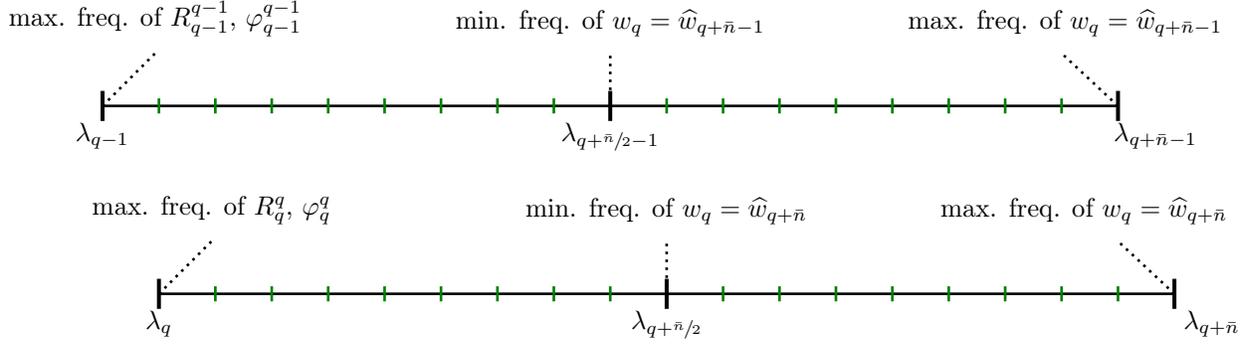

\subsubsection{Dodging techniques}

The inductive set-up described above needs to be complemented with assumptions on spatial support, as well as a methodology for propagating such information throughout the iteration.  To give an example of the kind of support properties we require, let us define the velocity increment $w_{q+1}=w_{q+1,R}+w_{q+1,\varphi}$ by
\begin{align}\label{looks:like:rain}
    w_{q+1,\varphi} = (-2\varphi_q^q)^{\sfrac 13} \varrho_{\qbn,\varphi} \, , \qquad w_{q+1,R} = \left( r_q^{\sfrac 13}(-2\varphi_q^q)^{\sfrac 13} + (R_q^q-\pi_q^q\Id)^{\sfrac 12} \right) \varrho_{\qbn,R} \, ,
\end{align}
where $\varrho_{\qbn,\varphi}$ and $\varrho_{\qbn,R}$ satisfy estimates identical to \eqref{eq:sat:after} after replacing $\la_{q+1}$ with $\la_\qbn$ and using the new definition of $r_q= \la_{q+\half}\la_{\qbn}^{-1}$. Then $w_{q+1,\varphi}$ and $w_{q+1,R}$ satisfy identical estimates, namely
\begin{align*}
\left\| \nabla^{N} w_{q+1,\varphi} \right\|_3 &\les \left\| \varphi_q \right\|_{1}^{\sfrac 13} \left\| \nabla^{N} \varrho_{\qbn,\varphi} \right\|_3 \approx \lambda_{\qbn}^{-\beta+N} r_q^{-\sfrac 13}  \, , \\
\left\| \nabla^{N} w_{\qbn,R} \right\|_3 &\les \left(\left\| R_q^q \right\|_{\sfrac 32}^{\sfrac 12} + r_q^{\sfrac 13} \left\| \varphi_q^q \right\|_1^{\sfrac 13} \right) \left\| \nabla^{N} \varrho_{\qbn,R} \right\|_3 \approx \lambda_{\qbn}^{-\beta+N} r_q^{-\sfrac 13}  \, .
\end{align*}
Now consider the Nash error obtained from adding $w_{q+1,\varphi}$, which we may estimate\footnote{Note the inverse divergence gain of $\la_\qbn$, which is larger than the minimum frequency $\la_{q+\half}$ of $w_{q+1}$. One can test the validity of this estimate by computing the one-dimensional version, where $\div^{-1}$ is simply integration.} by
\begin{align}\notag
    \big{\|} \div^{-1} ( \varrho_{\qbn,\varphi} (\varphi_q^q)^{\sfrac 13} \nabla u_q )  \big{\|}_{\sfrac 32} &\les \big{\|}  \underbrace{\div^{-1} \varrho_{\qbn,\varphi}}_{\textnormal{$L^{\sfrac 32}$ size $\la_\qbn^{-1}r_q^{\sfrac 23}$}} \underbrace{(\varphi_q^q)^{\sfrac 13}}_{\textnormal{$L^3$ size $\la_\qbn^{-\beta}r_q^{-\sfrac 13}$}} \cdot \underbrace{\nabla \hat u_q}_{\textnormal{$L^3$ size $\la_q^{-\beta+1}r_{q-\bn}^{-\sfrac 13}$}}  \big{\|}_{\sfrac 32} \notag\\
    &\qquad + \big{\|}  \underbrace{\div^{-1} \varrho_{\qbn,\varphi}}_{\textnormal{$L^{\sfrac 32}$ size $\la_\qbn^{-1}r_q^{\sfrac 23}$}} \underbrace{(\varphi_q^q)^{\sfrac 13}}_{\textnormal{$L^3$ size $\la_\qbn^{-\beta}r_q^{-\sfrac 13}$}} \cdot \underbrace{\left(\nabla u_q - \nabla \hat u_q \right)}_{\textnormal{$L^3$ size $\la_{\qbn-1}^{-\beta+1}r_{q-1}^{-\sfrac 13}$}}  \big{\|}_{\sfrac 32} \, . \notag
\end{align}
Since this error term oscillates at frequency $\la_{\qbn}$, we expect its size to be $\la_{q+2\bn}^{-2\beta}$ (the analogue of $\delta_{q+2}$ from \cite{NV22}, for example).  After a bit of arithmetic, one may check that the first term satisfies a sharp estimate when $\beta\rightarrow \sfrac 13$ (analogous to $\delta_{q+1}^{\sfrac 12} \delta_{q}^{\sfrac 12}\la_q \la_{q+1}^{-1}\leq \delta_{q+2}$ from a $C^{\sfrac 13 -}$ iteration, which is the size of the Nash error).  The second term, however, is far too large, due to the fact that  $\nabla u_q - \nabla \hat u_q$ has much larger $L^3$ norm than $\nabla \hat u_q$. The only way to close the estimate for the Nash error is then if 
$$  \supp w_{q+1} \cap \supp \left( u_q - \hat u_q \right) = \emptyset  \qquad \impliedby \supp w_{q+1} \cap  \left(\supp \hat w_{q+1} \cup \supp \hat w_{q+2} \dots \cup \supp \hat w_{\qbn-1} \right) = \emptyset \, , $$
where we have recalled that our ``hat'' notation gives that $u_q - \hat u_q = \hat w_{q+1} + \hat w_{q+2} + \dots + \hat w_{\qbn-1}$.  

There is however a clear obstruction to this assertion. Consider the velocity increments $\hat w_{q'}$ defined analogously to \eqref{looks:like:rain} for $q+1 \leq q' \leq q+ \half$. These velocity increments are constructed using intermittent Mikado flows $\varrho_{q',R}$ and $\varrho_{q',\varphi}$ which have pipe spacing $\la_{q'-\half}^{-1}$ and pipe thickness $\la_{q'}^{-1}$. Since the thickness of these pipes is \emph{larger} than the spacing of the pipes we plan to use at step $q$, namely $\la_{q+\half}^{-1}$, there is no way to arrange the support of $\hat w_{\qbn}$ to be disjoint from the support of $\hat w_{q+1}, \dots \hat w_{q+\half}$ without some adjustments to the definition of $\hat w_{q+n}$. We have solved this issue through the creation of a new type of multi-scale, intermittent building block, which we call an \emph{intermittent Mikado bundle}. 

An intermittent Mikado bundle $\mathbb{B}_{\varphi}$ or $\mathbb{B}_{R}$ is a stationary solution of Euler given as a product of two intermittent pipe flows (for example with velocity parallel to $e_1$) such as
\begin{equation}\label{intermittent:pipe:bundle}
    \mathbb{B}_{\varphi}(x,y) := e_1 \varrho_{\qbn,\varphi}(x_2,x_3) \tilde\varrho_{q+1}(x_2,x_3) \, .
\end{equation}
Here we have that $\varrho_{\qbn,\varphi}$ has minimum frequency $\la_{q+\half}$ and maximum frequency $\la_{\qbn}$, exactly as described in the last subsection.  However, $\tilde \varrho_{q+1}$ is essentially a standard, non-intermittent Mikado flow with minimum and maximum frequency $\la_{q+1}$ (so that the frequency support of $\mathbb{B}_\varphi$ is essentially the same as that of $\varrho_{\qbn,\varphi}$). We refer to $\tilde \varrho_{q+1}$ as the \emph{bundling pipe}. The purpose of $\tilde \varrho_{q+1}$ is to enable the support of $\hat w_\qbn$ to be chosen disjointly from the supports of $\hat w_{q+1}, \dots, \hat w_{q+\half}$, all of which are comprised of pipes with thickness \emph{less} than $\la_{q+1}^{-1}$ and spacing \emph{more} than $\la_{q+1}^{-1}$. This disjointness is ensured by \emph{shifting} the supports of the bundling pipes.  The idea that a compactly supported building block enjoys a degree of freedom according to shifts was first utilized in \cite{DK22} and further developed using intermittency in \cite{BMNV21, NV22}. The new development relative to the strategy used in \cite{NV22} is the introduction of the bundling pipe, combined with inductive assumptions on the support of previous velocity increments.  As in \cite{BMNV21} and \cite{NV22}, each bundle is pre-multiplied by space-time cutoff functions $\{a_i\}_{i\in \mathcal{I}}$ which form a partition of unity and localize the bundles to a particular space-time region. The timescale on which a cutoff function $a_i$ is non-zero is inversely proportional to the Lipschitz norm of $\hat u_q$ on the support of  $a_i$. The machinery for these cutoff functions was developed in \cite{BMNV21} and \cite{NV22}. Finally, on the support of each $a_i$, we compose with a local flow map $\Phi_{q,i}$ of $\hat u_q$. Then a more accurate description of $w_{q+1,\varphi}$ would be
\begin{equation} \label{pert:tuezday}
 w_{q+1,\varphi} = \sum_{i}  \left(-2\varphi_q^q\right)^{\sfrac 13} a_i\left( \nabla \hat u_q \right) \mathbb{B}_{i, \varphi} \circ \Phi_{q,i} \, .
\end{equation}
For a pictorial depiction of the support of $\mathbb{B}_{i,\varphi}$, we refer to Figure~\ref{figure:pipes}. 

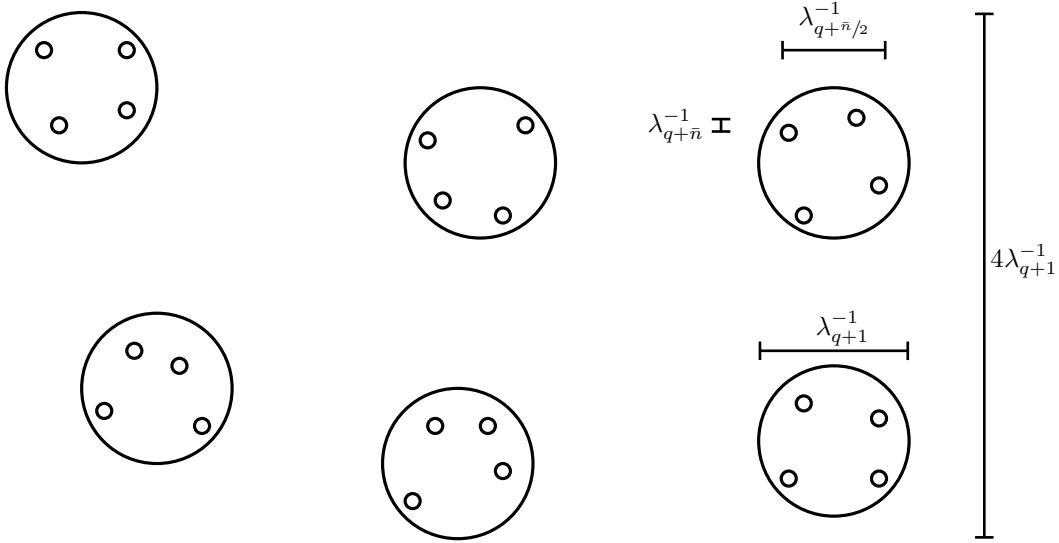
\begin{figure}
\begin{center}
\begin{tikzpicture}
\draw[very thick](-6,0) circle (1);
\draw[very thick](-6+1,-5+1) circle (1);
\draw[very thick](-6.6+.3,-.5) circle (.1);
\draw[very thick](-6.6+.9,-5.5+1.2) circle (.1);
\draw[very thick](-5.4,-.3) circle (.1);
\draw[very thick](-5.4+1,-5.5+1) circle (.1);
\draw[very thick](-6.5,0.5) circle (.1);
\draw[very thick](-6.3+1,-4.5+1) circle (.1);
\draw[very thick](-5.4,0.5) circle (.1);
\draw[very thick](-5.4+.7,-4.5+.8) circle (.1);
\draw[very thick](-1+.3,-1) circle (1);
\draw[very thick](-1,-5) circle (1);
\draw[very thick](-1.6+.4,-.5-1) circle (.1);
\draw[very thick](-1.6,-5.5) circle (.1);
\draw[very thick](-0.4,-.5-1.2) circle (.1);
\draw[very thick](-0.4,-5.1) circle (.1);
\draw[very thick](-1.6+.2,0.5-1.2) circle (.1);
\draw[very thick](-1.3,-4.5) circle (.1);
\draw[very thick](-0.4+.3,0.5-1) circle (.1);
\draw[very thick](-0.6,-4.5) circle (.1);
\draw[very thick](4,-1) circle (1);
\draw[very thick](4,-5+.3) circle (1);
\draw[very thick](3.6,-.5-1.2) circle (.1);
\draw[very thick](3.4,-5.5+.3) circle (.1);
\draw[very thick](4.6,-.5-.8) circle (.1);
\draw[very thick](4.6,-5.5+.3) circle (.1);
\draw[very thick](3.4,0.5-1.1) circle (.1);
\draw[very thick](3.6,-4.5+.3) circle (.1);
\draw[very thick](4.3,0.5-0.9) circle (.1);
\draw[very thick](4.6,-4.7+.3) circle (.1);
\draw[|-|,line width=1pt] (6, 1) -- (6,-6);
\draw[|-|,line width=1pt] (3, -3.5) -- (5,-3.5);
\draw[|-|,line width=1pt] (3.3, .5) -- (4.7,.5);
\draw[|-|,line width=1pt] (2.5, -.6) -- (2.5,-.4);
\draw[black] (7.2,-2.3) node[left=.1cm]{$\small 4\lambda^{-1}_{q+1}$};
\draw[black] (4.7,-3.2) node[left=.1cm]{$\small \lambda^{-1}_{q+1}$};
\draw[black] (4.7,.9) node[left=.1cm]{$\small \lambda^{-1}_{q+\half}$};
\draw[black] (2.5,-.5) node[left=.1cm]{$\small \lambda^{-1}_{\qbn}$};
\end{tikzpicture}
\end{center}
\caption{Depicted here is the cross-section of the support of a collection of intermittent pipe bundles $\mathbb{B}_{i,\varphi}$ (not drawn to scale). The large circles represent the support of various $\varrho_{q+1,i}$'s.  Each of these is shifted to ensure disjointness of $\hat w_{\qbn}$ with $\hat w_{q+1},\dots, \hat w_{q+\half}$. The small circles represent the support of of various $\varrho_{\qbn,i, \varphi}$'s. Each of these is shifted to ensure disjointness of $\hat w_{\qbn}$ with $\hat w_{q+\half+q},\dots,\hat w_{\qbn-1}$.}\label{figure:pipes}
\end{figure}

The final piece of our dodging technology is needed to handle the material derivatives of nonlinear error terms such as
$$  -  \div^{-1} \mathbb{P}_{\neq 0} \div \left( \sfrac 12 |w_{q+\bn,\varphi}|^2 w_{q+\bn,\varphi} \right) \, . $$
Since $w_{\qbn,\varphi}$ is defined using intermittent Mikado bundles with active frequencies in between $\la_{q+\half}$ and $\la_{\qbn}$, we must divide this error term up according to the frequencies in this range.  A standard Littlewood-Paley decomposition $\sum_{q'=q+\half}^{\qbn} \mathbb{P}_{\approx \la_{q'}}$ would however destroy any useful spatial support properties of these error terms.  As a few simple heuristic computations show (similar to those made above for the Nash error), the material derivative $D_{t,{q'-1}}=\pa_q + \hat u_{q'-1}\cdot \nabla$ of the component of this error term at frequency $\la_{q'}$ would be too large unless the error is disjoint from the velocity increments $\hat w_{q+1}, \dots , \hat w_{q'-1}$. In order to ensure this disjointness, we replace the standard Littlewood-Paley decomposition with a \emph{synthetic Littlewood-Paley} decomposition $\sum_{q'=q+\half}^{\qbn} \tP_{\approx \lambda_{q'}}$.  This decomposition uses kernels that are compactly supported in space, so that the support of $\tP_{\approx \lambda_{q'}}$ is contained inside a ball of radius $\approx\lambda_{q'}^{-1}$ around the support of the input. Combined with the spatially localized inverse divergence developed in \cite{BMNV21} and \cite{NV22}, we therefore have that
$$ \supp \hat w_{k} \, \cap \, \supp \left[ \div^{-1} \tP_{\approx \lambda_{q'}}  \div \left( \sfrac 12 |w_{\qbn,\varphi}|^2 w_{\qbn,\varphi}  \right) \right] \qquad \textnormal{for $q+1\leq k \leq q'-1$} \, . $$
Then we can upgrade the natural material derivative estimate for this error term, corresponding to $D_{t,q}$, to $D_{t,q'-1}=D_{t,q} + \left( \hat w_{q+1} + \dots + \hat w_{q'-1} \right)\cdot \nabla$, \emph{for free}. It is simple to check that the synthetic Littlewood-Paley decomposition obeys essentially the same properties as a standard Littlewood-Paley decomposition with respect to integration (inverse divergence) and differentiation.  We have checked these simple computations in the companion paper \cite[section~4.3]{GKN23} and recall them here in the appendix.

\subsubsection{Intermittent pressure}\label{ss:ip:intro}

The primary purpose of our intermittent pressure $\pi_q^q$ is to ensure that $\pi_q^q\Id -R_q^q$ is positive-definite, identical to the role of the pressure decrement $p_{q}-p_{q+1}$ in previous $C^0$ Nash iterations for the Euler equations. In $C^0$ schemes, for example,  $\pi_q^q=2\norm{R_q^q}_{C^0}$ is typically chosen to be constant in space, as in either the Onsager theorem \cite{Isett2018, BDLSV17} or the partial results \cite{Is22, DK22} for the strong $C^0$ Onsager conjecture. In the intermittent iteration from \cite{NV22}, $\pi_q^q$ was chosen essentially as $2|R_q^q|$ and added into the system at step $q+1$. This strategy is acceptable because the pressure adjustment does not contribute to the new error $R_{q+1}$. However, it will not work for the relaxed local energy inequality; as we will explain, this is because we must \emph{invert the divergence} on a current error term containing $\pi_q^q$, which requires quite a bit of structure.  

To illustrate this point, notice that when we add $u_{q+1}-u_q = w_{q+1} =
\hat w_{\qbn}$ into \eqref{eqn:ER:intro:new}, one of the error terms in the relaxed local energy inequality will be 
$$  (\partial_t + u_q \cdot \nabla) \left(|w_{q+1}|^2 + \kappa_q^q \right) =  (\partial_t + \hat u_q \cdot \nabla) \left( |w_{q+1}|^2 + \kappa_q^q \right) \, , $$
where $\kappa_q^q=\sfrac 12 \tr (R_q^q - \pi_q^q\Id)$. Note that we have used \emph{dodging} to assert that $(u_q - \hat u_q)\cdot\nabla|w_{q+1}|^2=0$.
The presence of $\pi_q^q$ arises from its prior construction at {\it step $q$} and subsequent integration into the relaxed system. Consequently, $\ka_q^q$
cancels out the low frequency portion of $|w_{q+1}|^2$, leaving only a high-frequency remainder for which one can invert the divergence. We remark that without the presence of $\kappa_q^q$, $(\pa_q+\hat u_q\cdot\nabla)|w_{q+1}|^2$ is quite threatening since $w_{q+1}$ may be \emph{increasing} the energy along the trajectories of $\hat u_q$ in some localized region.  Referring to \eqref{pert:tuezday}, we see that this will indeed be the case in the regions where
\begin{equation}\label{p ㄹressure:error:intro}
(\pa_t + \hat u_q \cdot \nabla)|w_{q+1,R}|^2= (\pa_t + \hat u_q \cdot \nabla)\left(\left( r_q^{\sfrac 13}(-2\varphi_q^q)^{\sfrac 13} + (R_q^q-\pi_q^q\Id)^{\sfrac 12} \right)^2 a_i^2(\nabla \hat u_q) \right) |\BB_{i,R}|^2 \circ \Phi_{q,i} > 0 \, . 
\end{equation}
Due to the fact that $\Phi_{q,i}$ is only coherent on a timescale determined by the local Lipschitz norm of $\nabla \hat u_q$, we must ``turn off'' the $\BB_{i,R}$ bundle after some time has passed; in other words, $\Dtq(a_i^2)<0$ in some region of space.  But since the $a_i$'s form a partition of unity, that means that there exists some $i'$ such that $\Dtq(a_{i'}^2)>0$ in that region of space. Therefore this term cannot be expected to have the correct sign in general, and so $\kappa_q^q$ serves an essential role.

The intermittent pressure $\pi_q^q$ associated to $R_q^q$, however, is not solely generated at step at step $q$. Instead, its components have been progressively formed in preceding steps and carried over through inductive steps.
To understand the idea, let us consider the construction of a new intermittent pressure $\pi_{q+1}^{q+\bn}$, which guarantees the positive-definiteness of  $\pi_{q+1}^{q+\bn}\Id - R_{q+1}^{q+\bn}$. Following \cite{NV22}, a straightforward choice is $\pi_{q+1}^{q+\bn} =2|R_{q+1}^{q+\bn}|$. Incorporating this new intermittent pressure $\pi_{q+1}^{q+\bn}$ to the Euler-Reynold equations adds $-2|R_{q+1}^{q+\bn}|$ to a new pressure increment $p_{q+1}-p_q$, which subsequently generates a new error in the relaxed local energy inequality,
\begin{align*}
    -2(\pa_t + \hat{u}_q\cdot \na)|R_{q+1}^{q+\bn}|.
\end{align*}
We first remark that handling such an error term in $C^0$ schemes is relatively simple because the pressure increment can be chosen to be constant in space, as discussed ealier. In an intermittent scheme, however, this term is more complex due to the smearing of the frequency support of $R_{q+1}^{q+\bn}$. 
Indeed, it will contain quite a large range of active frequencies which contribute nontrivially to the amplitude of $|R_{q+1}^{q+\bn}|$. While we shall give a heuristic explanation of the active frequencies in $|R_{q+1}^{q+\bn}|$ below, we emphasize that any low frequency portions will likely have to be added in previous steps and carried inductively until they are needed. 

To this end, we analyze $|R_{q+1}^{q+\bn}|$. For simplicity, imagine that $R_{q+1}^{q+\bn}$ only contains the nonlinear error term $\div^{-1}\tP_{\approx\la_{q+\bn}} \div(w_{q+1} \otimes  w_{q+1})$,
and set $ w_{q+1} = (R_{q}^{q})^{\sfrac 12}\BB_{q+\bn}$ where $\mathbb{B}_{q+\bn}$ is a scalar-valued, $L^2$-normalized intermittent function with maximum frequency $\la_{q+\bn}$.  With these simplifications, we have
\begin{equation}\label{simpli:compu}
 |R_{q+1}^{q+\bn}| \approx \big{|} \underbrace{\nabla R_{q}^{q}}_{\textnormal{freq. $\la_{q}$}} \underbrace{\div^{-1}\tP_{\approx\la_{q+\bn}} \mathbb{B}_{q+\bn}^2}_{\textnormal{freq. $\la_{q+\bn}$}} \big{|} =
 \underbrace{\left| \nabla R_{q}^{q} \right| \tP_{=0} \left|\div^{-1}\tP_{\approx\la_{q+\bn}} \mathbb{B}_{q+\bn}^2 \right|}_{=\tP_{\les \la_{q}}|R_{q+1}^{q+\bn}|} + \underbrace{\left| \nabla R_{q}^{q} \right| \tP_{\neq 0} \left|\div^{-1}\tP_{\approx\la_{q+\bn}} \mathbb{B}_{q+\bn}^2 \right|}_{=\tP_{ [\la_{q+\half}, \la_{q+\bn}]}|R_{q+1}^{q+\bn}|}\, .
\end{equation}
As we see, the leading order portion of $R_{q+1}^{q+\bn}$ is composed of both low and high frequency components, but taking the magnitude of the high frequency part generates a low frequency component. Decomposing the magnitude further, we notice that the last term of \eqref{simpli:compu} exhibits a similar structure with the stress errors arising at step $q+1$, and its material derivative can be absorbed to new current errors.
The second-to-last term of \eqref{simpli:compu}, on the other hand, is essentially a rescaled version of $|R_{q}^{q}|$, which is 
therefore \emph{inductive} in nature. 
Repeating the analysis above under the simplification $R_{q'}^{q'} \approx \na R_{q'-\bn}^{q'-\bn} \div^{-1}\tP_{\approx \la_{q'}}\BB_{q'}^2$, we obtain a wavelet decomposition of $|R_{q+1}^{q+\bn}|$ as 
\begin{align*}
    |R_{q+1}^{q+\bn}|
    \approx
    \tP_{ [\la_{q+\half}, \la_{q+\bn}]}|R_{q+1}^{q+\bn}|
    + \frac{\la_{q}}{\la_{q+\bn}} 
    \tP_{ [\la_{q-\frac{\bn}2},\la_{q}]}|R_{q}^{q}| 
    +  \frac{\la_{q-\bn}}{\la_q} 
    \tP_{ [\la_{q-\frac32\bn},\la_{q-\bn]}}|R_{q-\bn}^{q-\bn}| + \cdots \, ,
    \end{align*}
where we terminate the decomposition after finitely many steps, at which the last term of the sum is no longer intermittent and bounded uniformly in space. 
Since the term $\tP_{ [\la_{q+\half}, \la_{q+\bn}]}|R_{q+1}^{q+\bn}|$ is handled at step $q+1$, the pressure $\tP_{ [\la_{q'-\half}, \la_{q'}]}|R_{q'}^{q'}|$ must be added at the same step that $R_{q'}^{q'}$ is created. We then summarize that when we produce a stress error of the form $R_{q'}^{q'}$, we not only add the pressure $|\na R_{q'-\bn}^{q'-\bn}|\tP_{\neq 0}|\div^{-1} \tP_{\approx \la_{q'}}\BB_{q'}^2|$ at the same step of iteration, but also add rescaled\footnote{In practice, due to the presence of other errors in $R_q^q$, the rescaling factor will be raised to the power of $2\be$; see~\eqref{eq:ind.pr.anticipated}.} copies that appear in wavelet decomposition of future stress errors. We refer to these rescaled pressure terms needed for future wavelet decompositions as \emph{anticipated pressure}; see the discussion in subsection~\ref{sec:cutoff:stress}. In general, the collection of all pieces of pressures that are constructed for $R_{q+1}^k$ as above at multiple steps
forms a \emph{intermittent pressure} $\pi_{q+1}^k$.

Finally, we note that at a technical level, we will construct intermittent pressure $\pi_q^k$ associated to $R_q^k$ so that $\pi_q^k \la_k^N$ dominates $|\nabla_x^N R_q^k|$ for $N\rightarrow \infty$ as $\beta\rightarrow \sfrac 13$.  In fact, we will construct $\pi_q^k$ to dominate material derivatives of $R_q^k$, spatial and material derivatives of $\varphi_q^k$ and $ \nabla \hat u_k$, and even spatial and material derivatives of \emph{itself}.  We refer to these bounds as \emph{pointwise bounds}, and the net effect is that in fact almost every estimate in this manuscript can be reduced to a pointwise bound in terms of an intermittent pressure $\pi_q^{k}$, save for the $L^{\sfrac 32}$ and $L^\infty$ norms of $\pi_q^{k}$ itself.  The fact that the pressure plays such a fundamental role in every estimate of this paper contrasts very sharply with every convex integration scheme for fluid equations to date, in which the pressure plays a more passive role.

\subsection{Organization}

Lastly, we provide a brief summary of the contents covered in the subsequent sections. Section \ref{sec:ind} lists the inductive hypotheses, introduces the inductive statement in Proposition~\ref{prop:main}, and presents the proof of Theorem \ref{thm:main} assuming Proposition~\ref{prop:main}. 

The proof of Proposition~\ref{prop:main} is then contained in sections~\ref{stuff:from:part:1}--\ref{sec:intermittent:pressure}.  In particular, section~\ref{stuff:from:part:1} first sets up the inductive proof by defining the intermittent Mikado bundles and cutoff functions needed for the definition of the velocity increment, before defining and estimating the velocity increment and its potentials.  In addition, section~\ref{stuff:from:part:1} defines and estimates most of the new Reynolds stress errors, the transport and current Nash errors (which are closely related to the new Reynolds stress errors), the associated pressure increments, and the mollification errors.  The proofs of the assertions in this section are contained in \cite{GKN23}, as they bear some similarities to the constructions in \cite{BMNV21} and \cite{NV22}.  Next, section~\ref{sec:dodging} introduces the new dodging techniques for intermittent Mikado bundles, which form a crucial component of the wavelet-based scheme. Section~\ref{sec:RLE:errors} then considers the main new components of the current error generated by adding the new velocity increment $w_{q+1}$ to the relaxed local energy inequality. We define and estimate both the current errors and associated pressure increments in this section.  The final substantive portion of the proof of Proposition~\ref{prop:main} is contained in section~\ref{sec:intermittent:pressure}.  Here we gather all pressure increments constructed in the previous sections and define a new pressure increment and a new intermittent pressure, which are shown to satisfy the proper inductive assumptions.  We then finalize the definition of all stress and current errors to conclude the proof of the inductive proposition.

Finally, section~\ref{sec:params} specifies parameter choices and delineates useful inequalities resulting from these choices.  Then in Appendix~\ref{sec:app}, we collect the abstract technical components needed throughout the paper.

\subsection*{Acknowledgements} 
The authors acknowledge the hospitality and working environment at the Institute for Advanced Study during the special year on the h-principle, when they first started working on this project. VG was supported by the NSF under Grants DMS-FRG-1854344 and DMS-1946175 while at Princeton University. HK was supported by the NSF under Grant DMS-1926686 while a member at the IAS and would like to extend gratitude to the employer, the Forschungsinstitut f\"ur Mathematik (FIM) at ETH Z\"urich. MN was supported by the NSF under Grant DMS-1926686 while a member at the IAS.  The authors thank Camillo De Lellis and Vlad Vicol for their commentary on drafts of this manuscript.

\section{Inductive proposition and proof of the main theorem}\label{sec:ind}
The main goal of this section is to state the inductive proposition and prove the main theorem; this is the content of subsection~\ref{ss:friday}.  However, we must first introduce a number of notations in subsection~\ref{sec:not.general}, before listing the inductive assumptions in subsections~\ref{ss:relaxed}--\ref{sec:inductive:secondary:velocity}. Specifically, subsection~\ref{ss:relaxed} introduces the relaxed equations, subsection~\ref{sec:cutoff:inductive} introduces the inductive cutoff functions, subsection~\ref{sec:pi:inductive} introduces the intermittent pressure, subsection~\ref{ss:dodging} introduces the inductive ``dodging'' hypotheses, and subsection~\ref{sec:inductive:secondary:velocity} introduces the inductive velocity bounds.  Throughout subsections~\ref{ss:relaxed}--\ref{sec:inductive:secondary:velocity}, we identify precisely where in the manuscript each inductive assumption is proven.  We shall assume that all inductive assumptions at the $q^{\textnormal{th}}$ step hold on the domain $[-\tau_{q-1},T+\tau_{q-1}]\times \T^3$.

\subsection{General notations and parameters}\label{sec:not.general}

We start by introducing the primary parameters 
$$ \beta\, , \bn\, , b\, , \la_q\, , \delta_q\, , r_q \, , \Ga_q  \, , \varepsilon_\Gamma $$
which appear in the inductive hypotheses. First, we choose an $L^3$-based regularity index $\beta\in[\sfrac17,\sfrac 13)$. Next, we fix $\bn\in 6\N$ (used to describe active frequencies in certain inductive objects; see \eqref{eqn:ER:intro:new} and the subsequent discussion) such that\index{$\beta$} \index{$\bn$}
\begin{equation}
        \beta < \frac 13 \cdot \frac{\sfrac \bn 3}{\sfrac \bn 3 + 2} - \frac{2}{\sfrac \bn 3 +2} \, , \qquad \beta< \frac {2}{3} \cdot \frac{\sfrac{\bn}{2}-1}{\bn} \, . \label{eq:choice:of:bn}
    \end{equation}
Next, we choose $b\in (1,\sfrac{25}{24})$ close to $1$ such that
 \begin{subequations}\label{ineq:b}
    \begin{align}
        \beta < \frac{1}{3b^{\bn}} \cdot \frac{1+b+\dots+b^{\sfrac \bn 3 -1}}{1+b+\dots+b^{\sfrac \bn 3 +1}}& - \frac{2\left(1+(b-1)(1+\cdots+b^{\sfrac \bn 2 -1})^2\right)}{1+b+\dots+b^{\sfrac \bn 3 +1}} \, , \quad  \frac{2}{3b^{\sfrac{\bn}{2}}} \cdot \frac{1+\dots+b^{\sfrac \bn 2 - 2}}{1+\dots+b^{\bn-1}} \label{ineq:b:first} \\
        b^{\bn} < 2& \, , \qquad \frac{(b^{\sfrac \bn 2 - 1}+\dots+b+1)^2}{b^{\sfrac \bn 2 -1}+\dots+b+1} (b-1) < (b-1)^{\sfrac 12} \, . \label{ineq:b:second}
    \end{align}
    \end{subequations}
We now define the frequency parameter $\la_q$, the amplitude parameter $\de_q$, the intermittency parameter $r_q$, and the multi-purpose parameter $\Gamma_q$ by \index{$\la_q$} \index{$\de_q$} \index{$r_q$} \index{$\Ga_q$}
\begin{align}
    &\la_q = 2^{\left \lceil (b^q) \log_2 a \right \rceil}\approx a^{(b^q)} \, , \qquad  \de_q = \la_q^{-2\beta} \, ,
    \label{eq:def:la:de}\\
    r_q = \frac{\la_{q+\half}\Ga_q}{\la_\qbn}&\, , \qquad
    \Ga_q = 2^{\left \lceil \varepsilon_\Gamma \log_2 \left( \frac{\la_{q+1}}{\la_q} \right) \right \rceil}  
    \approx \left( \frac{\la_{q+1}}{\la_q} \right)^{\varepsilon_\Gamma} \approx \la_{q}^{(b-1)\varepsilon_\Gamma} \, .   \label{eq:deffy:of:gamma}
    \end{align}
The large positive integer $a$ and the small positive number $0<\varepsilon_\Gamma\ll (b-1)^2<1$ are defined in \eqref{i:choice:of:a} and \eqref{i:choice:ep} of subsection \ref{sec:para.q.ind}, respectively.  Note that the intermittency parameter $r_q$ is determined by the ``$\sfrac 12$ rule'' as in \cite{NV22}.

We now introduce further parameters
$$ \tau_q \, ,  \Lambda_q\, , \Tau_q\, , \badshaq\, . $$
As described in the introduction, we often decompose $u_q=\hat u_q + (u_q-\hat u_q)$.  One may imagine that the gradient of velocity $\nabla\hat u_{q'}$ will have spatial derivative cost $\approx\la_{q'}$ and $L^3$ norm $ \approx \tau_{q'}^{-1}\approx \de_{q'}^{\sfrac12} r_{q'-\bn}^{-\sfrac13}\la_{q'}$.  It will in fact be necessary to slightly adjust the definition of $\tau_q^{-1}$ using the parameter $\La_q$ (slightly larger than $\la_q$), which accounts for small spatial frequency losses due to mollification, and introduce a parameter $\Tau_q^{-1}$ (much larger than $\tau_q^{-1})$, which accounts for temporal frequency losses due to mollification.  We set
\begin{align}\label{eq:defn:tau}
\la_q < \La_q = {\la_q \Ga_q^{10} } \, , \qquad 
\tau_q^{-1} &= \delta_q^{\sfrac 12} \la_q r_{q-\bn}^{-\sfrac 13} \Ga_q^{35} \ll \Tau_q^{-1} \, ,
\end{align}
and refer to \eqref{v:global:par:ineq} for the precise definition of $\Tau_q$. For the $L^\infty$ norm of $R_q^q$ (and other inductive objects), we use the parameter $\badshaq$, which will satisfy (we refer to \eqref{eq:badshaq:choice} for the precise choice of $\badshaq$)
 \begin{align*}
    \la_q^{\frac 1{\bn}}\les\Ga_q^{\badshaq} \les \la_q^{\frac{12}{\bn}}\, . 
\end{align*}
 
Finally, we will inductively propagate spatial and material derivative estimates, where we use the notation and paramaters
\begin{align*}
    D_{t,q} = \pa_t  + (\hat u_q\cdot \na)\, , \qquad \NcutSmall\, , \Nindt\, , \Nind\, , \Nfin \, .
\end{align*}
The integers $\mathsf{N}_{\bullet}$ above quantify the number of spatial and material derivative estimates propagated inductively and satisfy the ordering (see subsection~\ref{sec:para.q.ind} for the precise choices)
\begin{align*}
1\ll \NcutSmall\ll \Nindt \ll  \Nind \ll \Nfin\, . 
\end{align*}
In particular, $\Nindt$ helps us keep track of both sharp and lossy material derivative estimates.  For this purpose, we use the following notation, which roughly says that ``the first $N_*$ material derivatives cost $\tau^{-1}$, while additional derivatives cost $\Tau^{-1}$.'' We also list a few other notations in the subsequent two remarks.\index{$\MM{n,N_*,\tau^{-1},\Tau^{-1}}$}
\begin{remark}[\bf Geometric upper bounds with two bases]\label{not:M}
    For all $n\geq 0$, we define
    \[\MM{n,N_*,\tau^{-1},\Tau^{-1}} := \tau^{-\min\{n,N_*\}} \Tau^{-\max\{n-N_*,0\}}\,.\]
\end{remark}

\begin{remark}[\bf Space-time norms]\label{rem:notation:space:time}
In the remainder of the paper, we shall always measure objects using uniform-in-time norms $\sup_{t\in[T_1,T_2]}\|\cdot (t)\|$, where $\| \cdot (t)\|$ is any of a variety of norms used to measure functions defined on $\T^3\times[T_1,T_2]$ but restricted to time $t$. In a slight abuse of notation, we shall always abbreviate these space-time norms with simply $\| \cdot \|$.
\end{remark}

\begin{remark}[\bf Space-time balls]\label{rem:notation:space:time:balls}
For any set $\Omega\subseteq \mathbb{T}^3\times \mathbb{R}$, we shall use the notations
\begin{subequations}\label{eq:space:time:balls}
\begin{align}
B(\Omega,\lambda^{-1}) &:= \left\{ (x,t) \, : \, \exists \, (x_0,t) \in \Omega \textnormal{ with } |x-x_0| \leq \lambda^{-1} \right\}\\
B(\Omega,\lambda^{-1},\tau) &:= \left\{ (x,t) \, : \, \exists \, (x_0,t_0) \in \Omega \textnormal{ with } |x-x_0| \leq \lambda^{-1} \, , |t-t_0| \leq \tau \right\}
\end{align}
\end{subequations}
for space and space-time neighborhoods of $\Omega$ of radius $\lambda^{-1}$ in space and $\tau$ in time, respectively. 
\end{remark}

\subsection{Relaxed equations}\label{ss:relaxed}

We assume that there exists a given $q$-independent continuous function $E=E(t,x)\geq 0$ such that the approximate solution $(u_q, p_q, R_q, \ph_q, -\pi_q)$ at the $q^{\textnormal{th}}$ step satisfies the Euler-Reynolds system\index{Euler-Reynolds system}
\begin{align}\label{eqn:ER}
\begin{cases}
\partial_t u_q + \div (u_q \otimes u_q) + {\nabla p_q} = \div(R_q -\pi_q \Id) \\
\div\, u_q = 0
\end{cases}
\end{align}
and the relaxed local energy identity\index{relaxed local energy identity} with dissipation measure $E$\footnote{{The dissipation measure $E$ is the Duchon-Robert measure of the limiting solution.}}
\begin{align}\label{ineq:relaxed.LEI}
\pa_t \left( \frac 12 |u_q|^2 \right)
    + \div\left( \left(\frac 12 |u_q|^2 + p_q\right) u_q\right)
    = (\pa_t + \hat u_q \cdot \na ) \ka_q  
    + \div((R_q-\pi_q\Id) \hat u_q) + \div \ph_q - E\, .
\end{align}
In the above equation, we have set $\ka_q = \sfrac{\tr(R_q -\pi_q \Id)}2$, and we use the decomposition and notations\index{$\hat u_q$}\index{$\hat w_q$}
\begin{equation}\label{eq:hat:no:hat}
u_q = \underbrace{\hat u_{q-1} + \hat w_q}_{=: \hat u_q} + \hat w_{q+1} + \dots + \hat w_{q+\bn-1} =: \hat u_{q+\bn-1} \, 
\end{equation}
for the velocity field.  The stress error $R_q$ has a decomposition
\begin{align}
R_q &= \sum_{k=q}^{q+\bn-1} R_q^k \, . \label{eq:ER:decomp:basic} 
\end{align}\index{$R_q$}
where each $R_q^k$ is a symmetric matrix. The pressure $\pi_q$ has a decomposition
\begin{align}
\pi_q &= \sum_{k=q}^{{\infty}} \pi_q^k \, . \label{eq:pi:decomp:basic}
\end{align}\index{$\pi_q$}
Similarly, the current error $\varphi_q$ has a decomposition
\begin{align}
\ph_q &= \sum_{k=q}^{q+\bn-1} \ph_q^k \,. \label{eq:LEI:decomp:basic} 
\end{align}\index{$\ph_q$}

The Reynolds stress $R_q$ and current error $\varphi_q$ defined above will have frequency support in modes no larger than $\lambda_{q+\bn-1}$ (effectively speaking). We correct the errors which live at frequencies no higher than $\lambda_q$, denoted by $R^q_q$ and $\varphi_q^q$, respectively. More generally, we denote the portions of $R_q$ and $\varphi_q$ with spatial derivative cost $\lambda_{k}$ by $R_q^k$ and $\varphi_q^k$, respectively.  We verify the inductive assumptions in this section in Lemma~\ref{lem:relaxed:new}.

\subsection{Inductive assumptions for velocity cutoff functions}
\label{sec:cutoff:inductive}
Given the intermittent nature of the velocity vector field $\hat u_{q'}$, the cost of its associated material derivative $D_{t,q'}$ can vary significantly across different level sets of the gradient of velocity. To address this issue, we introduce velocity cutoff functions $\psi_{i,q'}$, which are defined inductively. The cutoffs $\psi_{i,q'}$\index{$\psi_{i,q}$} partition the domain into distinct level sets of $\nabla \hat u_{q'}$, where $|\nabla \hat u_{q'}|\approx \tau_{q'}^{-1} \Ga_{q'}^{i}$ on $\supp \psi_{i,q'}$. We first record here the key properties, including a partition of unity property, maximum value $\imax$ needed to index the cutoffs, derivative estimates, and Lebesgue norms. The local $L^\infty$ estimates for velocity increment $\hat w_{q'}$ and velocity $\hat u_{q'}$, obtained as a consequence of the definition of $\psi_{i,q'}$, can be found in subsection \ref{sec:inductive:secondary:velocity}. The concrete construction of $\psi_{i,q+\bn}$ of \eqref{eq:inductive:partition}--\eqref{eq:inductive:timescales} for $q\mapsto q+1$ (i.e., $q'=q+\bn$) is contained in \cite[section~9]{GKN23}, and the verification of the inductive assumptions here is contained as well in \cite[section~9]{GKN23} and recalled in Proposition~\ref{prop:verified:vel:cutoff}.  Finally, we note that all assumptions in subsection~\ref{sec:cutoff:inductive} are assumed to hold for all $q-1\leq q'\leq q+\bn-1$.  

First, we assume that the velocity cutoff functions\index{velocity cutoffs} form a partition of unity:\index{$i$ and $\imax$}
\begin{align}\label{eq:inductive:partition}
    \sum_{i\geq 0} \psi_{i,q'}^6 \equiv 1, \qquad \mbox{and} \qquad \psi_{i,q'}\psi_{i',q'}=0 \quad \textnormal{for}\quad|i-i'| \geq 2 \, .
\end{align}
Second, we assume that there exists an $\imax = \imax(q') \geq 0$, bounded uniformly in $q'$, such that
\begin{align}
\imax(q') &\leq \frac{\badshaq+12}{(b-1)\varepsilon_\Ga} \, ,
\label{eq:imax:upper:lower} \\
\psi_{i,q'} \equiv 0 \quad \mbox{for all} \quad i > \imax(q')\,,
\qquad &\mbox{and} \qquad
\Gamma_{q'}^{\imax(q')} \leq 
\Ga_{q'-\bn}^{\sfrac{\badshaq}{2}+18}
 \delta_{q'}^{-\sfrac 12}r_{q'-\bn}^{-\sfrac {2}3} \, .
\label{eq:imax:old}
\end{align}
For all $0 \leq i \leq \imax$, we assume the following pointwise derivative bounds for the cutoff functions $\psi_{i,q'}$. First, for mixed space and material derivatives and multi-indices $\alpha,\beta \in {\mathbb N}^k$, $k \geq 0$, $0 \leq |\alpha| + |\beta| \leq \Nfin$, we assume that
\begin{align}
&\frac{{\bf 1}_{\supp \psi_{i,q'}}}{\psi_{i,q'}^{1- (K+M)/\Nfin}} \left|\left(\prod_{l=1}^k D^{\alpha_l} D_{t,q'-1}^{\beta_l}\right) \psi_{i,q'}\right| \leq \Gamma_{q'} (\Gamma_{q'}  \lambda_{q'})^{|\alpha|} 
\MM{|\beta|,\NindSmall - \NcutSmall,  \Gamma_{q'-1}^{i+3}  \tau_{q'-1}^{-1}, \Gamma_{q'-1} \Tau_{q'-1}^{-1}} \, .
\label{eq:sharp:Dt:psi:i:q:old}
\end{align}
Next, with $\alpha, \beta,k$ as above, $N\geq 0$ and $D_{q'}:=\hat w_{q'}\cdot\nabla$, we assume that
\begin{align}
&\frac{{\bf 1}_{\supp \psi_{i,q'}}}{\psi_{i,q'}^{1- (N+K+M)/\Nfin}} \left| D^N \left( \prod_{l=1}^k D_{q'}^{\alpha_l} D_{t,q'-1}^{\beta_l}\right)  \psi_{i,q'} \right| \notag\\
&\qquad \leq \Gamma_{q'} ( \Gamma_{q'}  \lambda_{q'})^N
(\Gamma_{q'}^{i-5} \tau_{q'}^{-1})^{|\alpha|}
\MM{|\beta|,\Nindt-\NcutSmall,  \Gamma_{q'-1}^{i+3}  \tau_{q'-1}^{-1}, \Gamma_{q'-1}  \Tau_{q'-1}^{-1}}
\label{eq:sharp:Dt:psi:i:q:mixed:old}
\end{align}
for $0 \leq N+ |\alpha| + |\beta| \leq \Nfin$. Moreover, for $0\leq i \leq \imax(q')$, we assume the $L^1$ bound\index{$\CLebesgue$}
\begin{align}
\norm{\psi_{i,q'}}_{1}  \leq \Gamma_{q'}^{-3i+\CLebesgue} \qquad \mbox{where} \qquad \CLebesgue = \frac{6+b}{b-1} \, .
\label{eq:psi:i:q:support:old}
\end{align}
Lastly, we assume that local timescales dictated by velocity cutoffs at a fixed point in space-time are decreasing in $q$. More precisely, for all $q' \leq q+\bn-1$ and all $q''\leq q'-1$, we assume 
\begin{equation}
    \psi_{i',q'}  \psi_{i'',q''} \not \equiv 0 \quad \implies \quad 
    \tau_{q'} \Gamma_{q'}^{-i'} \leq \tau_{q''} \Gamma_{q''}^{-i'' -25}   \, . \label{eq:inductive:timescales}
\end{equation}
This will be useful when we upgrade material derivatives from $D_{t,q''}$ to $D_{t,q'}$.

\subsection{Inductive bounds on the intermittent pressure}\label{sec:pi:inductive}

The intermittent pressure $\pi_q$ is designed to give pointwise control on errors and velocity increments. Towards this end, we shall use the phrase ``pointwise estimates''\index{pointwise estimates} to refer to such bounds on stress errors, current errors, or velocities in terms of various $\pi$'s. We introduce estimates for $\pi_q$ in subsection~\ref{sec:pik:inductive} and establish precise relations between the intermittent pressure and errors/velocity increments in subsection~\ref{sec:pik:maj}. (The $L^p$ estimates of the errors will follow consequently.) Furthermore, as we discussed in the introduction, the intermittent pressure is constructed to anticipate the low-frequency parts of future pressure increments. We record the relevant properties in subsection~\ref{sec:pik:ant}. All inductive assumptions appearing in subsection~\ref{sec:pi:inductive} will be verified for $q\mapsto q+1$ in Section~\ref{sec:intermittent:pressure}. 

\subsubsection{\texorpdfstring{$L^{\sfrac 32}$}{tpdfs1}, \texorpdfstring{$L^\infty$}{tpdfs2}, and pointwise bounds for \texorpdfstring{$\pi_q^k$}{tpdfs3}} \label{sec:pik:inductive}

We assume that for $q\leq k \leq q+\bn-1$ and $N+M\leq 2\Nind$, $\pi_q^k$ satisfies 
\begin{subequations}\label{eq:pressure:inductive}
\begin{align}
\norm{ \psi_{i,k-1} D^N D_{t,k-1}^M  \pi_q^k }_{\sfrac 32}  
 &\leq \Ga_q\Ga_k \delta_{k+\bn} \Lambda_k^N \MM{M, \NindRt, \Gamma_{k-1}^{i} \tau_{k-1}^{-1} ,  \Tau_{k-1}^{-1} } \, .
\label{eq:pressure:inductive:dtq} \\
\norm{ \psi_{i,k-1} D^N D_{t,k-1}^M \pi_q^k }_{\infty}  
 &\leq \Ga_q {\Gamma_k^{\badshaq+1}} \Lambda_k^N \MM{M, \NindRt, \Gamma_{k-1}^{i} \tau_{k-1}^{-1} ,  \Tau_{k-1}^{-1} } \, , \label{eq:pressure:inductive:dtq:uniform} \\
 \label{eq:ind:pi:by:pi}
    \left|\psi_{i,k-1} D^N D_{t,k-1}^M \pi_q^{k}\right| &\leq \Ga_q {\Gamma_k}\pi_q^k  \Lambda_k^N \MM{M, \NindRt, \Gamma_{k-1}^{i} \tau_{k-1}^{-1} , \Tau_{k-1}^{-1} } \, .
\end{align}
\end{subequations}
For $q+\bn\leq k \leq q+\Npr-1$ and $N+M\leq 2\Nind$, we assume that $\pi_q^k$ satisfies
\begin{subequations}\label{eq:pressure:inductive:largek}
\begin{align}
\norm{ \psi_{i,q{+\bn-1}} D^N D_{t,q{+\bn-1}}^M  \pi_q^k }_{\sfrac 32}  
 &\leq \Ga_q \Ga_{{k}} \delta_{k+\bn} \Lambda_{q+\bn-1}^N \MM{M, \NindRt, \Gamma_{{q+\bn-1}}^{i} \tau_{{q+\bn-1}}^{-1} ,  \Tau_{q+\bn-1}^{-1} }
\label{eq:pressure:inductive:dtq:largek} \\
\norm{ \psi_{i,q{+\bn-1}} D^N D_{t,q{+\bn-1}}^M \pi_q^k }_{\infty}  
 &\leq \Ga_q \Gamma_{{q+\bn-1}}^{\badshaq+1} \Lambda_{q+\bn-1}^N \MM{M, \NindRt, \Gamma_{q+\bn-1}^{i} \tau_{q+\bn-1}^{-1} ,  \Tau_{q+\bn-1}^{-1} } \, , \label{eq:pressure:inductive:dtq:uniform:largek} \\
 \label{eq:ind:pi:by:pi:largek}
    \left|\psi_{i,q+\bn-1} D^N D_{t,q+\bn-1}^M \pi_q^{k}\right| &\leq {\Gamma_q}\pi_q^k  \Lambda_{q+\bn-1}^N \MM{M, \NindRt, \Gamma_{q+\bn-1}^{i} \tau_{q+\bn-1}^{-1} ,  \Tau_{q+\bn-1}^{-1} } \, .
\end{align}
\end{subequations}
The bounds in \eqref{eq:pressure:inductive} and \eqref{eq:pressure:inductive:largek} are verfied in Lemma~\ref{lem:verify.ind.pressure1}.

\subsubsection{Lower and upper bounds for \texorpdfstring{$\pi_q^k$}{tpdfs4}}\label{sec:pik:ant}
For $ k\geq q$, we assume that $\pi_q^{k}$ has the lower bound
\begin{align}\label{low.bdd.pi}
    \pi_q^{k} \geq \de_{k+\bn} \, .
\end{align} 
For all $q+\bn-1 \leq k' < k \leq q+\Npr-1$ (see \eqref{defn:Npr} for the definition of $\Npr$\index{$\Npr$}), we assume that $\pi_q^k$ has the upper bound
\begin{align}\label{ind:pi:upper}
    \pi_q^k \leq \pi_q^{k'} \, .
\end{align}
For all $k\geq q+\Npr$, we assume that
\begin{equation}\label{defn:pikq.large.k}
    \pi_q^k \equiv \Ga_k \de_{k+\bn} \, .
\end{equation}
We finally assume that for all $q\leq q' < q''<\infty$,
\begin{subequations}\label{eq:ind.pr.anticipated}
\begin{align}
    \frac{\de_{q''+\bn}}{\de_{q'+\bn}} \pi_q^{q'} &< 2^{q'-q''}\pi_q^{q''} \,, \qquad \text{if}\, q+\half \leq q'' \label{eq:ind.pr.anticipated.1}\\
    \frac{\de_{q''+\bn}}{\de_{q'+\bn}} \pi_q^{q'} &< \pi_q^{q''} \,,\qquad \text{otherwise} \, . \label{eq:ind.pr.anticipated.2}
\end{align}
\end{subequations}
This final bound says that the $\pi_q^k$'s obey a scaling law which may be roughly translated as ``any $\pi_q^{k+m}$ for $m>0$ can be bounded from below by an appropriately rescaled $\pi_q^k$.''  All the bounds in this subsubsection are verified in Lemma~\ref{lem:lower:upper}.

\subsubsection{Pointwise bounds for errors, velocities, and velocity cutoffs}\label{sec:pik:maj}
We assume that we have the pointwise estimates
\begin{subequations}\label{eq:inductive:pointwise}
\begin{align}
    \label{eq:ind:stress:by:pi}
    \left|\psi_{i,k-1} D^N D_{t,k-1}^M R^k_q\right| &< \Ga_q \Ga_k^{-8} \pi^k_q \Lambda_k^N \MM{M, \NindRt, \Gamma_{k-1}^{i  +20} \tau_{k-1}^{-1} ,  \Tau_{k-1}^{-1}\Ga_{k-1}^{10} } \, ,\\
    \label{eq:ind:current:by:pi}
    \left|\psi_{i,k-1} D^N D_{t,k-1}^M 
    {\ph_q^k}\right|
    &< \Ga_q \Ga_k^{-12}(\pi^k_q)^\frac32 r_k^{-1} \Lambda_k^N \MM{M, \NindRt, \Gamma_{k-1}^{i  +20} \tau_{k-1}^{-1} ,  \Tau_{k-1}^{-1} \Ga_{k-1}^{10} } \, , \\
    \label{eq:ind:velocity:by:pi}
    \left|\psi_{i,k-1} D^N D_{t,k-1}^M \hat w_k \right| &< \Ga_q r_{k-\bn}^{-1} (\pi_q^k)^{\sfrac12} {\Lambda}_k^{N} \MM{M, \NindRt, \Gamma_{k-1}^{i} \tau_{k-1}^{-1} ,  \Tau_{k-1}^{-1} {\Ga_{k-1}^2}}
\end{align}
\end{subequations}
for $q\leq k\leq q+\bn-1$, where the first bound holds for $N+M \leq 2\Nind$, the second bound holds for $N+M\leq \sfrac{\Nind}{4} $, and the third bound {holds for $N+M\leq \sfrac{3\Nfin}{2}$}. The first and second bounds above are verified in Lemma~\ref{lem:verify.ind.pressure2}, where we also verify the bound given below in \eqref{eq:Rnnl:inductive:dtq}. The proof of \eqref{eq:ind:velocity:by:pi} is given in Lemma~\ref{prop:velocity:domination}, where we also verify the bound given below in \eqref{eq:psi:q:q'}.

\begin{remark}[\bf $L^p$ estimates on Reynolds errors from pointwise estimates]\label{sec:stress:inductive}
The estimates on $R^k_q$ in \eqref{eq:ind:stress:by:pi} and the estimates on $\pi_q^k$ in \eqref{eq:pressure:inductive} imply that for $q\leq k \leq q+\bn-1$ and $N+M\leq 2\Nind$, $R_q^k$ satisfies
\begin{subequations}\label{eq:Rn:inductive}
\begin{align}
\norm{ \psi_{i,k-1} D^N D_{t,k-1}^M  R_q^k }_{\sfrac 32}  
 &\leq \Ga_q^2 \Ga_k^{-7} \delta_{k+\bn} \Lambda_k^N \MM{M, \NindRt, \Gamma_{k-1}^{i+20} \tau_{k-1}^{-1} ,  \Tau_{k-1}^{-1}\Gamma_q^{10} } \, ,
\label{eq:Rn:inductive:dtq} \\
\norm{ \psi_{i,k-1} D^N D_{t,k-1}^M R_q^k }_{\infty}  
 &\leq \Ga_q^2 \Ga_k^{-7} {\Gamma_k^{\badshaq}} \Lambda_k^N \MM{M, \NindRt, \Gamma_{k-1}^{i+20} \tau_{k-1}^{-1} ,  \Tau_{k-1}^{-1}\Gamma_q^{10} }\, . \label{eq:Rn:inductive:dtq:uniform}
\end{align}
\end{subequations}
\end{remark}

While the main $L^p$ estimates on the Reynolds stress follow from the pointwise estimates in terms of the pressure, we are forced to assume that $R_q^k$ has a decomposition $R_q^k = R_q^{k,l}+R_q^{k,*}$, where $R_q^{k,*}$ satisfies the stronger bound
\begin{align}
\norm{ D^N D_{t,k-1}^M  R_q^{k,*} }_{\infty}
\leq \Ga_q^2 \Tau_k^{2\Nindt}\delta_{k+2\bn} \Lambda_k^N \MM{M,\Nindt, \tau_{k-1}^{-1},\Tau_{k-1}^{-1}}
\label{eq:Rnnl:inductive:dtq} 
\end{align}
for all $N+M \leq 2\Nind$. The extra superscript $l$ stands for ``local,'' in the sense that $R_q^{k,l}$ is a stress error over which we maintain control of the spatial support (see Hypothesis~\ref{hyp:dodging4}), whereas $\ast$ refers to non-local terms which are negligibly small.  The reader can safely ignore such non-local error terms. 

Finally, we assume that for all $q\leq q'\leq q+\bn-1$,
\begin{align}\label{eq:psi:q:q'}
    \sum_{i=0}^{\imax} \psi_{i,q'}^2 \delta_{q'} r_{q'-\bn}^{-\sfrac23} \Gamma_{q'}^{2i} &\leq {2^{q-q'}} \Ga_{q'} {r_{q'-\bn}^{-2}} \pi_{q}^{q'} \, .
\end{align}
Using the heuristic that $|\nabla \hat u_{q'}| \mathbf{1}_{\supp \psi_{i,q'}} \approx \tau_{q'}^{-1} \Ga_{q'}^i$ and \eqref{eq:defn:tau}, this inequality says that $|\nabla \hat u_{q'}|^2 \leq r_{q'-\bn}^{-2} \la_{q'}^2 \pi_q^{q'}$.

\subsection{Dodging inductive hypotheses}\label{ss:dodging}
In this subsection, we list ``dodging''\index{dodging} inductive hypotheses, which encode the spatial support information which is crucial to our wavelet-inspired scheme. As discussed in the introduction, one component of this is dodging between velocity increments, which is detailed in Hypothesis~\ref{hyp:dodging1}. To construct a new velocity increment with such dodging, it is necessary to keep a record of the density of previous velocity increments, as stated in Hypothesis~\ref{hyp:dodging2}.  Stronger statements than Hypotheses~\ref{hyp:dodging1} and~\ref{hyp:dodging2} will be proved in Section~\ref{sec:dodging} (see Lemma~\ref{lem:dodging}).  As byproducts of the careful construction of the velocity increments, we also have Hypothesis~\ref{hyp:dodging4} and Hypothesis~\ref{hyp:dodging5}, which present dodging properties between velocity increments and stress errors/intermittent pressure. These bounds will be utilized later in the stress current estimate discussed in subsection~\ref{sec:stress:current}. Hypothesis~\ref{hyp:dodging4} for $q\mapsto q+1$ has been already verified in \cite[Lemma~8.16]{GKN23} and is recalled in Lemma~\ref{l:divergence:stress:upgrading}. Lastly, Hypothesis~\ref{hyp:dodging5} for $q\mapsto q+1$ will be verified in subsection~\ref{sec.pr.ind.verify}. We formulate these hypotheses using the notation from Remark~\ref{rem:notation:space:time:balls}.

\begin{hypothesis}[\bf Effective dodging]\label{hyp:dodging1}\index{effective dodging}
For $q',q''\leq q+\bn-1$ that satisfy $0<|q''-q'|\leq \bn-1$, we have that
\begin{equation}\label{eq:ind:dodging}
    B\left(\supp \hat w_{q'} , \lambda_{q'}^{-1}\Gamma_{q'+1} \right) \cap  B\left( \supp \hat w_{q''} , \lambda_{q''}^{-1}\Gamma_{q''+1} \right)= \emptyset \, .
\end{equation}
\end{hypothesis}

\begin{hypothesis}[\textbf{Density and direction of old pipe bundles}]\label{hyp:dodging2}
There exists a $q$-independent constant $\const_D$ such that the following holds.  Let $\bar q', \bar q''$ satisfy $q \leq \bar q'' < \bar q'\leq q+\bn-1$, and set\footnote{The reasoning behind the choice of $d(\bar q', \bar q'')$ is as follows. The set should be small enough that it can be contained in the support of a single $\bar q''$ velocity cutoff.  Since these functions oscillate at frequencies no larger than $\approx\lambda_{q''}$, the first number inside the minimum ensures that this is the case.  The set should also be no larger than the size of a periodic cell for pipes of thickness $\bar q'$, which is ensured by the second number inside the minimum.}
\begin{equation} \label{eq:diam:def}
d(\bar q', \bar q'') :=  \min\left[ (\lambda_{\bar q''}\Gamma_{\bar q''}^7)^{-1} , (\lambda_{\bar q' - \half} \Gamma_{\bar q' - \bn})^{-1} \right] \, .
\end{equation}
Let $t_0\in\R$ be any time and $\Omega\subset\T^3$ be a convex set of diameter at most $d(\bar q', \bar q'')$.  Let $i$ be such that $\Omega \times \{t_0\}\cap \supp \psi_{i,\bar q''} \neq \emptyset$. 
Let $\Phi_{\bar q''}$ be the flow map such that
\begin{align*}
\begin{cases}
\pa_t \Phi_{\bar q''} + \left(\hat u_{\bar q''} \cdot \na \right) \Phi_{\bar q''} = 0\\
\Phi_{\bar q''}(t_0,x) = x \, .
\end{cases}
\end{align*}
We define $\Omega(t)=\Phi_{\bar q''}(t)^{-1}(\Omega)$.\footnote{For any set $\Omega'\subset \T^3$, $\Phi_{\bar q''}(t)^{-1}(\Omega')=\{x\in \T^3: \Phi_{\bar q''}(t,x) \in \Omega'\}$. We shall also sometimes use the notation $\Omega\circ \Phi_{\bar q''}(t)$.} 
Then there exists a set\footnote{Heuristically this set is $\cup_t \suppp_x \hat w_{\bar q'}(\cdot,t) \cap \Omega(t)$, but in order to ensure that $(\partial_t+\hat u_{\bar q''} \cdot \nabla)\mathbf{1}_L\equiv 0$, $L$ does not include any ``time cutoffs" which turn pipes on and off.} $L=L(\bar q',\bar q'', \Omega, {t_0})\subseteq \T^3\times \R$ such that for all $t\in(t_0-\tau_{\bar q''}\Gamma_{\bar q''}^{-i+2},t_0+\tau_{\bar q''}\Gamma_{\bar q''}^{-i+2})$,
\begin{align}\label{eq:ind:dodging2}
    (\partial_t + \hat u_{\bar q''}\cdot\nabla) \mathbf{1}_{L}(t,\cdot) \equiv 0 \qquad \textnormal{and} \qquad \supp_x \hat w_{\bar q'}(x,t) \cap \Omega(t) \subseteq L \cap \{t\} \, .
\end{align}
Furthermore, there exists a finite family of continuously differentiable curves $\{\ell_{j,L}\}_{j=1}^{\const_D}$ of length at most $2d(\bar q', \bar q'')$ which satisfy 
\begin{equation}\label{eq:concentrazion}
    L \cap \{t=t_0\} \subseteq \bigcup_{j=1}^{\const_D} B\left( \ell_{j,L} , 3\la_{\bar q'}^{-1} \right) \, .
\end{equation}
Finally, there exist $q$-independent sets $\Xi_1,\Xi_1', \dots, \Xi_{\bn}, \Xi_{\bn}' \subset \mathbb{Q}^3\cap \mathbb{S}^2$ such that for all curves $\ell_{j,L}$ used to control the support of $\hat w_{\overline q'}$, the tangent vector to the curve $\ell_{j,L}$ belongs to a $\Gamma_{0}^{-1}$-neighborhood of a vector $\xi\in \Xi_{\overline{q}' \, \textnormal{mod} \, \bn} \cup \Xi'_{\overline{q}' \, \textnormal{mod} \, \bn}$.
\end{hypothesis} 
\begin{remark}[\bf Segments of deformed pipes of thickness $\la_{\bar q'}^{-1}$]\label{rem:deformed:pipes}
We will refer to a $3\la_{\bar q'}^{-1}$ neighborhood of a continuously differentiable curve of length at most $2(\la_{\bar q'-\sfrac{\bn}{2}}\Ga_{\bar q'-\bn})^{-1}$ as a ``segment of deformed pipe;'' see Definition~\ref{def:sunday:sunday}. Since $(\la_{\bar q'-\sfrac{\bn}{2}}\Ga_{\bar q'-\bn})^{-1}$ will be the scale to which our high-frequency pipes will be periodized, Hypothesis~\ref{hyp:dodging2} then asserts that at each step of the iteration, our algorithm can use at most a finite number of high-frequency pipe segments inside any single periodic cell.
\end{remark}

\begin{hypothesis}[\bf Stress dodging]\label{hyp:dodging4}
For all $k,q''$ such that $q\leq q'' \leq k-1$ and $q\leq k \leq q+\bn-1$, we assume that
\begin{equation}\label{eq:ind:stressdodging:equiv}
    B\left(\supp \hat w_{q''} , \lambda_{q''}^{-1}\Gamma_{q''+1} \right) \cap  \supp R_{q}^{k,l}= \emptyset \, .
\end{equation}
\end{hypothesis}

\begin{hypothesis}[\bf Pressure dodging]\label{hyp:dodging5}
We assume that for all $q<k\leq q+\bn-1$, $k\leq k'$, and $N+M\leq 2\Nind$,
\begin{subequations}
\begin{align}
\label{eq:pinl:inductive:dtq} 
    \left|\psi_{i,k-1} D^N D_{t,k-1}^M \left(\hat w_{k}\pi_q^{k'}\right)\right| &<
    \Ga_q 
    {\Ga_k^{-100}} \left(\pi_q^k\right)^{\sfrac32} r_k^{-1} \Lambda_k^{N} \MM{M, \NindRt, \Gamma_{k-1}^{i+\blue{1}} \tau_{k-1}^{-1} , \Gamma_{k}^{-1} \Tau_{k}^{-1} } \, .
\end{align}
\end{subequations}
\end{hypothesis}

\subsection{Inductive velocity bounds}
\label{sec:inductive:secondary:velocity}
In this section, we present inductive, local $L^\infty$-bounds for velocity increments and total velocity, which are derived from the construction of velocity cutoffs. Additionally, we introduce velocity increment potentials, which express velocity increments as the $\dpot^{\rm th}$ divergence of the velocity increment tensor potential, up to a small homogeneous error. This representation will be useful to deal with the pressure current error (see subsection~\ref{sec:pressure:current} for more details). All inductive assumptions in subsection~\ref{sec:inductive:secondary:velocity} except for \eqref{est.upsilon.ptwise} at $q\mapsto q+1$ have been verified in \cite[sections~9, 10]{GKN23} (see Proposition~\ref{prop:inductive:velocity:bdd:verified}), and we prove \eqref{est.upsilon.ptwise} for $q\mapsto q+1$ in Lemma~\ref{prop:velocity:domination}.

\subsubsection{Velocities and velocity increments}
In this subsection, we assume that $0\leq q'\leq q+\bn-1$.  First, for $0 \leq i \leq \imax$, $k\geq 1$, $\alpha, \beta \in \N^k$, we assume that
\begin{align}
&\norm{\left( \prod_{l=1}^k D^{\alpha_l} D_{t,q'-1}^{\beta_l} \right) \hat w_{q'} }_{L^\infty(\supp \psi_{i,q'})} \leq \Gamma_{q'}^{i+2}\de_{q'}^{\sfrac12} r_{q'-\bn}^{-\sfrac13} (\la_{q'}\Ga_{q'})^{|\alpha|} \MM{|\beta|,\Nindt, \Gamma_{{q'}}^{i+3}  \tau_{q'-1}^{-1},  \Gamma_{{q'-1}} \Tau_{q'-1}^{-1}}
\label{eq:nasty:D:wq:old}
\end{align}
for $|\alpha|+|\beta|  \leq {\sfrac{3\Nfin}{2}+1}$. We also assume that for $N\geq 0$,
\begin{subequations}\label{eq:nasty}
\begin{align}
&\norm{ D^N \Big( \prod_{l=1}^k D_{q'}^{\alpha_l} D_{t,q'-1}^{\beta_l} \Big) \hat w_{q'}}_{L^\infty(\supp \psi_{i,q'})} \notag\\
&\qquad \leq
(\Gamma_{{q'}}^{i+2}\delta_{q'
}^{\sfrac 12}r_{q'-\bn}^{-\sfrac 13})^{|\alpha|+1} (\la_{q'}\Ga_{q'})^{N+|\alpha|} \MM{|\beta|,\Nindt,  \Gamma_{q'}^{i+3}  \tau_{q'-1}^{-1},  \Gamma_{q'-1} \Tau_{q'-1}^{-1}}   \label{eq:nasty:Dt:wq:old} \\
&\qquad \leq
\Gamma_{q'}^{i+2} \delta_{q'}^{\sfrac 12} r_{q'-\bn}^{-\sfrac 13} (\lambda_{q'}\Ga_{q'})^N (\Gamma_{q'}^{i-5}  \tau_{q'}^{-1})^{|\alpha|}  \MM{|\beta|,\Nindt, \Gamma_{q'}^{i+3}  \tau_{q'-1}^{-1},  \Gamma_{q'-1} \Tau_{q'-1}^{-1}}
\label{eq:nasty:Dt:wq:WEAK:old}
\end{align}
\end{subequations}
whenever $N+|\alpha|+|\beta|\leq  {\sfrac{3\Nfin}{2}+1}$. Next, we assume
\begin{align}
&\norm{\left( \prod_{l=1}^k D^{\alpha_l} D_{t,q'}^{\beta_l} \right) D \hat u_{q'} }_{L^\infty(\supp \psi_{i,q'})}  \leq \tau_{q'}^{-1}\Ga_{q'}^{i-4} (\lambda_{q'}\Ga_{q'})^{|\alpha|} \MM{|\beta|,\Nindt,\Gamma_{q'}^{i-5} \tau_{q'}^{-1},   \Gamma_{q'-1} \Tau_{q'-1}^{-1}}
\label{eq:nasty:D:vq:old}
\end{align}
for $|\alpha|+|\beta| \leq {\sfrac{3\Nfin}{2}}$.  In addition, we assume the lossy bounds
\begin{subequations}\label{eq:bob:old}
\begin{align}
\norm{\left( \prod_{l=1}^k D^{\alpha_l} D_{t,q'}^{\beta_l} \right)  \hat u_{q'}}_{L^\infty(\supp \psi_{i,q'})} &\leq \tau_{q'}^{-1}\Ga_{q'}^{i+2} \la_{q'} (\lambda_{q'}\Ga_{q'})^{|\alpha|} \MM{|\beta|,\Nindt,\Gamma_{q'}^{i-5} \tau_{q'}^{-1},   \Gamma_{q'-1} \Tau_{q'-1}^{-1}}
\label{eq:bob:Dq':old} \\
\left\| D^{|\alpha|} \partial_t^{|\beta|} \hat{u}_{q'} \right\|_{L^\infty} & \leq \Lambda_{q'}^{\sfrac 12} \Lambda_q^{|\alpha|} \Tau_{q'}^{-|\beta|} \, , \label{eq:bobby:old}
\end{align}
\end{subequations}
hold, where the first bounds holds for $|\alpha|+|\beta| \leq \sfrac{3\Nfin}{2}+1$, and the second bound holds for $|\alpha|+|\beta|\leq 2\Nfin$.

\begin{remark}[\bf Upgrading material derivatives for velocity and velocity cutoffs]
\label{rem:D:t:q':orangutan}
We recall from \cite[Remark~2.9]{GKN23} that we have the bound
\begin{align}
&\norm{ D^N  D_{t,q'}^{M}  \hat w_{q'} }_{L^\infty(\supp \psi_{i,q'})} \les \Gamma_{q'}^{i+2} \delta_{q'}^{\sfrac 12} r_{q'-\bn}^{-\sfrac 13} (\la_{q'}\Ga_{q'})^N 
\MM{M,\Nindt, \Gamma_{q'}^{i-5}  \tau_{q'}^{-1},  \Gamma_{q'-1} \Tau_{q'-1}^{-1}}
\label{eq:nasty:Dt:uq:orangutan}
\end{align}
for all $N+M \leq {\sfrac{3\Nfin}{2}+1}$.
We also have
that for all $N+M \leq \Nfin$,
\begin{align}
\frac{{\bf 1}_{\supp \psi_{i,q'}}}{\psi_{i,q'}^{1- (N+M)/\Nfin}} \left| D^N  D_{t,q'}^{M}  \psi_{i,q'} \right| &\leq \Gamma_{q'} (\lambda_{q'}\Ga_{q'})^N
\MM{M,\Nindt-\NcutSmall, \Gamma_{q'}^{i-5} \tau_{q'}^{-1}, \Gamma_{q'-1}  \Tau_{q'-1}^{-1}} \notag\\
&< \Gamma_{q'} (\lambda_{q'}\Ga_{q'})^N \MM{M,\Nindt,\Gamma_{q'}^{i-4}\tau_{q'}^{-1},\Gamma_{q'-1}^2\Tau_{q'-1}^{-1}} \, .
\label{eq:nasty:Dt:psi:i:q:orangutan}
\end{align}
\end{remark}

\subsubsection{Velocity increment potentials}\label{sec:ind.vel.inc.pot}
We assume that for all $q-1 < q' \leq q+\bn -1$ and $\hat w_{q'}$ as in \eqref{eq:hat:no:hat}, there exists a velocity increment potential 
\index{velocity increment potential} $\hat\upsilon_{q'}$\index{$\hat\upsilon_{q'}$} and an error $\hat e_{q'}$ such that $\hat w_{q'}$ can be decomposed as\footnote{See \eqref{i:par:10} of subsection~\ref{sec:para.q.ind} for the definition of $\dpot$.}
\index{$\dpot$}
\begin{align}\label{exp.w.q'}
\hat w_{q'} = \div^{\dpot} \hat \upsilon_{q'} +\hat e_{q'}  \,, 
\end{align}
which written component-wise gives $
\hat w_{q'}^\bullet = 
\pa_{i_1}\cdots \pa_{i_\dpot} \hat\upsilon_{q'}^{(\bullet, i_1, \cdots, i_\dpot)} + \hat e_{q'}^\bullet
$. Next, we assume that $\hat\upsilon_{q'}$ and $\hat e_{q'}$ satisfy
\begin{align}\label{supp.upsilon.e.ind}
B\left(\supp(\hat w_{q''}), \frac14 \la_{q''}\Ga_{q''}^2\right) \cap \left(\supp(\hat \upsilon_{q'}) \cup \supp(\hat e_{q'})\right) = \emptyset 
\end{align}
for any $q+1 \leq q'' < q'$.  In addition, we assume that $\hat\upsilon_{q',k}^\bullet :=\la_{q'}^{\dpot-k}\partial_{i_1}\cdots \partial_{i_k} \hat\upsilon_{q'}^{(\bullet,i_1,\dots,i_\dpot)}$, $0\leq k \leq \dpot$, satisfies the pointwise estimates
\begin{align}
    &\left|\psi_{i,q'-1} D^ND_{t,q'-1}^{M} 
   \hat \upsilon_{q',k}
    \right|< \Ga_q \Ga_{q'}\left(\pi_q^{q'}\right)^{\sfrac12} r_{q'-\bn}^{-1}
     (\la_{q'}{\Ga_{q'}})^N
     \MM{M, \Nindt, \Ga_{q'-1}^i \tau_{q'-1}^{-1}, \Tau_{q'-1}^{-1}\Ga_{q'-1}^2}
    \label{est.upsilon.ptwise}
\end{align}
for $N+M \leq \sfrac{3\Nfin}{2}$. Finally, we assume that for $N+M \leq \sfrac{3\Nfin}{2}$, $\hat e_{q'}$ satisfies the estimates
\begin{align} 
    \norm{D^ND_{t,q'-1}^{M} \hat e_{q'}}_{\infty} 
    \leq \de_{q'+2\bn}^3 \Tau_{q'}^{5\Nindt}\la_{q'}^{-10} 
    (\la_{q'} {\Ga_{q'}})^N
     \MM{M, \Nindt, \tau_{q'-1}^{-1}, \Tau_{q'-1}^{-1}\Ga_{q'-1}^2}
    \label{est.e.inf}\, .
\end{align}

\subsection{Inductive proposition and the proof of Theorem~\ref{thm:main}}\label{ss:friday}
In this section, we introduce the inductive proposition and give a proof of the main theorem. 

\begin{proposition}[\bf Inductive proposition]\label{prop:main}
Fix $\be\in (0, \sfrac13)$, and choose $\bn$ satisfying \eqref{eq:choice:of:bn}, $b\in (1, \sfrac{25}{24})$ satisfying \eqref{ineq:b}, $T>0$, and a continuous positive function $E(t,x)\geq 0$. Then there exist parameters $\varepsilon_\Gamma$, $\badshaq$, $\Npr$, $\NcutSmall$, $\Nindt$, $\Nind$, $\Nfin$, depending only on $\be$, $b$, and $\bn$  (see Section \ref{sec:para.q.ind}) such that 
    we can find sufficiently large $a_*=a_*(b,\be, \bn, T)$ such that for $a\geq a_*(b,\be, \bn, T)$, the following statements hold for any $q\geq 0$. Suppose that an approximate solution $(u_q, p_q, R_q, \ph_q, -\pi_q)$ of the Euler-Reynolds system \eqref{eqn:ER} and the relaxed local energy identity \eqref{ineq:relaxed.LEI} with dissipation measure $E$ on the time interval $[-\tau_{q-1},T+\tau_{q-1}]$ is given, and suppose that there exists a partition of unity $\{\psi_{i,q'}^6\}_{i\geq 0}$ of $[-\tau_{q-1},T+\tau_{q-1}] \times \T^3$ for $q-1\leq q'\leq q+\bn-1$ such that 
    \begin{itemize}
        \item $\psi_{i,q'}$ satisfies \eqref{eq:inductive:partition}--\eqref{eq:inductive:timescales}, and
        \item the velocity $u_q$ and the errors $R_q$, $\ph_q$, and $\pi_q$ may be decomposed as in \eqref{eq:hat:no:hat}--\eqref{eq:LEI:decomp:basic} so that
        \eqref{eq:pressure:inductive}--\eqref{eq:psi:q:q'}, Hypotheses~\ref{hyp:dodging1}--Hypothesis \ref{hyp:dodging5}, \eqref{eq:nasty:D:wq:old}--\eqref{eq:bob:old}, and \eqref{exp.w.q'}--\eqref{est.e.inf} hold. 
    \end{itemize}
     Then there exist a new partition of unity $\{\psi_{i, q+\bn}^6\}_{i\geq 0}$ of $[-\tau_q,T+\tau_q] \times \T^3$ satisfying \eqref{eq:inductive:partition}--\eqref{eq:inductive:timescales} for $q'=q+\bn$, and a new approximate solution $(u_{q+1}, p_{q+1}, R_{q+1}, \ph_{q+1}, -\pi_{q+1})$ satisfying \eqref{eqn:ER} and \eqref{ineq:relaxed.LEI} on $[-\tau_q,T+\tau_q]$ with dissipation measure $E$ satisfying the following conditions.  The approximate solution may be decomposed as in \eqref{eq:hat:no:hat}--\eqref{eq:LEI:decomp:basic} for $q\mapsto q+1$ so that \eqref{eq:pressure:inductive}--\eqref{eq:psi:q:q'} and Hypothesis \ref{hyp:dodging1}--Hypothesis \ref{hyp:dodging5} hold for $q\mapsto q+1$, and \eqref{eq:nasty:D:wq:old}--\eqref{eq:bob:old} and \eqref{exp.w.q'}--\eqref{est.e.inf} hold for $q\mapsto q+1$.
\end{proposition}
\noindent Assuming for the moment that Proposition~\ref{prop:main} holds, we prove Theorem \ref{thm:main}.

\begin{proof}[\bf Proof of Theorem \ref{thm:main}.]
Let $\upbeta\in (0,\sfrac 13)$ be fixed as in Theorem~\ref{thm:main}.  Since Proposition~\ref{prop:main} holds for $\beta\in (0, \sfrac 13)$, we can choose $\beta \in (\upbeta, \sfrac13)$, $\bn$ large, and $b$ close enough to $1$ such that
\begin{align}
     \upbeta < \beta - (b-1)^{\sfrac 12} \, , \quad \frac{1}{1-3\upbeta} <
      \beta b^{3\bn} \left( \left[ (b-1) \left( (b^{\sfrac{\bn}{2}-1}+ \dots + b + 1)^2 + 2000 b^{\bn} \right) + \frac{4b^{\bn-1}}{1+\dots+b^{\sfrac{\bn}{2}-1}} \right] \right)^{-1} \, ,
    \label{eq:morning:useful}
\end{align}
and \eqref{eq:choice:of:bn} and \eqref{ineq:b} hold. Let $a_*=a_*(b,\beta,\bn,T)$ be as in Proposition~\ref{prop:main}, fix $a\geq a_*$, and define
\begin{align*}
    \Ga_q =1 \quad \text{for } -\bn\leq q\leq -1 \, , \qquad
    r_q =\begin{cases}
    \la_{q+\half}\la_{q+\bn}^{-1} &\text{for } -\half \leq q<0\\
    \la_0^{\sfrac12} \la_{\bn}^{-\sfrac12} &\text{for }-\bn\leq q<-\half \, .
    \end{cases}
\end{align*}

\noindent\texttt{Step 1: Construction of the initial approximate solution $(u_0, p_0, R_0, \ph_0, - \pi_0)$ on $[-1,T+1]$.} \\
We first define
\begin{align*}
    \hat u_{-1} = 0 \, , \quad \hat w_{q'} &= 0 \quad\text{for $0\leq q'\leq \bn-2$} \, ,  
    \qquad  p_0 = \ph_0 =0 \, , \qquad \pi_0^k = \Ga_k \de_{k+\bn} \, , 
    \\
    \psi_{i,q'} &= 
    \begin{cases}
    1 & i=0 \\
    0 & \text{otherwise}
    \end{cases} \quad \text{for } 0\leq q' \leq \bn-1\, , 
\end{align*}
so that
\begin{align*}
   \hat u_{q'} =0 \quad\text{for }0\leq q'\leq \bn-2 \, , \qquad
   \pi_0 = \sum_{k=0}^\infty\Ga_k \de_{k+\bn}\, . 
\end{align*}
Let $\bar\vartheta_{e_3, \la_{\bn-1}, r_{-1}}$ and $\bar\varrho_{e_3, \la_{\bn-1}, r_{-1}}$ be 
smooth functions such that they depend only on the $x_1, x_2$ variables, are $\left(\sfrac{\T^3}{\la_{\half-1}\Ga_{-1})}\right)$-periodic, have support contained in pipes of thickness $\la_{\bn-1}^{-1}$, and satisfy
\begin{align*}
    \bar\varrho_{e_3, \la_{\bn-1}, r_{-1}} 
    = \la_{\bn-1}^{-\dpot} \div^\dpot\bar\vartheta_{e_3, \la_{\bn-1}, r_{-1}}, \quad 
    \norm{\na^n\bar\vartheta_{e_3, \la_{\bn-1}, r_{-1}}}_{L^\infty} +
    \norm{\na^n\bar\varrho_{e_3, \la_{\bn-1}, r_{-1}}}_{L^\infty}
    \leq C_0 \la_{\bn-1}^n r_{-1}^{-1}       
\end{align*}
for all $n\leq 2\Nfin + \dpot$ and some positive constant $C_0=C_0(2\Nfin, \dpot)$. To lighten notation, we abbreviate $\bar\varrho_{e_3, \la_{\bn-1}, r_{-1}}$ and $\bar\vartheta_{e_3, \la_{\bn-1}, r_{-1}}$ by $\rho_0$ and $\theta_0$, respectively. 
We then define
\begin{align*}
    u_0 &= \hat w_{\bn-1} = e(t) \rho_0(x_1, x_2) e_3,
    \quad 
    e(t) := \Ga_{\bn}^{-100}\de_{2\bn} r_{-1}\exp(-\tau_0^{-1}(t+1))
    \\
    R_0 &= e'(t) \begin{pmatrix}  
    0 & 0 & (\la_{\bn-1}^{-\dpot}\div^{\dpot-1}\theta_0)_1\\
    0& 0& (\la_{\bn-1}^{-\dpot}\div^{\dpot-1}\theta_0)_2\\
    (\la_{\bn-1}^{-\dpot}\div^{\dpot-1}\theta_0)_1 & (\la_{\bn-1}^{-\dpot}\div^{\dpot-1}\theta_0)_2 &0
    \end{pmatrix} = R_0^{\bn-1,l}\,  ,
\end{align*}
where $(g)_k$ denotes the $k^{\rm th}$ component of the vector $g$. Notice that $R_0$ is constructed to satisfy $\pa_t u_0 = \div R_0$. 

From the construction, we have 
$\div u_0=\div(|u_0|^2 u_0)=0$ and $\div (u_0 \otimes u_0 ) =0$, so that one can easily see that $(u_0, p_0, R_0, \ph_0, - \pi_0)$ satisfies the Euler-Reynolds system \eqref{eqn:ER} and the relaxed local energy identity \eqref{ineq:relaxed.LEI} with $E(t,x):=-e(t)e'(t) |\rho_0|^2$
on the time interval $[-1, T+1]$.  We will now check that the constructed approximate solution and the partitions of unity satisfy the remaining inductive assumptions appearing in Section \eqref{ss:relaxed}--\eqref{sec:inductive:secondary:velocity} on $[-1, T+1]$.

We first note that for $0\leq q'\leq \bn-1$, it is immediate that $\psi_{i,q'}$ satisfies \eqref{eq:inductive:partition}--\eqref{eq:inductive:timescales}.  Next, letting $R_0^k=0$ for $0\leq k\leq \bn-2$, $R_0^{\bn-1} = R_0^{\bn-1,l}$, and $\ph_0^k=0$ for $0\leq k \leq \bn-1$,
we have the decompositions \eqref{eq:hat:no:hat}--\eqref{eq:LEI:decomp:basic}. Using the convention $B(A, r) =\emptyset$ for the empty set $A$ and setting $\hat \upsilon_{\bn-1} =e(t) \theta_0 e_3$, $\hat \upsilon_{q'}=e_{q''}=0$ for $0\leq q' \leq \bn-2$ and $0\leq q''\leq \bn-1$, we can easily verify \eqref{eq:pressure:inductive}--\eqref{eq:psi:q:q'}, Hypotheses \ref{hyp:dodging1}, \ref{hyp:dodging4}, and \ref{hyp:dodging5}, \eqref{eq:nasty:D:wq:old}, \eqref{eq:nasty}, and \eqref{exp.w.q'}--\eqref{est.e.inf} for $q=0$.  As for Hypothesis \ref{hyp:dodging2}, it is enough to consider $q'=\bn-1$. Since $q''<q=\bn-1$ and $i$ needs to be $0$, recalling that $\hat u_{q''} =0$, we have $\Phi_{q''}(t,x)=x$ and hence $\Omega(t)=\Omega$. Then, \eqref{eq:ind:dodging2} is equivalent to 
\begin{align*}
    \supp \hat w_{\bn-1} \cap \Omega 
    \subset L \cap \Omega \, .
\end{align*}
Therefore, we choose $L$ as the collection of the $(\sfrac{\T}{\la_{\half-1}\Ga_{\half-1}})^3$-periodic pipes of thickness $\la_{\bn-1}^{-1}$ containing the support of $\hat w_{\bn-1}$.  To finish verifying the hypothesis, we may adjust the definition of $\Xi_{\bar q' \textnormal{ mod } \bn}$ so that it contains $e_3$, and adjust the other sets of vector directions so that none of them also contain $e_3$ (). Lastly, considering \eqref{eq:nasty:D:vq:old} and \eqref{eq:bob:old}, it is enough to prove it when $q'=\bn-1$ and $i=0$. Since we have $\hat u_{\bn-1} = \hat w_{\bn-1}$ and 
\begin{align*}
    \norm{\left( \prod_{l=1}^k D^{\alpha_l} \pa_t^{\beta_l} \right)\hat w_{\bn-1} }_{L^\infty(\T^3)}  \leq C_0
    \Gamma_{\bn}^{-100}\de_{2\bn}\la_{\bn-1}^{|\al|} 
    \tau_0^{-|\be|}\, , 
\end{align*}
applying Lemma~\ref{rem:upgrade.material.derivative.end} to $p=\infty$, $v=0$, $w=\hat w_{\bn-1}$, $\Omega=\T^3$, $N_* = \sfrac{7\Nfin}{4}$, we obtain \eqref{eq:nasty:D:vq:old} and \eqref{eq:bob:old}. 
\smallskip

\noindent\texttt{Step 2: From Proposition \ref{prop:main} to Theorem \ref{thm:main}.}  In \texttt{Step 1}, we checked that the inductive assumptions hold at the base case $q=0$ of the induction.  We now apply Proposition~\ref{prop:main} inductively with $E(t,x)=-e(t)e'(t) |\rho_0|^2$ to produce a sequence of approximate solutions $(u_q, p_q, R_q, \ph_q, -\pi_q)$ such that all inductive assumptions hold for all $q\geq 0$. In order to compute the Besov norm of the limiting solution, we will use the inductive bound \eqref{eq:psi:i:q:support:old} for $\psi_{i,q}$, the bounds \eqref{eq:nasty:D:wq:old} for $\hat w_q$ and $\nabla \hat w_q$, and the $q$-independent bound in \eqref{eq:imax:upper:lower} for $\imax$. Then by the definition of $B^{\beta'}_{3,\infty}(\T^3)$ in \eqref{eq:bosov}, interpolation, and the standard characterization of $W^{1,3}(\T^3)$ in terms of difference quotients, we have that
\begin{align}
    \left\| \hat w_{q} \right\|_{C^0 B^{\beta'}_{3,\infty}} \leq \left\| \nabla \hat w_{q} \right\|_{C^0 W^{1,3}}^{\beta'} \left\| \hat w_{q} \right\|_{C^0 L^3}^{(1-\beta')} \les \delta_q^{\sfrac 12} \Ga_q^{\frac{\CLebesgue}{3}+3} r_{q-\bn}^{-\sfrac 13} \la_q^{\beta'}  \leq \la_q^{-\beta + (b-1)\varepsilon_\Gamma \left( \frac{6+b}{b-1} + 3 \right) + \frac 13 \left( \frac{b^{\sfrac{\bn}{2}}-1}{b^{\sfrac{\bn}{2}}} \right) + \beta'}  \, . \label{nighttime:mess}
\end{align}
Using that $0<\varepsilon_\Gamma\ll (b-1)^2$ (see immediately below \eqref{eq:deffy:of:gamma}) and $(b^{\sfrac \bn 2 -1}+ \dots + b + 1)(b-1)<(b-1)^{\sfrac 12}$ (from \eqref{ineq:b:second}), we have that the exponent for $\la_q$ in \eqref{nighttime:mess} is no larger than $-\beta + (b-1)^{\sfrac 12} + \beta'$.  Now from \eqref{eq:morning:useful}, we have that the series $\sum_{q\geq \bn} \hat w_{q}$ is absolutely summable in $C^0_t B^{\upbeta}_{3,\infty}$.  Therefore the limiting velocity field $u = u_0 +\sum_{q\geq \bn} \hat w_{q} \in C^0_t B^{\upbeta}_{3,\infty}$, as desired. Since $R_q, \pi_q \to 0$ in $C^0_t L^{\sfrac32}$, we obtain from the equation $-\Delta p_q = \div\div(u_q\otimes u_q + \pi_q\Id -R_q)$ that the limiting pressure $p\in C^0_t L^{\sfrac32}$. Using in addition that $\ph_q \to 0$ in $C^0_t L^1$, the limiting pair $(u,p)$ therefore solves the Euler equations \eqref{eqn:Euler} and satisfies
\begin{align}\notag
\pa_t\left(\frac12 |u|^2\right)
+ \div\left( \left(\frac12|u|^2+p\right)u\right)
= e(t)e'(t) |\rho_0|^2 \leq 0
\end{align}
in the sense of distributions. In particular, the strict local energy inequality holds in the interior of $\supp(\rho_0)$, which leads to the total kinetic energy dissipation.  

In order to conclude the proof of the theorem, we only need to show that $u\in C^0_t L^{\frac{1}{(1-3\upbeta)}}$. Using the $L^3$ and $L^\infty$ bounds on $\hat w_q$ provided by \eqref{eq:ind:velocity:by:pi} and \eqref{eq:pressure:inductive}, we sum over $0\leq i \leq \imax(q)$ as before, arriving at
\begin{align*}
    \norm{\hat w_q}_{L^3}  \leq C \delta_{q+\bn}^{\sfrac 12}r_{q-\bn}^{-1}
    \Ga_q^5 
    \qquad \mbox{and} \qquad
    \norm{\hat w_q}_{L^\infty}  
    \leq C \Gamma_q^{\sfrac{\badshaq}2 + 5} r_{q-\bn}^{-1} \, ,
    \end{align*}
where the constant $C$ depends only on our upper bound for $\imax(q)$, and so only on $\beta$ and $b$ through \eqref{eq:imax:upper:lower}. Using interpolation, the definition \eqref{eq:badshaq:choice} of $\badshaq$, $\varepsilon_\Gamma\ll(b-1)^2$, and the above established bounds, for $p\in [3,\infty)$ we obtain
\begin{align}
    \left\| \hat w_q \right\|_{L^p} 
    \leq 
    \left\| \hat w_q \right\|_{L^3}^{\frac{3}{p}} 
    \left\| \hat w_q \right\|_{L^\infty}^{1-\frac 3p} 
     &\leq 
     C 
\de_{q+\bn}^{\frac{3}{2p}}
\Ga_q^{\badshaq{\left(\frac12-\frac3{2p}\right)}} \Ga_q^5 r_{q-\bn}^{-1} \notag\\
&\leq C \la_{q-\bn}^{\frac{3}{2p}(-2\be b^{2\bn}) +\left(\frac12-\frac 3{2p}\right) 6 b^{\bn} \left[ (b-1)(b^{\sfrac{\bn}{2}-1}+\dots+b+1)^2 + 2000(b-1) b^{\bn} + \frac{4b^{\bn-1}}{1+\dots+b^{\sfrac{\bn}{2}-1}} \right] } \, ,
\label{bounding:L3}
\end{align}
where the constant $C=C(\be, b)\geq 1$. In order to ensure that the exponent of $\lambda_{q-\bn}$ on the right side of \eqref{bounding:L3} is strictly negative, a short computation shows that it will suffice to choose
\begin{align}
     p < p_*(\beta,b) =:
      \beta b^{2\bn} \left( b^{\bn} \left[ (b-1)(b^{\sfrac{\bn}{2}-1}+ \dots + b + 1)^2 + 2000 (b-1)b^{\bn} + \frac{4b^{\bn-1}}{1+\dots+b^{\sfrac{\bn}{2}-1}} \right] \right)^{-1}
    \,.
    \notag
\end{align}
Then by \eqref{eq:morning:useful}, we can choose $p=(1-3\upbeta)^{-1}$, and so $u\in C^0_t L^{\frac{1}{1-3\upbeta}}$.
\end{proof}

\section{Convex integration set-up and preliminary results}\label{stuff:from:part:1}

The first goal of this section is to set up the proof of Proposition~\ref{prop:main} by defining a premollified velocity increment $w_{q+1}$.  The second goal is to record lemmas and estimates related to the creation of new Euler-Reynolds stress errors.  We do not give the full proofs of these results, opting instead to give a few ideas when necessary and refer to \cite{GKN23} for further details.

First, we will define a premollified velocity increment (still only heuristically, and not with the final notation) by
\begin{align}\label{pert:fried:egg}
    w_{q+1} = w_{q+1,\varphi} + w_{q+1,R} \approx \sum_j a_j(\nabla \hat u_q, \varphi_q^q) \mathbb{B}_{j,\varphi} \circ \Phi_{q,j} + \sum_{j'} a_{j'}(\nabla \hat u_q, \varphi_q^q, R_q^q) \mathbb{B}_{{j'}, R} \circ \Phi_{q,{j'}} \, .
\end{align}
The intermittent Mikado bundles $\mathbb{B}_{j,\bullet}$ are defined in subsection~\ref{subsec:result:par1:1}. Then in subsection~\ref{subsec:result:par1:2}, we define all the components of the coefficient functions $a_\bullet$.  These components include both non-inductive cutoffs, which serve to partition time, measure the level sets of the intermittent pressure and its derivatives, and localize the supports of the bundles, and the inductive cutoffs, which partition the level sets of the gradient of the velocity field.  The flow maps $\Phi_{q,j}$ are then defined and estimated on the intersection of the supports of velocity and time cutoffs. Estimates for the cutoff functions are record throughout, and ``aggregation lemmas'' which serve to combine local pointwise bounds into global $L^p$ bounds are given in subsubsection~\ref{sss:aggregation}.  With these ingredients in hand, we define and estimate the premollified velocity increment precisely in subsection~\ref{sec:vel.inc}.  We also define and estimate the associated pressure increments and current errors created by these pressure increments, as well as the velocity increment potentials, which are necessary for example in Lemma~\ref{lem:nezezzary}. Then sections~\ref{sec:new.pressure.stress}--\ref{subsec:result:par1:6} recall the definitions, estimates, and associated pressure increments of new stress errors, new transport and Nash current errors, and new mollification current errors, respectively.

\subsection{Intermittent Mikado bundles}
\label{subsec:result:par1:1}
In this subsection, we introduce the intermittent Mikado bundles\index{$\pxi$}\index{$\diamond$}\index{$I$}\index{$\xi$}\index{$\BB_{\pxi,\diamond}$}\index{$\WW_{\pxi,\diamond}^I$}
\begin{equation}\label{int:pipe:bundle:short}
    \BB_{\pxi,\diamond} = \rhob_\pxi^\diamond \sum_I \zetab_\xi^{I,\diamond} \WW_{\pxi,\diamond}^I \, \quad \diamond=R,\ph\,.
\end{equation}
The scalar function $\rhob_\pxi^\diamond$ is a mildly intermittent (meaning that in any $L^p$-estimates we can simply appeal to its $L^\infty$ norm) ``bundling pipe'' defined in Proposition~\ref{prop:bundling}. The notation $\diamond$ distinguishes objects used to correct stress errors $R$ from objects used to correct current errors $\varphi$.  We put $\diamond$ in the exponent of $\rhob_\pxi^\diamond$ since $\rhob_\pxi^\diamond$ is defined as the second or third power of a master function.  The vector $\xi$ is the direction of the velocity field of the bundle and satisfies $\xi\cdot \nabla \BB_{\pxi,\diamond}\equiv 0$. We use the notation $\pxi$ as a short-hand for the numerous indices on which the placement of the bundling pipe will depend; these indices will be detailed in Section~\ref{sec:dodging} when we carry out the placement. The direction vectors $\xi$ are chosen to belong to sets $\Xi_j, \Xi_j' \subset (\mathbb{S}^2 \cap \mathbb{Q}^3)$\index{$\Xi$, $\Xi'$}, where $1 \leq j \leq \bn$.  These sets are used in geometric lemmas (\cite[Propositions~4.1 and 4.2]{GKN23}) to engender symmetric tensors or vectors as ``positive convex integrals'' of the simple tensors and vectors, respectively, coming from the geometric lemmas. The next components in the bundle are a strongly anisotropic cutoff function $\zetab_{\xi}^{I,\diamond}$, which is defined in Definition~\ref{def:etab} and used to localize the support of the highly intermittent Mikado flows $\WW_{(\xi), \diamond}^I$, recalled in Propositions~\ref{prop:pipeconstruction} and \ref{prop:pipe.flow.current}. The strongly anisotropic cutoff function is indexed by $I$ and depends on the vector direction $\xi$, but \emph{not} any further indices like the bundling pipe flow, due to the fact that we do not ``place'' the strong anisotropic cutoff functions. These cutoffs also satisfy a normalization property depending on $\diamond$; this is used for example to enact the cubic cancellation in subsection~\ref{ss:ssO:current}.  Finally, we note that in Section~\ref{sec:dodging}, the intermittent Mikado flow $\WW_{(\xi), \diamond}^I$ will be placed carefully on the intersection of the support of an anistropic cutoff $\zetab_\xi^{I,\diamond}$ indexed by $I$, and the supports of further cutoff functions, which we detail later and depend on indices abbreviated with the notation $\pxi$. While one might guess that the Mikado flows for $\diamond=R,\varphi$ are defined as powers of a master flow, in fact certain distinct cancellation properties are needed for $\varphi$ and $R$; see for example Proposition~\ref{prop:pipeconstruction}, item~\eqref{item:pipe:means}, which ensures that the $L^2$-normalized pipes do not create a low-frequency term in the trilinear current oscillation error in subsection~\ref{ss:ssO:current}.

\begin{proposition}[\bf Bundling pipe flows $\rhob_{\xi,\diamond}^k$ for Reynolds and current correctors]\label{prop:bundling}
The master scalar functions $\ov\rhob_{\xi,k}$ for $k\in \{1,\dots,\Ga_q^6\}$ and subsidiary bundling pipe flows $\rhob_{\xi,k,R}:=\ov\rhob_{\xi,k}^3$ for Reynolds correctors and $\rhob_{\xi,k,\varphi}:=\ov\rhob_{\xi,k}^2$ for current correctors defined in \cite[Proposition~4.9]{GKN23} satisfy the following.
\begin{enumerate}[(i)]
    \item\label{i:bundling:1} $\ov\rhob_{\xi,k}$ is simultaneously $\left( \sfrac{\T}{\lambda_{q+1}\Gamma_q^{-4}}\right)^3$-periodic 
    and $\left(\sfrac{\Tthreexi}{\lambda_{q+1}\Gamma_q^{-4} n_\ast}  \right)$-periodic,\footnote{$n_*$ is a positive integer independent of $q$ such that $n_*\xi\in \Z^3$ for all $\xi\in \Xi\cup \Xi'$. We denote by $\T^3_\xi$ a rotation of the standard torus so that it has a face perpendicular to $\xi$.\index{$n_*$} \index{$\T^3_\xi$}}
    and satisfies $\xi \cdot \nabla \ov\rhob_{\xi,k}\equiv 0$.
    \item\label{i:bundling:2} 
    The support of $\ov\rhob_{\xi,k}$ is contained in a pipe (cylinder) centered around a
    line parallel to $\xi$ with the cross-sectional radius $\pi(4\lambda_{q+1}\Gamma_q^{-1} n_*)^{-1}$. Dividing a face of $\left(\sfrac{\Tthreexi}{\lambda_{q+1}\Gamma_q^{-4} n_\ast}  \right)$ perpendicular to $\xi$ into a grid of squares with side-length $2\pi(4\lambda_{q+1}\Gamma_q^{-1} n_*)^{-1}$, there are $\Gamma_q^{6}$ disjoint possible placements of the cylinder, indexed by $k$. In particular, we have $\supp \ov\rhob_{\xi,k} \cap \supp \ov\rhob_{\xi,k'} = \emptyset$ for $k\neq k'$.
    \item\label{i:bundling:3} $\displaystyle \int_{\T^3} \ov\rhob_{\xi,k}^6=1$.
    \item\label{i:bundling:4} For all $n\leq 3\Nfin$, $p\in[1,\infty]$, and $k,k'$,
    \begin{equation}\label{e:fat:pipe:estimates:1}
    \left\| \nabla^n \rhob_{\xi,k,R} \right\|_{L^p(\mathbb{T}^3)} \lesssim \left(\Gamma_q^{-1}\lambda_{q+1}\right)^n \Gamma_q^{-3\left(\frac 2p -1\right)} \, , \qquad \left\| \nabla^n \rhob^k_{\xi,k',\varphi} \right\|_{L^p(\mathbb{T}^3)} \lesssim \left(\Gamma_q^{-1}\lambda_{q+1}\right)^n \Gamma_q^{-3\left(\frac 2p - \frac 23 \right)} \, .
    \end{equation}
\end{enumerate}
\end{proposition}
\noindent We will later choose $\chib^\diamond_{(\xi)} = \chib_{\xi,m,\diamond}$\index{$\rhob^\diamond_{\pxi}$}  for some $m=m_{(\xi),\diamond}$. The indices $m=m_{(\xi),\diamond}$ encode the placement of the bundling pipes, which will be determined later in Section 
\ref{sec:dodging}.

\begin{definition}[\bf Strongly anisotropic cutoff]\label{def:etab}
For a given vector direction $\xi$, the set of strongly anisotropic cutoffs $\{ \etab_{\xi}^I\}_I$ is
a partition of unity which divides the orthogonal space $\xi^\perp \in \T^3$ into a grid of squares $\mathcal{S}_{I}=\mathcal{S}_{I, \xi}$ of side-length $\approx\lambda_{q+\half}^{-1}$.  Each cutoff satisfies
    \begin{align}\label{eq:sa:summability}
        (\xi\cdot\na) \etab_\xi =0 \, , \quad  \left\| \nabla^N \etab_\xi^I \right\|_\infty \lesssim \lambda_{q+\half}^N \, \forall N\leq 3\Nfin \, , \quad
        \sum_{I} (\etab_\xi^{I})^{6} = 1 \, , \quad  \etab_\xi^{I} =\begin{cases}
        1 &\textnormal{  on  } \frac34 \mathcal{S}_{I}\\
        0 &\textnormal{  outside  } \frac54 \mathcal{S}_{I}
        \end{cases} \, .
    \end{align}
\end{definition}

We now recall the $L^2$-normalized intermittent pipe flows used for the Reynolds corrector $w_{q+1,R}$.

\begin{proposition}[\bf Intermittent pipe flows for Reynolds corrector]
\label{prop:pipeconstruction}
The intermittent pipe flows $\mathcal{W}_{\xi,\la, r}^k$ for Reynold correctors defined in \cite[Proposition~4.5]{GKN23} satisfy the following properties.
\begin{enumerate}[(1)]
\item\label{item:pipe:1}
     The flow $\mathcal{W}_{\xi,\la, r}^k$ has the form $\xi \varrho^k_{\xi,\lambda,r}$ for some vector $\xi\in \mathbb{Q}^3\cap \mathbb{S}^2$ and a scalar function $\varrho^k_{\xi,\lambda,r}$. There exists a pairwise symmetric tensor potential $\vartheta_{\xi, \la, r}^k : \mathbb{R}^2\rightarrow\mathbb{R}^{2^\dpot}$ such that $\varrho^k_{\xi,\lambda,r}=\la^{-\dpot}\div^D(\vartheta_{\xi, \la, r}^k)$.
     
     
     Then there exists $\mathcal{U}^k_{\xi,\lambda,r}:\mathbb{T}^3\rightarrow\mathbb{R}^3$ such that
    if $\{\xi,\xi',\xi''\} \subset \mathbb{Q}^3 \cap \mathbb{S}^2$ form an orthonormal basis of $\R^3$ with $\xi\times\xi'=\xi''$, then we have\footnote{The double index $ii$ indicates that $\div^{\Dpot-2} \left(\vartheta_{\xi,\lambda,r}^k \right)$ is a $2$-tensor, and we are summing over the diagonal components. The factor of $\sfrac 13$ appears because each component on the diagonal of this $3\times 3$ matrix is $\Delta^{-1} \varrho_{\xi,\lambda,r}^{k}$. The formula then follows from the identity $\curl \curl = -\Delta$ for divergence-free vector fields.}
    \begin{equation}
    \mathcal{U}_{\xi,\lambda,r}^k
     =  - \frac 13  \xi' \underbrace{\lambda^{-\Dpot} \xi''\cdot \nabla \left(\div^{\Dpot-2} \left(\vartheta_{\xi,\lambda,r}^k \right)\right)^{ii}}_{=:\varphi_{\xi,\lambda,r}^{\prime \prime k}}
     +  \frac 13
     \xi'' \underbrace{\lambda^{-\Dpot} \xi'\cdot \nabla \left(\div^{\Dpot-2} \left(\vartheta_{\xi,\lambda,r}^k \right)\right)^{ii}
     }_{=:\varphi_{\xi,\lambda,r}^{\prime k}}
     \label{eq:UU:explicit}
        \,, 
    \end{equation}
and thus
\begin{equation}
\label{eq:WW:explicit}
\curl \mathcal{U}^k_{\xi,\lambda,r} = \xi \lambda^{-\Dpot }\div^\Dpot  \left(\vartheta^k_{\xi,\lambda,r}\right) = \xi \varrho^k_{\xi,\lambda,r} =: \mathcal{W}^k_{\xi,\lambda,r}
\,,
\end{equation}
and 
\begin{equation}
    \xi \cdot \nabla \vartheta_{\xi,\lambda,r} =  (\xi \cdot \nabla) \mathcal{W}^k_{\xi,\lambda,r} 
    = (\xi \cdot \nabla )\mathcal{U}^k_{\xi,\lambda,r}
    = 0
    \,.
    \label{eq:derivative:along:pipe}
\end{equation}
    \item\label{item:pipe:2} The sets of functions $\{\varrho_{\xi,\lambda,r}^k\}_{k}$ and $\{\vartheta_{\xi,\lambda,r}^k\}_{k}$ are simultaneously $\left(\frac{\mathbb{T}^3}{\lambda r}  \right)$-periodic and $\left(\frac{\Tthreexi}{\lambda r n_\ast}  \right)$-periodic. 
    Each function in these sets has the support contained in 
    a pipe (cylinder) centered around a 
    line parallel to $\xi$ with the cross-sectional radius $\pi(4\la n_*)^{-1}$. Dividing a face of $\left(\frac{\Tthreexi}{\lambda r n_\ast}  \right)$ perpendicular to $\xi$ into a grid of squares with side-length $2\pi(4\la n_*)^{-1}$, there are $r^{-2}$ possible placements of the cylinder, indexed by $k$. In particular, we have $\supp \varrho_{\xi,\lambda,r}^k\cap \supp\varrho_{\xi,\lambda,r}^{k'}=\emptyset$ when $k\neq k'$ and same holds for $\vartheta_{\xi,\lambda,r}^k$.
    
    \item\label{item:pipe:3} $\mathcal{W}^k_{\xi,\lambda,r}$ is a stationary, pressureless solution to the Euler equations.
    \item\label{item:pipe:4} $\displaystyle{\dashint_{\mathbb{T}^3} \mathcal{W}^k_{\xi,\lambda,r} \otimes \mathcal{W}^k_{\xi,\lambda,r} = \xi \otimes \xi }$.
    \item\label{item:pipe:means}
    $\displaystyle{\dashint_{\mathbb{T}^3} |\mathcal{W}^k_{\xi,\lambda,r}|^2 \mathcal{W}^k_{\xi,\lambda,r} = \dashint_{\T^3} (\varrho_{\xi,\lambda,r}^k)^2 \mathcal{U}^k_{\xi,\lambda,r} = \int_{\T^3} \varrho^k_{\xi,\lambda,r} \mathcal{U}_{\xi,\lambda,r}^k = 0 \, . }$
    \item\label{item:pipe:5} For all $n\leq 3 \Nfin$, 
    \begin{equation}\label{e:pipe:estimates:1}
    {\left\| \nabla^n\vartheta^k_{\xi,\lambda,r} \right\|_{L^p(\mathbb{T}^3)} \lesssim \lambda^{n}r^{\left(\frac{2}{p}-1\right)} }, \qquad {\left\| \nabla^n\varrho^k_{\xi,\lambda,r} \right\|_{L^p(\mathbb{T}^3)} \lesssim \lambda^{n}r^{\left(\frac{2}{p}-1\right)} }
    \end{equation}
    and
    \begin{equation}\label{e:pipe:estimates:2}
    {\left\| \nabla^n\mathcal{U}^k_{\xi,\lambda,r} \right\|_{L^p(\mathbb{T}^3)} \lesssim \lambda^{n-1}r^{\left(\frac{2}{p}-1\right)} }, \qquad {\left\| \nabla^n\mathcal{W}^k_{\xi,\lambda,r} \right\|_{L^p(\mathbb{T}^3)} \lesssim \lambda^{n}r^{\left(\frac{2}{p}-1\right)} }.
    \end{equation}
    \item\label{item:pipe:3.5} We have that $\supp \vartheta_{\xi,\lambda,r}^k \subseteq B\left( \supp\varrho_{\xi,\lambda,r} ,2\lambda^{-1}\right)$.
    \item\label{item:pipe:6} Let $\Phi:\mathbb{T}^3\times[0,T]\rightarrow \mathbb{T}^3$ be the periodic solution to the transport equation
\begin{align}
\label{e:phi:transport}
\partial_t \Phi + v\cdot\nabla \Phi =0\,, 
\qquad 
\Phi|_{t=t_0} &= x\, ,
\end{align}
with a smooth, divergence-free, periodic velocity field $v$. Then
\begin{equation}\label{eq:pipes:flowed:1}
\nabla \Phi^{-1} \cdot \left( \mathcal{W}^k_{\xi,\lambda,r} \circ \Phi \right) = \curl \left( \nabla\Phi^T \cdot \left( \mathcal{U}^k_{\xi,\lambda,r} \circ \Phi \right) \right).
\end{equation}
\end{enumerate}
\end{proposition}

We will choose $\WW_{(\xi), \diamond}^I:=\mathcal{W}^{m'}_{\xi,\lambda_{q+\bn},\sfrac{\lambda_{q+\half}\Gamma_q}{\lambda_{q+\bn}}}$ for some $m'=m'_{(\xi),\diamond,I}$. The indices $m'=m'_{(\xi),\diamond,I}$ 
stand for the placement of the pipes, which will be determined later in Section 
\ref{sec:dodging}.  Next, we introduce the $L^3$-normalized intermittent pipe flows.

\begin{proposition}[\bf Intermittent pipe flows for current corrector]
\label{prop:pipe.flow.current}
The intermittent pipe flows $\mathcal{W}_{\xi,\la, r}^k$ for current correctors defined in \cite[Proposition~4.6]{GKN23} satisfy the same properties as in Proposition \ref{prop:pipeconstruction}, but item~\ref{item:pipe:4} is not imposed, and items~\ref{item:pipe:means}--\ref{item:pipe:5} are replaced by
\begin{enumerate}[(1)]
\setcounter{enumi}{4}
    \item\label{item:pipe:means:current}
    $\displaystyle{\dashint_{\mathbb{T}^3} |\mathcal{W}^k_{\xi,\lambda,r}|^2 \mathcal{W}^k_{\xi,\lambda,r} = |\xi|^2\xi} \, $, \quad $\displaystyle \dashint_{\T^3} (\varrho_{\xi,\lambda,r}^k)^2 \mathcal{U}^k_{\xi,\lambda,r} = \dashint_{\T^3} \varrho^k_{\xi,\lambda,r} \mathcal{U}^k_{\xi,\lambda,r} = 0$.
    \item\label{item:pipe:5:current} For all $n\leq 3 \Nfin$, 
    \begin{equation}\label{e:pipe:estimates:1:current}
    {\left\| \nabla^n\vartheta^k_{\xi,\lambda,r} \right\|_{L^p(\mathbb{T}^3)} \lesssim \lambda^{n} r^{\left(\frac{2}{p}-\frac23\right)} }, \qquad {\left\| \nabla^n\varrho^k_{\xi,\lambda,r} \right\|_{L^p(\mathbb{T}^3)} \lesssim \lambda^{n}r^{\left(\frac{2}{p}-\frac23\right)} }
    \end{equation}
    and
    \begin{equation}\label{e:pipe:estimates:2:current}
    {\left\| \nabla^n\mathcal{U}^k_{\xi,\lambda,r} \right\|_{L^p(\mathbb{T}^3)} \lesssim \lambda^{n-1}r^{\left(\frac{2}{p}-\frac23\right)} }, \qquad {\left\| \nabla^n\mathcal{W}^k_{\xi,\lambda,r} \right\|_{L^p(\mathbb{T}^3)} \lesssim \lambda^{n}r^{\left(\frac{2}{p}-\frac23\right)} } \, .
    \end{equation}
\end{enumerate}
\end{proposition}

\subsection{Cutoff functions}\label{subsec:result:par1:2}
We introduce and recall basic facts about time cutoffs, flow maps, intermittent pressure cutoffs, mildly and strongly anisotropic checkerboard cutoff functions, velocity cutoff functions, and cutoff aggregation lemmas.  Further details are contained in each subsubsection.

\subsubsection{Time cutoffs}
\label{sec:cutoff:temporal:definitions}
We introduce two collections of basic temporal cutoffs,\index{$\chi_{i,k,q}$}\index{$\tilde\chi_{i,k,q}$} which will be useful for estimating flow maps defined on various regions of space-time.  We refer the reader to \cite[subsection~5.1]{GKN23}, for further details (although these are the simplest cutoffs contained in this manuscript and do not require sophisticated definitions). The first set of temporal cutoff functions $\chi_{i,k,q}(t)$ satisfies
\begin{subequations}\label{eq:chi}
\begin{align}
t\in \supp \chi_{i,k,q} \iff& t\in \left[ (k-1)\sfrac 12 \tau_q \Gamma_{q}^{-i-2}, (k+1) \sfrac 12 \tau_q \Gamma_{q}^{-i-2} \right]  \, , \label{eq:chi:support}\\
|\partial_t^m \chi_{i,k,q}| &\les (\Gamma_{q}^{i+2} \tau_{q}^{-1})^m \qquad \textnormal{for }m\geq 0 \, ,
\label{eq:chi:cut:dt} \\
\chi_{i,k_1,q}(t)\chi_{i,k_2,q}(t) &\equiv 0 \qquad \qquad\textnormal{unless }|k_1-k_2|\leq 1 \,  \label{e:chi:overlap}\\
 \sum_{k \in \Z} \chi_{i,k,q}^6 &\equiv 1 \, .
 \label{eq:chi:cut:partition:unity}
\end{align}
\end{subequations}
The second set of temporal cutoff functions $\tilde\chi_{i,k,q}$ has a slightly longer timescale and is only needed for technical reasons; see Lemma~\ref{lem:dodging}.  They are supported in the set of times $t$ satisfying
\begin{equation}\label{eq:chi:tilde:support}
\left| t-\tau_q \Gamma_{q}^{-i-2} k \right| \leq \tau_q\Gamma_{q}^{-i} \,.
\end{equation}
In addition, if $(i,k)$ and $(\istar,\kstar)$ are such that $\supp \chi_{i,k,q} \cap \supp \chi_{\istar,\kstar,q}\neq\emptyset$ and $\istar\in\{i-1,i,i+1\}$, then as a consequence of these definitions and a sufficiently large choice of $\lambda_0$, we have that $\supp \chi_{i,k,q} \subset \supp \tilde\chi_{\istar,\kstar,q}$.

\subsubsection{Velocity cutoffs}\index{$\psi_{i,q}$}\index{velocity cutoffs}  Recall from subsection~\ref{sec:cutoff:inductive} that one should heuristically think of $\psi_{i,q}$ as the characteristic function of the set where $|\nabla \hat u_{q}| \approx \tau_{q}^{-1}\Ga_{q}^i$. The most important properties of the velocity cutoffs are the partition of unity property in \eqref{eq:inductive:partition}, the maximum value of $i$ in \eqref{eq:imax:upper:lower}, the derivative bounds for the cutoffs in~\eqref{eq:nasty:Dt:psi:i:q:orangutan}, and the estimate~\eqref{eq:psi:q:q'}, which controls a weighted sum of velocity cutoffs (and by extension $\nabla \hat u_q$ itself) by a pressure increment. We will construct the new velocity cutoffs $\psi_{i,\qbn}$ in Proposition~\ref{prop:verified:vel:cutoff} after constructing the velocity increment in Definition~\ref{def:wqbn}.

\subsubsection{Estimates on flow maps}\label{s:deformation}
\label{sec:cutoff:flow:maps}

The first lemma and corollary in this subsubsection assert that forwards and backwards Lagrangian trajectories of the velocity field $ \hat u_{q'}$ restricted to the local Lipschitz timescale determined by the cutoff functions stay contained in the support of a few neighboring velocity cutoff functions. The results are discussed further in \cite[Section~5.2]{GKN23}.

\begin{lemma}[\bf Lagrangian paths don't jump many supports]
\label{lem:dornfelder}
Let $q'\leq q+\bn-1$ and $(x_0,t_0)$ be given. Assume that the index $i$ is such that $\psi_{i,q'}^2(x_0,t_0) \geq \kappa^2$, where $\kappa\in\left[\frac{1}{16},1\right]$. Then the forward flow $(X(t),t) := (X(x_0,t_0;t),t)$ of the velocity field $\hat u_{q'}$ originating at $(x_0,t_0)$ has the property that $\psi_{i,q'}^2(X(t),t) \geq\sfrac{\kappa^2}{2}$ for all $t$ such that $|t - t_0|\leq \tau_{q'} \Gamma_{q'}^{-i+4}$.
\end{lemma}

\begin{corollary}[\bf Backwards Lagrangian paths don't jump many supports]
\label{cor:dornfelder}
Suppose $(x_0,t_0)$ is such that $\psi^2_{i,q'}(x_0,t_0)\geq \kappa^2$, where $\kappa\in\left[\sfrac{1}{16},1\right]$. For $\abs{t-t_0}\leq\tau_{q'}\Gamma_{q'}^{-i+3}$, define $x$ to satisfy $x_0=X(x,t;t_0)$. Then we have that $\psi_{i,q'}(x,t)\neq 0$.
\end{corollary}

We can now define and estimate the flow\index{flow maps}\index{$\Phi_{i,k,q}$, $\Phi_{(i,k)}$} of the vector field $\hat u_{q'}$ for $q'\leq q+\bn-1$ on the support of a velocity and time cutoff function. These estimates are also contained in \cite[subsection~5.2]{GKN23} and are proven in a similar manner to all such estimates for convex integration schemes.
\begin{definition}[\bf Flow maps]\label{def:transport:maps} We define $\Phi_{i,k,q'}(x,t)=\Phi_{(i,k)}(x,t)$ by
\begin{equation}\label{e:Phi}
\left\{\begin{array}{l}
(\partial_t + \hat u_{q'} \cdot\nabla) \Phi_{i,k,q'} = 0 \\
\Phi_{i,k,q'}(x,k{\tau_{q'}}\Gamma_{q'}^{-i-2})=x\, .
\end{array}\right.
\end{equation}
\end{definition}
\noindent We denote the inverse of the gradient $D\Phi_{(i,k)}$ by $(D\Phi_{(i,k)})^{-1}$, in contrast to $D\Phi_{(i,k)}^{-1}$, which is the gradient of $\Phi_{(i,k)}^{-1}$.

\begin{corollary}[\bf Deformation bounds]
\label{cor:deformation}
For $k \in \Z$, $0 \leq i \leq  i_{\rm max}$, $q'\leq q+\bn-1$, and $2 \leq N \leq \sfrac{3\Nfin}{2}+1$, we have the following bounds on the set $\Omega_{i,q',k}:=\supp \psi_{i,q'}(x,t){\tilde\chi_{i,k,q'}(t)}$.
\begin{subequations}
\begin{align}
\norm{D\Phi_{(i,k)} - {\rm Id}}_{L^\infty(\Omega_{i,q',k})} + \norm{(D\Phi_{(i,k)})^{-1} - {\rm Id}}_{L^\infty(\Omega_{i,q',k})} &\lesssim \Gamma_{q'}^{-1}
\label{eq:Lagrangian:Jacobian:1}\\
\norm{D^N\Phi_{(i,k)} }_{L^\infty(\Omega_{i,q',k})} + \norm{D^N\Phi^{-1}_{(i,k)} }_{L^\infty(\Omega_{i,q',k})} + \norm{D^{N-1}\left((D\Phi_{(i,k)})^{-1}\right) }_{L^\infty(\Omega_{i,q',k})} & \lesssim \Gamma_{q'}^{-1} (\lambda_{q'}\Ga_{q'})^{N-1} \label{eq:Lagrangian:Jacobian:2}
\end{align}
\end{subequations}
Furthermore, we have the following bounds for $1\leq N+M\leq \sfrac{3\Nfin}{2}$ and $0\leq N'\leq N$:
\begin{subequations}
\begin{align}
\left\| D^{N-N'} D_{t,q'}^M D^{N'+1} \Phi_{(i,k)} \right\|_{L^\infty(\Omega_{i,q',k})} &\leq  (\lambda_{q'}\Ga_{q'})^{N} \MM{M,\NindSmall,\Gamma_{q'}^{i} \tau_{q'}^{-1},\Tau_{q'-1}^{-1}\Gamma_{q'-1}}\label{eq:Lagrangian:Jacobian:5}\\
\left\| D^{N-N'} D_{t,q'}^M D^{N'} (D \Phi_{(i,k)})^{-1} \right\|_{L^\infty(\Omega_{i,q',k})} &\leq (\lambda_{q'}\Ga_{q'})^{N} \MM{M,\NindSmall,\Gamma_{q'}^{i} \tau_{q'}^{-1},{\Tau}_{q'-1}^{-1}\Gamma_{q'-1}} \, . \label{eq:Lagrangian:Jacobian:6}
\end{align}
\end{subequations}
\end{corollary}

\subsubsection{Intermittent pressure cutoffs}\label{sec:cutoff:stress}

In this subsubsection, we record the main properties of the cutoff functions for the intermittent pressure $\pi_\ell$\index{intermittent pressure} obtained from mollifying $\pi_q^q$ (see the mollification estimates in Lemma~\ref{lem:upgrading}). These cutoff functions are rigorously defined and estimated in \cite[Section~5.3]{GKN23}. The definition is essentially a level-set chopping of $\pi_\ell$, which we now discuss.  We shall need estimates for derivatives of error terms on the support of a particular cutoff, and for the derivatives of the cutoff itself; such estimates follow from the pointwise nature of the bounds in \eqref{eq:ind:pi:by:pi} and \eqref{eq:inductive:pointwise}. These pointwise bounds follow from the construction of the intermittent pressure; see the pointwise bounds in Proposition~\ref{lem.pr.invdiv2.c}, \eqref{est.S.by.pr.final2:c} and \eqref{est.S.prbypr.pt:c}.  These bounds in turn follow from the definition of the pressure increment $\sigma_H$; for example let us consider a Reynolds stress error term $H=G \rho\circ\Phi$.  Here $G$ oscillates at a low spatial frequency $\la_G$ and a Lagrangian frequency $\nu$ (corresponding to a material derivative operator $D_t$) and has $L^p$ norm size $\const_{G,p}$ for $p=\sfrac 32, \infty$, $\Phi$ is a flow map, and $\rho$ oscillates at a high frequency $\la_\rho$ and has $L^{p}$ norm size $\const_{\rho,p}$ for $p=\sfrac 32,\infty$: \index{pressure increment}\index{$\NcutSmall$, $\NcutLarge$}
\begin{subequations}\label{eq:fried:egg:evening}
\begin{align}
\pr(\rho) := \left( 
 \left(\const_{\rho,\infty} \Gamma^{-\NcutSmall}\right)^2 + \sum_{N=0}^{\NcutLarge} (\la_\rho \Ga)^{-2N} |D^N \rho|^2 \right)^\frac12 - \const_{\rho,\infty} \Gamma^{-\NcutSmall} \, , \\
\pr(G) := \left( \left( \const_{G,\infty} \Gamma^{-\NcutSmall} \right) + \sum_{N=0}^{\NcutLarge}\sum_{M=0}^{\NcutSmall} (\la_G \Ga)^{-2N}  (\nu\Ga) ^{-2M} |D^N D_{t}^M G|^2 \right)^\frac12 - \const_{G,\infty} \Gamma^{-\NcutSmall}  \, , \\
\sigma_H = \pr(G) \left(  \pr(\rho) - \int_{\T^3} \pr(\rho) \right) \circ \Phi =: \sigma_H^+ - \sigma_H^- \, .
\end{align}
\end{subequations}
We pause to highlight a few aspects of this definition.  First, $\pr(\rho)$ and $\pr(G)$ are supported on the supports of $\rho$ and $G$, respectively, and scale linearly with $\rho$ and $G$.\footnote{Stress errors and pressure errors have the same scaling.  Current errors, however, scale like $(\textnormal{pressure})^{\sfrac 32}$.  Therefore to make a pressure increment for a current error, one uses $L^1$ norms instead of $L^{\sfrac 32}$ norms and then raises the sum to the $\sfrac 13$ power, and the $\const_{\rho,1}$ to the $\sfrac 23$ power, and similarly for $G$.}  Second, note that $\sigma_H$ is effectively mean zero, so that we may invert the divergence.  Next, we point out that the portion $\sigma_H^-$ depends only on the mean of $\pr(\rho)$ but not any further properties of $\rho$, and is therefore effectively low frequency.  Therefore, we may actually dominate $\sigma_H^-$ by \emph{anticipated} pressure, i.e. rescaled versions of old, lower frequency, intermittent pressure; see the bounds in~\eqref{est.pr.vel.inc-}, \eqref{heatsie}, \eqref{eq:ct.p.4}, and \eqref{eq:co.p.4}.  Finally, $\sigma_H^+$ incorporates a weighted sum of at most $\NcutSmall$ material derivatives and $\NcutLarge$ spatial derivatives, where the weight uses a small factor $\Gamma$ (typically $\Ga_q$).  The effect of $\Gamma$ in the sum is that the terms for $N$ and $M$ very large are actually quite \emph{small}, since $G$ for example oscillates at frequency $\lambda_G$, which is less than $\lambda_G \Gamma$. This implies that in fact for $N$ and $M$ larger than $\NcutLarge$ and $\NcutSmall$,
$$  \left| D^N D_t^M H \right| \les \left(  \delta_{\rm tiny} + \sigma_H^+ \right) (\lambda_G \Gamma)^N (\nu \Gamma)^M \, , \qquad \qquad  \mbox{where} \quad  \delta_{\rm tiny} \ll 1 \, . $$
That is, $\sigma_H^+$ dominates $H$ and essentially \emph{all} of its derivatives using a slightly weakened multiplicative derivative cost, up to a negligible constant $\delta_{\rm tiny}$. In a similar way, we can also obtain
\begin{align*}
|D^ND_t^M \si_H^+| \lec  \left(  \delta_{\rm tiny} + \sigma_H^+ \right) (\lambda_G \Gamma)^N (\nu \Gamma)^M\,  ,
\end{align*}
which implies that $\sigma_H^+$ dominates essentially all of its derivatives.

We now recall from \cite[Section~5.3]{GKN23} the key properties of the intermittent stress cutoffs. The properties below include a partition of unity property, a comparability estimate to $\pi_\ell$ itself, control on the maximum index $\jmax$\index{$j$ and $\jmax$} used in the cutoffs, derivative estimates, and Lebesgue norms.

\begin{lemma}[\bf Intermittent pressure cutoffs]\label{lem:D:Dt:Rn:sharp}
There exist cutoff functions\index{$\omega_{j,q}$} $\{\omega_{j,q}^{{6}} \}_{j\geq 0}$ satisfying the following properties.
\begin{enumerate}[(i)]
    \item We have that
    \begin{align}
\sum_{j\geq 0} \omega_{j,q}^{{6}}(t,x) \equiv 1 \, , \qquad \omega_{j,q} \omega_{j',q} \equiv 0 \quad \textnormal{if} \quad |j-j'| >1  \, .
\label{eq:omega:cut:partition:unity}
\end{align}
\item We have the comparability properties
\begin{subequations}
\begin{align}
\sfrac 14 \delta_{q+\bn} \Ga_q^{2j} &\leq {\bf 1}_{\supp (\omega_{j,q})} \pi_\ell \label{pt.est.pi.lem.lower} \\
\sfrac 18 \sum_j \omega_{j,q} \delta_{q+\bn} \Gamma_q^{2j} &\leq \pi_\ell \, .
\label{pt.est.pi.cutoff}
\end{align}
\end{subequations}
\item  There exists $\jmax = \jmax(q) = \inf \left\{ j \, : \, \frac 14 \Ga_q^{2j} \delta_{q+\bn} \geq \Ga_q^{3+\badshaq} \right\}$ which is bounded independently of $q$, and
\begin{align}\label{eq:omega:j:is:zero}
\omega_{j,q} \equiv 0 \qquad \mbox{for all} \qquad j > \jmax \, .
\end{align}
Moreover, we have the bound
\begin{align}\label{ineq:jmax:use}
\Gamma_{q}^{2j_{\rm max}} \leq \de_{q+\bn}^{-1} \Gamma_q^{\badshaq+6} \, .
\end{align}
\item For $q\geq 0$, $0 \leq i \leq \imax$, $0 \leq j \leq \jmax$, and $N + M \leq \Nfin$, we have
\begin{align}
\frac{{\bf 1}_{\supp \psi_{i,q}} |D^N D_{t,q}^M \omega_{j,q}|}{\omega_{j,q}^{1-(N+M)/\Nfin}} 
\les (\Gamma_{q}^5 \Lambda_q)^N \MM{M, \Nindt,\Gamma_{q}^{i+4} \tau_{q}^{-1}, {\Tau}_{q}^{-1}} \, .
\label{eq:D:Dt:omega:sharp}
\end{align}
\item For any $r \geq \sfrac 32$ and $0\leq j \leq \jmax$, we have that 
\begin{align}
\norm{\omega_{j,q}}_{L^r}  \lesssim  \Gamma_{q}^{\frac{3(1-j)}{r}} \, .
\label{eq:omega:support}
\end{align}
\end{enumerate}
\end{lemma}

\subsubsection{Mildly and strongly anisotropic checkerboard cutoffs}
\label{sec:cutoff:checkerboard:definitions}

We now introduce the mildly and strongly anisotropic cutoffs, which localize the support of the intermittent Mikado bundles defined in \eqref{int:pipe:bundle:short}. These cutoffs are always composed with flow maps $\Phi_{i,k,q}$, which ensures that the cutoffs retain their general shape and dimensions on the local Lipschitz timescale.  In addition, a number of summability properties are imposed, which help ensure the quadratic and cubic cancellation properties and more; see subsection~\ref{ss:ssO:current}. Finally, the cutoffs satisfy spatial and material derivative estimates.  Of particular importance is that the anisotropic shape of the cutoffs is adapted to a particular bundle $\BB_{\pxi,\diamond}$, so that the spatial derivative in the (flowed) $\xi$ direction satisfies a better estimate.  This is crucial in subsection~\ref{ss:ssO:current}, for example; see \eqref{eq:osc:c:eq:1}, in which the divergence does \emph{not} land on the costliest objects, due to the fact that the derivative is in the (flowed) $\xi$ direction.  The anisotropy also helps limit the size of the support on which a bundle exists, aiding in the placement; see section~\ref{sec:dodging}. The proofs of these two lemmas are straightforward; we refer to \cite[Lemmas~5.14, 5.15]{GKN23} for further details.\index{$\zeta_{q,\diamond,i,k,\xi,\vecl}$}\index{$\zetab_{\xi}^I$}

\begin{lemma}[\bf Mildly anistropic checkerboard cutoffs]\label{lem:checkerboard:estimates}
Given $q$, $\xi\in\Xi\cup \Xi'$, $i\leq \imax$, and $k\in\mathbb{Z}$, there exist cutoff functions\index{$\vecl$}
\begin{equation}\label{eq:checkerboard:definition}
    \zeta_{q,\diamond,i,k,{\xi},\vec{l}}\,(x,t) = \mathcal{X}_{q,\xi,\vec{l},\diamond}\left(\Phi_{i,k,q}(x,t)\right) \, ,
\end{equation}
which satisfy the following properties.
\begin{enumerate}[(i)]
    \item\label{item:checkeeee} At the time $\tau_{q}\Ga_{q}^{-i-2}$ at which $\Phi_{i,k,q}$ is the identity, the support of $\mathcal{X}_{q,\xi,\vec{l},\diamond}$ is contained in a rectangular prism of dimensions no larger than $\sfrac 34 \lambda_q^{-1}\Gamma_q^{-8}$ in the direction of $\xi_z$, and $\sfrac 34 \Gamma_{q}^5(\lambda_{q+1})^{-1}$ in the directions perpendicular to $\xi_z$.
    \item\label{item:check:1} The material derivative $\Dtq (\zeta_{q,\diamond,i,k,{\xi},\vec{l}})$ vanishes.  
    \item\label{item:check:2} We have the summability properties for all $(x,t)\in \T^3\times \R$; 
    \begin{subequations}\label{eq:checkerboard:partition}
    \begin{align}
    \sum_{\vec{l}} \bigl(\zeta_{q,R,i,k,{\xi},\vec{l}}\,(x,t)\bigr)^{2} &\equiv 1 \, , \label{eq:summy:summ:1} \\
    \sum_{\vecl \, : \, l=l_0} \zeta_{q,\vp,i,k,\xi,\vecl}^2(x,t) &\equiv \mathcal{X}_{q,\xi,l_0}^4(\Phi_{i,k,q}(x,t)) \, , \label{eq:summy:summ:2} \\
    \sum_{\vecl \, : \, l^\perp = l^\perp_0} \zeta_{q,\vp,i,k,\xi,\vecl}^3(x_1,x_2,x_3) &= \mathcal{X}_{q,\xi,l_0^\perp}^3(\Phi_{i,k,q}(x,t)) \, . \label{eq:summy:summ:3}
    \end{align}
    \end{subequations}
    Furthermore, at the time for which $\Phi_{i,k,q}$ is the identity, $\sum \mathbb{P}_{\neq 0} \left( \chi_{q,\xi,l_0^\perp}^3 \right)$ is a $\sfrac{\T^3}{\Ga_q^{-5}\la_{q+1}}$-periodic function which satisfies
    \begin{equation}\label{eq:summy:summ:0}
    \left \langle \sum_{l^\perp} \chi_{q,e_1,l^\perp}^3(x_2,x_3) \right \rangle = c_3 \, .
\end{equation}
Here and in the rest of the article, $\langle \cdot \rangle$\index{$\langle \cdot \rangle$} denotes the spatial average over the periodic domain $\T^3$. 
\item\label{item:check:3}  Let $A=(\nabla\Phi_{(i,k)})^{-1}$. Then for all $N_1+N_2+M\leq \sfrac{3\Nfin}{2}+1$,
    \begin{align}\label{eq:checkerboard:derivatives}
        \bigl\| D^{N_1} \Dtq^M  ({\xi^\ell A_\ell^j \partial_j} )^{N_2} \zeta_{q,\diamond,i,k,\xi,\vec{l}} \bigr\|_{L^\infty\left(\supp \psi_{i,q}\tilde\chi_{i,k,q} \right)} &\lesssim \left(\Gamma_{q}^{-5} \lambda_{q+1} \right)^{N_1} \left(\Gamma_q^8\lambda_{q}\right)^{N_2} \notag\\
        &\qquad \qquad \times \MM{M,\Nindt,\Gamma_{q}^{i}\tau_q^{-1},\Tau_q^{-1}\Gamma_{q}^{-1}} \, .
    \end{align}
\end{enumerate}
\end{lemma}

\begin{lemma}[\bf Strongly anisotropic checkerboard cutoff function]\label{lem:finer:checkerboard:estimates}
The cutoff functions $\etab_\xi^I\circ\Phiik$ defined using $\etab$ from Definition~\ref{def:etab} satisfy the following properties:
\begin{enumerate}[(1)]
    \item\label{item:check:check:1} The material derivative $\Dtq (\etab_\xi^I\circ \Phiik)$ vanishes.
    \item\label{item:check:check:2} For all fixed values of $q,i,k,\xi$, each $t\in\mathbb{R}$, and all $x=(x_1,x_2,x_3)\in\mathbb{T}^3$,
    \begin{equation}\label{eq:checkerboard:partition:check}
    \sum_{I} (\etab_\xi^I \circ \Phiik)^6(x,t) = 1 \, .
    \end{equation}
    \item\label{item:check:check:3} 
    Let $A=(\nabla\Phi_{(i,k)})^{-1}$.  Then for all $N_1+N_2+M\leq \sfrac{3\Nfin}{2}+1$, we have that
    \begin{align}\label{eq:checkerboard:derivatives:check}
        \bigl\| D^{N_1} \Dtq^M  ({\xi^\ell A_\ell^j \partial_j} )^{N_2} \etab_\xi^I \circ \Phiik \bigr\|_{L^\infty\left(\supp \psi_{i,q}\tilde\chi_{i,k,q} \right)} &\lesssim \lambda_{q+\lfloor \sfrac \bn 2 \rfloor}^{N_1}  \MM{M,\Nindt,\Gamma_{q}^{i}\tau_q^{-1},\Tau_q^{-1}\Gamma_{q}^{-1}} \, .
    \end{align}
    \item For fixed $q,i,k,\xi$, we have that
    \begin{align} \label{lem.cardinality}
    \# \left\{(\vecl,I) : \supp \left(\zeta_{q,i,k,\xi,\vecl} \, \etab^I_{\xi} \circ \Phi_{(i,k)}\right) \neq \emptyset \right\} \lec \Ga_q^8 \la_{q} \la_{q+\half}^2 \,.
\end{align}
\end{enumerate}
\end{lemma}

\subsubsection{Definition of the cumulative cutoff function}
\label{sec:cutoff:total:definitions}
With all the various cutoffs functions defined previously, we introduce the cumulative cutoff functions\index{cumulative cutoff function}\index{$\eta_{i,j,k,\xi,\vecl,\diamond}$}
\begin{align}
 \eta_{i,j,k,\xi,\vec{l},\diamond}\,(x,t) &= \psi_{i,q}^\diamond(x,t) \omega_{j,q}^\diamond(x,t) \chi_{i,k,q}^\diamond(t)\zeta_{q,\diamond,i,k,\xi,\vec{l}}\,(x,t) \label{def:cumulative:current}
  \, .
\end{align}
The $\diamond$ in the superscript is equal to $2$ if $\diamond=\varphi$ (so that they are cubic-summable to $1$) and $3$ if $\diamond=R$ (so that they are square-summable to $1$). We have the following $L^p$ estimates on $\eta_{i,j,k,\xi,\vecl,\diamond}$, which may be proven directly using \eqref{eq:omega:support} and \eqref{eq:psi:i:q:support:old}; we refer to \cite[Lemma~5.17]{GKN23} for details.

\begin{lemma}[\bf Cumulative support bounds for cutoff functions]\label{lemma:cumulative:cutoff:Lp}
For $r_1,r_2\in [1,\infty]$ with $\frac{1}{r_1}+\frac{1}{r_2}=1$ and any $0\leq i \leq \imax$, $0\leq j,\leq \jmax$, $\xi\in\Xi,\Xi'$, and $\diamond=\varphi,R$, we have that for each $t$,
\begin{align}
    \sum_{\vecl} \left| \supp_x \left( \eta_{i,j,k,\xi,\vecl,\diamond}(t,x) \right) \right| &\lesssim \Gamma_{q}^{\frac{-3i + \CLebesgue}{r_1} + \frac{-3j}{r_2}+3} \, . \label{eq:supp:cumul:varphi}
\end{align}
We furthermore have that
\begin{align}
    \sum_{i,j,k,\xi,\vecl,I,\diamond} \mathbf{1}_{\supp \eta_{i,j,k,\xi,\vecl,\diamond} \rhob_\pxi^\diamond \zetab_\xi^I} \approx \sum_{i,j,k,\xi,\vecl,\diamond} \mathbf{1}_{\supp \eta_{i,j,k,\xi,\vecl,\diamond} \rhob_\pxi^\diamond} \lesssim 1 \, . \label{eq:desert:cowboy:sum}
\end{align}
\end{lemma}

\subsubsection{Cutoff aggregation lemmas}\label{sss:aggregation}

We now introduce several ``aggregation lemmas.''  The basic idea of these lemmas is that since specific error terms include various cutoffs, such as $\psi_{i,q}$, $\omega_{j,q}$, $\zeta_{q,\diamond,i,k,\xi,\vecl}$, or $\zetab_{\xi}^I$, they must be estimated individually first.  This produces bounds which depend on indices $i,j,\vecl,I$.  However, by summing over such indices, we may \emph{aggregate} the localized estimates into global (and thus more useful) pointwise estimates and $L^p$ estimates. The proofs of these aggregation estimates are quite short and follow from direct computation; we refer the reader to \cite[Section~5.6]{GKN23} for details.
\begin{corollary}[\bf Aggregated $L^p$ estimates]\label{rem:summing:partition}
Let $\theta\in{(}0,3]$, and $\theta_1,\theta_2\geq 0$ with $\theta_1+\theta_2=\theta$. Let $H=H_{i,j,k,\xi,\vecl,\diamond}$ or $H=H_{i,j,k,\xi,\vecl,I,\diamond}$ be a function with 
\begin{align}
    \supp H_{i,j,k,\xi,\vecl,\diamond} \subseteq \supp \eta_{i,j,k,\xi,\vecl,\diamond} \qquad \textnormal{or} \qquad \supp H_{i,j,k,\xi,\vecl,I,\diamond} \subseteq \supp  \eta_{i,j,k,\xi,\vecl,\diamond} \zetab_{\xi}^{I,\diamond}\circ\Phiik \, . \label{eq:agg:assump:1}
\end{align}
Let $p\in[1,{\infty)}$ and let $\theta_1,\theta_2\in[0,3]$ be such that $\theta_1+\theta_2=\sfrac 3p$. Assume that there exists $\const_H,N_*,M_*,N_x,M_t$ and $\lambda,\Lambda,\tau,\Tau$ such that
\begin{subequations}
\begin{align}
  \left\| D^N \Dtq^M H_{i,j,k,\xi,\vecl,\diamond} \right\|_{L^p} &\lesssim \sup_{t\in\R} \left( \left| \supp_x \left( \eta_{i,j,k,\xi,\vecl,\diamond} (t,x) \right) \right|^{\sfrac 1p} \right) \notag\\
  &\qquad \qquad \times \const_H \Gamma_q^{\theta_1 i + \theta_2 j}  \MM{N,N_x,\lambda,\Lambda} \MM{M,M_t,\tau^{-1}\Gamma_q^i,\Tau^{-1}} \label{eq:agg:assump:2} \\
  \left\| D^N \Dtq^M H_{i,j,k,\xi,\vecl,I,\diamond} \right\|_{L^p} &\lesssim \sup_{t\in\R} \left( \left| \supp_x \left( \eta_{i,j,k,\xi,\vecl,\diamond} \zetab_\xi^{I,\diamond}\circ\Phiik  (t,x)\right) \right|^{\sfrac 1p} \right) \notag\\
  &\qquad \qquad \times \const_H \Gamma_q^{\theta_1 i + \theta_2 j} \MM{N,N_x,\lambda,\Lambda} \MM{M,M_t,\tau^{-1}\Gamma_q^i,\Tau^{-1}} \,  \label{eq:agg:assump:3}
\end{align}
\end{subequations}
for $N\leq N_*,M\leq M_*$. Then in the same range of $N$ and $M$,
\begin{subequations}
\begin{align}\label{eq:agg:conc:1}
 \left\| \psi_{i,q} \sum_{i',j,k,\xi,\vecl,\diamond} D^N \Dtq^M H_{i',j,k,\xi,\vecl,\diamond} \right\|_{L^p} &\lesssim \Gamma_q^{{3+\theta_1 \CLebesgue}} \const_H \MM{N,N_x,\lambda,\Lambda} \MM{M,M_t,\tau^{-1}\Gamma_q^{i+1},\Tau^{-1}} \\
 \label{eq:agg:conc:2}
 \left\| \psi_{i,q} \sum_{i',j,k,\xi,\vecl,I,\diamond} D^N \Dtq^M H_{i',j,k,\xi,\vecl,I,\diamond} \right\|_{L^p} &\lesssim \Gamma_q^{{3+\theta_1 \CLebesgue}} \const_H \MM{N,N_x,\lambda,\Lambda} \MM{M,M_t,\tau^{-1}\Gamma_q^{i+1},\Tau^{-1}} \, .
\end{align}
\end{subequations}
\end{corollary}

\begin{corollary}[\bf Aggregated pointwise estimates]\label{lem:agg.pt}
Let $H=H_{i,j,k,\xi,\vecl,\diamond}$ or $H=H_{i,j,k,\xi,\vecl,I,\diamond}$ be a function with 
\begin{align}
    \supp H_{i,j,k,\xi,\vecl,\diamond} \subseteq \supp \eta_{i,j,k,\xi,\vecl,\diamond} \qquad \textnormal{or} \qquad \supp H_{i,j,k,\xi,\vecl,I,\diamond} \subseteq \supp  \eta_{i,j,k,\xi,\vecl,\diamond} \zetab_{\xi}^{I,\diamond}\circ\Phiik \,  \label{eq:aggpt:assump:1}
\end{align}
and let $\varpi = \varpi_{i,j,k,\xi,\vecl,\diamond}$ or $\varpi = \theta_{i,j,k,\xi,\vecl,I,\diamond}$ be a {non-negative} function such that
\begin{align}
    \supp \varpi_{i,j,k,\xi,\vecl,\diamond} \subseteq \supp \eta_{i,j,k,\xi,\vecl,\diamond} \qquad \textnormal{or} \qquad \supp \varpi_{i,j,k,\xi,\vecl,I,\diamond} \subseteq \supp  \eta_{i,j,k,\xi,\vecl,\diamond} \zetab_{\xi}^{I,\diamond} \circ\Phiik\,  \label{eq:aggpt:assump:2}
\end{align}
Let $p\in({0},\infty)$ and assume that there exists $\lambda,\Lambda,\tau$ such that
\begin{subequations}
\begin{align}
  |D^N \Dtq H_{i,j,k,\xi,\vecl,\diamond}| &\lesssim \varpi_{i,j,k,\xi,\vecl,\diamond}^p \MM{N,N_x,\lambda,\Lambda} \MM{N,N_t,\tau^{-1}\Gamma_q^i,\Tau^{-1}} \, \label{eq:aggpt:assump:3} \\
  |D^N \Dtq H_{i,j,k,\xi,\vecl,I,\diamond}| &\lesssim \varpi_{i,j,k,\xi,\vecl,I,\diamond}^p \MM{N,N_x,\lambda,\Lambda} \MM{N,N_t,\tau^{-1}\Gamma_q^i,\Tau^{-1}} \, \label{eq:aggpt:assump:4} 
\end{align}
\end{subequations}
for $N\leq N_*,M\leq M_*$. Then in the same range of $N$ and $M$,
\begin{subequations}
\begin{align}\label{eq:aggpt:conc:1}
 \left|\psi_{i,q} \sum_{i',j,k,\xi,\vecl,\diamond} D^N \Dtq^M H_{i',j,k,\xi,\vecl,\diamond} \right| &\lesssim \left(\sum_{i,j,k,\xi,\vecl,\diamond}\varpi_{i,j,k,\xi,\vecl,\diamond}\right)^p \MM{N,N_x,\lambda,\Lambda} \MM{M,M_t,\tau^{-1}\Gamma_q^{i+1},\Tau^{-1}} \\
 \label{eq:aggpt:conc:2}
 \left| \psi_{i,q} \sum_{i',j,k,\xi,\vecl,I,\diamond} D^N \Dtq^M H_{i',j,k,\xi,\vecl,I,\diamond} \right| &\lesssim \left(\sum_{i,j,k,\xi,\vecl,I,\diamond}\varpi_{i,j,k,\xi,\vecl,I,\diamond}\right)^p \MM{N,N_x,\lambda,\Lambda} \MM{M,M_t,\tau^{-1}\Gamma_q^{i+1},\Tau^{-1}} \, .
\end{align}
\end{subequations}
\end{corollary}

\begin{corollary}[\bf Aggregated pointwise estimates with $\Ga_q^i$]\label{lem:agg.Dtq}
Let $H=H_{i,j,k,\xi,\vecl,I,\diamond}$ be a function with 
\begin{align}
     \supp H_{i,j,k,\xi,\vecl,\diamond} \subseteq \supp \eta_{i,j,k,\xi,\vecl,\diamond} \qquad \textnormal{or} \qquad \supp H_{i,j,k,\xi,\vecl,I,\diamond} \subseteq \supp  \eta_{i,j,k,\xi,\vecl,\diamond} \zetab_{\xi}^{I,\diamond}\circ\Phiik \,  \label{eq:aggDtq:assump:1}
\end{align}
and let $\varpi$ be a {non-negative} function and assume that there exists $\lambda,\Lambda,\tau,\Tau$ such that for $H = H_{i,j,k,\xi,\vecl,\diamond}$ or $H_{i,j,k,\xi,\vecl,I,\diamond}$
\begin{subequations}
\begin{align}
  \left|D^N \Dtq^M H\right| &\lesssim \tau_q^{-1} \Ga_q^i \psi_{i,q} \varpi \MM{N,N_x,\lambda,\Lambda} \MM{M,M_t,\tau^{-1}\Gamma_q^i,\Tau^{-1}} \, \label{eq:aggDtq:assump:4} 
\end{align}
\end{subequations}
for $N\leq N_*,M\leq M_*$. Then in the same range of $N$ and $M$,
\begin{subequations}
\begin{align}\label{eq:aggDtq:conc:1.0}
 \left|\psi_{i,q} \sum_{i',j,k,\xi,\vecl,\diamond} D^N \Dtq^M H_{i',j,k,\xi,\vecl,\diamond} \right| &\lesssim \Ga_q r_{q}^{-1} \la_q \left(\pi_q^q\right)^{\sfrac12} \varpi \MM{N,N_x,\lambda,\Lambda} \MM{M,M_t,\tau^{-1}\Gamma_q^{i+1},\Tau^{-1}} \\
\label{eq:aggDtq:conc:1}
 \left| \psi_{i,q} \sum_{i',j,k,\xi,\vecl,I,\diamond} D^N \Dtq^M H_{i',j,k,\xi,\vecl,I,\diamond} \right| &\lesssim \Ga_q r_{q}^{-1} \la_q \left(\pi_q^q\right)^{\sfrac12} \varpi \MM{N,N_x,\lambda,\Lambda} \MM{M,M_t,\tau^{-1}\Gamma_q^{i+1},\Tau^{-1}} \, .
\end{align}
\end{subequations}
\end{corollary}

\subsection{Premollified velocity increment, premollified velocity increment potential, pressure increment, velocity increment, and new velocity cutoffs}\label{sec:vel.inc}

We are now ready to define the premollified velocity increment $w_{q+1}$ and premollified velocity increment potential $\upsilon_{q+1}$,\index{$w_{q+1}$}\index{$\upsilon_{q+1}$} which we do in the first two subsubsections.  Then, we define the associated pressure increments and pressure current error.  Next, we define the velocity increment, which we denote with $\hat w_{\qbn}$, and the velocity increment potential, which we denote with $\hat \upsilon_{\qbn}$.\index{$\hat w_\qbn$}\index{$\hat \upsilon_\qbn$}  Finally, we define the new velocity cutoffs $\psi_{i,\qbn}$ and record a useful estimate for the velocity cutoffs in terms of the velocity pressure increment.

\subsubsection{Premollified velocity increment}
The premollified velocity increment $w_{q+1}$ will be divided into Reynolds and current correctors, and principal parts and divergence correctors:
\begin{align}\label{defn:w}
    w_{q+1} = w_{q+1,R} + w_{q+1,\varphi} \, , \qquad  w_{q+1}^{(p)} := w_{q+1,R}^{(p)} + w_{q+1,\varphi}^{(p)} \, , \qquad w_{q+1}^{(c)} := w_{q+1,R}^{(c)} + w_{q+1,\varphi}^{(c)} \, .
\end{align}
We will use the intermittent bundles and pipe flows from subsection~\ref{subsec:result:par1:1} to define $w_{q+1}$, which however requires choices of placement for both the bundling pipe $\rhob_\pxi^\diamond$ and $\WW_{\pxi,\diamond}^I$.  We delay this choice until Section~\ref{sec:dodging}, as many properties such as the basic estimates on the size of $w_{q+1}$ do not depend on the choices of placement. The principal portion $w_{q+1,\diamond}^{(p)}$ and divergence corrector portion $w_{q+1,\diamond}^{(c)}$, $\diamond=R,\ph$ by
\begin{subequations}\label{wqplusoneonediamond}
\begin{align}\label{wqplusoneonediamondp}
    w_{q+1,\diamond}^{(p)} = \sum_{i,j,k,\xi,\vecl,I} \underbrace{a_{(\xi),\diamond} \left(\chib_{(\xi)}^{\diamond} \etab^{I,\diamond}_\xi\right) \circ \Phiik  \curl \left( \nabla\Phi_{(i,k)}^T \mathbb{U}_{(\xi),\diamond}^I \circ \Phi_{(i,k)} \right)}_{=: w_{(\xi),\diamond}^{(p),I}} \, ,\\
    \label{wqplusoneonediamondc}
     w_{q+1,\diamond}^{(c)} = \sum_{i,j,k,\xi,\vecl,I} \underbrace{\nabla \left( a_{(\xi),\diamond} \left(\chib_{(\xi)}^\diamond \etab^{I,\ph}_\xi \right) \circ \Phiik \right)  \times 
    \left( \nabla\Phi_{(i,k)}^T \mathbb{U}_{(\xi),\diamond}^I \circ \Phi_{(i,k)} \right)}_{ =: w_{(\xi),\diamond}^{(c),I}}  \, .
\end{align}
\end{subequations}
The coefficient functions $a_{\pxi,\diamond}$\index{$\pxi$}\index{$a_{\pxi,\diamond}$} are defined by
\begin{align}
a_{\xi,i,j,k,\vecl,\ph} &= a_{(\xi),\ph} = 
\delta_{q+\bn}^{\sfrac 12}r_q^{-\sfrac13} \Gamma^{j-1}_{q} \psi_{i,q}^\ph \omega_{j,q}^{\varphi} \chi_{i,k,q}^\ph \zeta_{q,\varphi,i,k,\xi,\vecl}\,
|\na \Phi_{(i,k)}^{-1} \xi|^{-\sfrac23}
\td\gamma_{\xi}\left(\frac{\ph_{q,i,k}}{\delta_{q+\bn}^{\sfrac32}r_q^{-1}\Gamma^{3j-3}_{q}}\right) \, ,
\label{eq:a:xi:phi:def}\\
a_{\xi,i,j,k,\vecl,R}
&=a_{(\xi),R}=\delta_{q+\bn}^{\sfrac 12}\Gamma^{j-1}_{q} \psi_{i,q}^R \omega_{j,q}^R  \chi_{i,k,q}^R \zeta_{q,R,i,k,\xi,\vecl}\, \gamma_{\xi,\Gamma_q^9}\left(\frac{R_{q,i,k}}{\delta_{q+\bn}\Gamma_q^{2j-2}}\right) \, ,
\label{eq:a:xi:def}
\end{align}
where $\xi \in \Xi'_{q \textnormal{ mod } \bn}$ in~\eqref{eq:a:xi:phi:def}, and $\xi \in \Xi_{q \textnormal{ mod } \bn}$ in~\eqref{eq:a:xi:def}.
For convenience we shall often suppress the indices $\xi,i,j,k,\vecl$ with the shorthand notation $\pxi$. We also use the notation $f^\ph:= f^2$ and $f^R := f^3$ for $f=\ph_{i,q}, \om_{j,q}, \chi_{i,k,q}$, to ensure cubic and quadratic normalizations, respectively (recall that the sixth powers of cutoff functions are summable to one, cf.~\eqref{eq:omega:cut:partition:unity} and \eqref{eq:inductive:partition}). The functions $\tilde \gamma_\xi$ and $\gamma_{\xi,\Gamma_q^9}$ are defined in the geometric lemmas \cite[Propositions~4.1 and 4.2]{GKN23}, and they satisfy algebraic identities which enable the cubic and quadratic cancellations, respectively.  In the cubic case,
\begin{equation}\label{eq:alg:id}
    \frac12 \sum_{\xi} \td\gamma^3_{\xi}\left(\frac{\ph_{q,i,k}}{\delta_{q+\bn}^{\sfrac32}r_q^{-1}\Gamma^{3j-3}_{q}}\right) \xi = \frac{\ph_{q,i,k}}{\delta_{q+\bn}^{\sfrac32}r_q^{-1}\Gamma^{3j-3}_{q}}\, , 
\end{equation}
The coefficient functions satisfy the following estimates, which may be proved via the Fa'a di Bruno formula and direct computation.  We refer the reader to \cite[Lemma~6.4]{GKN23} for further details on these estimates.
\begin{lemma}[\bf Coefficient function estimates]
\label{lem:a_master_est_p}
For $N,N',N'',M$ with $N'',N'\in\{0,1\}$ and $N,M \leq {\sfrac{\Nfin} {3}}$, we have the following estimates.
\begin{subequations}\label{e:a_master_est_p}
\begin{align}
&\left\|D^{N-N''} D_{t,q}^M (\xi^\ell A_\ell^h \partial_h)^{N'} D^{N''} a_{\xi,i,j,k,\vec{l},\varphi}\right\|_{ r} \notag\\
&\qquad\lessg \left| \supp \eta_{i,j,k,\xi,\vecl,\varphi} \right|^{\sfrac 1r} \delta_{q+\bn}^{\sfrac 12} r_q^{-\sfrac 13} \Gamma^{j- {1}}_{q}  \left({\Gamma_{q}^{-5}\lambda_{q+1}}\right)^N  \left({\Gamma_q^{  5}\Lambda_{q}}\right)^{N'} \MM{M, \NindSmall, \tau_{q}^{-1}\Gamma_{q}^{i+4}, \Tau_{q}^{-1}}\label{e:a_master_est_p_phi} \, , \\
&\left\| D^{N-N''} D_{t,q}^M (\xi^\ell A_\ell^h \partial_h)^{N'} D^{N''} \left( a_{\xi,i,j,k,\vecl,\varphi} \left( \rhob_{(\xi)}^\varphi \zetab_\xi^{I,\varphi} \right) \circ \Phiik \right) \right\|_{ r}\notag\\
&\qquad\lessg \left| \supp \left( \eta_{i,j,k,\xi,\vecl,\varphi} \zetab_\xi^{I,\varphi} \right) \right|^{\sfrac 1r} \delta_{q+\bn}^{\sfrac 12} r_q^{-\sfrac 13} \Gamma_q^{j {+1}} \left({\lambda_{q+\lfloor \sfrac \bn 2 \rfloor}}\right)^N \left({\Gamma_q^{  5}\Lambda_{q}}\right)^{N'} \MM{M, \NindSmall, \tau_{q}^{-1}\Gamma_{q}^{i+4}, \Tau_{q}^{-1}}\label{e:a_master_est_p_phi:zeta}
\, ,\\
&\left\| D^{N-N''} D_{t,q}^M (\xi^\ell A_\ell^h \partial_h)^{N'} D^{N''} a_{\xi,i,j,k,\vecl,R}\right\|_{ r} \notag\\
&\qquad\lesssim \left| \supp \eta_{i,j,k,\xi,\vecl,R} \right|^{\sfrac 1r} \delta_{q+\bn}^{\sfrac 12} \Gamma_q^{j+ {4}} \left({\Gamma_{q}^{-5}\lambda_{q+1}}\right)^N \left({\Gamma_q^{ {13}}\Lambda_{q}}\right)^{N'} \MM{M, \NindSmall, \tau_{q}^{-1}\Gamma_{q}^{i+ {13}}, \Tau_{q}^{-1} {\Ga_q^8}}\label{e:a_master_est_p_R} \, , \\
&\left\| D^{N-N''} D_{t,q}^M (\xi^\ell A_\ell^h \partial_h)^{N'} D^{N''} \left( a_{\xi,i,j,k,\vecl,R} \left( \rhob_{(\xi)}^R \zetab_\xi^{I,R} \right) \circ \Phiik \right) \right\|_{ r}\notag\\
&\qquad\lessg \left| \supp \left( \eta_{i,j,k,\xi,\vecl,R} \zetab_\xi^{I,R} \right) \right|^{\sfrac 1r} \delta_{q+\bn}^{\sfrac 12} \Gamma_q^{j+ {7}} \left({\lambda_{q+\lfloor \sfrac \bn 2 \rfloor}}\right)^N \left({\Gamma_q^{ {13}}\Lambda_{q}}\right)^{N'} \MM{M, \NindSmall, \tau_{q}^{-1}\Gamma_{q}^{i+ {13}}, \Tau_{q}^{-1} {\Ga_q^8}}\label{e:a_master_est_p_R:zeta} \, .
\end{align}
\end{subequations}
In the case that $r=\infty$, the above estimates give that
\begin{subequations}
\begin{align}
\left\| D^{N-N''} D_{t,q}^M (\xi^\ell A_\ell^h \partial_h)^{N'} D^{N''} a_{\xi,i,j,k,\vec{l},R}\right\|_{ \infty} 
&\lessg {\Gamma_q^{\frac{\badshaq}{2}+ {7}}} \left(\Gamma_{q}^{-5}\lambda_{q+1}\right)^N \notag\\
&\qquad \qquad \times \left({\Gamma_q^{ {13}}\Lambda_{q}} \right)^{N'} \MM{M,\Nindt, \tau_{q}^{-1}\Gamma_{q}^{i+ {13}},\Tau_q^{-1} {\Ga_q^8}} \label{e:a_master_est_p_uniform_R} \, . \\
\left\| D^{N-N''} D_{t,q}^M (\xi^\ell A_\ell^h \partial_h)^{N'} D^{N''} a_{\xi,i,j,k,\vec{l},\varphi}\right\|_{ \infty} 
&\lessg {\Gamma_q^{\frac{\badshaq}{2}+ {2}}} r_q^{-\sfrac 13} \left(\Gamma_{q}^{-5}\lambda_{q+1}\right)^N \notag\\
&\qquad \qquad \times \left({\Gamma_q^8\Lambda_{q}} \right)^{N'} \MM{M,\Nindt, \tau_{q}^{-1}\Gamma_{q}^{i+4},\Tau_q^{-1}} \label{e:a_master_est_p_uniform_phi} \, ,
\end{align}
\end{subequations}
with analogous estimates (incorporating a loss of $\Ga_q^3$ for $\diamond=R$ and $\Ga_q^2$ for $\diamond=\varphi$) holding for the product $a_{(\xi),\diamond}\zetab_{\xi}^{I,\diamond}\rhob_{(\xi)}^\diamond$. Finally, we have the pointwise estimates
\begin{subequations}\label{e:a_master_est_p_pointwise}
\begin{align}
    \left| D^{N-N''} D_{t,q}^M (\xi^\ell A_\ell^h \partial_h)^{N'} D^{N''} a_{\xi,i,j,k,\vecl,R}\right| &\lesssim \Gamma_q^{12} \pi_\ell^{\sfrac 12} \left({\Gamma_{q}^{-5}\lambda_{q+1}}\right)^N \left({\Gamma_q^{ {13}}\Lambda_{q}}\right)^{N'} \MM{M, \NindSmall, \tau_{q}^{-1}\Gamma_{q}^{i+ {13}}, \Tau_{q}^{-1} {\Ga_q^8}} \label{e:a_master_est_p_R_pointwise} \\     \left| D^{N-N''} D_{t,q}^M (\xi^\ell A_\ell^h \partial_h)^{N'} D^{N''} a_{\xi,i,j,k,\vecl,\varphi}\right| &\lesssim \Gamma_q^{ {12}} \pi_\ell^{\sfrac 12} r_q^{-\sfrac{1}{3}} \left({\Gamma_{q}^{-5}\lambda_{q+1}}\right)^N \left({\Gamma_q^{ {5}}\Lambda_{q}}\right)^{N'} \MM{M, \NindSmall, \tau_{q}^{-1}\Gamma_{q}^{i+4}, \Tau_{q}^{-1}} \, . \label{e:a_master_est_p_phi_pointwise}
\end{align}
\end{subequations}
\end{lemma}

\subsubsection{Premollified velocity increment potentials}
We now introduce the premollified velocity increment potentials.  The basic idea is that the premollified velocity increment is defined as a product of high-frequency and low-frequency functions, and therefore can be written as the divergence of a potential which gains a factor of the high frequency.  In this case, the gain is $\la_{\qbn}^{-1}$, which comes from the thickness of the pipe. The premollified velocity increment potentials therefore satisfy a number of properties: their iterated divergence is equal to the premollified velocity increment, up to a small error which is negligible; they are supported in a neighborhood of the support of the premollified velocity increment itself; and each ``inverse divergence'' gains $\la_\qbn^{-1}$. We refer to \cite[Lemma~10.1]{GKN23} for more details.\index{$\upsilon_{q+1}$}  Note that item~\ref{item:eckel:1} gives that $\|w_{q+1}\|_3 \approx \delta_{\qbn}^{\sfrac 12} r_q^{-\sfrac 13}$ and $\|w_{q+1}\|_\infty\approx \Ga_q^{\sfrac{\badshaq}{2}}$, as expected.

\begin{lemma}[\bf Premollified velocity increment potential]\label{lem:rep:vel:inc:potential}
For a given $w_{q+1}^{(\texttt{l})}$, $\texttt{l}=p,c$, as in \eqref{defn:w}, there exists a tensor $\upsilon_{q+1}^{(\texttt{l})}$ and an error $e_{q+1}^{(\texttt{l})}$ such that the following hold.
\begin{enumerate}[(i)] 
    \item Let $\dpot$ be as in \eqref{i:par:10}. Then $w_{q+1}^{(\texttt{l})}$ can be written in terms of $\upsilon_{q+1}^{(\texttt{l})}$ and $e_{q+1}^{(\texttt{l})}$ as
     \begin{align}\label{exp.w.q+1}
w_{q+1}^{({p})} &= \div^{\dpot} \upsilon_{q+1}^{({p})} + e_{q+1}^{({p})}  \, , \qquad 
w_{q+1}^{({c})}= \div^{\dpot}(r_q\Ga_q^{-1} \upsilon_{q+1}^{({c})}) + r_q\Ga_q^{-1}e_{q+1}^{({c})} \, , 
\end{align}
or equivalently notated component-wise as $(w_{q+1}^{(p)})^\bullet = 
\pa_{i_1} \dots \pa_{i_\dpot} \upsilon_{q+1}^{(p, \bullet, i_1, \dots, i_\dpot)} + e_{q+1}^\bullet
$. 

\item $\upsilon_{q+1}^{(\texttt{l})}$ and $e_{q+1}^{(\texttt{l})}$ have the support property\footnote{For any smooth set $\Omega\subset\mathbb{T}^3$, we use $\Omega\circ\Phiik$ to denote the set $\Phiik^{-1}(\Omega)\subset\mathbb{T}^3\times\R$, i.e. the space-time set whose characteristic function is annihilated by $\Dtq$.}
\begin{align}
&\supp(\upsilon_{q+1}^{(\texttt{l})}), \, \supp(e_{q+1}^{(\texttt{l})}) \notag\\ &\subseteq \bigcup_{\xi,i,j,k,\vecl,I,\diamond} \supp \left(\chi_{i,k,q} \zeta_{q,\diamond,i,k,\xi,\vecl} \left(\rhob_{(\xi)}^\diamond \zetab_{\xi}^{I,\diamond} \right)\circ \Phiik \right) \cap
    B\left( \supp \varrho^I_{(\xi),\diamond} , 2 \lambda_\qbn^{-1} \right)\circ \Phiik \, .  \label{supp.upsilon.e.lem}
\end{align}
\item\label{item:eckel:1} For $0\leq k \leq \dpot$,  $(\upsilon_{q+1,k}^{(\texttt{l})})^\bullet :=\la_{q+\bn}^{\dpot-k}\partial_{i_1}\cdots \partial_{i_k} \upsilon_{q+1}^{(\texttt{l}, \bullet,i_1,\dots,i_\dpot)}$,\footnote{If $k=0$, we adopt the convention that $\partial_{i_1}\cdots\partial_{i_k}$is the identity operator.} satisfies the estimates
\begin{subequations}
\begin{align}
    &\norm{\psi_{{i,q}} D^N {D_{t,q}^{M}} 
   \upsilon_{q+1,k}^{(\texttt{l})}
    }_3\leq \Ga_q^{10} \de_{q+\bn}^\frac12 r_{q}^{-\frac13}
     \la_{q+\bn}^N
     \MM{M, \Nindt, \Ga_{q}^{i+14} \tau_{q}^{-1}, \Ga_{q}^{8}\Tau_{q}^{-1}}
    \label{est.new.upsilon.3}
    \\
    &\norm{\psi_{{i,q}} D^N {D_{t,q}^{M}} 
   \upsilon_{q+1,k}^{(\texttt{l})}
    }_\infty \leq 
    \Ga_q^{\frac{\badshaq}2+ 10} r_q^{-1}
    \la_{q+\bn}^{N}
     \MM{M, \Nindt, \Ga_{q}^{i+14} \tau_{q}^{-1}, \Ga_{q}^{8}\Tau_{q}^{-1}}
         \label{est.new.upsilon.inf}
\end{align}
\end{subequations}
for $N\leq \sfrac{\Nfin}4-2\dpot^2$ and $M\leq \sfrac{\Nfin}5$.
\item For $N\leq \sfrac{\Nfin}4-2\dpot^2$ and $M\leq \sfrac{\Nfin}5$, $e_{q+1}^{(\texttt{l})}$ satisfies
\begin{align} 
    \norm{ D^N D_{t,q}^{M} e_{q+1}^{(\texttt{l})}}_{\infty} 
    \leq \delta_{q+3\bn}^3\Tau_{\qbn}^{20\Nindt}\la_{q+\bn}^{-10}
    \la_{q+\bn}^{N}
     \MM{M, \Nindt,  {\tau_{q}^{-1}}, \Ga_{q}^{8}\Tau_{q}^{-1}}
    \label{est.new.e.inf}\, .
\end{align}
\end{enumerate}

\end{lemma}

\begin{remark}[\bf Cumulative premollified velocity increment potential]\label{rem:rep:vel:inc:potential}
We let $\upsilon_{q+1} : = \upsilon_{q+1}^{(p)} + r_q\Ga_q^{-1} \upsilon_{q+1}^{(c)}$ and $\upsilon_{q+1,k}^\bullet :=\la_{q+\bn}^{\dpot-k}\partial_{i_1}\cdots \partial_{i_k}
    \upsilon_{q+1}^{( \bullet,i_1,\dots,i_\dpot)}$. As a corollary of Lemma \ref{lem:rep:vel:inc:potential}, we have that $w_{q+1} = \div^{\dpot} \upsilon_{q+1} + e_{q+1}$, where $\upsilon_{q+1}$ and $e_{q+1}$ satisfy the properties \eqref{supp.upsilon.e.lem}--\eqref{est.new.e.inf} (up to implicit constants).
\end{remark}

\subsubsection{
Pressure increment
for premollified velocity increments}
The definition of the velocity pressure increment\index{velocity pressure increment} $\sigma_{\upsilon}$ is quite similar to that from \eqref{eq:fried:egg:evening}, after noticing that the premollified velocity increment has a natural high/low frequency product structure; we refer to \cite[Proposition~7.3]{GKN23} for details. The main difference in the definition is that the \emph{square} of velocity scales like Reynolds stresses and pressure, so we do not need to take the $\sfrac 12$ power on the sum like in \eqref{eq:fried:egg:evening}. The most important properties of the velocity pressure increment are as follows.  First, the square root of the positive portion of the pressure increment dominates the premollified velocity increment potential, as in~\eqref{est.vel.inc.pot.by.pr}; next, the velocity pressure increment 
 can be decomposed into ``positive and negative portions, $\sigma_{\upsilon}^\pm$,'' which satisfy estimates in~\eqref{sunday:sunday:sunday} asserting, among other things, that $\sigma_{\upsilon}^-$ is dominated by anticipated pressure; third, the positive portion of the velocity pressure increment obeys a good spatial support property in \eqref{supp:pr:vel.inc}, which follows from \eqref{supp.upsilon.e.lem} the results of section~\ref{sec:dodging}; and finally, there exists a negligible function of time which handles the mean of $\sigma_\upsilon$ in \eqref{th:billys:3} (see \eqref{sundayz}). Details may be found in \cite[Lemma~10.4]{GKN23}.

\begin{lemma}[\bf Pressure increment] \label{lem:pr.inc.vel.inc.pot}
For $\texttt{l}=p,c$, there exists a pressure increment $\si_{\upsilon^{(\texttt{l})}}=\si_{\upsilon^{(\texttt{l})}}^+ - \si_{\upsilon^{(\texttt{l})}}^-$ associated to the sum $\sum_{k=0}^\dpot \upsilon_{q+1,k}^{(\texttt{l})}$ defined in Lemma \ref{lem:rep:vel:inc:potential} such that the following properties hold.
\begin{enumerate}[(i)]
    \item We have that for all $k=0,1,\dots,\dpot$ and any $0\leq k \leq \dpot$, $N,M \leq \sfrac{\Nfin}5$,
\begin{align}
\left|\psi_{i,q}D^N D_{t,q}^M \upsilon_{q+1,k}^{(\texttt{l})}\right|
    \lec (\si_{\upsilon^{(\texttt{l})}}^+  + \de_{q+3\bn})^{\sfrac12}r_{q}^{-1} (\la_{q+\bn}\Ga_{q+\bn}^{{\sfrac1{10}}})^N\MM{M,\Nindt,\tau_q^{-1}\Ga_q^{i+16},\Tau_q^{-1}\Ga_{q}^9} \, .
    \label{est.vel.inc.pot.by.pr}
\end{align}
\item Set 
\begin{align}\label{defn:si.upsilon}
 \si_{\upsilon}^\pm:= \si_{\upsilon^{(p)}}^\pm
+\si_{\upsilon^{(c)}}^\pm \, , \qquad \si_{\upsilon}= \si_{\upsilon}^+-\si_{\upsilon}^- \, .
\end{align}
Then we have that for all $N \leq \sfrac{\Nfin}5$ and $M\leq \sfrac{\Nfin}5-\NcutSmall$ 
\begin{subequations}\label{sunday:sunday:sunday}
\begin{align}
\left|\psi_{i,q}D^N D_{t,q}^M \si_{\upsilon}^+\right|
    &\lec (\si_{\upsilon}^{+}+ \de_{q+3\bn})
    (\la_{q+\bn}\Ga_{q+\bn}^{\sfrac{1}{10}})^N\MM{M,\Nindt,\tau_q^{-1}\Ga_q^{i+16},\Tau_q^{-1}\Ga_{q}^9}\, , \label{est.pr.vel.inc}\\
\norm{\psi_{i,q}D^N D_{t,q}^M\si_{\upsilon}^+}_{\sfrac32}
    &\leq {\Ga_{q+\bn}^{-9}}\de_{q+2\bn}(\la_{q+\bn}\Ga_{q+\bn}^{\sfrac{1}{10}})^N \MM{M,\Nindt,\tau_q^{-1}\Ga_q^{i+16},\Tau_q^{-1}\Ga_{q}^9}\, , \quad \label{est.pr.vel.inc.32}\\
    \norm{\psi_{i,q}D^N D_{t,q}^M\si_{\upsilon}^+}_{\infty}
    &\leq \Ga_{q+\bn}^{\badshaq-9}(\la_{q+\bn}\Ga_{q+\bn}^{\sfrac{1}{10}})^N \MM{M,\Nindt,\tau_q^{-1}\Ga_q^{i+16},\Tau_q^{-1}\Ga_{q}^9}\, , \quad \label{est.pr.vel.inc.infty}\\
     \norm{\psi_{i,q}D^N D_{t,q}^M\si_{\upsilon}^-}_{\sfrac32}&\leq {\Ga_{q+\bn}^{-9}}\de_{q+2\bn}(\la_{q+\half}\Ga_{q+\half})^N\MM{M,\Nindt,\tau_q^{-1}\Ga_q^{i+16},\Tau_q^{-1}\Ga_{q}^9}\,  , \label{est.pr.vel.inc.-32}\\
     \norm{\psi_{i,q}D^N D_{t,q}^M\si_{\upsilon}^-}_{\infty}&\leq \Ga_{q+\bn}^{\badshaq-9}(\la_{q+\half}\Ga_{q+\half})^N\MM{M,\Nindt,\tau_q^{-1}\Ga_q^{i+16},\Tau_q^{-1}\Ga_{q}^9}\,  , \label{est.pr.vel.inc.-infty}\\
    {\left| \psi_{i,q} D^N D_{t,q}^M \si_{\upsilon}^-\right|}
    &\lec  \pi_\ell \Ga_q^{30}  {r_{q}^{\sfrac 43}} (\la_{q+\half}\Ga_{q+\half})^N\MM{M,\Nindt,\tau_q^{-1}\Ga_q^{i+16},\Tau_q^{-1}\Ga_{q}^9}\, . \label{est.pr.vel.inc-}
\end{align}
\end{subequations}
\item We have that for $q+1\leq q'' \leq \qbn-1$ and $q+1\leq q'\leq q+\half$, 
\begin{align}
    \supp(\si_{\upsilon}^+) \cap B(\hat w_{q''},\la_{q''}^{-1}\Ga_{q''{+1}}) \, , \quad
    \supp(\si_{\upsilon}^-)\cap B( \hat{w}_{q'}, \la_{q'}^{-1}\Ga_{q'{+1}}) = \emptyset \, . \label{supp:pr:vel.inc}
\end{align}
\item\label{i:presh:vel:mean} Define 
\begin{equation}\label{def:bmu:vel:presh}
    \bmu_{\sigma_{\upsilon}}(t) = \int_0^t \left \langle \Dtq \sigma_{\upsilon}  \right \rangle (s) \, ds \, .
\end{equation}
Then we have that for $0\leq M\leq 2\Nind$,
    \begin{align}\label{th:billys:3}
      \left|\frac{d^{M+1}}{dt^{M+1}} \bmu_{\sigma_{\upsilon}} \right| 
      \leq (\max(1, T))^{-1}\delta_{q+3\bn}^2 \MM{M,\Nindt,\tau_q^{-1},\Tau_{q+1}^{-1}} \, .
    \end{align}
\end{enumerate}
\end{lemma}

\medskip

\begin{remark}[\bf Pointwise bounds for principal and corrector parts]
From \eqref{exp.w.q+1}--\eqref{est.new.e.inf}, \eqref{est.vel.inc.pot.by.pr}, and \eqref{condi.Nfin0}, we have that for $N, M\leq \sfrac{\Nfin}5$,
\begin{subequations}\label{sunday:morning:1}
\begin{align}
    \left|\psi_{i,q}D^N D_{t,q}^M w_{q+1}^{(p)}\right|
    &\lec (\si_{\upsilon^{(p)}}^+  + \de_{q+3\bn})^{\sfrac12}r_{q}^{-1} (\la_{q+\bn}\Ga_{q+\bn}^{{\sfrac1{10}}})^N\MM{M,\Nindt,\tau_q^{-1}\Ga_q^{i+16},\Tau_q^{-1}\Ga_{q}^9} \, ,
    \label{est.vel.inc.p.by.pr} \\
    \left|\psi_{i,q}D^N D_{t,q}^M w_{q+1}^{(c)}\right|
    &\lec (\si_{\upsilon^{(c)}}^+  + \de_{q+3\bn})^{\sfrac12}\Ga_{q}^{-1} (\la_{q+\bn}\Ga_{q+\bn}^{{\sfrac1{10}}})^N\MM{M,\Nindt,\tau_q^{-1}\Ga_q^{i+16},\Tau_q^{-1}\Ga_{q}^9} \, .
    \label{est.vel.inc.c.by.pr}
\end{align}
\end{subequations} Note that thanks to the factor $r_q\Ga_q^{-1}$ in \eqref{exp.w.q+1}, the bound in \eqref{est.vel.inc.c.by.pr} has an extra gain compared to \eqref{est.vel.inc.p.by.pr}.  We may upgrade \eqref{sunday:morning:1} by applying Lemma~\ref{lem:upgrading.material.derivative}, so that in the same range of $N$ and $M$,
\begin{subequations}\label{sunday:morning:2}
\begin{align}
    \left|\psi_{i,\qbn-1} D^N D_{t,\qbn-1}^M w_{q+1}^{(p)}\right|
    &\lec (\si_{\upsilon^{(p)}}^+  + \de_{q+3\bn})^{\sfrac12}r_{q}^{-1} (\la_{q+\bn}\Ga_{q+\bn}^{{\sfrac1{10}}})^N \notag \\
    &\qquad\quad \times\MM{M,\Nindt,\tau_{\qbn-1}^{-1}\Ga_{\qbn-1}^{i-5},\Tau_{\qbn-1}^{-1}\Ga_{\qbn}^{-1}} \, ,
    \label{est.vel.inc.p.by.pr.upup}\\
    \left|\psi_{i,\qbn-1} D^N D_{t,\qbn-1}^M w_{q+1}^{(c)}\right|
    &\lec (\si_{\upsilon^{(c)}}^+  + \de_{q+3\bn})^{\sfrac12}\Ga_{q}^{-1} (\la_{q+\bn}\Ga_{q+\bn}^{{\sfrac1{10}}})^N \notag \\
    &\qquad\quad \times\MM{M,\Nindt,\tau_{\qbn-1}^{-1}\Ga_{\qbn-1}^{i-5},\Tau_{\qbn-1}^{-1}\Ga_{\qbn}^{-1}} \, .
    \label{est.vel.inc.c.by.pr.upup}
\end{align}
\end{subequations}
\end{remark}

Next, we recall that the pressure increment contributes to $\pi_q$ and hence $\kappa_q$ in \eqref{ineq:relaxed.LEI}, and thus one must account for the current error term $\phi_\upsilon\approx\div^{-1} \mathbb{P}_{\neq 0} \left(\Dtq \sigma_\upsilon \right)$. The first step, roughly speaking, is to decompose $\sigma_\upsilon$ using the synthetic Littlewood-Paley decomposition
\begin{equation}
    \tilde{\mathbb{P}}_{\lambda_{q+\half+1}} (\sigma_\upsilon) + \left( \sum_{k=q+\half+2}^{\qbn+1}  \tilde{\mathbb{P}}_{(\lambda_{k-1},\lambda_k]} (\sigma_\upsilon) \right) + \left( \Id -  \tilde{\mathbb{P}}_{\lambda_{\qbn+1}} \right)(\sigma_\upsilon) \, , \label{eq:decomp:showing:vel} 
\end{equation}
from subsection~\eqref{sec:LP}, obtaining \eqref{decomp:sigma:upsilon}; the term $\sigma_\upsilon^*$ comes from the last term in \eqref{eq:decomp:showing:vel} and is negligibly small, since $\la_{\qbn+1} > \la_\qbn$ and we only expect these functions to oscillate at frequencies no larger than $\la_\qbn$.  Then after applying the localized inverse divergence in Proposition~\ref{prop:intermittent:inverse:div} to $\Dtq$ of each of these terms, the important properties of these error terms are their pointwise estimates, and their spatial support properties, which may be obtained from Proposition~\ref{prop:intermittent:inverse:div} and the results of section~\ref{sec:dodging}. We note that the pointwise estimates in \eqref{pr:current:vel:loc:pt:q} are in terms of \emph{old} intermittent pressure, which prevents a loop of pressure creation and new current error creation.  Both \eqref{eq:decomp:showing:vel} and Proposition~\ref{prop:intermittent:inverse:div} produce negligible nonlocal remainders which we notate with the superscript ``$*$;'' these terms obey extremely strong estimates (as in \eqref{pr:current:vel:nonloc:infty1:q}) but do not obey spatial support properties (similar to \eqref{eq:Rnnl:inductive:dtq}). The proof of Lemma~\ref{lem:pr.current.vel.inc} is contained in \cite[Lemma~10.6]{GKN23}.  The methodology is essentially identical to the pressure increments for the current errors, for example those constructed for the current oscillation error in subsection~\ref{ss:ssO:current}. We refer also to the heuristics in subsection~\ref{sec:new.pressure.stress}, equation~\ref{sunday:eve}.

\begin{lemma}[\bf Current error from the velocity pressure increment]\label{lem:pr.current.vel.inc}
There exists a current error $\phi_{\upsilon}$ generated by $\sigma_\upsilon$ such that the following hold.
\begin{enumerate}[(i)]
\item\label{item:cpi:1} We have the decomposition and equalities 
\begin{subequations}
\begin{align}
\sigma_{\upsilon}= \sigma_{\upsilon}^* + \sum_{m'=q+\half+1}^{\qbn} \sigma_{\upsilon}^{m'} \, , &\qquad 
\phi_{{\upsilon}} = \underbrace{\phi_{{\upsilon}}^*}_{\textnormal{nonlocal}} + \underbrace{\sum_{m'=q+\half+1}^{\qbn} \phi_{{\upsilon}}^{m'}}_{\textnormal{local}} \label{decomp:sigma:upsilon} \\
\div \phi_\upsilon(t,x) = \Dtq \sigma_{ \upsilon}(t,x) &- \bmu_{\sigma_\upsilon}'(t) \, . \label{sundayz}
\end{align}
\end{subequations}
\item For all $N\leq \sfrac{\Nfin}5$ and $M\leq \sfrac{\Nfin}5-\NcutSmall-1$ and $q+\half+1\leq m' \leq q+\bn$,
\begin{align}
\left|\psi_{i,q} D^ND_{t,q}^M\phi_{{\upsilon}}^{m'}\right|
&\lec \Ga_m^{-100}(\pi_q^{m'})^{\sfrac32}r_m^{-1} (\lambda_{m}\Ga_{m'})^N \MM{M,\Nindt, \tau_q^{-1}\Gamma_{{q}}^{i+16},\Tau_q^{-1}\Ga_q^9}
\label{pr:current:vel:loc:pt:q}
\, .
\end{align}
\item For all $N\leq 3\Nind$ and $M\leq 3\Nind$,
\begin{align}
    \norm{D^N\Dtq^M\phi_{{\upsilon}}^{*}}_{\infty}
    \lec \delta_{q+3\bn}^{\sfrac 32} \Tau_{q+\bn}^{2\Nindt} \lambda_{q+\bn+2}^{-10}
    (\lambda_{q+\bn}\Ga_{q+\bn})^N \MM{M,\Nindt, {\tau_q^{-1}},\Tau_q^{-1}\Ga_q^9} \, .
 \label{pr:current:vel:nonloc:infty1:q}
\end{align}
\item For all $q+1\leq q'\leq q+\half$, $q+\half+2\leq m \leq q+\bn $, and $q+1\leq q''\leq m-1$, we have that
\begin{align}
  \supp(\phi_{{\upsilon}}^{q+\half+1})
    \cap B(\hat{w}_{q'}, \la_{q+1}^{-1}\Ga^2_{q})=\emptyset \, , \qquad
    \supp(\phi_{{\upsilon}}^{m}) \cap \supp \hat w_{q''} = \emptyset  \label{supp:pr:vel:q} \, .
\end{align}
\end{enumerate}
\end{lemma}

\subsubsection{Velocity increment, new velocity field, and new velocity cutoffs}
We can now define the velocity increment, velocity increment potential, and new velocity field.\index{$\hat w_q$}\index{$\hat \upsilon_\qbn$}\index{$u_{q+1}$}\index{$\hat u_\qbn$}

\begin{definition}[\bf Definition of $\hat w_\qbn$, $\hat u_\qbn = u_{q+1}$, and $\hat\upsilon_\qbn$]\label{def:wqbn}
Let $\mathcal{\tilde P}_{q+\bn,x,t}$\index{$\mathcal{\tilde P}_{q+\bn,x,t}$} denote a space-time mollifier which is a product of compactly supported kernels at spatial scale $\lambda_{q+\bn}^{-1}\Gamma_{q+\bn-1}^{-\sfrac 12}$ and temporal scale $\Tau_{q+1}^{-1}$.  We assume that both kernels have vanishing moments up to $10\Nfin$ and are $C^{10\Nfin}$ differentiable and define
\begin{equation}\label{def.w.mollified}
    \hat w_{q+\bn} := \mathcal{\tilde P}_{q+\bn,x,t} w_{q+1} \, , \quad u_{q+1} = \hat u_{\qbn} = u_q + \hat w_{\qbn} = \hat u_{\qbn-1} + \hat w_\qbn \, , \quad  \hat \upsilon_{q+\bn,k} := \mathcal{\tilde P}_{q+\bn,x,t} \upsilon_{q+1,k} \, .
\end{equation}\index{$\hat w_\qbn$}\index{$u_{q+1}$}
\end{definition}
\noindent The velocity increment and velocity increment potential satisfy the following pointwise estimates in terms of the velocity pressure increment $\sigma_\upsilon^+$. These follow easily from standard mollification estimates and \eqref{est.vel.inc.c.by.pr}, and we refer to \cite[Proposition~10.7]{GKN23} for details.
\begin{lemma}[\bf Pointwise estimates for $\hat \upsilon_{q+\bn}$]\label{lem:pt:mol:vip}
    For $N+M \leq \sfrac{3\Nfin}{2}$ and $0\leq k \leq \dpot$, we have that $\hat\upsilon_{q+\bn,k}$ satisfies the estimates
\begin{align}
    \left|\psi_{i,q+\bn-1} D^ND_{t,q+\bn-1}^{M} 
   \hat \upsilon_{q+\bn,k}
    \right|&< {\Ga_{q+\bn}} \left(\si_{\upsilon}^+  + 2\de_{q+3\bn} \right)^{\sfrac12} r_{q}^{-1}
      \notag \\
     & \qquad \times(\la_{q+\bn}{\Ga_{q+\bn}})^N \MM{M, \Nindt, \Ga_{q+\bn-1}^i \tau_{q+\bn-1}^{-1}, \Tau_{q+\bn-1}^{-1}{\Ga_{q+\bn-1}^2}} \, . \label{est.upsilon.ptwise.verify}
\end{align}
\end{lemma}

With the velocity increment in hand, we can also define the new velocity cutoff functions. The construction of the velocity cutoffs is carried out in \cite[Section~9]{GKN23} and is similar to the construction in \cite{BMNV21}. These velocity cutoffs are defined inductively, which is motivated by the following analogy: propagation of material derivative estimates for $D_{t,q}= \partial_t + \hat u_{q} \cdot \nabla$ in \emph{any} convex integration scheme should be done inductively by decomposing $D_{t,q}= \partial_t + \hat u_q \cdot \nabla = D_{t,q-1} + \hat w_{q} \cdot \nabla$. The main difference with \cite{NV22} is that the Chebyshev inequality used to control the support of $\psi_{i,q}$ utilizes inductive $L^3$ bounds rather than inductive $L^2$ bounds.
\begin{proposition}[\bf New velocity cutoff functions]\label{prop:verified:vel:cutoff}
There exists a set of new velocity cutoff functions $\psi_{i,q+\bn}$ such that all the listed inductive assumptions in subsections~\ref{sec:cutoff:inductive} hold.
\end{proposition}
\noindent Next, with the new velocity cutoffs in hand, we can verify nearly all of the inductive assumptions for the velocity increment $\hat w_\qbn$ and velocity increment potential $\hat \upsilon_\qbn$.
\begin{proposition}[\bf Inductive bounds for new velocity]\label{prop:inductive:velocity:bdd:verified}
The velocity increment $\hat w_{q+\bn}$, total velocity $\hat u_\qbn$, and velocity increment potential $\hat \upsilon_\qbn$ satisfy \eqref{eq:hat:no:hat} and all the inductive assumptions listed in subsection~\ref{sec:inductive:secondary:velocity}, save for \eqref{est.upsilon.ptwise}, which we verify by hand later in Lemma~\ref{prop:velocity:domination}. 
\end{proposition}

Finally, we can dominate a weighted sum of the velocity cutoffs by the pressure increment $\sigma_{\upsilon}^+$.  This estimate is identical to \cite[Lemma~10.8]{GKN23}.  The proof of this estimate follows from the facts that the velocity pressure increment dominates $\nabla \hat w_{\qbn}$ from \eqref{est.vel.inc.pot.by.pr} and $\pi_{q}^\qbn$ dominates $\nabla \hat u_{\qbn-1}$ from \eqref{eq:psi:q:q'} and \eqref{eq:nasty:D:vq:old}, and so their sum dominates $\nabla \hat u_{\qbn}$, which is itself comparable to the weighted sum of the velocity cutoffs on the left-hand side of \eqref{eq:psi:q:qplusbn:ineq:0:recall}. This result will be used in Lemma~\ref{prop:velocity:domination}.
\begin{lemma}[\bf New inductive cutoffs are dominated by the pressure increment]\label{lem:pr.vel.dom.cutoff}
    The new velocity cutoff functions $\psi_{i,q+\bn}$ satisfy
    \begin{align}
    \sum_{i=0}^{\imax} \psi_{i,q+\bn}^2 \delta_{q+\bn} r_{q}^{-\sfrac 23} \Gamma_{q+\bn}^{2i} &\les  {r_{q}^{-2}} \left(\pi_{q}^{q+\bn} + \si_\upsilon^+ + \de_{q+3\bn}\right) \, .  \label{eq:psi:q:qplusbn:ineq:0:recall}
\end{align}
\end{lemma}

\subsection{Stress errors}\label{sec:new.pressure.stress}

Upon adding $\hat w_{\qbn}$ to \eqref{eqn:ER}, we must define a new stress error $S_{q+1}$ such that
\begin{align}
\div(S_{q+1}) 
&= \underbrace{(\pa_t + \hat u_q \cdot \na) w_{q+1} + w_{q+1}\cdot\nabla \hat u_q}_{=:\, \div S_{TN}}
+\underbrace{\div \left(\wp_{q+1}\otimes \wp_{q+1} + R_\ell-\pi_\ell\Id\right)}_{=:\, \div S_O} \notag \\
&\qquad + \underbrace{\div \left(\wp_{q+1}\otimes_s \wc_{q+1}+\wc_{q+1}\otimes \wc_{q+1}\right)
}_{=:\, \div S_C} 
+ \underbrace{\div \left( R_q^q - R_\ell + \left( \pi_\ell - \pi_q^q \right) \Id \right) }_{=:\div S_{M1}}
\label{ER:new:error}\\
&\qquad + 
\underbrace{(\pa_t + \hat u_{q}\cdot \na )(\hat w_{q+\bn}-w_{q+1}) 
+ ((\hat w_{q+\bn}-w_{q+1})\cdot \na) \hat u_q + \div(\hat w_{q+\bn}\otimes \hat w_{q+\bn} - w_{q+1}\otimes w_{q+1} )}_{=:\div S_{M2}}
\, . \notag
\end{align}
We have used the notations $R_\ell$ and $\pi_\ell$ to denote the mollifications of $R_q^q$ and $\pi_q^q$; see Lemma~\ref{lem:upgrading}. Also, $u\otimes_s v = u\otimes v + v\otimes u$. \index{$\otimes_s$}
Note also that the transport and Nash error $S_{TN}$ involves only $\hat u_q$, but not $u_q$; this will be justified in Section~\ref{sec:dodging} by showing that $w_{q+1}$ has disjoint support from $u_q - \hat u_q$.  We then define the primitive stress error $\overline R_{q+1} = R_{q}-R_q^q+S_{q+1}$, so that $(u_{q+1}, p_q, \overline R_{q+1}, -(\pi_q-\pi_q^q))$ solves
    \begin{align}
    \partial_t u_{q+1} 
    + \div \left( u_{q+1} \otimes u_{q+1} \right) + \nabla p_q    
    = \div (  -(\pi_q-\pi_q^q)\Id + \ov R_{q+1}) \, , \qquad \div \, u_{q+1} = 0 \, . \label{ER:new:equation}
    \end{align}
Although a few of the terms in \eqref{ER:new:error} are already in divergence form and sufficiently small, in general we must use Proposition~\ref{prop:intermittent:inverse:div} to produce each error term. The definition and estimation of the Reynolds stress error terms is contained in \cite[Section~8]{GKN23}. We note first that the Reynolds stresses live at frequencies between (and including) $\la_{q+1}$ and $\la_\qbn$, which we denote with a superscript ``$m$,'' as in $S_{O}^m$ for the oscillation stress error at frequency $\lambda_m$.  Second, from \eqref{eq:inverse:div}, the inverse divergence first produces a spatially localized term, which we use the superscript ``$l$'' to indicate. Finally, the inverse divergence produces a nonlocal term from \eqref{eq:inverse:div:error:stress}, which we indicate with a superscript ``$*$;'' as usual, the nonlocal terms are negligible and can be ignored. We now summarize the definitions.\index{$\ov R_{q+1}$}\index{$S_{q+1}$}

\begin{definition}[\bf $\ov R_{q+1}$ and $S_{q+1}^m$]\label{def:of:new:stresses} The transport-Nash error $S_{TN}$, oscillation error $S_O$, divergence corrector error $S_C$, and mollification error $S_{M}=S_{M1}+S_{M2}$ can be decomposed into local and nonlocal portions at frequencies $\la_{q+1}$, \dots, $\la_\qbn$ as
\begin{equation}
    S_{\bullet} = \sum_{m=q+1}^{\qbn} S_\bullet^m = \sum_{m=q+1}^{\qbn} S_\bullet^{m,l} + S_\bullet^{m,*} \, .
\end{equation}
Then we define $S_{q+1}^m := S^{m,l}_{q+1}+ S^{m,*}_{q+1}$ for all $q+1 \leq m \leq q+\bn$ by
\begin{subequations}\label{defn:newstress}
\begin{align}
     S^{m,l}_{q+1} &:= S^{m,l}_{O} + S^{m,l}_{TN} + S^{m,l}_{C} + S^{m,l}_M \, , \qquad S^{m,*}_{q+1} := S^{m,*}_{O} + S^{m,*}_{TN} + S^{m,*}_{C} + S^{m,*}_M \, .
\end{align}
\end{subequations}
Recalling \eqref{eq:ER:decomp:basic} and \eqref{eq:Rnnl:inductive:dtq}, we then define the primitive\footnote{We refer to these stress errors as ``primitive'' since we will have to add a few small nonlocal pieces in \eqref{billystrings:eq:1}.  The local piece $R_{q+1}^{m,l}$ requires no such adjustment, and so we do not notate it with a bar.} stress errors by
\begin{align}\label{defn:primitive.stress}
\ov R_{q+1} &:= \sum_{m=q+1}^{q+\bn} \ov R_{q+1}^{m} \, , \qquad
\ov R_{q+1}^m = R_q^m + S_{q+1}^m \\ \label{defn:local.stress}
R_{q+1}^{m,l} &:= R_q^{m,l} + S_{q+1}^{m,l}\, , \qquad \ov R_{q+1}^{m,*} := R_q^{m,*} + S_{q+1}^{m,*}\, .
\end{align}
\end{definition}

Next, associated to each $S_{q+1}^{m,l}$ is a pressure increment $\sigma_{S^m}$, which is defined as in \eqref{eq:fried:egg:evening}. We also define functions of time $\bmu_{\sigma_{S^m}}$ by 
\begin{align}\label{defn:bmu:stress}
\bmu_{\sigma_{S^m}}:= \int_0^t \langle D_{t,q}\si_{S^m}   \rangle(s) ds\, ,  
\end{align}
which will be useful in \eqref{Sunday:Sunday:Sunday}. The pressure increments dominate the local portions of the Reynolds stresses $S_{q+1}^{m,l}$ for frequencies $m> q+\half$, while the local portions at any lower frequencies are actually dominated by anticipated pressure.  We recall the basic estimates and support properties on stresses and their pressure increments in the following lemma, and refer to \cite[Section~8]{GKN23} for details.  The support properties are consequences of the results of Section~\ref{sec:dodging}.
\begin{lemma}[\bf Collected properties of stress error terms and pressure increments]\label{lem:prop.si.pre.stress}
For each $q+\half+1 \leq m \leq q+\bn$, $\si_{S^m}$ satisfies the following properties.
\begin{enumerate}[(i)]
\item For any $0\leq k \leq \dpot$ and $N+M\leq 2\Nind$, we have that
    \begin{subequations}
        \begin{align}
        \left|\psi_{i,q}D^N D_{t,q}^M S^{m,l}_{q+1}\right|
    &\lec \left(\si_{S^m}^+ +  {\de_{q+3\bn}}\right)  (\la_{m}\Ga_m)^N  \label{est.S.m.pt.sim.stress}
    \MM{M, \Nindt, \Ga_q^{i+18} \tau_q^{-1}, \Tau_q^{-1}\Ga_q^9} \, .
    \end{align}
    \end{subequations}
\item For $N, M \leq {\sfrac{\Nfin}{200}}$, we have that
\begin{subequations}
    \begin{align}
        \norm{\psi_{i,q}D^N D_{t,q}^M \si_{S^m}^+}_{\sfrac32}
    &\lec \Ga_m^{-9} \de_{m+\bn}  (\la_{m}\Ga_m)^N 
    \MM{M, \Nindt, \Ga_q^{i+18} \tau_q^{-1}, \Tau_q^{-1}\Ga_q^9} \label{est:si.m+.32.Dtq.stress}\\
    \norm{\psi_{i,q}D^N D_{t,q}^M \si_{S^m}^+}_{\infty}
    &\lec {\Ga_m^{\badshaq-9}}   (\la_{m}\Ga_m)^N 
    \MM{M, \Nindt, \Ga_q^{i+18} \tau_q^{-1}, \Tau_q^{-1}\Ga_q^9} \label{est:si.m+.infty.Dtq.stress}\\
    \left|\psi_{i,q}D^N D_{t,q}^M \si_{S^m}^+\right|
    &\lec \left(\si_{S^m}^+ +  {\de_{q+3\bn}}\right)  (\la_{m}\Ga_m)^N 
    \MM{M, \Nindt, \Ga_q^{i+18} \tau_q^{-1}, \Tau_q^{-1}\Ga_q^9} \label{est:si.m+.pt.stress}\\
    \left|\psi_{i,q}D^N D_{t,q}^M \si_{S^m}^-\right|
    &\lec {\Ga_{q+\half}^{-100}} \pi_q^{q+\half}  (\la_{q+\half}\Ga_{q+\half})^N 
    \MM{M, \Nindt, \Ga_q^{i+18} \tau_q^{-1}, \Tau_q^{-1}\Ga_q^9} \, .  \label{heatsie}
\end{align}
\end{subequations}
\item\label{item:sat:one} $\si_{S^m}$ and $\si_{S^m}^+$ have the support properties
\begin{subequations}
\begin{align}
    B(\supp \hat w_{q'}, \la_{q'}^{-1}{\Ga_{q'+1}})
    \cap \supp\si_{S^m} = \emptyset \qquad \forall q+1\leq q'\leq q+\half  \, , \label{supp.si.m.stress}\\
    B(\supp \hat w_{q'}, \la_{q'}^{-1}{\Ga_{q'+1}})
    \cap \left(\supp \si_{S^m}^+  \cup \supp S^{m,l}_{q+1}\right) = \emptyset \qquad \forall q+1\leq q'\leq m-1\, . \label{supp.si.m+.stress}
\end{align}
\end{subequations}
\item\label{upgrade:item:5} For $0\leq M\leq 2\Nind$, the function of time $\bmu_{\sigma_{S^m}}$ defined in \eqref{defn:bmu:stress} satisfies
\begin{align}\label{est:bmu:stress}
      \left|\frac{d^{M+1}}{dt^{M+1}} \bmu_{\sigma_{S^m}} \right| 
      \lec (\max(1, T))^{-1}\delta_{q+3\bn} \MM{M,\Nindt,\tau_q^{-1},\Tau_{q+1}^{-1}} \, .
    \end{align}
\end{enumerate}
\end{lemma}

With the previous lemma in hand, we may upgrade the material derivatives using Lemma~\ref{lem:upgrading.material.derivative}. We may also verify Hypothesis~\ref{hyp:dodging4}.  The proof of this lemma follows easily from Lemma~\ref{lem:prop.si.pre.stress}.
\begin{lemma}[\bf Upgrading material derivatives and verifying Hypothesis~\ref{hyp:dodging4}]\label{l:divergence:stress:upgrading}  The new stress errors $S_{q+1}^m=S_{q+1}^{m,l}+S_{q+1}^{m,*}$ satisfy the following.
\begin{enumerate}[(i)]
    \item\label{upgrade:item:1} $R_{q+1}^{m,l}$ satisfies Hypothesis~\ref{hyp:dodging4} with $q$ replaced by $q+1$.
    \item\label{upgrade:item:2} For $q+2\leq m\leq q+\half$, the symmetric stresses $S_{q+1}^{m,l}$ obey the estimates
\begin{align}\label{eq:lo:upgrade:1}
    \left|\psi_{i,m-1} D^N D_{t,m-1}^M S_{q+1}^{m,l} \right| 
    &\lec \Ga^{-50}_{m} \pi_{q}^{m}
    \La_{m}^N \MM{M,\Nindt, \Ga_{m-1}^{i-5}\tau_{m-1}^{-1} , \Tau_{q}^{-1}\Ga_q^9}
\end{align}
for $N,M \leq \sfrac{\Nfin}{10}$.  For the same range of $N,M$, the symmetric stress $S_{q+1}^{q+1,l}$ obeys the estimates
\begin{align}\label{eq:lo:upgrade:2}
    \left|\psi_{i,q} D^N D_{t,q}^M S_{q+1}^{q+1,l} \right| 
    &\lec \Ga^{-50}_{{q+1}} \pi_{q}^{q+1}
    \La_{{q+1}}^N \MM{M,\Nindt, \Ga_{q}^{i+19}\tau_{q}^{-1}, \Tau_{q}^{-1}\Ga_q^9} \, .
\end{align}
\item\label{upgrade:item:3} For $q+\half+1\leq m \leq \qbn$ and $N,M\leq \sfrac{\Nfin}{100}$, the symmetric stresses $S_{q+1}^{m,l}$ obey the estimates
\begin{subequations}
\begin{align}
\left| \psi_{i,m-1} D^N D_{t,m-1}^M S^{m,l}_{q+1} \right| &\lesssim \left( \sigma_{S_O^m}^+ + \sigma_{S_C^{m,l}}^+ + \mathbf{1}_{\{m=\qbn\}} \left( \sigma_{S_{TN}}^+ + \sigma_\upsilon^+ \right) + \delta_{q+3\bn} \right) \notag\\
&\qquad \times (\lambda_{m}\Ga_m)^N \MM{M,\Nindt,\Gamma_{m-1}^{i-5} \tau_{m-1}^{-1}, \Tau_{q}^{-1} \Ga_q^9 }   \, .  \label{eq:stress:Linfty:upgraded}
\end{align}
\end{subequations}
\item\label{upgrade:item:4} For all $q+1\leq m \leq \qbn$ and $N+M\leq 2\Nind$, the symmetric stresses $S_{q+1}^{m,*}$
\begin{align}
\left\| D^N D_{t, m-1}^M S^{m,*}_{q+1} \right\|_{L^\infty}
    &\leq \Ga_{q+1}^2 \Tau_{q+1}^{2\Nindt} \delta_{q+3\bn}^2 \la_m^N \MM{M,\Nindt,\tau_{m-1}^{-1},\Tau_{m-1}^{-1}}  \, .
    \label{eq:nlstress:upgraded}
\end{align}
\end{enumerate}
\end{lemma}

Finally, the pressure increments $\sigma_{S^m}$ produce current errors in exactly the same way the pressure increment $\sigma_\upsilon$ for the velocity produced a current error in Lemma~\ref{lem:pr.current.vel.inc}. We refer to the discussion preceding Lemma~\ref{lem:pr.current.vel.inc}, as well as \cite[Section~8]{GKN23} for further details. Let us pause, however, to explain why one may expect that these current errors do not require a new pressure increment, but are instead dominated by old intermittent pressure.  Consider the portion of the highest shell of the oscillation error given by
\begin{equation}\notag
    \div^{-1}\left(\tP_{\la_{\qbn}} \div \left( w_{q+1,R} \otimes w_{q+1,R} \right)\right)\approx \nabla R_q^q \tP_{\la_\qbn} \, \div^{-1} \left( |\WW_{q+1,R}|^2 \right) \, ,
\end{equation}
where $\tP_{\la_\qbn}$ is a synthetic Littlewood-Paley projector defined in subsection~\ref{sec:LP}. In order to dominate the right-hand side pointwise, we recall \eqref{eq:ind:stress:by:pi} and create a pressure increment $\sigma_{S_O^\qbn}$: 
\begin{equation}\notag
    \sigma_{S_O^\qbn} = \la_q \pi_q^q \mathbb{P}_{\neq 0} \left| \tP_{\la_\qbn} \div^{-1} \left(  |\WW_{q+1,R}|^2 \right) \right| \, . 
\end{equation}
Considering the current error produced by the pressure increment in the highest shell $\tP_{\la_\qbn}$ only,  
we now apply another synthetic Littlewood-Paley projector $\tP_{\la_\qbn}$ to the high frequency object, apply $\Dtq$, and invert the divergence. 
Imagining that $\Dtq$ is analogous to multiplication by $\nabla \hat u_q$, 
what we must estimate is then
\begin{align}
      \left| \nabla \hat u_q \right| \la_q \pi_q^q \div^{-1} \left[ \tP_{\la_\qbn} \left| \tP_{\la_\qbn} \div^{-1} \left(  |\WW_{q+1,R}|^2 \right) \right| \right] \, . \label{eq:sunday:to:bound}
\end{align}
Using \eqref{eq:psi:q:q'} and \eqref{eq:nasty:D:vq:old}, we may control $|\nabla \hat u_q|$ by $\la_q (\pi_q^q)^{\sfrac 12} r_q^{-1}$. Next, we use \eqref{e:pipe:estimates:1} to control $\WW_{q+1,R}$ pointwise by $r_q^{-1}$, and we will gain two factors of $\la_\qbn^{-1}$ from the inverse divergences. Finally, recalling that $r_q^{-2}\approx \la_\qbn \la_q^{-1}$, $\delta_q\approx \la_q^{-\sfrac 23}$, using the pressure scaling law \eqref{eq:ind.pr.anticipated}, and writing everything out, we may bound \eqref{eq:sunday:to:bound} by
\begin{equation}\label{sunday:eve}
\la_q (\pi_q^q)^{\sfrac 12} r_q^{-1} \la_q \pi_q^q r_q^{-2} \la_{\qbn}^{-2} \leq (\pi_q^q)^{\sfrac 32} \la_q \la_\qbn^{-1} r_q  \leq (\pi_q^\qbn)^{\sfrac 32} r_\qbn^{-1} \, . 
\end{equation}
Notice that this bound is \emph{exactly} consistent with \eqref{eq:ind:current:by:pi}, and so the current error coming from the addition of the new pressure increment is dominated by \emph{anticipated} pressure and requires no new pressure increment!  This avoids the possibility of a loop of pressure creation and new current error creation.  The reader who absorbs these heuristics will be able to understand essentially every estimate in this paper related to the intermittent pressures $\pi_q^{q'}$.

\begin{lemma}[\bf Current error from the stress pressure increment]\label{lem:stress:pressure:current.stress}
For every $m\in\{q+\half+1,\dots,q+\bn\}$, the current error $\phi_{{S^{m}}}$ satisfies the following properties.
\begin{enumerate}[(i)]
    \item\label{i:pc:2:ER.stress} We have the decompositions and equalities
    \begin{subequations}
    \begin{align}\label{eq:desert:decomp:ER.stress}
        \phi_{{S^{m}}} &= \phi_{{S^{m}}}^* + \sum_{m'=q+\half+1}^{{m}} \phi_{{S^{m}}}^{m'} \, , \qquad \qquad 
         \phi_{S^{m}}^{m'} = \phi_{S^{m}}^{m',l} + \phi_{S^{m}}^{m',*} \, \\
         \div \phi_{{S^{m}}}
          &= D_{t,q}\si_{{S^{m}}} - \bmu_{\sigma_{S^m}}' \, . \label{Sunday:Sunday:Sunday}
    \end{align}
    \end{subequations}
    \item\label{i:pc:3:ER.stress} For $q+\half+1 \leq m' \leq m$ and $N,M\leq  2\Nind$,
    \begin{subequations}
    \begin{align}
        &\left|\psi_{i,q} D^N \Dtq^M \phi_{S^{m}}^{m',l} \right| < \Ga_{m'}^{-100} \left(\pi_q^{m'}\right)^{\sfrac 32} r_{m'}^{-1} (\la_{m'} \Ga_{m'}^2)^M \MM{M,\Nindt,\tau_q^{-1}\Ga_q^{i+17},\Tau_q^{-1}\Ga_q^9} \label{s:p:c:pt.stress} \\
        &\left\| D^N \Dtq^M \phi_{S^{m}}^{m',*} \right\|_\infty + \left\| D^N\Dtq^M \phi_{S^{m}}^{*} \right\|_\infty < \Tau_\qbn^{2\Nindt} \delta_{q+3\bn}^{\sfrac 32} (\la_{m}\Ga_{m}^2)^N \tau_q^{-M} \label{e:p:c:nonlocal.stress} \, .
    \end{align}
    \end{subequations}
    \item\label{i:pc:4:ER.stress} For all $q+\half+1\leq m' \leq m$ and all $q+1\leq q' \leq m'-1$, 
    \begin{align}
        B\left( \supp \hat w_{q'}, \sfrac 12 \lambda_{q'}^{-1} \Ga_{q'+1} \right) \cap \supp \left( \phi^{m',l}_{S^{m}} \right) = \emptyset \label{s:p:c:supp.stress} \, .
    \end{align}
\end{enumerate}
\end{lemma}

\subsection{Transport and Nash current errors}\label{ss:cur:T:and:N}
In this section, we treat the transport and Nash current errors, which are defined by
\begin{align}\label{eq:curr:trans:redux}
   \div \left(\ov \phi_T + \ov \phi_N\right) &= (\pa_t + \hat u_q \cdot \na ) \left(\frac 12 |w_{q+1}|^2+ \ka_q^q  - \frac {\tr (S_{q+1})}2 \right) \\
   &\qquad + (\na \hat u_q) : \left(w_{q+1}\otimes w_{q+1} + R_q - \pi_q^q \Id -  \ov R_{q+1}\right) - \bmu'_T - \bmu'_N\, .\notag
\end{align}
These error terms are in fact closely related to the new Reynolds stress error terms $S_{q+1}$, as one can see from \eqref{eq:curr:trans:redux}.  For this reason we have included the proofs of the estimates for these error terms in \cite[subsection~8.7]{GKN23}. We however pause to discuss the heuristics behind these error terms, expanding upon the discussion in subsubsection~\ref{sss:L3}. First note that the cost of the material derivative operator in the first term is closely related to the size of $\nabla \hat u_q$ in the second term, and so we expect both terms to have the same size. Let us therefore consider the term 
$$  \left\|  \div^{-1}  \left[ \nabla \hat u_q \, : \, \tP_{\la_{\qbn}} \left( w_{q+1,R} \otimes w_{q+1,R} \right) \right] \right\|_1  \, . $$
From \eqref{eq:nasty:D:vq:old} and \eqref{eq:psi:i:q:support:old}, we have that the $L^3$ norm of $\nabla \hat u_q$ is essentially $\delta_q^{\sfrac 12} r_{q-\bn}^{-\sfrac 13} \la_q$. Next, we expect the inverse divergence to gain $\la_{q+\bn}^{-1}$, due to the presence of the synthetic Littlewood-Paley projector.  Finally, recall that $w_{q+1,R}$ uses $L^2$ normalized pipes multiplied by coefficient functions $a_{\pxi,\diamond}$, which from Lemma~\ref{lem:a_master_est_p} have $L^3$ norm $\delta_\qbn^\frac12$ (after aggregation).  Combining these estimates and using that the high-frequency pipes decouple from the rest of the expression, we need that
\begin{align*}
    \delta_q^{\sfrac 12} r_{q-\bn}^{-\sfrac 13} \la_q \la_\qbn^{-1} \delta_{\qbn} \leq \delta_{q+2\bn} r_{q+\bn}^{-1} \, .
\end{align*}
We can draw an analogy with \eqref{sunday:sunday:sunday:night} by noticing that $\delta_{\qbn}$ behaves essentially like $\delta_{q+1}$, $\delta_{q+2\bn}$ behaves essentially like $\delta_{q+2}$, and we have an extra $\sfrac 23$ power of $r_{\qbn}$ leftover, which behaves essentially like $r_{q+1}$.  Then writing out everything in terms of $\la_q$ as in \eqref{sunday:sunday:sunday:night}, we see that this error term is sharp in $B^{\sfrac 13}_{3,\infty}$.

The following lemma contains the precise estimates satisfied by the transport and Nash current errors.  We refer to \cite[subsection~8.7]{GKN23} for details.  As usual, though, the main properties are as follows: the error terms are contained in frequency shells between (and including) $\la_{q+1}$ and $\la_\qbn$; the shells at $\la_{q+\half+1}$ and above require a pressure increment, while the lower shells are dominated by existing intermittent pressure; the pressure increments satisfy heuristic estimates exactly like those outlined in the above discussion; pressure increments and error terms themselves both satisfy support properties; and finally, nonlocal error terms are negligibly small.

\begin{lemma}[\bf Current error and pressure increment from \eqref{eq:curr:trans:redux}]\label{lem:ct:general:estimate}
There exist vector fields $\ov\phi_{TN}=\ov\phi_T + \ov \phi_N$ and functions of time $\bmu_{TN} = \bmu_T + \bmu_N$ such that \eqref{eq:curr:trans:redux} holds.  Furthermore, we may decompose $\ov\phi_{TN}$ into components
\begin{align}
    \ov \phi_{TN} = \sum_{m=q+1}^\qbn \ov \phi_{TN}^m = \sum_{m=q+1}^\qbn \ov\phi_{TN}^{m,l} + \ov\phi_{TN}^{m,*}
\end{align}
which satisfy the following.
\begin{enumerate}[(i)]
    \item The errors $\ov\phi_{TN}^m$ vanish uniformly for $q+2\leq m' \leq q+\half-1$. Furthermore, the errors $\ov \phi^{m}_{TN}$ for $m=q+1,q+\sfrac{\bn}{2}$ require no pressure increment. More precisely, we have that for $N,M\leq \sfrac{\Nfin}{100}$,
\begin{subequations}
\begin{align}
    \label{eq:ct:lowshell:nopr:1}
    \left|\psi_{i,q} D^N \Dtq^M \ov \phi^{q+1,l}_{TN} \right| &\les \Ga_{m}^{-100} \left(\pi_q^{m}\right)^{\sfrac32} r_{m}^{-1} \la_{m}^N \MM{M,\Nindt, \tau_q^{-1} \Ga_q^{i+15}, \Tau_q^{-1}\Ga_q^8} \, .
\end{align}
\end{subequations}
\item For $q+\half+1 \leq m \leq \bn$, there exists functions $\si_{\ov \phi^m_{TN}} = \si_{\ov \phi^m_{TN}}^+ - \si_{\ov \phi^m_{TN}}^-$ such that
\begin{subequations}
\begin{align}
    \label{eq:ct.p.1}
    \left|\psi_{i,q} D^N \Dtq^M \ov \phi^{m,l}_{{TN}}\right| &\les \left( (\si_{\ov \phi^{m}_{TN}}^+)^{\sfrac 32} r_m^{-1} + {\de_{q+3\bn}^2}\right)  \left(\lambda_{m}\Gamma_q\right)^N \MM{M,\Nindt,\tau_q^{-1}\Gamma_{q}^{i+17},\Tau_q^{-1}{\Ga_q^9}}\\
    \label{eq:ct.p.2}
    \left|\psi_{i,q} D^N \Dtq^M \si_{\ov \phi^m_{TN}}^+ \right| &\les \left(\si_{\ov\phi^m_{TN}}^+ +\de_{q+3\bn}\right) \left(\lambda_{m}\Gamma_q\right)^N \MM{M,\Nindt,\tau_q^{-1}\Gamma_{q}^{i+{18}},\Tau_q^{-1}{\Ga_q^9}}\\
    \label{eq:ct.p.3}
    \norm{\psi_{i,q} D^N \Dtq^M \si_{\ov\phi^m_{TN}}^+}_{\sfrac32} &\les \de_{m+\bn} \Gamma_{m}^{-9} \left(\lambda_{m}\Gamma_q\right)^N \MM{M,\Nindt,\tau_q^{-1}\Gamma_{q}^{i+{18}},\Tau_q^{-1}{\Ga_q^9}}\\
    \label{eq:ct.p.3.1}
    \norm{\psi_{i,q} D^N \Dtq^M \si_{\ov\phi^m_{TN}}^+}_{\infty} &\les \Gamma_{m}^{\badshaq-9} \left(\lambda_{m}\Gamma_q\right)^N \MM{M,\Nindt,\tau_q^{-1}\Gamma_{q}^{i+{18}},\Tau_q^{-1}{\Ga_q^9}}\\
    \label{eq:ct.p.4}
    \left|\psi_{i,q} D^N \Dtq^M \si_{\ov \phi^m_{TN}}^- \right| &\les \left(\frac{\la_q}{\la_{q+\floor{\bn/2}}}\right)^{\sfrac23} \pi_q^q  \left(\lambda_{q+\floor{\bn/2}}\Gamma_q\right)^N \MM{M,\Nindt,\tau_q^{-1}\Gamma_{q}^{i+{18}},\Tau_q^{-1}{\Ga_q^9}}
\end{align}
\end{subequations}
for all $N,M \leq \sfrac{\Nfin}{100}$. 
Furthermore, we have that for $q+1 \leq m' \leq m-1$ and $q+1\leq q'' \leq q+\half$,
\begin{align}\label{eq:ct.p.6}
    \supp \si_{\ov\phi^m_{TN}}^- \cap B\left( \supp \hat w_{q''}, \la_{q''}^{-1}\Ga_{q''+1} \right) =
    \supp \si_{\ov \phi^m_{TN}}^+ \cap B\left(\supp \hat w_{m'}, \la_{m'}^{-1}\Ga_{m'+1} \right) = \emptyset \, .
\end{align}
\item When $m=q+2,\dots,q+\bn$ and $q+1\leq q' \leq m-1$, the local parts satisfy
\begin{align}
B\left( \supp \hat w_{q'}, \lambda_{q'}^{-1} \Gamma_{q'+1} \right) \cap \supp \ov \phi_{TN}^{m,l} = \emptyset \, . \label{ct:support:first}
\end{align}
\item For $m=q+1, \dots, q+\bn-1$ and $N,M\leq \sfrac{\Nind}{4}$, the non-local parts ${\ov \phi}^{m,*}_{O}$ satisfy
\begin{subequations}\label{eq:ctnl:estimate:1}
\begin{align}
\left\| D^N D_{t, q}^M {\ov \phi}^{m,*}_{{TN}} \right\|_{L^\infty}
    &\les \delta_{q+3\bn}^{\sfrac32} (\la_{m}\Ga_{m-1})^{N} \tau_{q}^{-M} \, .
\end{align}
For $q+\bn$ and $N,M\leq \sfrac{\Nind}{4}$, the non-local part ${\ov \phi}^{\qbn,*}_{O}$ satisfies
\begin{align}
\left\| D^N D_{t,\qbn-1}^M {\ov \phi}^{\qbn,*}_{{TN}} \right\|_{L^\infty} &\les \delta_{q+3\bn}^{\sfrac32} (\la_{m}\Ga_{m-1})^{N}\MM{M,\Nindt,\tau_{\qbn-1}^{-1},\Tau_{\qbn-1}^{-1}\Ga_{\qbn-1}} \, .
\end{align}
\end{subequations}
\item For $M \leq 2\Nind$, the time function $\bmu_{TN}$ satisfies
\begin{align}\label{sat:morn:10:57}
\left| \frac{d^{M+1}}{dt^{M+1}} \bmu_{TN} \right| \les \left( \max(1,T) \right)^{-1} \delta_{q+3\bn}^2 \MM{M,\Nindt,\tau_q^{-1},\Tau_{q+1}^{-1}} \, .
\end{align}
\end{enumerate}
\end{lemma}

Just as for the premollified velocity increment and the new Reynolds stress errors, the pressure increment for the transport and Nash current errors produces a new current error, which is dominated by anticipated pressure.  Applying the same heuristics as in \eqref{sunday:eve} shows that the current error coming from the pressure increment for the transport and Nash current error is dominated by anticipated pressure.  We refer to \cite[Lemma~8.21]{GKN23} for details.

\begin{lemma}[\bf Current error from the transport/Nash current error pressure increment]\label{lem:ctn:pressure:current}
For \\ every $m'\in\{q+\half+1,\dots,q+\bn\}$, there exist current errors $\phi_{{\ov\phi_{TN}^{m'}}}$ associated to the pressure increments $\si_{\ov\phi_{TN}^{m'}}$ and functions $\bmu_{\si_{\ov\phi_{TN}^{m'}}}$ of time
that satisfy the following properties.  
\begin{enumerate}[(i)]
\item\label{i:ctn:pc:2} We have the decompositions and equalities
\begin{subequations}
\begin{align}\label{eq:ctn:desert:decomp}
    \div \phi_{{\ov\phi_{TN}^{m'}}} + \bmu_{\si_{\ov\phi_{TN}^{m'}}}'
    &= D_{t,q}  \si_{\ov\phi_{TN}^{m'}}\, , \\
    \phi_{{\ov\phi_{TN}^{m'}}} = \phi_{{\ov\phi_{TN}^{m'}}}^* + \sum_{m=q+\half+1}^{{m'}} \phi_{{\ov\phi_{TN}^{m'}}}^{m} \, , \qquad 
    \phi_{\ov \phi_{TN}^{m'}}^{m} &= \phi_{\ov \phi_{TN}^{m'}}^{m,l} + \phi_{\ov \phi_{TN}^{m'}}^{m,*} \, .           
\end{align}
\end{subequations}
\item\label{i:pc:3:TN} For $q+\half+1 \leq m \leq m'$ and $N,M\leq  2\Nind$,
\begin{subequations}
\begin{align}
    &\left|\psi_{i,q} D^N \Dtq^M \ov\phi_{\ov \phi_{TN}^{m'}}^{m,l} \right| < \Ga_{m}^{-100} \left(\pi_q^m\right)^{\sfrac 32} r_m^{-1} (\la_m \Ga_m^2)^M \MM{M,\Nindt,\tau_q^{-1}\Ga_q^{i+{18}},\Tau_q^{-1}\Ga_q^9} \, , \label{eq:desert:estimate:1:TN} \\
    &\left\| D^N \Dtq^M \phi_{\ov \phi_{TN}^{m'}}^{m,*} \right\|_\infty  < \Tau_\qbn^{2\Nindt} \delta_{q+3\bn}^{\sfrac 32} (\la_{m'}\Ga_{m'}^2)^N \tau_q^{-M},\label{eq:desert:estimate:21:TN}\\
    &\left\| D^N\Dtq^M \phi_{\ov\phi_{TN}^{m'}}^{*} \right\|_\infty < \Tau_\qbn^{2\Nindt} \delta_{q+3\bn}^{\sfrac 32} (\la_{q+\bn}\Ga_{q+\bn}^2)^N \tau_q^{-M} \label{eq:desert:estimate:22:TN} \, .
\end{align}
\end{subequations}
\item\label{i:pc:4:TN} For all $q+\half+1\leq m \leq m'$ and all $q+1\leq q' \leq m-1$, 
\begin{align}
        B\left( \supp \hat w_{q'}, \sfrac 12 \lambda_{q'}^{-1} \Ga_{q'+1} \right) \cap \supp \left( \phi^{m,l}_{\ov \phi_{TN}^{m'}} \right) = \emptyset \label{eq:desert:dodging:TN} \, .
\end{align}
\item \label{i:pc:5:TN} For $M\leq 2\Nind$, the mean part $\bmu_{\si_{\ov\phi_{TN}^{m'}}}$ satisfies
\begin{align}\label{eq:desert:mean:tn}
    \left|\frac{d^{M+1}}{dt^{M+1}}
    \bmu_{\si_{\ov\phi_{TN}^{m'}}} \right| 
    \leq (\max(1, T))^{-1}\delta_{q+3\bn} \MM{M,\Nindt,\tau_q^{-1},\Tau_{q+1}^{-1}} \, .
\end{align}
\end{enumerate}
\end{lemma}

\subsection{Mollification}\label{subsec:result:par1:6}

In this section, we present two lemmas which estimate various mollification terms.  These lemmas are essentially technical in nature and the reader can safely skip them.  We refer to \cite[Lemmas~3.1, 8.22]{GKN23} for details.

\begin{lemma}[\bf Mollification and upgrading material derivative estimates]\label{lem:upgrading}
Assume that \emph{all} inductive assumptions listed in subsections~\ref{ss:relaxed}-\ref{sec:inductive:secondary:velocity} hold. Let $\Pqxt$\index{$\Pqxt$} be a space-time mollifier for which the kernel is a product of $\mathcal{P}_{q,x}(x)$, which is compactly supported in space at scale $\Lambda_q^{-1}\Gamma_{q-1}^{-\sfrac 12}$, and $\mathcal{P}_{q,t}(t)$, which is compactly supported in time at scale $\Tau_{q-1}\Gamma_{q-1}^{\sfrac 12}$; we further assume that both kernels have vanishing moments up to $10\Nfin$ and are $C^{10\Nfin}$-differentiable. Define
\begin{align}
    &R_\ell = \Pqxt R_q^q \, , \qquad
    \pi_\ell = \Pqxt \pi_q^q \, , \qquad \varphi_\ell = \Pqxt \varphi_q^q 
    \label{def:mollified:stuff}
\end{align}
on the space-time domain $[-\sfrac{\tau_{q-1}}2,T+\sfrac{\tau_{q-1}}2]\times \T^3$. 
For $q'$ such that $q<q'\leq q+\bn-1$, we define $\mathcal{P}_{q',x,t}$\index{$\mathcal{P}_{q',x,t}$} in an analogous way after making the appropriate parameter substitutions, and we set $R_\ell^{q'}=\mathcal{P}_{q',x,t}R_q^{q'}$ and $\pi_\ell^{q'}=\mathcal{P}_{q',x,t}\pi_q^{q'}$. For $q'$ with $q+\bn\leq q' <q+\Npr$, we define $\overline{\mathcal{P}}_{q+\bn-1,x,t}$\index{$\overline{\mathcal{P}}_{q+\bn-1,x,t}$} analogously at the spatial scale $\Lambda_{q+\bn-1}^{-1}\Ga_{q+\bn-1}^{-\sfrac12}$ and temporal scale $\Tau_{q+\bn-1}\Ga_{q+\bn-1}^{-\sfrac12}$ and set $\pi_\ell^{q'} = \overline{\mathcal{P}}_{q+\bn-1,x,t} \pi_q^{q'}$. Then the following hold. 
\begin{enumerate}[(i)]
\item\label{item:moll:two} The inductive assumptions in \eqref{eq:pressure:inductive} are replaced with upgraded bounds for all $N+M\leq \Nfin$:
\begin{subequations}\label{eq:pressure:upgraded}
\begin{align}
\norm{ \psi_{i,q} D^N D_{t,q}^M  \pi_\ell }_{\sfrac 32}  
&\les \Gamma_q^2 \delta_{q+\bn}\left(\Lambda_q\Gamma_q\right)^N \MM{M, \NindRt, \Gamma_{q}^{i} \tau_q^{-1}, \Tau_{q}^{-1} } \, ,
\label{eq:pressure:inductive:dtq:upgraded} \\
\norm{ \psi_{i,q} D^N D_{t,q}^M \pi_\ell }_{\infty} 
&\les \Gamma_q^{2+\badshaq} \left(\Lambda_q\Gamma_q\right)^N \MM{M, \NindRt, \Gamma_{q}^{i} \tau_q^{-1}, \Tau_{q}^{-1} } \, , \label{eq:pressure:inductive:dtq:uniform:upgraded} \\
\left|\psi_{i,q} D^N D_{t,q}^M \pi_\ell\right| &\les \Gamma_q^3 \pi_\ell \left(\Lambda_q\Gamma_q\right)^N \MM{M, \Nindt, \Gamma_{q}^{i} \tau_q^{-1}, \Tau_{q}^{-1} }  \, . \label{eq:pressure:inductive:dtq:pointwise}
\end{align}
\end{subequations}
We record the following additional bounds for $\pi_\ell^k$ with $q<k\leq q+\bn-1$ and $N+M\leq \Nfin$:
\begin{subequations}\label{eq:pressure:upgraded:higher}
\begin{align}
\norm{ \psi_{i,k-1} D^N D_{t,k-1}^M  \pi_\ell^k }_{\sfrac 32}  
&\les \Ga_k^2 \delta_{k+\bn}\left(\Lambda_k\Gamma_{k-1}\right)^N \MM{M, \NindRt, \Gamma_{k-1}^{i+2} \tau_{k-1}^{-1}, \Tau_{k-1}^{-1}\Ga_{k-1} } \, ,
\label{eq:pressure:inductive:dtq:upgraded:higher} \\
\norm{\psi_{i,k-1} D^N D_{t,k-1}^M \pi_\ell^k }_{\infty} 
&\les \Ga_k^{2+\badshaq} \left(\Lambda_k\Gamma_{k-1}\right)^N \MM{M, \NindRt, \Gamma_{k-1}^{i+2} \tau_{k-1}^{-1}, \Tau_{k-1}^{-1}\Ga_{k-1} } \, , \label{eq:pressure:inductive:dtq:uniform:upgraded:higher} \\
\left|\psi_{i,k-1} D^N D_{t,k-1}^M \pi_\ell^k \right| &\leq 2\Ga_k^3 \pi_\ell^k \left(\Lambda_k\Gamma_k\right)^N \MM{M, \Nindt, \Gamma_{k-1}^{i+3} \tau_{k-1}^{-1}, \Tau_{k-1}^{-1} \Ga_{k-1}^2 }  \, . \label{eq:pressure:inductive:dtq:pointwise:higher}
\end{align}
\end{subequations}
For $\pi_\ell^k$ with $q+\bn\leq k<q+\Npr$ and $N+M\leq \Nfin$, we also have that
\begin{subequations}\label{eq:pressure:upgraded:higher:much}
\begin{align}
\norm{ \psi_{i,q+\bn-1} D^N D_{t,q+\bn-1}^M  \pi_\ell^k }_{\sfrac 32} 
&\les \Ga_k^2 \delta_{k+\bn}\left(\Lambda_{q+\bn-1}\Gamma_{q+\bn-1}\right)^N \notag\\
&\qquad\times\MM{M, \NindRt, \Gamma_{q+\bn-1}^{i+2} \tau_{q+\bn-1}^{-1}, \Tau_{q+\bn-1}^{-1}\Ga_{q+\bn-1} } \, ,
\label{eq:pressure:inductive:dtq:upgraded:higher:much} \\
\norm{ \psi_{i,q+\bn-1} D^N D_{t,q+\bn-1}^M  \pi_\ell^k }_{\infty} 
&\les \Ga_{{ q+\bn-1}}^{2+\badshaq} \left(\Lambda_{q+\bn-1}\Gamma_{q+\bn-1}\right)^N\notag\\
&\qquad\times \MM{M, \NindRt, \Gamma_{q+\bn-1}^{i+2} \tau_{q+\bn-1}^{-1}, \Tau_{q+\bn-1}^{-1}\Ga_{q+\bn-1} } \, , \label{eq:pressure:inductive:dtq:uniform:upgraded:higher:much} \\
\left|\psi_{i,q+\bn-1} D^N D_{t,q+\bn-1}^M  \pi_\ell^k \right| &\leq 2\Ga_k^3 \pi_\ell^k \left(\Lambda_{q+\bn-1}\Gamma_{q+\bn-1}^2\right)^N \notag\\
&\qquad\times\MM{M, \Nindt, \Gamma_{q+\bn-1}^{i+3} \tau_{q+\bn-1}^{-1}, \Tau_{q+\bn-1}^{-1}\Ga_{q+\bn-1}^2 }  \, . \label{eq:pressure:inductive:dtq:pointwise:higher:much}
\end{align}
\end{subequations}
We finally record the additional estimate
\begin{equation}\label{ind:pi:lower}
   \frac12 \delta_{q+\bn} \leq \pi_\ell \leq 2\pi_q^q \leq 4\pi_\ell \, , \qquad \frac12 \delta_{k+\bn} \leq \pi_\ell^k \leq 2\pi_q^k \leq 4\pi_\ell^k \, .
\end{equation}
\item\label{item:moll:three} The inductive assumptions in \eqref{eq:ind:stress:by:pi}--\eqref{eq:ind:velocity:by:pi} for $k=q$ are replaced with the following upgraded bounds for all $N+M\leq \Nfin$ in the first two inequalities, and $N+M\leq\sfrac{3\Nfin}{2}$ in the third: 
\begin{subequations}\label{eq:inductive:pointwise:upgraded}
\begin{align}
    \left|\psi_{i,q} D^N D_{t,q}^M R_\ell\right| &\les \Gamma_q^{-7} \pi_\ell \left(\Lambda_q\Gamma_q\right)^N \MM{M, \Nindt, \Gamma_{q}^{i} \tau_q^{-1}, \Tau_{q}^{-1} } \label{eq:inductive:pointwise:upgraded:1} \\
    \left|\psi_{i,q} D^N D_{t,q}^M \varphi_\ell \right| &\les \Gamma_q^{-11} \pi_\ell^{\sfrac 32} r_q^{-1} \left(\Lambda_q\Gamma_q\right)^N \MM{M, \Nindt, \Gamma_{q}^{i} \tau_q^{-1}, \Tau_{q}^{-1} } \label{eq:inductive:pointwise:upgraded:2} \\
    \left|\psi_{i,q} D^N D_{t,q}^M \hat w_k \right| &\les r_{k-\bn}^{-1} \pi_\ell^{\sfrac 12} \left(\Lambda_q\Gamma_q\right)^N \MM{M, \NindRt, \Gamma_{q}^{i} \tau_q^{-1}, \Tau_{q}^{-1} } \, . \label{eq:inductive:pointwise:upgraded:3}
\end{align}
\end{subequations}
For $k$ such that $q<k\leq q+\bn-1$, we have for $N+M\leq \Nfin$ the additional bound
\begin{align}
     \left|\psi_{i,k-1} D^N D_{t,k-1}^M R_\ell^k\right| &\les \Gamma_q^{-7} \pi_\ell^k \left(\Lambda_k\Gamma_k\right)^N \MM{M, \Nindt, \Gamma_{k-1}^{i+{23}} \tau_{k-1}^{-1}, \Tau_{k-1}^{-1} \Ga_{k-1}^{{12}} } \, . \label{eq:inductive:pointwise:upgraded:1:higher}
\end{align}
\item\label{item:moll:four}
For $k$ such that $q<k\leq q+\bn-1$ and $N+M\leq 2\Nind$, we have that
\begin{align}\label{eq:diff:moll:higher:statement}
    &\norm{D^N D_{t,k-1}^M \left(\pi_q^k - \pi_\ell^k\right)}_\infty +  \norm{D^N D_{t,k-1}^M \left(R_q^k - R_\ell^k\right)}_\infty \notag \\ 
    &\qquad \qquad \lec \Gamma_{k+1} \Tau_{k+1}^{4\Nindt} \delta_{k+3\bn}^2 (\Lambda_k \Ga_{k-1})^N \MM{M,\Nindt,\tau_{k-1}^{-1}\Ga_{k-1},\Tau_{k-1}^{-1}\Gamma_{k-1}^{ 11}} \, ,
\end{align}    
and for $k$ with $q+\bn\leq k <q+\Npr$ and $N+M\leq 2\Nind$, 
\begin{align}
    \norm{D^N D_{t,q+\bn-1}^M \left(\pi_q^k - \pi_\ell^k\right)}_\infty 
    &\lec \Gamma_{q+\bn+1} \Tau_{q+\bn+1}^{{4}\Nindt} \delta_{q+4\bn}^2 (\Lambda_{q+\bn-1} \Ga_{q+\bn-1})^N\notag\\
    &\quad\times
    \MM{M,\Nindt,\tau_{q+\bn-1}^{-1}\Ga_{q+\bn-1},\Tau_{q+\bn-1}^{-1}\Gamma_{q+\bn-1}} \, . \label{eq:diff:moll:higher:statement2}
\end{align}
\end{enumerate}
\end{lemma}

\begin{lemma}[\bf Mollification current errors]\label{lem:moll:curr} There exist current errors $\ov\phi_M^{q+1}$ and $\ov\phi_M^\qbn$ and a function of time $\bmu_M^\qbn$ such that the following hold.
\begin{enumerate}[(i)]
    \item We have the equalities
    \begin{subequations}
    \begin{align}
        \div \ov\phi_M^{q+1} &= \div \left( \varphi_q^q - \varphi_\ell \right) \\
        \div \ov\phi_{M}^{\qbn} + (\bmu_M^\qbn)' &= \frac 12 \div \left( |\hat w_{\qbn}|^2 \hat w_\qbn - |w_{q+1}^2| w_{q+1} \right) + (\hat w_{q+\bn} - w_{q+1}) \cdot (\pa_t u_q + (u_q\cdot \na) u_q + \na p_q )  \notag\\
        &\qquad + \left( (\pa_t + \hat u_q \cdot \nabla) \frac 12 \tr + (\nabla \hat u_q) :  \right) \left(\hat w_{q+\bn}\otimes \hat w_{q+\bn} - w_{q+1}\otimes w_{q+1}\right) \, .
    \end{align}
    \end{subequations}
    \item For all $N+ M\leq \sfrac{\Nind}4$, the mollification errors $\ov \phi_M^{q+1}$ and $\ov \phi_M^{q+\bn}$ satisfy 
\begin{subequations}
    \begin{align}
        \norm{D^N D_{t,q}^M \ov \phi_M^{q+1}}_\infty
        &\leq \delta_{q+3\bn}^{\sfrac32} \lambda_{q+1}^N \MM{M,\Nindt, \tau_{q}^{-1},\Gamma_{q}^{-1}\Tau_{q}^{-1}} \, , \label{est:curr.mollification1}\\
        \norm{ D^N D_{t,q+\bn-1}^M \ov \phi_M^{q+\bn}}_\infty
        &\leq \Ga_\qbn^9 {\delta_{q+3\bn}^{\sfrac 32}}\Tau_\qbn^{2\Nindt} \left(\lambda_{q+\bn}\Gamma_{q+\bn}\right)^N \MM{M, \NindRt, \tau_{q+\bn-1}^{-1}, \Tau_{q+\bn-1}^{-1}\Ga_{\qbn-1} } \, . 
    \label{est:curr.mollification}
    \end{align}
\end{subequations}
\noindent In addition, the mean portion $\bmu_{M}^{\qbn}$ satisfies
\begin{align}\label{eq:sat:evening:6:32}
        \left| \frac{d^{M+1}}{dt^{M+1}}\bmu_{M}^{\qbn} \right|
        \leq (\max(1,T))^{-1} \de_{q+3\bn} \MM{M,\Nindt,\tau_q^{-1},\Tau_{q+1}^{-1}} \quad \text{for }M\leq \sfrac{\Nind}4 \, .
\end{align}
\end{enumerate}
\end{lemma}

\section{New velocity increment dodging}\label{sec:dodging}
In our wavelet-inspired scheme scheme, the building block flows are the intermittent Mikado bundles presented in subsection~\ref{subsec:result:par1:1}. The goal of this section is to specify the placement of these bundles.  Specifically, we treat both the ``bundling pipe'' $\rhob^\diamond_\pxi$ and the highly intermittent pipe $\WW^I_{\pxi,\diamond}= \curl \UU^I_{\pxi,\diamond}$ appearing in \eqref{wqplusoneonediamond}.\index{intermittent Mikado bundle} The placement of each intermittent bundle will be chosen to ensure disjointness from various other intermittent Mikado bundles.  In subsection~\ref{ss:straight:pipe}, we consider a single prototypical term from \eqref{wqplusoneonediamond} at the fixed time slice at which the flow map $\Phiik$ is the identity.  This ensures that $\rhob_\pxi^\diamond$ and $\WW_{\pxi,\diamond}^I$ are supported in perfectly straight, periodized cylinders. We then assume that the support of $a_{\pxi,\diamond}$ is inhabited by various \emph{deformed} bundles and show that one can choose the supports of $\rhob_\pxi^\diamond$ and $\WW_{\pxi,\diamond}^I$ to be disjoint from these bundles.  We prove these placement lemmas in abstraction, assuming that $e_3$ is the vector tangent to the pipe we are trying to place and that none of the existing deformed pipes from previous generations have tangent vectors in a neighborhood of $e_3$.  The modifications required to replace $e_3$ with a general vector direction are purely cosmetic.  We note that the required assumptions on the tangent vectors is ensured by geometric lemmas which gives $\bn$ different choices for the set of rational vector directions; see the errata for~\cite{GKN23}.  Then in subsection~\ref{ss:stress:oscillation:2}, we can apply our abstract lemmas to choose the placements for each term in \eqref{wqplusoneonediamond} inductively on the various indices present in the sum.  It is in this subsection that we verify Hypotheses~\ref{hyp:dodging1}--\ref{hyp:dodging2}.

\subsection{Straight pipe dodging}\label{ss:straight:pipe}

Consider a rectangular prism $\Omega_0$, which in practice will be related to the support of the cutoffs $a_{\pxi,\diamond}$ and $\zetab_\xi^{I,\diamond}$ from a single term of~\eqref{wqplusoneonediamond}.  Assume that $\Omega_0$ is inhabited by deformed pipes of thickness $\la_{q+1}^{-1}, \dots \la_{q+\bn-1}^{-1}$, corresponding to $\hat w_{q+1}, \dots, \hat w_{\qbn-1}$, and $\la_\qbn^{-1}$, corresponding to terms from \eqref{wqplusoneonediamond} whose placements have already been specified. If the prism $\Omega_0$ is not too large, and the new bundle we are trying to place is sufficiently sparse, then we can place the new bundle to dodge the existing bundles. Furthermore, the pipes in the new bundle will be placed at distance no smaller than $\la_{q+i}^{-1}\Gamma_{q+i}$ away from a given deformed pipe of thickness $\la_{q+i}^{-1}$. We call this additional property {\it effective dodging}\index{effective dodging}.  To see the importance of this property, consider the portion of the oscillation error given by
$$  \tP_{\la_{q+i}} \div (w_{q+1,R} \otimes w_{q+1,R}) \, , $$
where $\tP_{\la_{q+i}}$ is a synthetic Littlewood-Paley projector defined in~\ref{sec:LP}. By the properties of this operator, the support of this error term will be contained in a $\la_{q+i}^{-1}$ neighborhood of the support of $w_{q+1,R}$.  By effective dodging, this error term will have spatial support which is disjoint from pipes of thickness $\la_{q+1}^{-1},\dots, \la_{q+i-1}^{-1}$.  We will repeatedly use effective dodging in conjunction with the synthetic Littlewood-Paley projectors to ensure good spatial support properties of error terms and pressure increments. 

Our first dodging proposition uses the bundling pipes to dodge pipes with thickness \emph{at least} $\lhalf^{-1}$ and \emph{at most} $\lambda_{q+1}^{-1}$, corresponding to $\hat w_{q+\half}$ and $\hat w_{q+1}$, respectively. We record and prove a statement for $\xi=e_3$ and leave the case for general direction vectors to the reader.  Note importantly that in our actual iteration, the choice of vector directions depends on $q \textnormal{ mod } \bn$, in order to ensure that vector directions from different generations never coincide; this orthogonality is necessary for the technique we use. The argument proceeds by first isolating a single rectangular prism $\Omega_0$ of dimensions $\la_{q+1}^{-1}\Gamma_q^5 \times \la_{q+1}^{-1}\Gamma_q^5\times \lambda_{q}^{-1}\Gamma_{q}^{-8}$ (corresponding to the support of a mildly anisotropic cutoff from Lemma~\ref{lem:checkerboard:estimates}), and choosing the support of the bundling pipe to effectively dodge all given pipes of thicknesses $\la_{q+1}^{-1}, \cdots \la_{q+\floor{\sfrac{\bn}2}}^{-1}$ in the prism. 

\begin{lemma}[\bf Using bundling pipes to dodge very old, thick pipes]\label{lem:coarsepipe}
 Let $\Omega_0$ be a rectangular prism of dimensions $\lambda_{q+1}^{-1}\Gamma_q^{5}\times \lambda_{q+1}^{-1}\Gamma_q^{5} \times \lambda_q^{-1}\Gamma_q^{-8}$. Suppose that there exists a $q$-independent constant $\const_P$ such that at most $\const_P$ segments of deformed pipe (in the sense of Definition~\ref{def:sunday:sunday}) with thickness $\la_{q'+\bn}^{-1}$ and spacing $({\lambda_{q'+\half}\Gamma_{q'}})^{-1}$ for some $q-\bn < q' \leq q-\half$ have non-empty intersection with $\Omega_0$. Suppose furthermore there exist $q$-independent sets $\Xi_1,\Xi_1', \dots, \Xi_{\bn}, \Xi_{\bn}' \subset \mathbb{Q}^3\cap \mathbb{S}^2$ and a constant $\const_{\rm angle}$, which is independent of $q$ but may depend on $\bn$, such that for all curves $\ell$ around which one of the $\const_P$ segments of deformed pipe is concentrated, the tangent vector to the curve $\ell$ belongs to a $\Gamma_{0}^{-1}$-neighborhood of a vector $\xi \in \Xi_{\overline{q}' \, \textnormal{mod} \, \bn} \cup \Xi'_{\overline{q}' \, \textnormal{mod} \, \bn}$, and 
 $$ B_{\const_{\rm angle}}(e_3) \bigcap \bigcup_{q'=q+1}^{q+\sfrac{\bn}{2}} \Xi_{q' \textnormal{ mod } \bn} \cup \Xi'_{q' \textnormal{ mod } \bn} =\emptyset \, ; $$
 that is, the ball of radius $\const_{\rm angle}$ around $e_3$ has empty intersection with the sets $\Xi_{q' \textnormal{ mod } \bn}, \Xi'_{q' \textnormal{ mod } \bn}$. Let $E_0\subset \Omega_0$ denote the support of these deformed segments inside $\Omega_0$. Then for $\diamond=\varphi,R$, there exists $k\in\{1,\dots,\Gamma_q^{6}\}$ and a bundling pipe flow $\rhob_{e_3,k,\diamond}$ defined as in Proposition \ref{prop:bundling} such that
    \begin{align}\label{dodging:rhob}
        B\left( \supp \boldsymbol \chib_{e_3,k,\diamond} , \lambda_{q+1}^{-1} \Gamma_q^2 \right) \cap E_0 = \emptyset \quad \text{i.e.,} \quad         B\left( E_0, \lambda_{q+1}^{-1} \Gamma_q^2 \right) \cap  \supp \boldsymbol \chib_{e_3,k,\diamond} = \emptyset \, .
    \end{align}
\end{lemma}
\begin{proof}
    We first divide the face $[0,\lambda_{q+1}^{-1}\Gamma_q^5]^2$ of the prism into the grid of squares of sidelength $\approx\lambda_{q+1}^{-1}\Gamma_q$, and 
    we will find a set of squares in which we can place a new bundling pipe flow $\rhob_{e_3,k,\diamond}$. Since the set of squares will be placed $(\sfrac{\mathbb{T}}{\la_{q+1}\Gamma_q^{-4}})^2$-periodically, we have from \eqref{i:bundling:2} that
    \begin{align*}
        \text{(the possible number of placement of a set of squares)} = \left(\frac{\text{spacing}}{\text{thickness}}\right)^2
        = \left(\frac{\la_{q+1}^{-1} \Gamma_q^4}{\la_{q+1}^{-1}\Gamma_q} \right)^2 = \Gamma_q^6 \, .
    \end{align*}
    By assumption there exist at most $\const_P$ number of deformed pipe segments in the prism. When we enlarge these segments by a factor of $\lambda_{q+1}^{-1}\Gamma_q^2$ and project the enlarged neighborhood onto the face $[0,\lambda_{q+1}^{-1}\Gamma_q^{5}]^2$, we claim that each projection will be contained in a $\approx \lambda_{q+1}^{-1}\Gamma_q^2$-neighborhood of a curve of length at most $\approx\lambda_{q+1}^{-1}\Gamma_q^5$. This claim will follow from showing that the curve $\ell$ around which the segment of deformed pipe is concentrated intersects the rectangular prism $\Omega_0$ in a curve of length at most $\approx\lambda_{q+1}^{-1}\Gamma_q^5$, and then applying~(4.5) and~(4.6). In order to measure the length of $\ell \cap \Omega_0$, we use that the tangent vector to $\ell$ belongs to $\Gamma_{0}^{-1}$ neighborhood of a vector $\xi$, which itself satisfies $|e_3-\xi| \geq \const_{\rm angle}$, or equivalently, $\langle e_3, \xi \rangle < 1-\delta$ for some $\delta=\delta(\const_{\rm angle})$.  Assuming that $\Gamma_{0}^{-1}$ is sufficiently small, depending on $\const_{\rm angle}$, it is then impossible for $\ell \cap \Omega_0$ to have length longer than a constant multiplied by the width of $\Omega_0$ \emph{in the $e_1$ and $e_2$ directions}, which is $\lambda_{q+1}^{-1}\Gamma_q^5$.  This is a geometric consequence of the fact that the tangent vector to $\ell$ and $e_3$ have inner product bounded from above away from $1$ by a small by quantified amount, which depends on $\const_{\rm angle}$.

It then follows that\footnote{A fully rigorous version of this estimate would utilize a standard covering argument which is predicated on the geometric constraints imposed by Lemma~\ref{lem:axis:control}, or even Definition~\ref{def:sunday:sunday}; we however content ourselves with a slightly heuristic version and refer the reader to \cite[Proposition~4.8]{BMNV21} for further details.}
    \begin{align*}
        &\text{(the number of grid squares occupied by given enlarged segments)}\\
        &\qquad \qquad \les \textnormal{number of segments} \times \frac{\textnormal{area occupied by an enlarged segment}}{\textnormal{area of a grid square}} \\
        &\qquad \qquad \lesssim \const_P
        \times \frac{\lambda_{q+1}^{-2}\Gamma_q^7}{\Gamma_q^2\lambda_{q+1}^{-2}}\\
        &\qquad \qquad =\const_P \Gamma_q^5 \, ,
    \end{align*}
    which is less than $\Gamma_q^6$ for sufficiently large $\la_0$. Therefore, from the pigeonhole principle, there exists a set of squares in which we can place the pipe $\rhob_{e_3,k,\diamond}$ satisfying \eqref{dodging:rhob}. 
\end{proof}

We now use the intermittent pipe flows from Propositions~\ref{prop:pipeconstruction} or \ref{prop:pipe.flow.current} to dodge pipes with thickness \emph{at least} $\lambda_{q+\bn}^{-1}$ and \emph{at most} than $\lambda_{q+\half+1}^{-1}$, corresponding to $\hat w_{\qbn}$ and $\hat w_{q+\half+1}$.  This dodging is carried out on a prism of dimensions roughly $\la_{q+\half}^{-1}\times \la_{q+\half}^{-1}\times \la_q^{-1}$ (corresponding to the support of a product of a mildly anisotropic cutoff from Lemma~\ref{lem:checkerboard:estimates} and a strongly anistropic checkerboard cutoff from Lemma~\ref{lem:finer:checkerboard:estimates}). Combining this result with the previous lemma, we will have successfully dodged pipes of thicknesses in between $\lambda_{q+1}^{-1}$ and $\lambda_{q+\bn}^{-1}$. As before, we present the statement for $\xi=e_3$ and omit further details.

\begin{lemma}[\bf Using very intermittent pipes to dodge newer, less thick pipes]\label{lem:finepipe}
Let $\Omega_1$ be a rectangular prism of dimensions $\la_{q+\half}^{-1} \times \la_{q+\half}^{-1}\times \lambda_{q}^{-1}\Gamma_{q}^{-8}$ with the long side in the $e_3$ direction. Suppose that there exists a $q$-independent constant $\const_P$ such that for each $q''$ with $q-\sfrac \bn 2 < q'' \leq q$ and any convex subset $\Omega'\subset\Omega_1$ with $\textnormal{diam}\left(\Omega'\right) \lesssim \lambda_{{q''}+\half}^{-1}\Gamma_{{q''}}^{-1}$, at most $\const_P\Gamma_{{q''}}$ segments of deformed pipes of thickness $\la_{{q''}+\bn}^{-1}$ and spacing $\lambda_{q''+ \sfrac \bn 2}\Gamma_{q''}^{-1}$ have non-empty intersection with $\Omega'$.  Suppose furthermore there exist $q$-independent sets $\Xi_1,\Xi_1', \dots, \Xi_{\bn}, \Xi_{\bn}' \subset \mathbb{Q}^3\cap \mathbb{S}^2$ and a constant $\const_{\rm angle}$, which is independent of $q$ but may depend on $\bn$, such that for all curves $\ell$ around which one of the $\const_P$ segments of deformed pipe is concentrated, the tangent vector to the curve $\ell$ belongs to a $\Gamma_{0}^{-1}$-neighborhood of a vector $\xi \in \Xi_{\overline{q}' \, \textnormal{mod} \, \bn} \cup \Xi'_{\overline{q}' \, \textnormal{mod} \, \bn}$, and 
 $$ B_{\const_{\rm angle}}(e_3) \bigcap \bigcup_{q'=q+1}^{q+\sfrac{\bn}{2}} \Xi_{q' \textnormal{ mod } \bn} \cup \Xi'_{q' \textnormal{ mod } \bn} =\emptyset \, ; $$
 that is, the ball of radius $\const_{\rm angle}$ around $e_3$ has empty intersection with the sets $\Xi_{q' \textnormal{ mod } \bn}, \Xi'_{q' \textnormal{ mod } \bn}$. For fixed ${q''}$, let $E_{{q''}}$ denote the support of such segments inside $\Omega_1$.  Then for either $\diamond=R$ or $\diamond=\varphi$, there exists $k$ and a corresponding intermittent pipe flow $\WW_{e_3, \diamond}:=\mathcal{W}^k_{\xi,\lambda_{q+\bn},\sfrac{\lambda_{q+\half}\Gamma_q}{\lambda_{q+\bn}}}$ constructed as in Propositions~\ref{prop:pipeconstruction} or \ref{prop:pipe.flow.current} such that for all $q-\half < {q''} \leq q$,
    \begin{align*}
        B \left(\supp \WW_{e_3, \diamond}, \Gamma_{{q''}+\bn}^2\la_{{q''}+\bn}^{-1} \right) \cap E_{{q''}} = \emptyset \quad \text{i.e.,} \quad  B \left( E_{{q''}} , \Gamma_{{q''}+\bn}^2\la_{{q''}+\bn}^{-1} \right) \cap \supp \WW_{e_3, \diamond} = \emptyset \, .
    \end{align*}
\end{lemma}
\begin{proof}
As in the previous lemma, since we want to place a new pipe which enjoys {\it effective dodging} with previously placed deformed pipes, instead of considering the previously placed pipes themselves, we consider a thickened neighborhood of them. More precisely, for a deformed pipe of thickness $2\la_{q+i}^{-1}$, we consider instead a neighborhood of it of thickness $\Gamma_{q+i}^{2}\la_{q+i}^{-1}$ and call these new objects `thickened pipes'.  Then, it is enough to place a new pipe that dodges these thickened pipes, so that a new pipe effectively dodges all previously placed deformed pipes.  
    
We divide the face of $\Omega_1$ into a grid of squares of sidelength $\la_{q+\bn}^{-1}$. Since a new pipe will be placed $(\sfrac{\T}{\la_{q+\half}\Gamma_q})^3$-periodically, we have from Proposition~\ref{prop:pipeconstruction} or \ref{prop:pipe.flow.current} that
\begin{align}\label{card.possible}
    \text{(the possible number of placement of a set of squares)}= \left(\frac{\text{spacing}}{\text{thickness}}\right)^2
    = \left(\frac{\la_{q+\bn}}{\la_{q+\half}\Gamma_q} \right)^2 \, .
\end{align}

Now, we count the number of grid squares occupied by given enlarged segments and compare it to this number. From the assumption that there exists $\const_P$ which controls the density of thickened, deformed pipe segments of thickness $\lambda_{{q''}+\bn}^{-1}$ that can intersect a ball $\Omega'$ of volume $\approx (\lambda_{{q''}+\half}\Gamma_{{q''}})^{-3}$, we have that the total number of thickened pipe segments that can intersect $\Omega_1$ is at most
\begin{align*}
    \const_P \Ga_{{q''}} \times \frac{\textnormal{length of $\Omega_1$}}{\lambda_{{q''}+\half}^{-1}\Gamma_{{q''}}^{-1}} \times \frac{\textnormal{(width of $\Omega_1$})^2}{\min\left(\lambda_{{q''}+\half}^{-1}\Gamma_{{q''}}^{-1},\textnormal{width of $\Omega_1$}\right)^2} \leq \const_P \Ga_{{q''}} \times \frac{\la_{{q''}+\half}\Ga_{{q''}}^3}{\Gamma_q^8\la_{q}} \, .
\end{align*}
When we project all these thickened segments onto the face of $\Omega_1$, each projection will be contained in a $\approx\lambda_{{q''}+\bn}^{-1}\Gamma_{{q''}+\bn}^{2}$-neighborhood of a curve of length at most $\approx\la_{q+\half}^{-1}$; the control on the length of this projected curve is a consequence of~(4.5) and~(4.6) and control on the length of the original curve intersected with $\Omega_1$.   Control of the length of the original curve intersected with $\Omega_1$ \emph{in the case $q'' \neq q$} is a consequence of the fact that $e_3$ is quantitatively orthogonal to the tangent vector of the curve, so that the length of the curve intersected with $\Omega_1$ is proportional to the width of $\Omega_1$ \emph{in the $e_1$ and $e_2$ directions}, which is at most $\lambda_{q+\sfrac \bn 2}^{-1}$; note importantly that the diameter of $\Omega'$ is irrelevant here.  In the case $q'' = q$, the diameter of $\Omega'$ is $\lambda_{q+\sfrac \bn 2}\Gamma_{q}^{-1}$, so that upon projection, the same bound holds for the curve around which a deformed pipe segment contained in $\Omega'$ is concentrated. Therefore, the number of grid squares occupied by each enlarged pipe projection is 
$$ \frac{\textnormal{area occupied by an enlarged segment}}{\textnormal{area of a grid square}} \approx \frac{\lhalf^{-1}\lambda_{{q''}+\bn}^{-1}\Gamma^2_{{q''}+\bn}}{\lambda_{q+\bn}^{-2}} \, . $$
Thus the total number of grid squares covered by the union of all projections is 
\begin{align}\label{card.{q''}}
\sim \sum_{{q''}=q-\half +1}^{q}\const_P \Gamma_{{q''}}
\times
\frac{\la_{{q''}+\half}\Ga_{{q''}}^3}{\lambda_{{q''}+\bn}\Gamma_{{q''}+\bn}^{-{2}}}   \frac{\la_{q+\bn}^2\la_{q+\half}^{-1}}{\Gamma_q^8\la_{q}} \, ,
\end{align}
or the product of the two numbers computed above and summed over ${q''}$.  This number will be less than the the number in \eqref{card.possible} if 
\begin{align*}
    \bn \const_p \Gamma_{q+\bn}^2 \Gamma_q^{-2} \frac{\lambda_{{q''}+\half}\lambda_{q+\half}}{\lambda_q \lambda_{{q''}+\bn}} \leq 1
\end{align*}
for $q-\half +1\leq {q''} \leq q$, which is precisely \eqref{eq:dodging:parameterz}.
\end{proof}

\subsection{Dodging for new velocity increment}
\label{ss:stress:oscillation:2}

Before getting to the main result of this subsection in Lemma~\ref{lem:dodging}, we set up a few preliminaries and notations.  First, recall that in the previous subsection, we ignored the fact that the objects in \eqref{wqplusoneonediamond} contain flow maps $\Phiik$ which depend on time. In order to control the geometry of pipes which are deformed by these flow maps on a velocity field on a local Lipschitz timescale, we recall \cite[Lemma~3.7]{NV22}.

\begin{lemma}[\bf Control on Axes, Support, and Spacing]
\label{lem:axis:control}
Consider a convex neighborhood  of space $\Omega\subset \mathbb{T}^3$. Let $v$ be an incompressible velocity field, and define the flow $X(x,t)$ and inverse $\Phi(x,t)=X^{-1}(x,t)$, which solves $\partial_t \Phi + v\cdot\nabla \Phi =0$ with $\Phi|_{t=t_0} = x$. Define $\Omega(t):=\{ x\in\mathbb{T}^3 : \Phi(x,t) \in \Omega \} = X(\Omega,t)$. For an arbitrary $C>0$, let $\tau>0$ be a timescale parameter and $\Gamma > 3$ a large multiplicative prefactor such that the vector field $v$ satisfies the Lipschitz bound $\sup_{t\in [t_0 - \tau,t_0+\tau]} \norm{\nabla v(\cdot,t) }_{L^\infty(\Omega(t))} \lesssim \tau^{-1} \Gamma^{-2}$. Let $\mathcal{W}^k_{\xi,\lambda,r}:\mathbb{T}^3\rightarrow\mathbb{R}^3$ be a set of straight pipe flows constructed as in Proposition~\ref{prop:pipeconstruction} and Proposition~\ref{prop:pipe.flow.current}
which are $(\sfrac{\mathbb{T}}{\lambda r})^3$-periodic and concentrated in the $\pi(4\la n_*)^{-1}$-neighborhoods of axes $\{A_i\}_{i\in\mathcal{I}}$ oriented in the vector direction $\xi$ for $\xi\in\Xi,\Xi'$ 
Then $\mathcal{W}:=\mathcal{W}^k_{\xi,\lambda,r}(\Phi(x,t)):\Omega(t)\times[t_0-\tau,t_0+\tau]$ satisfies the following conditions:
\begin{enumerate}[(1)]
	\item  We have the inequality
	\begin{equation}\label{eq:diameter:inequality}
	\textnormal{diam}(\Omega(t)) \leq \left(1+\Gamma^{-1}\right)\textnormal{diam}(\Omega) \, .
	\end{equation}
\item Let $x$ and $y$ belong to $A_i\cap\Omega$ for some $i$, where the axes $A_i$ are defined above.  Denote the length of the axis $A_i(t):=X(A_i\cap\Omega,t)$ in between $X(x,t)$ and $X(y,t)$ by $L(x,y,t)$.  Then
    \begin{equation}\label{e:axis:length}
    L(x,y,t) \leq \left(1+\Gamma^{-1}\right)\left| x-y \right| \, .
    \end{equation}
    \item The support of $\mathcal{W}$ is contained in a $\displaystyle\left(1+\Gamma^{-1}\right)\twopi (4n_\ast\lambda)^{-1}$-neighborhood of the set
    \begin{equation}\label{e:axis:union}
       \bigcup_{i} A_i(t) \, .
    \end{equation}
\item $\mathcal{W}$ is ``approximately periodic'' in the sense that for distinct axes $A_i,A_j$ with $i\neq j$, we have
\begin{equation}\label{e:axis:periodicity:1}
    \left(1-\Gamma^{-1}\right) \dist(A_i\cap\Omega,A_j\cap\Omega)
    \leq \dist\left(A_i(t),A_j(t)\right)
    \leq \left(1+\Gamma^{-1}\right) \dist(A_i\cap\Omega,A_j\cap\Omega) \, .
\end{equation}
\end{enumerate}
\end{lemma}

A consequence of Lemma~\ref{lem:axis:control} is that a set of $(\sfrac{\T}{\lambda r})^3$-periodic intermittent pipe flows which are flowed by a locally Lipschitz vector field on the Lipschitz timescale can be decomposed into ``segments of deformed pipe.''  Furthermore, any neighborhood of diameter $\approx (\lambda r)^{-1}$ contains at most a finite number of such segments of deformed pipe.

\begin{definition}[\bf Segments of deformed pipes]\label{def:sunday:sunday}\index{segments of deformed pipes}
A single ``segment of deformed pipe with thickness $\la^{-1}$ and spacing $(\la r)^{-1}$'' is defined as a $3\lambda^{-1}$ neighborhood of a continuously differentiable curve of length at most $2(\la r)^{-1}$.
\end{definition}

We now recall and set a few notations which will be used in Lemma~\ref{lem:dodging}. We first recall from \eqref{eq:space:time:balls} the notations $B(\Omega,\lambda^{-1})$ and $B(\Omega,\lambda^{-1},\tau)$ for space and space-time balls, respectively, around a space-time set $\Omega$.  Using these notations and the definition of $\hat w_\qbn$ from Definiton~\ref{def:wqbn}, we have that
\begin{equation}
    \supp \hat w_\qbn \subseteq B\left( \supp w_{q+1}, \sfrac 12 \lambda_\qbn^{-1} , \sfrac 12 \Tau_q \right) \, . \label{eq:dodging:useful:support}
\end{equation}
Next, we recall the formula in \eqref{eq:WW:explicit} for an intermittent Mikado flow and the notation in~\eqref{int:pipe:bundle:short} for a intermittent bundle, and we set\index{$\varrho_{\pxi,\diamond}^I$}
\begin{align}\label{eq:pipez:thursday}
    \varrho_{(\xi),\diamond}^{I} := \xi \cdot \WW_{(\xi),\diamond}^I \, .
\end{align}
Next, in slight conflict with \eqref{eq:space:time:balls}, we shall also use the notation
\begin{align}\label{eq:ballz:useful}
    B\left(\supp \varrho_{\pxi,\diamond}^I,\lambda^{-1}\right) := \left\{ x\in\T^3 \, : \, \exists y \in \supp \varrho_{\pxi,\diamond}^I \, , |x-y| \leq \lambda^{-1} \right\}
\end{align}
throughout this section, despite the fact that $\supp\varrho_{\pxi,\diamond}^I$ is not a set in space-time, but merely a set in space. We shall also use the same notation but with $\varrho_{\pxi,\diamond}^I$ replaced by $\rhob_\pxi^\diamond$. Finally, for any smooth set $\Omega\subseteq \mathbb{T}^3$ and any flow map $\Phi$ defined in Definition~\ref{def:transport:maps}, we use the notation 
\begin{equation}\label{eq:flowing:sets}
\Omega \circ \Phi := \left\{(y,t): t\in \R, \Phi(y,t)\in \Omega\right\} = \supp \left(\mathbf{1}_{\Omega}\circ \Phi\right) \, .
\end{equation}
With this definition, $\Omega\circ\Phi$ is the smooth space-time set whose characteristic function is annihilated by $\Dtq$.

We are now ready to verify the main dodging lemma for $w_{q+1}$, which as noted in \eqref{eq:dodging:useful:support} implies similar dodging properties for $\hat w_\qbn$. 

\begin{lemma}[\bf Dodging and preventing self-intersections for $w_{q+1}$ and $\hat w_\qbn$]\label{lem:dodging} 
We construct $w_{q+1}$ so that the following hold.
\begin{enumerate}[(i)]
    \item\label{item:dodging:more:oldies} Let {$q+1\leq q' \leq q+ \sfrac \bn 2$} and fix indices $\diamond,i,j,k,\xi,\vecl$, which we abbreviate by $(\pxi,\diamond)$, for a coefficient function $a_{\pxi,\diamond}$ (c.f.~\eqref{eq:a:xi:phi:def}, \eqref{eq:a:xi:def}).  Then
    \begin{equation}
        B\left( \supp \hat w_{q'}, \frac 12 {\lambda_{q+1}^{-1}\Ga_q^2}, {2 \Tau_q} \right) \cap \supp \left( \tilde \chi_{i,k,q} \zeta_{q,\diamond,i,k,\xi,\vecl} \, \rhob_{\pxi}^{\diamond}\circ \Phi_{(i,k)} \right) = \emptyset \, . \label{eq:oooooldies}
    \end{equation}
    \item\label{item:dodging:1} Let $q'$ satisfy $q+1\leq q' \leq q+\bn-1$, fix indices $(\pxi,\diamond,I)$, and assume that $\Phiik$ is the identity at time $t_{\pxi}$, cf. Definition~\ref{def:transport:maps}. Then we have that
\begin{align}
    B \left(\supp \hat w_{q'}, \frac 14 \lambda_{q'}^{-1} \Gamma_{q'}^2, 2\Tau_{q} \right)  \cap  \supp &\left( \tilde \chi_{i,k,q} \zeta_{q,\diamond,i,k,\xi,\vecl} \left(\rhob_{(\xi)}^\diamond \zetab_{\xi}^{I,\diamond} \right)\circ \Phiik \right) \notag \\
    &\cap
    B\left( \supp \varrho^I_{(\xi),\diamond} , \frac 12 {\lambda_{q'}^{-1} \Gamma_{q'}^2}\right)\circ \Phiik
    = \emptyset \, . \label{eq:dodging:oldies:prep}
    \end{align}
As a consequence we have
\begin{equation}\label{eq:dodging:oldies}
        B\left( \supp \hat w_{q'}, \frac 14 {\lambda_{q'}^{-1} \Gamma_{q'}^2}, 2\Tau_q \right) \cap  \supp w_{q+1} = \emptyset \, .
    \end{equation}
    \item\label{item:dodging:2} Consider the set of indices $\{(\pxi,\diamond,I)\}$, whose elements we use to index the correctors constructed in \eqref{wqplusoneonediamond}, and let $\ttl, \ov \ttl \in \{p,c\}$ denote either principal or divergence corrector parts. Then if $(\ov\diamond,(\ov \xi), \ov I) \neq (\diamond,(\xi),I)$, we have that for any $\ttl, \ov \ttl$,
    \begin{equation}\label{eq:dodging:newbies}
        \supp w_{\pxi,\diamond}^{(\ttl),I} \cap \supp w_{(\ov \xi),\ov \diamond}^{(\ov \ttl), \ov I} = \emptyset \, .
    \end{equation}
    \item\label{item:dodging:zero}  $\hat w_\qbn$ satisfies Hypothesis~\ref{hyp:dodging2} with $q$ replaced by $q+1$.
\end{enumerate}
\end{lemma}

\begin{remark}[\bf Verifying Hypothesis~\ref{hyp:dodging1}]\label{rem:checking:hyp:dodging:1}
We claim that \eqref{eq:dodging:oldies} and \eqref{eq:dodging:useful:support} imply that Hypothesis~\ref{hyp:dodging1} holds with $q+1$ replacing all instances of $q$.  To check this, we must show that \eqref{eq:ind:dodging} holds for $q',q''\leq \qbn$ and $0<|q'-q''|\leq \bn-1$.  By induction on $q$ and the symmetry of $q''$ and $q'$, the only case we must check is the case that $q+\bn=q''$ and $0<\qbn-q'\leq\bn-1$. But it is a simple exercise in set theory to check that for $q+1\leq q'\leq \qbn-1$, \eqref{eq:dodging:oldies} is equivalent to $\supp \hat w_{q'} \cap B(\supp w_{q+1},\sfrac 14 \lambda_{q'}^{-1}\Ga_{q'}^2,2\Tau_q) = \emptyset$. Then using \eqref{eq:dodging:useful:support} and the inequalities $\lambda_{q'}^{-1}\Ga_{q'}^2 \geq \la_\qbn^{-1}$, $b<2 \implies \Ga_{q'+1}\ll \Ga_{q'}^2$ implies that \eqref{eq:ind:dodging} holds.
\end{remark}

\begin{proof}[Proof of Lemma~\ref{lem:dodging}]
We split the proof up into steps, in which we first carry out some preliminary set-up before verifying item~\eqref{item:dodging:more:oldies}, items~\eqref{item:dodging:1}--\eqref{item:dodging:2}, and finally item~\eqref{item:dodging:zero}.
\smallskip

\noindent\texttt{Step 0: Ordering of cutoff functions and set-up.} Consider all coefficient functions $a_{\xi,i,j,k,\vecl,\diamond}$ utilized at stage ${q+1}$, cf.~\eqref{eq:a:xi:def} and \eqref{eq:a:xi:phi:def}. Using natural numbers $z\in\mathbb{N}$ as indices, we choose an ordering of the tuples $(i,j,k,\xi,\vecl,\diamond)$ such that for any choice of $(i,j,k,\xi,\vecl,\diamond)$ and $(\istar,\jstar,\kstar,\xistar,\vecl^*,\diamond^*)$, we have
\begin{equation}\label{eq:tricky:ordering}
 i<\istar \implies (i,j,k,\xi,\vecl,\diamond) <_{\textnormal{ordering}} (\istar,\jstar,\kstar,\vecl^*,\diamond^*) \, ,
 \end{equation}
where the implied inequality holds for the natural numbers assigned to each tuple in our chosen ordering. This automatically provides an ordering for the coefficient functions $a_{\xi,i,j,k,\vecl,\diamond}$ and associated pipe bundles $\BB_{\pxi,\diamond}\circ\Phiik$. We will place pipe bundles inductively according to this ordering so that all the conclusions in the statement of Lemma~\ref{lem:dodging} hold. \eqref{eq:tricky:ordering} ensures that timescales are decreasing with respect to this ordering and mitigates the fact that the number of possible overlaps between $\psi_{i,q}\chi_{i,k,q}$ and $\psi_{i+1,q}\chi_{i+1,k',q}$ could be of order $\Gamma_q$ (see \ref{eq:chi:support}).\footnote{This is not strictly necessary-- one can always adjust the choice of parameters to accommodate a spare $\Gamma_q$.}  To lighten the notation, we will abbreviate the newly ordered coefficient and cutoff functions and associated intermittent pipe bundles (cf. \eqref{eq:a:xi:phi:def},  \eqref{eq:a:xi:def}, and \eqref{int:pipe:bundle:short}) as 
$$ a_z \, , \psi_z \, , \omega_z \, , \chi_z \, , \zeta_z \, , \qquad  \quad (\BB\circ\Phi)_z = \rhob_z \circ \Phi_z \sum_I ( \zetab_z \WW_z^I ) \circ \Phi_z \, , $$
respectively, where $z\in\mathbb{N}$ corresponds to the ordering. Now for fixed $z$ and $a_z$, we will place $\rhob_z$ and $\WW_z^I$ with two goals in mind.  First, we must dodge the velocity increments $\hat w_{q'}$ for $q+1 \leq q' \leq q+ \half$ and $\hat w_{q''}$ for $q + \half + 1 \leq q'' \leq q+\bn-1$. Second, we must dodge all pipe bundles $(\BB\circ \Phi)_{\hat z}$ with coefficient functions $a_{\hat{z}}$ such that $\hat z < z$ in the aforementioned ordering.
\smallskip

\noindent\texttt{Step 1: Proof of item~\eqref{item:dodging:more:oldies}.} We will apply Lemma~\ref{lem:coarsepipe} with the following choices. We recall that at the time $t_z$ at which $\Phi_z$ is the identity, the cutoff function $\eta_z$ contains a checkerboard cutoff function $\zeta_{z}$ which from Lemma~\ref{lem:checkerboard:estimates}, item~\eqref{item:checkeeee} is contained in a rectangular prism of dimensions no larger than $\sfrac 34 \lambda_q^{-1}\Gamma_q^{-8}$ in the direction of $\xi_z$, and $\sfrac 34 \Gamma_{q}^5(\lambda_{q+1})^{-1}$ in the directions perpendicular to $\xi_z$. Thus we set 
\begin{equation}\notag
\Omega_0 = \supp \zeta_z \cap \{t=t_z\} \, .
\end{equation}
Notice that $\textnormal{diam}(\Omega_0)\leq \lambda_q^{-1}\Gamma_q^{-8}$, which satisfies \eqref{eq:diam:def} for $\bar q', \bar q''$ chosen as $\bar q''=q$ and $\bar q'=q'$ as in \eqref{item:dodging:more:oldies} so that $q+1\leq q' = \bar q' \leq q+\half$. Then by applying Hypothesis~\ref{hyp:dodging2} at level $q$ with $\bar q' = q'$, $\bar q''=q$, $\Omega=\Omega_0$ as defined above, $t_0=t_z$, and $\Phi_{\bar q''} = \Phi_z$, we have that for each $q+1\leq q' \leq q+\half$, there exists a set $L(q',q,\Omega_0,t_z)$ such that \eqref{eq:ind:dodging2} and \eqref{eq:concentrazion} hold. Now we set
$$ E_0 := \bigcup_{q'=q+1}^{q+\half} L(q',q,\Omega_0,t_z) \cap \{t=t_z\} \, , \qquad \mathcal{C}_P = \mathcal{C}_D \bn \, . $$
We now appeal to the conclusion of Lemma~\ref{lem:coarsepipe} to choose a placement for $\rhob_z$ such that
\begin{equation}\label{eq:desert:dodging:0}
B\left( \supp \rhob_z , \lambda_{q+1}^{-1}\Gamma_q^2 \right) \cap E_0 = \emptyset \, .
\end{equation}

An immediate consequence of \eqref{eq:desert:dodging:0}, Hypothesis~\ref{hyp:dodging2}, and \eqref{e:Phi} is that for $t$ such that $|t-t_z|\leq \tau_q\Ga_q^{-i_z+2}$,
\begin{equation}\label{eq:desert:dodging:1}
 \Dtq \left( \mathbf{1}_{B\left(\supp \rhob_z, \la_{q+1}^{-1}\Ga_q^2\right) \circ \Phi_z} \mathbf{1}_{L(q',q,\Omega_0,t_z)} \right)(t,x) \equiv 0 \, 
\end{equation}
in distribution sense.\footnote{Here, we used that for any subset $A\subset \mathbb{T}^3$, 
$\iint \mathbf{1}_{A\circ\Phi} D_{t,q}\phi\, dxdt =
\iint \mathbf{1}_{A}\circ\Phi D_{t,q}\phi \, dxdt
=\lim_{k\to \infty}
\iint \eta_k\circ\Phi D_{t,q}\phi \, dxdt =0
$
holds for any $\phi\in C_c^\infty(\mathbb{T}^3\times\mathbb{R})$, where $\eta_k$ is a sequence of smooth functions converging to $\mathbf{1}_{A}$ in $L^1(\mathbb{T}^3)$.} This in turn implies that in the same range of $t$, 
\begin{equation}
    B\left(\supp \rhob_z, \la_{q+1}^{-1}\Ga_q^2\right) \circ \Phi_z(t) \cap L(q',q,\Omega_0,t_z) \cap \{t=t_z\} = \emptyset \, . \label{eq:friday:dodging:2}
\end{equation}
Next, we claim that \eqref{eq:friday:dodging:2} implies that 
\begin{equation}\label{eq:friday:dodging:1}
    B\left(\supp \rhob_z \circ \Phi_z \cap (\T^3\times \{t\}), \sfrac 34 \la_{q+1}^{-1}\Ga_q^2\right) \cap L(q',q,\Omega_0,t_z) = \emptyset \quad \textnormal{for $|t-t_z|\leq \tau_q\Ga_q^{-i_z+2}$} \, ,
\end{equation}
which we now prove. Indeed, this follows from the fact that on the Lipschitz timescale $\tau_q\Ga_q^{-i_z+2}$, spatial distances can change by at most a multiplicative factor of $(1\pm\Ga_q^{-1})$ due to \eqref{eq:nasty:D:vq:old} (see also Lemma~\ref{lem:axis:control}, which contains similar assertions). Finally, we claim that
\begin{align}
     &B\left( \bigcup_{|t-t_z|\leq \frac 12 \tau_q\Ga_q^{-i_z+2} } \supp \rhob_z \circ \Phi_z \cap (\T^3\times\{t\}) , \frac 12 \la_{q+1}^{-1}\Ga_q^2, 2\Tau_q \right) \notag\\
     &\qquad \qquad \subseteq \bigcup_{|t-t_z|\leq \tau_q\Ga_q^{-i_z+2} }  B\left( \supp \rhob_z \circ \Phi_z \cap (\T^3\times\{t\}) , \sfrac 34 \la_{q+1}^{-1}\Ga_q^2 \right) \, . \label{eq:desert:dodging:2}
\end{align}
Assuming that \eqref{eq:desert:dodging:2} holds, we have then from \eqref{eq:friday:dodging:1}, \eqref{eq:chi:tilde:support}, and Hypothesis~\ref{hyp:dodging2} that
$$  B\left( \supp \left(\tilde\chi_z \zeta_z \rhob_z \circ \Phi_z\right) , \sfrac 12 \la_{q+1}^{-1}\Ga_q^2, 2\Tau_q \right) \cap \supp \hat w_{q'} = \emptyset \, , $$
which is equivalent to \eqref{eq:oooooldies} after using the same sort of set-theoretic reasoning as in Remark~\ref{rem:checking:hyp:dodging:1}. To prove \eqref{eq:desert:dodging:2}, suppose that $(\td x, \td t)$ belongs to the set on the left-hand side of the inclusion in \eqref{eq:desert:dodging:2}.  Then by definition, there exists $(t_0,x_0)$ such that $|\td t - t_0| \leq 2\Tau_q$, $|\td x - x_0|\leq \sfrac 12 \la_{q+1}^{-1}\Ga_q^2$, $|t_0-t_z|\leq \sfrac 12 \tau_q\Ga_q^{-i_z+2}$, and $(x_0,t_0)\in\supp \rhob_z\circ\Phi_z \cap (\T^3\times \{t_0\})$. Then from \eqref{eq:imax:old} and \eqref{v:global:par:ineq}, we have that $|\td t-t_z|\leq \tau_q\Ga_q^{-i_z+2}$. So we need to find $x'$ such that $(x',\td t)\in \supp \rhob_z\circ\Phi_z\cap (\T^3\times\{\td t\})$ and $|x'-\td x| <\sfrac 34 \la_{q+1}^{-1}\Ga_q^2$. Now from \eqref{e:Phi}, \eqref{eq:bobby:old}, Corollary~\ref{cor:deformation}, and \eqref{v:global:par:ineq}, we have that
\begin{align*}
    \left\| {\Phi_z}(\td t, \cdot) - \Phi_z(t_0,\cdot) \right\|_{L^\infty(\T^3)} &\les \Tau_q \left\| \partial_t \Phi_z \right\|_{L^\infty\left( \T^3\times  (t_z-\tau_q\Ga_q^{-i_z+2}, t_z+\tau_q\Ga_q^{-i_z+2})\right)} \\
    &\les \Tau_q \left\| \hat u_q \right\|_{\infty} \left\| \nabla \Phi_z \right\|_{L^\infty\left( \T^3\times  (t_z-\tau_q\Ga_q^{-i_z+2}, t_z+\tau_q\Ga_q^{-i_z+2})\right)} \\
    &\les \Ga_q^{-1} \la_{q+1}^{-1} \, .
\end{align*}
Therefore, although it may not be the case that $(x_0,\td t) \in \supp \rhob_z \circ \Phi_z \cap (\T^3\times \{\td t\})$, there must exist $x'$ such that $|x'-x_0|\leq \Ga_q^{-\sfrac 12}\la_{q+1}^{-1}$ and $(x',\td t) \in \supp \rhob_z \circ \Phi_z \cap (\T^3\times \{\td t\})$.\footnote{We have that $(x_0,t_0)\in \supp \rhob_z\circ \Phi_z$ if and only if $\Phi_z(x_0,t_0)\in \supp \rhob_z$. Using the bound on the difference between $\Phi_z(\td t)$ and $\Phi_z(t_0)$, we may say that although $\Phi_z(x_0,\td t)$ is not necessarily in the support of $\rhob_z$, it is very close.}  Since $|x'-\td x|\leq |x'-x_0|+|x_0-\td x| \leq \Ga_q^{-\sfrac 12}\la_{q+1}^{-1} + \sfrac 12 \la_{q+1}\Ga_q^2 < \sfrac 34\la_{q+1}^{-1}\Ga_q^2$, we have thus concluded the proof of \eqref{eq:desert:dodging:2}.

\medskip

\noindent\texttt{Step 2: Proofs of items~\eqref{item:dodging:1} and \eqref{item:dodging:2}.}  In order to proceed with this portion of the proof, we assume inductively that a version of Hypothesis~\ref{hyp:dodging2} holds for the portion of $w_{q+1}$ already constructed. More precisely, we extend the ordering on $z \in \mathbb{N}$ from \texttt{Step 0} to ordered pairs $(z,I) \in \mathbb{N}^2$ such that $\hat z < z \implies (\hat z, \hat I) < (z,I)$ for any $\hat I, I$ (that is, we fix $z$ and finish placing an entire bundle for all its various values of $I$ before moving to different $\hat z$).  We thus assume the following inductive hypothesis.
\begin{hypothesis}[\bf Density of already placed pipe bundles in $w_{q+1}$]\label{hyp:dodging22}
There exists a geometric constant $\mathcal{C}_{\rm pipe}$ such that the following holds. Fix $z$ and set
\begin{align*}
    \mathfrak{s}_{\hat z, \hat I}(t) &:= \supp \left[\chi_{\hat i,\hat k,q} \zeta_{(\hat \xi)} \left(\rhob^{\diamond}_{(\hat \xi)}\zetab^{\hat I,\diamond}_{\hat \xi}\right) \circ \Phi_{(\hat i, \hat k)}\right] \cap  B\left(\varrho^{\hat I}_{(\hat \xi),\diamond} , \frac 12 {\lambda_{q+\bn}^{-1} \Gamma_{q}^2}\right) \circ \Phi_{(\hat i, \hat k)}\cap (\T^3 \times \{t\})
    \, , \\
    w_{z,I} &:= \sum_{(\hat z,\hat I) < (z,I)} {a_{\hat z} (\rhob_{\hat z} \zetab_{\hat z}^{\hat I} \WW^{\hat I}_{\hat z})\circ\Phi^{\hat I}_{\hat z}} \,, \qquad \mathfrak{S}_{z, I}(t) 
    := \bigcup_{(\hat z,\hat I) < (z,I)} \mathfrak{s}_{\hat z, \hat I}(t) \, .
\end{align*}
Let $i_z$ be the value of $i$ corresponding to $z$ and $a_z$, and let $t_0$ be any time and $\Omega \subset \mathbb{T}^3$ be a convex set with diameter at most $\left(\la_{q+\half}\Gamma_{q}\right)^{-1}$ such that $\Omega \times \{t_0\} \cap \supp \psi_{i_z,q} \neq \emptyset$. Let $\Phi$ solve $\Dtq\Phi=0$ with initial data $\Phi|_{t=t_0}=\Id$. We set $\Omega(t) = \Phi(t)^{-1}(\Omega)$ and 
\begin{align*}
    \mathcal{N}_{\Omega,z,I} = \# \left\{(\hat z,\hat I) < (z,I) \, : \, \exists t\in [t_0 - \tau_q \Ga_q^{-i_z-2}, t_0 + \tau_q \Ga_q^{-i_z-2}] \textnormal{ with } \mathfrak{s}_{\hat z, \hat I}(t) \cap \Omega(t) \neq \emptyset \right\} \, .
\end{align*}
Then there exists an $\Omega$-dependent set $L_{(z,I)}\subseteq\Omega$ consisting of at most $\mathcal{N}_{\Omega,z,I} \mathcal{C}_{\rm pipe}$ segments of deformed pipe segments with thickness $\la_{q+\bn}^{-1}$ such that for all $t\in [t_0 - \tau_q \Gamma_q^{-i_z-2}, t_0 + \tau_q \Gamma_q^{-i_z-2}]$,
\begin{align}\label{eq:ind:dodging22}
    \left[ \supp w_{z,I}(\cdot, t) \cap \Omega(t) \right] \subseteq \left[ \mathfrak{S}_{z,I}(t) \cap \Omega(t) \right] \subseteq \left[ \Phi(t)^{-1}(L_{(z,I)}) \cap \Omega(t) \right] \, .
\end{align}
\end{hypothesis}
One should understand this hypothesis as asserting that at all steps in the construction of $w_{q+1}$, there is no more than a finite number of pipe segments of thickness $\la_\qbn^{-1}$ in any set of diameter proportional to the size of a periodic cell. 
Indeed, $\mathcal{N}_{\Omega,z,I}$ is bounded independently of $z$, $I$, and hence $q$. From the finite maximal cardinality of the indices $(\hat i,\hat  j, \hat \xi, \hat  \diamond)$ and decreasing time scale with respect to the ordering \eqref{eq:tricky:ordering}, the indices $(\hat i,\hat  j, \hat k, \hat \xi, \hat  \diamond)$ takes a finite number, independent of $z$, $I$, and $q$. Fix these indices and we now count the remaining indices $(\hat \vecl, \hat I)$. Since $\Phiik$ and $\Phi$ are advected by the same velocity field $\hat u_q$, recalling Definition \ref{eq:checkerboard:definition}, it is enough to count the indices to have $\mathfrak{s}_{\hat z, \hat I}(t) \cap \Omega(t) \neq \emptyset$ at some fixed time $\bar t\in \supp \chi_{\hat i, \hat k, q}\cap [t_0 - \tau_q \Gamma_q^{-i_z-2}, t_0 + \tau_q \Gamma_q^{-i_z-2}]$. From the diameter bound on $\Omega$ and Lemma \ref{lem:axis:control}, we have $\text{diam}(\Omega(\bar t))\leq 2(\la_{q+\half}\Ga_q)^{-1}$, while the spatial derivative costs of $\zeta_{(\hat\xi)}$ and $\zetab_{\hat \xi}^{\hat I, \hat \diamond}\circ\Phi_{(\hat i, \hat k)}$ are $\la_{q+1}\Ga_q^{-5}$ and $\la_{q+\half}$, respectively, from \eqref{eq:checkerboard:derivatives} and \eqref{eq:checkerboard:derivatives:check}. Since the inverse of the derivative costs are much 
greater than $2(\la_{q+\half}\Ga_q)^{-1}$, only for finite number of indices $\hat\vecl$ and $\hat I$, the intersection $\mathfrak{s}_{\hat z, \hat I}(\bar t) \cap \Omega(\bar t) \neq \emptyset$ occurs, where the number is independent of $z$, $I$, and $q$. Therefore, we can set an upper bound of $\mathcal{N}_{\Omega,z,I}$ as a geometric constant $\const$. Lastly, we note that Hypothesis~\ref{hyp:dodging22} is vacuously true in the base case where $(z,I)$ is the smallest element in our ordering.


We will now justify the application of Lemma~\ref{lem:finepipe}. We recall from \eqref{eq:a:xi:phi:def} and \eqref{eq:a:xi:def} that at the time $t_z$ at which $\Phi_z$ is the identity, $a_z \rhob_{z} \zetab_{ z}^{ I} $ contains both a strongly anisotropic checkerboard cutoff function $\zetab_{z}^{I}$ and a mildly anistropic checkerboard cutoff function $\zeta_z$. The support of the product of these cutoff functions is contained in a rectangular prism of dimensions no larger than $\sfrac 34\lambda_q^{-1}\Gamma_q^{-8}$ in the direction of $\xi_z$ from item~\eqref{item:checkeeee} of Lemma~\ref{lem:checkerboard:estimates}, and $\lambda_{q+\floor{\bn/2}}^{-1}$ in the directions perpendicular to $\xi_z$ from Definition~\ref{def:etab}.  Thus we can contain the support of $a_z\rhob_z\zetab_z^I$ at time $t_z$ inside a prism of dimensions $\sfrac 34 \lambda_+q^{-1}\Ga_q^{-8}$ and $\lambda_{q+\floor{\bn/2}}^{-1}$, and so we set 
$$ \Omega_1=\supp \zeta_z \zetab_{z}^{I} \cap \{t=t_z\} \, . $$
By applying Hypothesis~\ref{hyp:dodging2} at level $q$ with $\bar q' = q''$ for each $ q+\half \leq q'' \leq \qbn -1$, $\bar q'' = q$, $\Omega'\subset\Omega_1$ any convex subset of diameter at most $(\lambda_{q''-\bn+\half}\Gamma_{q''-\bn})^{-1}$ (which satisfies \eqref{eq:diam:def}), $t_0=t_z$, and $\Phi_{\bar q''}$ as defined in Hypothesis~\ref{hyp:dodging2}, we have that there exists a set $L(q'',q,\Omega',t_z)$ such that \eqref{eq:ind:dodging2} and \eqref{eq:concentrazion} hold. We therefore see that the density condition of Lemma~\ref{lem:finepipe} is verified with 
$$ \mathcal{C}_P= \mathcal{C}_{\rm pipe} \const + \bn \const_D $$
from Hypothesis~\ref{hyp:dodging2}, which contributes the second term counting the number of old pipe segments belonging to $\hat w_{q''}$ for $q + \half + 1 \leq q'' \leq q+\bn-1$, and our inductive Hypothesis~\ref{hyp:dodging22}, which contributes the first term counting the number of current pipe segments belong to $w_{z,I}$.  We define $E_{q''}$ (including the endpoint case $q''=q+\bn$ which contains already placed pipes from $w_{z,I}$) as in Lemma~\ref{lem:finepipe} so that it contains the support of $\hat w_{q''}$ inside $\Omega_1$ if $q''<q+\bn$, and $w_{z,I}$ inside $\Omega_1$ in the endpoint case.

Now appealing to the conclusion of Lemma~\ref{lem:finepipe}, we may choose the support of $\WW^I_z=(\WW\circ\Phi)^I_z|_{t=t_z}$ so that for $q+\half+1\leq q''\leq q+\bn-1$,
\begin{align*}
    \suppp_x \, (\WW\circ\Phi)^I_z|_{t=t_z} \cap  B\left( \supp \hat w_{q''}(\cdot, t_z), \lambda_{q''}\Ga_{q''}^2 \right) \cap \Omega = \emptyset  \, , \\
    \suppp_x \, (\WW\circ\Phi)^I_z|_{t=t_z} \cap  B\left( {\mathfrak{S}_{z, I}(t_z)} , \lambda_{q+\bn}^{-1}\Gamma_{q+\bn}^2 \right) \cap \Omega = \emptyset \, . 
\end{align*}
Reasoning as in the final portion of \texttt{Step 1} and assuming for the moment that Hypothesis~\ref{hyp:dodging22} can be propagated throughout the construction of $w_{q+1}$, we have that the first assertion implies \eqref{eq:dodging:oldies:prep} and therefore also the weaker assertion \eqref{eq:dodging:oldies}. In addition, the second assertion verifies \eqref{eq:dodging:newbies} at $t=t_z$, and \eqref{eq:dodging:newbies} at all times follows from a similar type of argument (in fact simpler since no expansion by $2\Tau_q$ in time is required), but with Hypothesis~\ref{hyp:dodging22} replacing Hypothesis~\ref{hyp:dodging2}. 

We now verify that Hypothesis~\ref{hyp:dodging22} has been preserved by the placement of $\WW_{z}^I$. Note that we have placed $\left(\la_{q+\half} \Ga_q\right)^{-1}$-periodic straight pipes at time $t=t_z$ and have composed them with a diffeomorphism $\Phi_z$. Furthermore, this diffeomorphism and the vector field $\hat u_q$ obey the conditions and conclusions of Lemma~\ref{lem:axis:control} on the support of $a_{z} \rhob_{z} \zetab_{ z}^{ I}$.  Thus we have that there exists $\const_{\rm pipe}$ such that for any convex set $\Omega'$ of diameter at most $\left(\la_{q+\half} \Ga_q\right)^{-1}$ and for any time $t$, 
\begin{align*}
\supp (\WW \circ \Phi)^I_z(\cdot, t) \cap \Omega'(t) \cap \supp((a_{z} (\rhob_{z} \zetab_{ z}^{ I})\circ\Phi_z)(\cdot, t) 
\end{align*}
is contained in at most $\const_{\rm pipe}$ deformed pipe segments. Now fix a convex set $\Omega$ and a time $t_0$ as in Hypothesis~\ref{hyp:dodging22}, and let $(z,I)^+$ denote the next element after $(z,I)$ in our ordering. Then if $\Omega(t)$ has empty intersection with the spatial support of $a_{z} (\rhob_{z} \zetab_{z}^{ I})\circ\Phi_z$ for all times $t$, we define $L_{(z,I)^+}=L_{(z,I)}$. If not, we set
$$ L_{(z,I)^+} = L_{(z,I)} \cup \supp_x ( \WW \circ \Phi)_z^I(t_0) \cap \Omega  \, . $$
Then to see that we have verified Hypothesis~\ref{hyp:dodging22}, in particular \eqref{eq:ind:dodging22}, we use that all flow maps $\Phi_{\hat i,\hat k}$ and $\Phiik$ are advected by the same velocity field $\hat u_q$ and the observation above concerning the support of $(\WW\circ \Phi)_z^I$.
\smallskip


\noindent\texttt{Step 3: Proof of item~\eqref{item:dodging:zero}, or Hypothesis~\ref{hyp:dodging2} for $q+1$.} In order to verify Hypothesis~\ref{hyp:dodging2} at level $q+1$, we must consider $\bar q' = q+\bn$ and any $\bar q'' $ such that $q+1\leq \bar q'' < q+\bn$, any convex set $\Omega$ of diameter\footnote{Hypothesis~\ref{hyp:dodging2} in fact allows for sets of smaller diameter, for which the results follow trivially by embedding a smaller set into a set of the largest possible diameter.}
$$d(q+\bn,\bar q'') = \min\left[ (\lambda_{\bar q''}\Gamma_{\bar q''}^7)^{-1} , (\lambda_{q+\half}\Gamma_q)^{-1} \right] \, , $$
any time $t_0$, and any $i''$ such that $\Omega\times\{t_0\} \cap \supp \psi_{i'', \bar q''}\neq \emptyset$.  Given these choices and the flow map $\Phi_{\bar q''}$ as defined in Hypothesis~\ref{hyp:dodging2}, we must define $L(q+\bn,\bar q'',\Omega,t_0)$ satisfying \eqref{eq:ind:dodging2} and \eqref{eq:concentrazion}. We divide the proof into substeps, in which we first count the number of cutoff functions which can overlap with $\Omega(t)$, before defining $L(\qbn,\bar q'', \Omega,t_0)$ and verifying \eqref{eq:ind:dodging2} and \eqref{eq:concentrazion} in the second substep.
\smallskip

\noindent\texttt{Step 3a:  Counting overlap.} Let us pick any fixed but arbitrary $x\in\Omega$ and set
\begin{equation}\label{eq:sat:aft}
\tilde \Omega = B\left(x,3d(\qbn,\bar q'') + 3\la_{\qbn}^{-1}\right) \, . 
\end{equation}
Note that $\tilde\Omega$ contains a ball of radius $3\lambda_\qbn^{-1}$ around $\Omega$. We claim that the cardinality of the set of indices $(\xi,i,j,k,\vecl,\diamond,I)$ such that
\begin{align}
\suppp \left[ \chi_{i,k,q} \zeta_{q,\diamond,i,k,\xi,\vecl} \left( \rhob_\pxi^\diamond \zetab_\xi^{I,\diamond} \right)\circ \Phiik \right] &\cap B\left( \varrho_{\pxi,\diamond}^I, 3\la_\qbn^{-1} \right)\circ \Phiik \notag\\
&\cap \tilde\Omega\circ \Phi_{\bar q''} \cap \left[ \T^3 \times \{ |t-t_0|\leq 2\tau_{\bar q''}\Ga_{\bar q''}^{-i''+2} \} \right] \neq \emptyset  \label{friday:dodging:3} 
\end{align}
is bounded by a finite, $q$-independent constant $\const_{\rm counting}$. We first count the possible values for $i$ and $k$.  From the diameter bound on $\Omega$, \eqref{eq:sharp:Dt:psi:i:q:old}, Lemma~\ref{lem:dornfelder}, and Corollary~\ref{cor:dornfelder}, we have that $\tilde\Omega\circ \Phi_{\bar q''}\subseteq \displaystyle\cup_{\texttt{i}=i''-1}^{i''+1}\suppp \psi_{\texttt{i},\bar q''}$ restricted to $\T^3 \times \{ |t-t_0|\leq 2\tau_{\bar q''}\Ga_{\bar q''}^{-i''+2} \}$.
From \eqref{eq:inductive:timescales}, we have that if $t$ is in the same range as above, $(x,t)\in\tilde\Omega\circ \Phi_{\bar q''}$, and $\psi_{i,q}(x,t)\neq 0$, then it must be the case that $2\tau_{\bar q''}\Ga_{\bar q''}^{-i''+2}\leq \tau_q\Ga_q^{-i-7}$.  Thus if $i$ is such that $\psi_{i,q}(x,t)\neq 0$ at some $(x,t)\in\tilde\Omega\circ \Phi_{\bar q''}$ for $t$ in the same range as above, from \eqref{eq:chi:support} there exist at most two values of $k$ such that $\chi_{i,k,q}$ satisfies
\begin{align*}
\suppp \left(\psi_{i,q} \chi_{i,k,q}\right) \cap \tilde\Omega\circ \Phi_{\bar q''} 
\cap \left[ \T^3 \times \{ |t-t_0|\leq 2\tau_{\bar q''}\Ga_{\bar q''}^{-i''+2} \} \right]
\neq \emptyset 
\end{align*}  
Next, recall that we have bounds $\imax$ and $\jmax$ for the number of values of $i$ and $j$, and $\xi$ and $\diamond$ take a finite number of $q$-independent values.

In order to conclude the proof of the claim concerning intersections with \eqref{friday:dodging:3}, it only remains to count $\vecl$ and $I$ for fixed $(\xi, i, j, k, \diamond)$. Since $i$ and $k$ are fixed, we can drop the time cutoff $\chi_{i,k,q}$ and consider in the intersection in the time interval $[t_0 -2\tau_{\bar q''}\Ga_{\bar q''}^{-i''+2}, t_0 + 2\tau_{\bar q''}\Ga_{\bar q''}^{-i''+2}]\cap \supp\chi_{i,k,q}$. We then observe that from \eqref{eq:chi:tilde:support} and \eqref{eq:dodging:oldies:prep}, 
\begin{align}
    D_{t,\bar q''} &\mathbf{1}_{\left\{\zeta_{\mathfrak{i}} \left( \rhob_{\mathfrak{i}} \zetab_{\mathfrak{i}} \right) \circ \Phi_{\mathfrak{i}} \cap  B\left(\varrho_{\mathfrak{i}}, 3\la_{\qbn}^{-1} \right) \circ \Phi_{\mathfrak{i}}\right\}} \notag\\
    &= D_{t,q} \mathbf{1}_{\left\{\zeta_{\mathfrak{i}} \left( \rhob_{\mathfrak{i}} \zetab_{\mathfrak{i}} \right) \circ \Phi_{\mathfrak{i}} \cap  B\left(\varrho_{\mathfrak{i}}, 3\la_{\qbn}^{-1} \right) \circ \Phi_{\mathfrak{i}}\right\}} 
     + \left(D_{t,\bar q''} - \Dtq \right) \mathbf{1}_{\left\{\zeta_{\mathfrak{i}} \left( \rhob_{\mathfrak{i}} \zetab_{\mathfrak{i}} \right) \circ \Phi_{\mathfrak{i}} \cap  B\left(\varrho_{\mathfrak{i}}, 3\la_{\qbn}^{-1} \right) \circ \Phi_{\mathfrak{i}}\right\}} = 0 \, .
     \label{replaceing:flowmap}
\end{align}
on $\T^3 \times \{ |t-t_0|\leq 2\tau_{\bar q''}\Ga_{\bar q''}^{-i''+2} \}$.
Here, recalling \eqref{eq:checkerboard:definition}, the first term vanishes since $\Dtq \Phi_{\mathfrak{i}}\equiv 0$, while the second term vanishes due to dodging. It follows that the set $\left(\suppp  \left( \mathcal{X}_{q,\xi,\vecl,\diamond} \rhob_\pxi^\diamond \zetab_\xi^{I,\diamond} \right)\cap B\left( \varrho_{\pxi,\diamond}^I, 3\la_\qbn^{-1} \right)\right) \circ \Phiik$ remains the same even though the deformation is replaced by the one induced by the vector field $\hat u_{\bar q''}$. Therefore, applying the same argument to get the upper bound of $\mathcal{N}_{\Omega, z, I}$ in \texttt{Step 2}, we can count the remaining indices at some fixed time and conclude the proof the claim concerning the cardinality of the set of indices satisfying \eqref{friday:dodging:3}.

We define 
$$  \mathfrak{I}_{\Omega,t_0} = \left\{ (\xi,i,j,k,\vecl,\diamond,I) \textnormal{ such that \eqref{friday:dodging:3} holds} \right\} $$
and note that its cardinality is bounded by the $q$-independent constant $\const_{\rm counting}$. In the remainder of the proof we shall abbreviate a tuple of indices $(\xi,i,j,k,\vecl,\diamond,I)$ with $\mathfrak{i}$ and use $\mathfrak{i}$ as a subscript/superscript on any cutoff functions or flow maps which are part of $w_{q+1}$.
\smallskip

\noindent\texttt{Step 3b: Defining $L$ and checking \eqref{eq:ind:dodging2} and \eqref{eq:concentrazion}.} We now define
\begin{align*}
    L(q+\bn,\bar q'',\Omega, t_0) =  
    \tilde \Omega\circ \Phi_{\bar q''} 
    \bigcap \left[ \bigcup_{\mathfrak{I}_{\Omega,t_0}}  \supp \left[ \zeta_{\mathfrak{i}} \left( \rhob_{\mathfrak{i}} \zetab_{\mathfrak{i}} \right) \circ \Phi_{\mathfrak{i}} \right] \cap  B\left(\varrho_{\mathfrak{i}}, 3\la_{\qbn}^{-1} \right) \circ \Phi_{\mathfrak{i}} \right] \, .  
\end{align*}
The first claim easily follows from \eqref{replaceing:flowmap} and $D_{t,\bar q''}\Phi_{\bar q''}=0$.


In order to prove the second claim in \eqref{eq:ind:dodging2}, we first note that by the definition of $\mathfrak{I}_{\Omega,t_0}$, $L(\qbn,\bar q'',\Omega, t_0)$ contains $\supp w_{q+1} \cap \, \Omega\circ\Phi_{\bar q''}$.  Then due to the fact that in the definition of $L$ we have enlarged the support of each $\varrho_{\mathfrak{i}}$, the fact that $\tilde\Omega(t):=\Phi_{\bar q''}(t)^{-1}(\tilde\Omega)$ contains a ball of radius $2\lambda_\qbn^{-1}$ around $\Omega(t)$ for all $|t-t_0| \leq \tau_{\bar q''}\Ga_{\bar q''}^{-i''+2}$, and the fact that \eqref{friday:dodging:3} has doubled the timescale over which overlap is being considered, we have that the second claim in \eqref{eq:ind:dodging2} follows from Definition~\ref{def:wqbn}.

Next, we must check \eqref{eq:concentrazion}.  Note that at time $t_0$, $L$ is defined using intermittent pipe bundles which have been deformed on the Lipschitz timescale $2\tau_{\bar q''}\Ga_{\bar q''}^{-i''+2} \leq \tau_q \Ga_{q}^{-i-7}$.  Note furthermore that due to \texttt{Step 3a}, we are only considering $\const_{\rm counting}$ many such bundles, and that due to the fact that the diameter of $\tilde\Omega(t)$ is bounded on the Lipschitz timescale by a constant times the size $\la_{q+\half}^{-1}\Ga_q^{-1}$ of a periodic cell, each bundle may only contribute a $q$-independent number $\const_\xi$ of deformed pipe segments.  We set $\const_{D} = \const_\xi\const_{\rm counting}$, which we emphasize is independent of $q$, concluding the proof of \eqref{eq:concentrazion}.

Finally, we have that the last part of Hypothesis~\ref{hyp:dodging2}, which gives assumptions on the tangent vectors to the curves used to define $L$, is satisfied from~\eqref{eq:a:xi:phi:def},~\eqref{eq:a:xi:def}, and~\eqref{eq:Lagrangian:Jacobian:1}.  
\end{proof}

\section{Error estimates for the relaxed local energy inequality}\label{sec:RLE:errors}

In this section, we estimate the main error terms in the relaxed local energy inequality.  In subsection~\ref{ss:rle:new}, we carefully identify each error term.  Then in subsections~\ref{ss:ssO:current}--\ref{ss:dce:rle}, we estimate the oscillation current error, linear current error, stress current error,  and divergence corrector current errors, respectively.  Finally, we summarize the results and upgrade the material derivatives in subsection~\ref{upgrade:rle:ss}.

\subsection{Defining new current error terms}\label{ss:rle:new}
We will define $\ov \phi_{q+1}$ by adding $\hat w_{q+\bn}$ to $u_q$ on the left-hand side of \eqref{ineq:relaxed.LEI} and collecting new errors generated by the addition. Recall that in \eqref{ER:new:equation} we added $\hat w_\qbn$ to the Euler-Reynolds system and obtained the equation
\begin{align}\label{ER:new:equation:redux}
    \pa_t u_{q+1} + \div \left( u_{q+1} \otimes u_{q+1} \right) + \nabla p_q = \div\left( \ov R_{q+1} - (\pi_q-\pi_q^q)\Id \right)
\end{align}
for $ u_{q+1} = u_q + \hat w_\qbn$, where
\begin{subequations}
\begin{align}
    \ov R_{q+1} &= R_q - R_q^q + S_{q+1} \label{eq:ovR:def:redux} \\
    \div S_{q+1} &= \pa_t \hat w_\qbn + u_q \cdot \nabla \hat w_\qbn + \hat w_\qbn \cdot \nabla u_q + \div \left( \hat w_\qbn \otimes \hat w_\qbn + R_q^q - \pi_q^q\Id \right) \, . \label{ER:new:error:redux}
\end{align}
\end{subequations}
We recall that $\kappa_q = \sfrac 12 \tr(R_q-\pi_q\Id)$, $\kappa_q^q = \sfrac 12 \tr(R_q^q-\pi_q^q\Id)$ and set\footnote{We are using the definition of $S_{M1}$ from \eqref{ER:new:error} to achieve the third equality.}
\begin{subequations}
\begin{align}
  \overline{\kappa}_{q+1} &:= \frac12 {\tr(\ov R_{q+1}-(\pi_q-\pi_q^q)\Id) } =\kappa_q - \kappa_q^q + \frac12\tr(S_{q+1})
 =\kappa_q - \kappa_\ell + \frac12\tr(S_{q+1}-S_{M1})
  \, , \label{eq:def:ov:kappa} \\
  \overline\ph_{q+1} &= \ph_q-\ph_q^q +\overline{\phi}_{q+1} \, , \label{eq:def:ov:phi}
\end{align}
\end{subequations}
where $\div \ov \phi_{q+1}$ will include the new errors. We now introduce a function of time $\bmu_{\ov\phi_{q+1}}(t)$ to account for the fact that the new errors may not have zero mean for each time and expect to obtain
\begin{align}\label{RHS.LEI}
  &\pa_t \left( \frac 12 |u_{q+1}|^2 \right) + \div \left( \left( \frac 12 |u_{q+1}|^2 + p_q \right)u_{q+1} \right) \notag\\
  &\qquad \qquad = (\pa_t + \hat u_{q+1} \cdot \na )
    {\overline{\ka}_{q+1}}
+ \div\left((
\ov R_{q+1}
-\pi_q\Id+\pi_q^q\Id)  \hat u_{q+1}\right) + \div \overline\ph_{q+1} + \bmu_{\ov\phi_{q+1}} -E \, .
\end{align}
Towards this end, we first note that since $\div \, \hat u_q =0$, we have
\begin{align}
    &\div \left( (R_q-\pi_q\Id) \hat u_q \right) + \hat u_q \cdot \div \left( \ov R_{q+1} - R_q + \pi_q^q \Id \right) \notag \\
    &\qquad = \div \left( (\ov R_{q+1} - \pi_q \Id + \pi_q^q \Id) \hat u_q \right) + \nabla \hat u_q : \left( R_q - \pi_q^q \Id - \ov R_{q+1} \right) \, . \label{eq:desert:helpfull}
\end{align}
We now add and subtract $\div\left( (\ov R_{q+1}-\pi_q \Id + \pi_q^q\Id) \hat u_{q+1} \right)$ in the second to last equality below to obtain
\begin{align}
\pa_t &\left( \frac 12 |u_{q} + \hat w_{q+\bn}|^2 \right)
    + \div\left( \left(\frac 12 |u_{q} + \hat w_{q+\bn}|^2 + p_q\right) (u_{q} + \hat w_{q+\bn})\right)\notag\\
&\underset{\eqref{ineq:relaxed.LEI}}{=} (\pa_t + \hat u_q \cdot \na ) \ka_q + \div ((R_q-\pi_q \Id) \hat u_q) + \div \ph_q
+ (\pa_t + u_q \cdot\na )\left(\frac 12 |\hat w_{q+\bn}|^2\right) + \div \left(\frac 12 |\hat w_{q+\bn}|^2 \hat w_{q+\bn}\right)\notag\\
&\qquad+ \hat w_{q+\bn} \cdot \left(\pa_t u_q + (u_q\cdot \na) u_q + \na p_q \right) + \na u_q : \hat w_{q+\bn}\otimes \hat w_{q+\bn}   \notag\\
&\qquad+ u_q \cdot \left(\pa_t \hat w_{q+\bn} + (u_q \cdot \na) \hat w_{q+\bn} + (\hat w_{q+\bn}\cdot \na) u_q+ \div (\hat w_{q+\bn}\otimes \hat w_{q+\bn})\right)-E
\notag\\
&\underset{\eqref{eq:dodging:oldies}}{=} (\pa_t + \hat u_q \cdot \na ) \ka_q + \div ((R_q-\pi_q \Id) \hat u_q) + \div \ph_q
+ (\pa_t + \hat u_q \cdot\na )\left(\frac 12 |\hat w_{q+\bn}|^2\right) + \div \left(\frac 12 |\hat w_{q+\bn}|^2 \hat w_{q+\bn}\right)\notag\\
&\qquad+ \hat w_{q+\bn} \cdot \left(\pa_t u_q + (u_q\cdot \na) u_q + \na p_q \right) + \na \hat u_q : \hat w_{q+\bn}\otimes \hat w_{q+\bn}   \notag\\
&\qquad+ \hat u_q \cdot \left(\pa_t \hat w_{q+\bn} + (u_q \cdot \na) \hat w_{q+\bn} + (\hat w_{q+\bn}\cdot \na) u_q+ \div (\hat w_{q+\bn}\otimes \hat w_{q+\bn})\right)-E
\notag\\
&\underset{\substack{\eqref{eq:ovR:def:redux}, \\ \eqref{ER:new:error:redux}}}{=} (\pa_t + \hat u_q \cdot \na ) \left(\frac 12 |\hat w_{q+\bn}|^2 + \ka_q  \right) 
+ \div( (R_q-\pi_q \Id) \hat u_q)+ \div \left(\frac 12 |\hat w_{q+\bn}|^2 \hat w_{q+\bn}+ \ph_q\right) \notag\\
&\qquad 
+ \hat w_{q+\bn} \cdot \left(\pa_t u_q + (u_q\cdot \na) u_q + \na p_q\right) 
+ \na \hat u_q : \hat w_{q+\bn}\otimes \hat w_{q+\bn} 
 + \hat u_q \cdot \div (\ov R_{q+1}-R_q + \pi_q^q \Id)-E
\notag\\
&\underset{\eqref{eq:desert:helpfull}}{=} (\pa_t + \hat u_q \cdot \na ) \left(\frac 12 |\hat w_{q+\bn}|^2+ \ka_q  \right)+ \div \left(\frac 12 |\hat w_{q+\bn}|^2 \hat w_{q+\bn}+ \ph_q\right) + \div ((\ov R_{q+1}-\pi_q\Id+\pi_q^q\Id)\hat u_{q+1}) \notag\\
&\qquad 
+  \hat w_{q+\bn} \cdot (\pa_t u_q + (u_q\cdot \na) u_q + \na p_q)
- \div \left((\ov R_{q+1}-\pi_q\Id+\pi_q^q\Id) (\hat u_{q+1}- \hat u_q)\right) \notag\\
&\qquad + \na \hat u_q : (\hat w_{q+\bn}\otimes \hat w_{q+\bn} + R_q -\pi_q^q \Id - \ov R_{q+1})-E \notag \\
&= (\pa_t + \hat u_{q+1} \cdot \na ) {\overline{\ka}_{q+1}}
+ \div\left((
\ov R_{q+1}
-\pi_q\Id+\pi_q^q\Id)  \hat u_{q+1}\right) + \div \overline\ph_{q+1} + \bmu_{\ov\phi_{q+1}}-E
\, .
 \label{LHS.LEI}
\end{align}
We see that the final quantity on the right-hand side of \eqref{LHS.LEI} will hold provided that
\begin{align}
\div &\ov \phi_{q+1} + {\bmu'_{\ov\phi_{q+1}}} \notag\\
&\underset{\eqref{eq:def:ov:phi}}{=} (\pa_t + \hat u_q \cdot \na ) \left(\frac 12 |\hat w_{q+\bn}|^2+ \ka_q  \right) - (\pa_t + \hat u_{q+1}\cdot\nabla) \ov\kappa_{q+1} \notag\\
&\qquad 
- \div \left((\ov R_{q+1}-\pi_q\Id+\pi_q^q\Id) (\hat u_{q+1}- \hat u_q)\right) +  \hat w_{q+\bn} \cdot (\pa_t u_q + (u_q\cdot \na) u_q + \na p_q) \notag\\
&\qquad + \na \hat u_q : (\hat w_{q+\bn}\otimes \hat w_{q+\bn} + R_q -\pi_q^q \Id - \ov R_{q+1}) + \div \left(\frac 12 |\hat w_{q+\bn}|^2 \hat w_{q+\bn}+ \ph_q^q \right) \notag\\
&\underset{\eqref{eq:def:ov:kappa}}{=} (\pa_t + \hat u_q \cdot \nabla)\left( \frac 12 |\hat w_\qbn|^2 + \kappa_q^q - \frac 12 \tr (S_{q+1}) \right) + (\pa_t + \hat u_q \cdot \nabla) \ov \kappa_{q+1} - (\pa_t + \hat u_{q+1}\cdot\nabla) \ov\kappa_{q+1} \notag\\
&\qquad 
- \div \left((\ov R_{q+1}-\pi_q\Id+\pi_q^q\Id) (\hat u_{q+1}- \hat u_q)\right) +  \hat w_{q+\bn} \cdot (\pa_t u_q + (u_q\cdot \na) u_q + \na p_q) \notag\\
&\qquad + \na \hat u_q : (\hat w_{q+\bn}\otimes \hat w_{q+\bn} + R_q -\pi_q^q \Id - \ov R_{q+1}) + \div \left(\frac 12 |\hat w_{q+\bn}|^2 \hat w_{q+\bn}+ \ph_q^q \right) \notag\\
&\underset{\eqref{eq:def:ov:kappa}}{=} \underbrace{(\pa_t + \hat u_q \cdot \na ) \left(\frac 12 |w_{q+1}|^2+ {\ka_q^q}  - \frac12\tr(S_{q+1})\right)}_{=:\div \ov\phi_T + \bmu'_T} 
\notag\\
&\quad -
\underbrace{ \div ( (\hat u_{q+1}- \hat u_q) \overline\ka_{q+1})
 - \div \left( \left( \ov R_{q+1} - (\pi_q- \pi_q^q)\Id \right)(\hat u_{q+1}- \hat u_q)\right)}_{=:\div \ov\phi_R} \notag \\
 &\quad
+\underbrace{ w_{q+1} \cdot (\pa_t  u_q + ( u_q\cdot \na)  u_q + \na p_q)}_{=:\div \ov\phi_{L} + \bmu'_L }
+ \underbrace{\na \hat u_q : (w_{q+1}\otimes w_{q+1} + R_q {-\pi_q^q\Id} - \ov R_{q+1} )}_{=:\div \ov\phi_{N} +\bmu'_N} \, , \notag\\
&\quad +\underbrace{\div \left(\frac 12 |w_{q+1}^{(p)}|^2 w_{q+1}^{(p)}+ \ph_\ell \right)}_{=:\div\ov\phi_O} + \underbrace{\div\left( \frac{1}{2} |w_{q+1}^{(c)}|^2 w_{q+1}^{(p)} +( w_{q+1}^{(c)} \cdot w_{q+1}^{(p)} ) w_{q+1}^{(p)}  +  \frac{1}{2} |w_{q+1}|^2 w_{q+1}^{(c)}  \right)}_{=: \div \ov\phi_C} \notag \\
&\quad + \underbrace{\div(\ph_q^q-\ph_\ell)
+ (\pa_t + \hat u_q \cdot \na ) \frac 12\left(|\hat w_{q+\bn}|^2 - |w_{q+1}|^2  \right)
+ \frac 12\div \left( |\hat w_{q+\bn}|^2 \hat w_{q+\bn}-|w_{q+1}|^2 w_{q+1}\right)}_{=:\div \ov\phi_{M1} + \bmu_{M1}'} \notag \\
&\quad  +\underbrace{(\hat w_{q+\bn}-w_{q+1}) \cdot (\pa_t u_q + (u_q\cdot \na) u_q + \na p_q) 
  + \na \hat u_q : (\hat w_{q+\bn}\otimes \hat w_{q+\bn}-w_{q+1}\otimes w_{q+1})}_{=:\div \ov\phi_{M2} + \bmu_{M2}'} \, , \notag
\end{align}
where $\bmu_T, \bmu_N, \bmu_L, \bmu_{M1}, \bmu_{M2}$ are functions of time only and are given by
\begin{subequations}
\label{eq:meanz:def}
\begin{align}
    \bmu_T(t) &:= \int_0^t \left\langle (\pa_t + \hat u_q \cdot \na ) \left(\frac 12 |w_{q+1}|^2+ \ka_\ell  - \frac12\tr(S_{q+1})\right) \right\rangle(s) \, ds \\
    \bmu_N(t) &:= \int_0^t \left\langle \na \hat u_q : (w_{q+1}\otimes w_{q+1} + R_q -  R_{q+1} ) \right\rangle(s) \, ds \\
    \bmu_L(t) &:= \int_0^t \left\langle  w_{q+1} \cdot (\pa_t \hat u_q + (\hat u_q\cdot \na) \hat u_q + \na p_q) \right\rangle(s) \, ds \\
    \bmu_{M1}(t)&:=
    \int_0^t \left\langle
(\pa_t + \hat u_q \cdot \na ) \frac 12\left(|\hat w_{q+\bn}|^2 - |w_{q+1}|^2  \right)
    \right\rangle(s) \, ds\\
    \bmu_{M2}(t)&:=
    \int_0^t \left\langle
(\hat w_{q+\bn}-w_{q+1}) \cdot (\pa_t u_q + (u_q\cdot \na) u_q + \na p_q) \right\rangle(s) \, ds\notag\\
  &\quad+ \int_0^t \left\langle\na \hat u_q : (\hat w_{q+\bn}\otimes \hat w_{q+\bn}-w_{q+1}\otimes w_{q+1})
    \right\rangle(s) \, ds \, .
\end{align}
\end{subequations}
The notation $\langle \cdot \rangle$ denotes the spatial average of a tensor, and we define $\bmu_{\ov\phi_{q+1}}:= \bmu_T+ \bmu_N+\bmu_L+\bmu_{M1}+\bmu_{M2}$.

With these definitions in hand, we can rewrite \eqref{LHS.LEI} 
for $(u_{q+1}, p_q, \ov R_{q+1}, \ov \ph_{q+1}, -(\pi_q-\pi_q^q))$
as
\begin{align}
    \pa_t &\left( \frac 12 |u_{q+1}|^2 \right)
    + \div\left( \left(\frac 12 |u_{q+1}|^2 + p_q\right) u_{q+1}\right) \notag\\
    &=(\pa_t + \hat u_{q+1} \cdot \na ) (\overline{\ka}_{q+1}+ {\bmu_{\ov\phi_{q+1}}})
+ \div\left( {(\ov R_{q+1}-(\pi_q-\pi_q^q)\Id)}  \hat u_{q+1}\right) + \div \overline\ph_{q+1} -E\, . \label{eq:LEI:new}
\end{align}
The primitive current error $\overline{\phi}_{q+1}$\index{$\overline{\phi}_{q+1}$} will consist of $\overline{\phi}_{q+1}^k$, so that $\overline{\phi}_{q+1}=\sum_{q+1}^{q+\bn} {\overline{\phi}_{q+1}^k}$.

\subsection{Oscillation current error}\label{ss:ssO:current}
Recalling the definition of $\BB_{(\xi),\ph}$ and $\BB_{(\xi),R}$ from \eqref{int:pipe:bundle:short}, we have
\begin{align*}
(\BB_a\BB_b \BB_b)_{(\xi),\ph}
&= \chib_{(\xi),\vp}^{3} \sum_I (\etab_\xi^I)^{6}\mathbb{P}_{\neq 0} \left[(\WW_{(\xi),\ph}^I)_a (\WW_{(\xi),\ph}^I)_b
(\WW_{(\xi),\ph}^I)_b\right]
+\mathbb{P}_{\neq 0}\chib_{(\xi),\vp}^{3}
\xi_a\xi_b\xi_b
+ \xi_a\xi_b\xi_b\\
(\BB_a\BB_b \BB_b)_{(\xi),R}
&= \chib_{(\xi),R}^{3} \sum_I (\etab_\xi^I)^{9}\mathbb{P}_{\neq 0} \left[ (\WW_{\pxi,R}^I)_a (\WW_{\pxi,R}^I)_b
(\WW_{\pxi,R}^I)_b \right]
\end{align*}
where we used $\langle |\WW^I_{\pxi,\ph}|^2 \WW^I_{\pxi,\ph} \rangle =|\xi|^2\xi$ from Proposition~\ref{prop:pipe.flow.current} item~\eqref{item:pipe:means:current},
$\langle \chib_{(\xi),\vp}^3 \rangle =1$ from Proposition~\ref{prop:bundling} item~\eqref{i:bundling:3}, $\sum_I (\etab_\xi^I)^6=1$ from \eqref{eq:sa:summability}, and $\langle |\WW^I_{\pxi,R}|^2 \WW^I_{\pxi,R} \rangle =0$ from Proposition~\ref{prop:pipeconstruction} item~\eqref{item:pipe:means}. Using that all cross-terms from $|w_{q+1}^{(p)}|^2 w_{q+1}^{(p)}$ (as defined in \eqref{defn:w}) vanish due to Lemma~\ref{lem:dodging} item~\eqref{item:dodging:2}, and the fact that $(\nabla\Phiik^{-1} \xi) \cdot \nabla $ gives zero when applied to $(\WW_{\pxi,\diamond}^I \rhob_{\pxi,\diamond}\zetab_\xi^I) \circ \Phiik$, it then follows that
\begin{align}
    \frac 12 &\div\left(  |w_{q+1}^{(p)}|^2w_{q+1}^{(p)} \right) \notag\\
&= \frac 12 \div \left[ \sum_{i,j,k,\xi,\vecl,I,\diamond} a_{\pxi,\diamond}^3 \left(\chib_{\pxi,\diamond}^{3}(\etab_\xi^I)^{3\diamond} \right) \circ \Phiik \left| \nabla \Phiik^{-1} \WW_{\pxi,\diamond}^I \circ \Phiik\right|^2 \nabla \Phiik^{-1}\WW_{\pxi,\diamond}^I \circ \Phiik \right] \notag\\
&= \frac 12 \sum_{\xi, i,j,k,\vecl} \div \left( a_{(\xi),\ph}^3 |\na \Phi_{(i,k)}^{-1} \xi|^2 \na\Phi_{(i,k)}^{-1}\xi \right) + \sum_{\xi, i,j,k,\vecl} b_{(\xi),\ph} \mathbb{P}_{\neq 0}\left( \chib_{\pxi,\vp}^3 \right)(\Phi_{(i,k)}) \notag \\
&\qquad\quad + \sum_{\xi, i,j,k,\vecl,I,\diamond} b_{(\xi),\diamond}
\left(\chib_{\pxi,\diamond}^{3}(\etab_\xi^I)^{3\diamond}
\mathbb{P}_{\neq 0} (\varrho_{\pxi,\diamond}^I)^3\right)(\Phi_{(i,k)}) \, , \label{eq:osc:c:eq:1}
\end{align}
where the $\gamma$ component of $b_{\pxi,\diamond}$ is given by $b_{(\xi),\diamond}^\gamma = \sfrac 12 \xi_a\pa_\gamma \left(  a_{(\xi),\diamond}^3  |\na \Phi_{(i,k)}^{-1} \xi|^2 (\na\Phi_{(i,k)}^{-1})_a^\gamma \right)$, $\varrho_{\pxi,\diamond}$ is the pipe density defined in \eqref{int:pipe:bundle:short}, and we are using the notation $3\diamond$ in the power of $\zetab_\xi^I$ as a stand-in for $6$ or $9$ in the current and Reynolds cases, respectively. 

By the choice of $a_{(\xi),\ph}$, the first term cancels out $\ph_\ell$ up to a higher-frequency error term. Indeed, using \eqref{eq:a:xi:phi:def}, \eqref{eq:alg:id}, \eqref{eq:summy:summ:3} to yield a formula for the summation of $\zeta_{q,\vp,i,k,\xi,\vecl}^3$, \eqref{eq:summy:summ:0}, \eqref{eq:inductive:partition} at level $q$, \eqref{eq:chi:cut:partition:unity} at level $q$, and \eqref{eq:omega:cut:partition:unity}, we have that
\begin{align}
    \frac 12 \sum_{\xi, i,j,k,\vecl}
    &a_{(\xi),\ph}^3 |\na \Phi_{(i,k)}^{-1} \xi|^2 (\na\Phi_{(i,k)}^{-1})_a^\gamma \xi_a \notag \\ 
    &= \frac 12 \delta_{q+\bn}^{\sfrac 32}r_q^{-1}\sum_{\xi, i,j,k, \vecl} \Gamma^{3(j-1)}_{q} \psi_{i,q}^6 \omega_{j,q}^6 \chi_{i,k,q}^6 \zeta_{q,\vp,i,k,\xi,\vecl}^3 \td\gamma_{\xi}^3\left(\frac{\ph_{q,i,k}}{\delta_{q+\bn}^{\sfrac32}r_q^{-1}\Gamma^{3(j-1)}_{q}}\right) (\na\Phi_{(i,k)}^{-1})_a^\gamma \xi_a \notag \\
    &= \frac 12 \delta_{q+\bn}^{\sfrac 32}r_q^{-1}\sum_{\xi, i,j,k,l^\perp} \Gamma^{3(j-1)}_{q} \psi_{i,q}^6 \omega_{j,q}^6 \chi_{i,k,q}^6 \mathcal{X}_{q,\xi,l^\perp}^3 \circ \Phi_{(i,k)} \td\gamma_{\xi}^3\left(\frac{\ph_{q,i,k}}{\delta_{q+\bn}^{\sfrac32}r_q^{-1}\Gamma^{3(j-1)}_{q}}\right) (\na\Phi_{(i,k)}^{-1})_a^\gamma \xi_a \notag \\
    &= \sum_{\xi, i,j,k} \psi_{i,q}^6 \omega_{j,q}^6 \chi_{i,k,q}^6 {c_3} \left(-{\frac1{c_3}\ph^\gamma_\ell}\right) \notag \\
    &\quad + \frac 12 \sum_{\xi, i,j,k, l^\perp} \underbrace{\delta_{q+\bn}^{\sfrac 32}r_q^{-1} \Gamma^{3(j-1)}_{q} \psi_{i,q}^6 \omega_{j,q}^6 \chi_{i,k,q}^6 \td\gamma_{\xi}^3\left(\frac{\ph_{q,i,k}}{\delta_{q+\bn}^{\sfrac32}r_q^{-1}\Gamma^{3(j-1)}_{q}}\right) (\na\Phi_{(i,k)}^{-1})_a^\gamma \xi_a}_{=: \, 2 \td b^\gamma_{(\xi)}} \left(\mathbb{P}_{\neq 0} \mathcal{X}^3_{q,\xi,l^\perp}\right) \circ \Phi_{(i,k)} \notag \\
    &= -(\ph_\ell)^\gamma + \sum_{\xi, i,j,k,l^\perp} \td b^\gamma_{(\xi)} \left(\mathbb{P}_{\neq 0} \mathcal{X}^3_{q,\xi,l^\perp}\right) \left(\Phi_{(i,k)}\right)\, . \label{eq:osc:c:eq:2}
\end{align}

The inverse divergence of the remaining terms will form new current errors. We first recall the synthetic Littlewood-Paley decomposition (cf. Section \ref{sec:LP}). Since $\varrho_{\pxi,\diamond}^{I}$ is defined on the plane $\xi^\perp$ and is periodized to scale $\left(\lambda_{q+\bn}r_{q}\right)^{-1} {=(\la_{q+\floor{\bn/2}}\Ga_q)^{-1}}$ from~\eqref{int:pipe:bundle:short}, and Propositions~\ref{prop:pipeconstruction}, \ref{prop:pipe.flow.current}, we can decompose $\mathbb{P}_{\neq 0}$ in front of $(\varrho_{\pxi,\diamond}^{I})^3$ into
$$\mathbb{P}_{\neq 0}
   = \tP_{\la_{q+\floor{n/2}{+1}}}^\xi\mathbb{P}_{\neq 0}
    + {\sum_{m={q+\floor{n/2}+{2}}}^{q+\bn+1} }
    \tP_{(m-1, m]}^\xi + (1-\tP_{q+\bn+1}^\xi) \, . $$
Assuming we can apply the inverse divergence operator from Proposition~\ref{prop:intermittent:inverse:div} (with the adjustments set out Remark~\ref{rem:pointwise:inverse:div} for pointwise bounds), we define
\begin{subequations}
\label{eq:osc:c:eq:3}
\begin{align}
    \overline{\phi}^{q+1}_O
    &:=(\divH + \divR) 
    \sum_{\xi, i,j,k,\vecl} \underbrace{b_{(\xi),\ph} \left(\mathbb{P}_{\neq 0} \chib_{\pxi,\vp}^3\right)(\Phi_{(i,k)})}_{=:\,t_{i,j,k,\xi,\vecl,\ph}^{q+1}}\label{osc.current.low} \\
    & \qquad\qquad + (\divH + \divR) \sum_{\xi, i,j,k,l^\perp} \underbrace{\pa_\gamma \td b^\gamma_{(\xi)} \left(\mathbb{P}_{\neq 0} \mathcal{X}^3_{q,\xi,l^\perp}\right) \left(\Phi_{(i,k)}\right)}_{=:\, \td t^{q+1}_{i,j,k,\xi,l^\perp,\ph}}
    \label{osc.current.low.1}\\
    \overline{\phi}^{q+\floor{n/2}+1}_O
    &:=(\divH + \divR) 
    \sum_{\xi, i,j,k,\vecl,I,\diamond} \underbrace{b_{(\xi),\diamond}
\left(\chib_{\pxi,\diamond}^{3}(\etab_\xi^I)^{3\diamond}
\tP^{\xi}_{q+\floor{n/2}+1}\mathbb{P}_{\neq 0} (\varrho_{\pxi,\diamond}^I)^3\right)(\Phi_{(i,k)})}_{=:\,t_{i,j,k,\xi,\vecl,I,\diamond}^{q+\floor{n/2}+1}}
     \label{osc.current.med0}\\
    \overline{\phi}^{m}_O
    &:=(\divH + \divR) 
    \sum_{\xi, i,j,k,\vecl,I,\diamond} \underbrace{b_{(\xi),\diamond}
\left(\chib_{\pxi,\diamond}^{3}(\etab_\xi^I)^{3\diamond}
\tP^{\xi}_{(m-1,m]}\mathbb{P}_{\neq 0} (\varrho_{\pxi,\diamond}^I)^3\right)(\Phi_{(i,k)})}_{=:\,t_{i,j,k,\xi,\vecl,I,\diamond}^{m}}
     \label{osc.current.med1}\\
    \overline{\phi}^{q+\bn}_O
    &:=
    \sum_{m=q+\bn}^{q+\bn+1}(\divH + \divR)
    \sum_{\xi, i,j,k,\vecl,I,\diamond} \underbrace{b_{(\xi),\diamond}
\left(\chib_{\pxi,\diamond}^{3}(\etab_\xi^I)^{3\diamond}
\tP^{\xi}_{(m-1,m]}\mathbb{P}_{\neq 0} (\varrho_{\pxi,\diamond}^I)^3\right)(\Phi_{(i,k)})}_{=\,t_{i,j,k,\xi,\vecl,I,\diamond}^{m}}
    \label{osc.current.high}\\
    &\quad +
   (\divH + \divR)\left[
    \sum_{\xi, i,j,k,\vecl,I,\diamond} b_{(\xi),\diamond}
\left(\chib_{\pxi,\diamond}^{3}(\etab_\xi^I)^{3\diamond}
(1-\tP_{q+\bn+1}^\xi)\mathbb{P}_{\neq 0} (\varrho_{\pxi,\diamond}^I)^3\right)(\Phi_{(i,k)})
    \right]\, ,
    \label{osc.current.high2}
\end{align}
\end{subequations}
where \eqref{osc.current.med1} is defined for $q+\half +1 < m < q+\bn$. We justify these applications and record estimates on the outputs in the following Lemma.

\begin{lemma}[\bf Current error and pressure increment]\label{lem:oscillation.current:general:estimate}
There exist current errors $\ov \phi_O^{m}$ for $m=q+1,\dots,q+\bn$ and pressure increments $\sigma_{\ov \phi_O^{m}}^+=\sigma_{\ov \phi_O^{m}}-\sigma_{\ov \phi_O^{m}}^-$ for $m=q+\half+1,\dots,q+\bn$ such that the following hold.
\begin{enumerate}[(i)]
\item\label{item:osc:c:1} We have the equality
$$ \frac 12 \div \left( |w_{q+1}^{(p)}|^2  w_{q+1}^{(p)} + \ph_\ell \right) = \sum_{m=q+1}^{q+\bn}\div \ov \phi_O^m = \sum_{m=q+1}^\qbn \div \left( \ov \phi_O^{m,l}+ \ov \phi_O^{m,*} \right) \, . $$ 
\item The lowest shell has no pressure increment; more precisely, $\sigma^+_{\ov \phi_O^{q+1}} \equiv 0$, and for $N,M\leq \sfrac{\Nfin}{100}$,
    \begin{align}
         \left|\psi_{i,q} D^N \Dtq^M \ov\phi^{q+1}_{O} \right| &< \Ga_q^{65} \La_q \la_{q+1}^{-1} (\pi_q^q)^{\sfrac 32} r_q^{-1} \la_{q+1}^N \MM{M,\Nindt, \tau_q \Ga_q^{i+14}, \Ga_q^{8}\Tau_q^{-1}} \, . \label{eq.cur.osc.low.1}
    \end{align}
\item For all $m=q+\half+1,\dots,q+\bn$, we have that
\begin{subequations}
\begin{align}
\label{eq:co.p.1}
\left|\psi_{i,q} D^N \Dtq^M \ov\phi^{m,l}_{O}\right| &\les \left( (\si_{\ov \phi^{m}_O}^+)^{\sfrac 32}r_m^{-1} + \de_{q+3\bn}^{2}\right) \left(\lambda_m\Gamma_q\right)^N \MM{M,\Nindt,\tau_q^{-1}\Gamma_{q}^{i+15},\Tau_q^{-1}\Ga_q^9}  \\
\label{eq:co.p.2}
\left|\psi_{i,q} D^N \Dtq^M \si_{\ov\phi^{m}_O}^+\right| &\les  \left(\si_{\ov\phi^m_O}^+ +\de^2_{q+3\bn}\right) \left(\lambda_{m}\Gamma_q\right)^N \MM{M,\Nindt,\tau_q^{-1}\Gamma_{q}^{i+16},\Tau_q^{-1}\Ga_q^9}\\
\label{eq:co.p.3}
\norm{\psi_{i,q} D^N \Dtq^M \si_{\ov\phi^{m}_O}^+}_{\sfrac32} &< \de_{m+\bn} \Gamma_{m}^{-9} \left(\lambda_{m}\Gamma_q\right)^N \MM{M,\Nindt,\tau_q^{-1}\Gamma_{q}^{i+16},\Tau_q^{-1}\Ga_q^9}\\
 \label{eq:co.p.3.5}
\norm{\psi_{i,q} D^N \Dtq^M \si_{\ov\phi^{m}_O}^+}_{\infty} &< \Ga_{m}^{\badshaq-9} \left(\lambda_{m}\Gamma_q\right)^N \MM{M,\Nindt,\tau_q^{-1}\Gamma_{q}^{i+16},\Tau_q^{-1}\Ga_q^9} \\
\label{eq:co.p.4}
\left|\psi_{i,q} D^N \Dtq^M \si_{\ov\phi^{m}_O}^-\right| &< \left(\frac{\la_q}{\la_{q+\floor{\bn/2}}}\right)^\frac23 \pi_q^q  \left(\lambda_{q+\floor{\bn/2}}\Gamma_q\right)^N \MM{M,\Nindt,\tau_q^{-1}\Gamma_{q}^{i+16},\Tau_q^{-1}\Ga_q^9}
\end{align}
\end{subequations}
for all $N,M \leq \sfrac{\Nfin}{100}$. Furthermore, we have that for all $m'\geq q+\half+1$,
\begin{equation}\label{eq:o.p.6}
\begin{split}
B\left( \supp \hat w_{q'}, \lambda_{q'}^{-1} \Gamma_{q'+1} \right) \cap \supp \ov \phi^{m,l}_O &= \emptyset \qquad \forall q+1\leq q' \leq  m-1\\
B\left( \supp \hat w_{q'}, \lambda_{q'}^{-1} \Gamma_{q'+1} \right) \cap \supp (\si_{\ov \phi^{m'}_O}^+) &= \emptyset \qquad \forall q+1\leq q' \leq  m'-1\\
B\left( \supp \hat w_{q'}, \lambda_{q+1}^{-1} \Gamma_{q}^2 \right) \cap \supp (\si_{\ov \phi^{m'}_O}^-) &= \emptyset \qquad \forall q+1\leq q' \leq  q+\half \, .
\end{split}
\end{equation}
\item For all $m=q+1,\dots,q+\bn$ and $N,M \leq 2\Nind$, the non-local part $\bar\phi^{m,*}_{O}$ satisfies
\begin{align}
\left\| D^N D_{t, q}^M \bar\phi^{m,*}_{O} \right\|_{L^\infty}
    &\leq \Tau_\qbn^{2\Nindt}\delta_{q+3\bn}^{\sfrac32}\la_{m}^{N}\tau_{q}^{-M}\, .
    \label{eq:osc.current:estimate:1}
\end{align}
\end{enumerate}
\end{lemma}
\begin{proof}
The equality in \eqref{item:osc:c:1} follows from \eqref{eq:osc:c:eq:1}--\eqref{eq:osc:c:eq:3}, assuming for the moment that all quantities in \eqref{eq:osc:c:eq:3} are well-defined.  We now split the proof up into cases.  We first treat $\ov \phi_O^{q+1}$ as defined in \eqref{osc.current.low}--\eqref{osc.current.low.1} and prove \eqref{eq.cur.osc.low.1} and \eqref{eq:osc.current:estimate:1} for $m=q+1$. Next, we treat \eqref{osc.current.med0} and prove \eqref{eq:co.p.1}--\eqref{eq:osc.current:estimate:1} for $m=q+\half+1$ using Proposition~\ref{prop:intermittent:inverse:div} in conjunction with Remark~\ref{rem:pointwise:inverse:div}.  Afterwards we treat the intermediate shells from \eqref{osc.current.med1} and a portion of the last shell \eqref{osc.current.high}, and prove \eqref{eq:co.p.1}--\eqref{eq:osc.current:estimate:1} for $q+\half+2\leq m \leq q+\bn$ using Proposition~\ref{lem.pr.invdiv2.c}.  Finally, we treat \eqref{osc.current.high2}, which will be absorbed to the nonlocal part of the current error. We will therefore prove \eqref{eq:osc.current:estimate:1} using Proposition~\ref{prop:intermittent:inverse:div}. We fix the following choices throughout the proof,
\begin{align*}
    &v= \hat u_q \, , \quad \Phi=\Phiik \, , \quad \const_v = \La_q^{\sfrac12}\, , \quad
    \nu'=\Tau_q^{-1}\Ga_q^8 \, ,\quad  
    \la'=\Ga_q^{13}\La_q\, , \quad
    \\
     &2M_*=N_*=\frac{\Nfin}{3} \, , \quad M_t=\Nindt \, , \quad M_\circ = N_\circ = 2\Nind \, , \qquad K_\circ \textnormal{ as in \eqref{i:par:9.5}/\eqref{ineq:K_0}}
\end{align*}
while the remaining parameters will vary depending on the case.
\smallskip

\noindent\texttt{Case 1a:} Analysis for \eqref{osc.current.low}.
Fix $\xi$, $i$, $j$, $k$, and $\vecl$. In order to check the low-frequency assumptions in Part 1, high-frequency assumptions in Part 2, and nonlocal assumptions in Part 4 of Proposition~\ref{prop:intermittent:inverse:div}, we set 
\begin{subequations}
\begin{align}
    &p=\infty \, , \quad G= b_{\pxi,\varphi} \, , \quad \const_{G,\infty} = r_q^{-1} \delta_\qbn^{\sfrac 32} \Ga_q^{3j+34} \La_q \, , \quad
     \la=\Ga_q^{13}\La_q \, , \quad \nu=\tau_q^{-1}\Ga_q^{i+13} \, , \quad
     \pi=\Ga_q^{36}\pi_\ell^{\sfrac 32}r_q^{-1} \La_q\Ga_q^{10}  \, , \notag\\
    &\varrho = \mathbb{P}_{\neq 0}\chib_{(\xi),\ph}^{ 3} \, , \quad \vartheta^{i_1i_2\dots i_{\dpot-1}i_\dpot} = \delta^{i_1i_2\dots i_{\dpot-1}i_\dpot} \Delta^{-\sfrac \dpot 2}\varrho \, , \quad \const_{\ast,1}=\const_{\ast,\infty} = \Ga_q^6{\la_{q+1}^\alpha} \, , \quad \mu = \la_{q+1}\Ga_q^{-4} \, , \notag\\
    &\Upsilon=\Upsilon'=\la_{q+1}\Ga_q^{-4} \, , \quad \La=\la_{q+1}\Ga_q^{-1} \, , \quad \Ndec \textnormal{ as in \eqref{i:par:9}/\eqref{condi.Ndec0}} \, , \quad \dpot \textnormal{ as in \eqref{i:par:10}/\eqref{ineq:dpot:1}} \, , \notag
\end{align}
\end{subequations}
where $\alpha$ is chosen as in \eqref{eq:choice:of:alpha}.
Then we have that \eqref{eq:inv:div:NM}--\eqref{eq:inverse:div:DN:G} are satisfied by definition and by \eqref{e:a_master_est_p_phi}, \eqref{eq:DDpsi2}--\eqref{eq:DDv} hold from Corollary~\ref{cor:deformation} and \eqref{eq:nasty:D:vq:old} at level $q$, \eqref{eq:inv:div:extra:pointwise} holds from \eqref{e:a_master_est_p_phi_pointwise}, \eqref{item:inverse:i}--\eqref{item:inverse:ii} hold by definition and item~\ref{i:bundling:1} from Proposition~\ref{prop:bundling}, \eqref{eq:DN:Mikado:density} holds due to standard Littlewood-Paley theory, \eqref{eq:inverse:div:parameters:0} holds by definition and by \eqref{condi.Nfin0}, \eqref{eq:inverse:div:parameters:1} holds due to \eqref{condi.Ndec0}, \eqref{eq:inv:div:wut} holds by \eqref{condi.Nfin0}, \eqref{eq:inverse:div:v:global}--\eqref{eq:inverse:div:v:global:parameters} hold from Remark~\ref{rem:lossy:choices}, and \eqref{eq:riots:4} holds from \eqref{ineq:dpot:1}. 

From \eqref{eq:inverse:div} and \eqref{eq:inverse:div:error:stress}, we have that \eqref{osc.current.low} is well-defined.  From \eqref{eq:divH:formula}, \eqref{eq:inverse:div:sub:main}, \eqref{eq:inv:div:extra:conc}, \eqref{ind:pi:lower}, and \eqref{condi.Nfin0}, we have that for $N,M\leq \sfrac{\Nfin}{7}$, 
\begin{align*}
    \left| D^N D_{t,q}^M \mathcal{H}\, t_{i,j,k,\xi,\vecl,\varphi}^{q+1} \right| &\lesssim \Ga_q^{60} (\pi_q^q)^{\sfrac 32} \La_q r_q^{-1} \lambda_{q+1}^{-1} \lambda_{q+1}^N \MM{M,\Nindt,\tau_q^{-1}\Ga_q^{i+13},\Tau_q^{-1}\Ga_q^8} \, .
\end{align*}
Notice that from \eqref{item:div:local:i}, the support of $\divH t^{q+1}_{i,j,k,\xi,\vecl,\varphi}$ is contained in the support of $t_{i,j,k,\xi,\vecl}$, which is contained inside the support of $\eta_{i,j,k,\xi,\vecl,\varphi}$ from the definition of $b_{\pxi,\varphi}$.  Thus we may apply Corollary~\ref{lem:agg.pt} with $H = \mathcal{H}t^{q+1}_{i,j,k,\xi,\vecl,\varphi}$, $\varpi=\Ga_q^{60}(\pi_q^q)^{\sfrac 32}r_q^{-1}\lambda_{q+1}^{-1}\Lambda_q\mathbf{1}_{\supp \eta_{i,j,k,\xi,\vecl,\varphi}}$, and $p=1$ to deduce that
\begin{align}\label{eq:cur:osc:ugh:1}
   \left| \psi_{i,q} D^N \Dtq^M \sum_{i',j,k,\xi,\vecl,\varphi} \mathcal{H}\, t_{i',j,k,\xi,\vecl,\varphi}^{q+1} \right| \lesssim \Ga_q^{60}\pi_\ell^{\sfrac 32} \La_q r_q^{-1}\lambda_{q+1}^{-1} \MM{M,\Nindt,\tau_q^{-1}\Ga_q^{i+14},\Tau_q^{-1}\Ga_q^8} \, .
\end{align}
From \eqref{eq:inverse:div:error:stress:bound} and summing over the values of $i,j,k,\xi,\vecl,\diamond$ which may be non-zero at a fixed point in time using \eqref{eq:inductive:partition} and \eqref{eq:imax:upper:lower} to control $i$, \eqref{e:chi:overlap} to control $k$, \eqref{ineq:jmax:use} to control $j$, Lemma~\ref{lem:checkerboard:estimates} to control $\vecl$, and using that $\xi$ takes only finitely many values, we have from \eqref{eq:inverse:div:error:stress:bound} and Remark~\ref{rem:lossy:choices} that for all $M_0,N_0\leq 2\Nind$,
\begin{align}\label{eq:cur:osc:ugh:2}
\left| D^N \Dtq^M \sum_{i,j,k,\xi,\vecl,\varphi} \divR t_{i,j,k,\xi,\vecl,\varphi}^{q+1} \right| \leq \la_\qbn^{-2} \delta_{q+3\bn}^{\sfrac 32} \Tau_\qbn^{2\Nindt} \la_{q+1}^N \tau_q^{-M} \, .
\end{align}
Combining \eqref{eq:cur:osc:ugh:1}--\eqref{eq:cur:osc:ugh:2} and using \eqref{ind:pi:lower}, we have an estimate consistent with \eqref{eq.cur.osc.low.1}, and an estimate consistent with \eqref{eq:osc.current:estimate:1} for $m=q+1$.
\medskip

\noindent\texttt{Case 1b:} Analysis for \eqref{osc.current.low.1}. 
Fix $\xi$, $i$, $j$, and $k$. In order to check the low-frequency assumptions in Part 1, high-frequency assumptions in Part 2, and nonlocal assumptions in Part 4 of Proposition~\ref{prop:intermittent:inverse:div}, we set 
\begin{subequations}
\begin{align}
    &p=\infty \, , \quad G= \partial_\gamma \td b_{\pxi}^\gamma \, , \quad \const_{G,\infty} = r_q^{-1} \delta_\qbn^{\sfrac 32} \Ga_q^{3j+34} \La_q \, , \quad \la=\Ga_q^{13}\La_q \, , \quad \nu=\tau_q^{-1}\Ga_q^{i+13} \, , \quad  \pi=\Ga_q^{36}\pi_\ell^{\sfrac 32}r_q^{-1} \La_q\Ga_q^{10}  \, , \notag\\
    &\varrho = \mathbb{P}_{\neq 0} \sum_{l^\perp} \mathcal{X}_{q,\xi,l^\perp}^{ {3}} \, , \quad \vartheta^{i_1i_2\dots i_{\dpot-1}i_\dpot} = \delta^{i_1i_2\dots i_{\dpot-1}i_\dpot} \Delta^{-\sfrac \dpot 2}\varrho \, , \quad \const_{\ast,1}=\const_{\ast,\infty} = {\la_{q+1}^{\alpha}} \, , \quad \mu = \const_\Gamma\la_{q+1}\Ga_q^{-5} \, , \notag\\
    &\Upsilon=\Upsilon'=\La=\la_{q+1}\Ga_q^{-5} \, , \quad \Ndec \textnormal{ as in \eqref{i:par:9}/\eqref{condi.Ndec0}} \, , \quad \dpot \textnormal{ as in \eqref{i:par:10}/\eqref{ineq:dpot:1}}  \, . \notag
\end{align}
\end{subequations}
Then we have that \eqref{eq:inv:div:NM}--\eqref{eq:inverse:div:DN:G} are satisfied by definition and by \eqref{eq:nasty:Dt:psi:i:q:orangutan} at level $q$, \eqref{eq:D:Dt:omega:sharp}, \eqref{eq:chi}, Corollary~\ref{cor:deformation} at level $q$, and \eqref{e:a_master_est_p_pointwise}, \eqref{eq:DDpsi2}--\eqref{eq:DDv} hold from Corollary~\ref{cor:deformation} and \eqref{eq:nasty:D:vq:old} at level $q$, \eqref{eq:inv:div:extra:pointwise} holds from \eqref{pt.est.pi.cutoff} and the same estimates which justified \eqref{eq:inv:div:NM}--\eqref{eq:inverse:div:DN:G}, \eqref{item:inverse:i}--\eqref{item:inverse:ii} hold from Lemma~\ref{lem:checkerboard:estimates}, \eqref{eq:DN:Mikado:density} holds due to standard Littlewood-Paley theory, \eqref{eq:inverse:div:parameters:0} holds by definition and by \eqref{condi.Nfin0}, \eqref{eq:inverse:div:parameters:1} holds due to \eqref{condi.Ndec0}, \eqref{eq:inv:div:wut} holds by \eqref{condi.Nfin0}, \eqref{eq:inverse:div:v:global}--\eqref{eq:inverse:div:v:global:parameters} hold from Remark~\ref{rem:lossy:choices}, and \eqref{eq:riots:4} holds from \eqref{ineq:dpot:1}. 

At this point, the remainder of the argument is essentially identical to that of \texttt{Case 1a}.  Indeed, the only differences are that the support of the localized output is contained inside the support of $\psi_{i,q}\omega_{j,q}\chi_{i,k,q}$, and so instead of appealing to the abstract aggregation lemma, we may appeal directly to \eqref{eq:inductive:partition} and \eqref{eq:omega:cut:partition:unity}. Eschewing further details, we have concluded the analysis of $\phi_O^{q+1}$ and proven \eqref{eq.cur.osc.low.1} and \eqref{eq:osc.current:estimate:1} at level $m=q+1$.
\smallskip

\noindent\texttt{Case 2: }Analysis for \eqref{osc.current.med0}.
Fix $\xi$, $i$, $j$, $k$, $\vecl$, $I$, and $\diamond$. In order to check the low-frequency, preliminary assumptions in Part 1 of Proposition~\ref{lem.pr.invdiv2.c}, we set
\begin{align}
    &p=1,\infty \, , \quad G_{\vp} = b_{\pxi,\vp} \left(\rhob_{\pxi,\vp}^3 (\zetab_\xi^I)^{3\vp} \right) \circ \Phiik \, , \quad G_{R} = r_q^{-1} b_{\pxi,R} \left(\rhob_{\pxi,R}^3 (\zetab_\xi^I)^{3R} \right) \circ \Phiik \, , \notag\\
    &\const_{G_\diamond,1} = \delta_\qbn^{\sfrac 32} r_q^{-1} \Ga_q^{3j+40}\Lambda_q \left| \supp \left( \eta_{i,j,k,\xi,\vecl,\diamond} \zetab_\xi^{I,\diamond} \right) \right|  {+ \la_{q+\bn}^{-15}} \, , \quad \const_{G_\diamond,\infty} = \delta_\qbn^{\sfrac 32} r_q^{-1} \Ga_q^{3j+40}\Lambda_q \, ,  \notag\\
    &\la = \la_{q+\half} \, , \quad  \nu=\tau_q^{-1}\Ga_q^{i+13} \, , \quad   \pi=\Ga_q^{35}\pi_\ell \La_q^{\sfrac 23} \, , \quad r_G=r_q \, .  \label{eq:desert:choices}
\end{align}
Then we have that \eqref{eq:inv:div:NM} is satisfied by definition, \eqref{eq:inverse:div:DN:G} is satisfied by \eqref{e:a_master_est_p_phi:zeta}, \eqref{e:a_master_est_p_R:zeta},  Corollary~\ref{cor:deformation}, \eqref{e:fat:pipe:estimates:1}, and Definition~\ref{def:etab}, \eqref{eq:DDpsi2}--\eqref{eq:DDv} hold from Corollary~\ref{cor:deformation} and \eqref{eq:nasty:D:vq:old} at level $q$, and \eqref{eq:inv:div:extra:pointwise} holds from \eqref{e:a_master_est_p_pointwise}. In order to check the high-frequency, preliminary assumptions in Part 1 of Proposition~\ref{lem.pr.invdiv2.c}, we set
\begin{subequations}
\begin{align}
    &\varrho_R = \tilde{\mathbb{P}}_{q+\half+1}^\xi \mathbb{P}_{\neq 0} (\varrho^I_{\pxi,R})^3 r_q \, , \quad \varrho_\vp = \tilde{\mathbb{P}}_{q+\half+1}^\xi \mathbb{P}_{\neq 0} (\varrho^I_{\pxi,\vp})^3 \, , \quad \vartheta_\diamond^{i_1i_2\dots i_{\dpot-1}i_\dpot} = \delta^{i_1i_2\dots i_{\dpot-1}i_\dpot} \Delta^{-\sfrac \dpot 2}\varrho_\diamond \, , \notag\\
    &\const_{\ast,1} = \Gamma_q  {\la_{q+\half+1}^{\alpha}}\, , \quad \const_{\ast,\infty} = \left( \frac{\lambda_{q+\half+1}}{\lambda_{q+\half}\Ga_q} \right)^{2} {\la_{q+\half+1}^{\alpha}} \, , \quad \mu = \Upsilon=\Upsilon'=\la_{q+\half}\Ga_q \, , \quad \La=\la_{q+\half+1} \, ,  \notag\\
    &\Ndec \textnormal{ as in \eqref{i:par:9}/\eqref{condi.Ndec0}} \, , \quad \dpot \textnormal{ as in \eqref{i:par:10}/\eqref{ineq:dpot:1}} \, \notag
\end{align}
\end{subequations}
Now \eqref{item:inverse:i}--\eqref{item:inverse:ii} hold by definition and from
\eqref{lem:special:cases} and \eqref{int:pipe:bundle:short}
(which specify the minimum frequency of $\mu=\la_{q+\half} {\Ga_q}$), \eqref{eq:DN:Mikado:density} holds due to Propositions~\ref{prop:pipeconstruction} and \ref{prop:pipe.flow.current} and estimate \eqref{eq:lowest:shell:inverse} from Lemma~\ref{lem:special:cases} applied with $\lambda r = \mu$, $\lambda=\lambda_\qbn$, $\lambda_0=\lambda_{q+\half {+1}}$, $\rho=\varrho_\diamond$, and $q=1$, \eqref{eq:inverse:div:parameters:0} holds by definition and by \eqref{condi.Nfin0}, \eqref{eq:inverse:div:parameters:1} holds due to \eqref{condi.Ndec0}, \eqref{eq:inv:div:wut} holds by \eqref{condi.Nfin0}, \eqref{eq:inverse:div:v:global}--\eqref{eq:inverse:div:v:global:parameters} hold from Remark~\ref{rem:lossy:choices}, and \eqref{eq:riots:4} holds from \eqref{ineq:dpot:1}. In order to check the additional assumptions in Part 2 of Proposition~\ref{lem.pr.invdiv2.c}, we set
\begin{align}
    &N_{**} \textnormal{ as in \eqref{i:par:10}} \,, \quad \NcutLarge, \NcutSmall \textnormal{ as in \eqref{i:par:6}} \, , \quad \Ga=\Ga_q^{ {\frac 1{10}}} \, , \quad \delta_{\rm tiny} = \delta^2_{q+3\bn} \, , \quad r_\phi = r_{q+\half+1} \, , \label{eq:desert:choices:1}\\
    &\delta_{\phi,p}^{\sfrac 32} = \const_{G_\diamond,p}\const_{*,p}(\lambda_{q+\half}\Ga_q)^{-1}r_{q+\half+1} \, , \quad \bm = 1 \, , \quad \mu_0 = \lambda_{q+\half+1}\Ga_q^{-1} \, , \quad \mu_{\bm} = \mu_1= \lambda_{q+\half+1}\Ga_q^{2} \, . \notag
\end{align}
Then \eqref{i:st:sample:wut:c}--\eqref{i:st:sample:wut:wut:c} hold from \eqref{condi.Nfin0}, \eqref{i:st:sample:wut:wut:wut:c} holds from \eqref{ineq:Nstarstar:dpot}, \eqref{eq:sample:prop:de:phi:c} holds by definition, \eqref{eq:sample:prop:Ncut:1:c} holds from \eqref{condi.Ncut0.1}, \eqref{eq:sample:prop:Ncut:2:c} holds from \eqref{condi.Ncut0.2}, \eqref{eq:sample:prop:Ncut:3:c} holds from \eqref{condi.Nfin0}, \eqref{eq:sample:prop:decoup:c} holds from \eqref{condi.Ndec0}, \eqref{eq:sample:prop:par:00:c} holds by definition, \eqref{eq:sample:prop:parameters:0:c} holds by definition and immediate computation, \eqref{eq:sample:riots:4:c} holds due to \eqref{ineq:dpot:1}, and \eqref{eq:sample:riot:4:4:c} holds due to \eqref{ineq:Nstarz:1}.

From \eqref{eq:inverse:div:error:stress:bound} and summing over the values of $i,j,k,\xi,\vecl,\diamond,I$ which may be non-zero at a fixed point in time in a manner similar to that from \texttt{Case 1a}, we have from \eqref{eq:inverse:div:error:stress:bound} and Remark~\ref{rem:lossy:choices} that for all $M_\circ,N_\circ\leq 2\Nind$,
\begin{align}\label{eq:cur:osc:ugh:2:redux}
\left| D^N \Dtq^M \sum_{i,j,k,\xi,\vecl,\diamond,I} \divR t_{i,j,k,\xi,\vecl,I,\diamond}^{q+\half+1} \right| \leq \la_\qbn^{-2} \delta_{q+3\bn}^{\sfrac 32} \Tau_\qbn^{2\Nindt} \la_{q+1}^N \tau_q^{-M} \, .
\end{align}
This verifies \eqref{eq:osc.current:estimate:1} at level $m=q+\half+1$. From \eqref{d:press:stress:sample:c}--\eqref{est.S.by.pr.final2:c} and \eqref{condi.Nfin0}, we have that there exists a pressure increment $\sigma_{ \divH t_{i,j,k,\xi,\vecl,I,\diamond}^{q+\half+1}} =\sigma_{ \divH t_{i,j,k,\xi,\vecl,I,\diamond}^{q+\half+1}}^+-\sigma_{ \divH t_{i,j,k,\xi,\vecl,I,\diamond}^{q+\half+1}}^-$ such that for $N,M\leq \sfrac{\Nfin}{7}$, 
\begin{align}\label{eq:desert:cowboy:1}
    \left| D^N \Dtq^M \divH t_{i,j,k,\xi,\vecl,I,\diamond}^{q+\half+1}  \right| \lesssim \left( \left( \sigma_{ \divH t_{i,j,k,\xi,\vecl,I,\diamond}^{q+\half+1}}^+ \right)^{\sfrac 32} r_{q+1}^{-1} + \delta_{q+3\bn}^2 \right) (\lambda_{q+\half+1}\Ga_q)^N \MM{M,\Nindt,\tau_q^{-1}\Ga_q^{i+14},\Tau_q^{-1}\Ga_q^{9}} \, .
\end{align}
From \eqref{eq:inverse:div:linear} and \eqref{est.S.pr.p.support:1:c}, we have that 
\begin{align}
    \supp \left( \sigma_{ \divH t_{i,j,k,\xi,\vecl,I,\diamond}^{q+\half+1}}^+ \right) \subseteq \supp \left( \divH t_{i,j,k,\xi,\vecl,I,\diamond}^{q+\half+1} \right) \subseteq \supp \left( a_{\pxi,\diamond} \left(\rhob_\pxi^\diamond \zetab_\xi^I \right)\circ\Phiik \right) \, . \label{eq:ocdc:support} 
\end{align}

Now define
\begin{align}\label{eq:desert:pressure:def}
    \sigma_{\ov\phi_O^{q+\half+1}}^{\pm} = \sum_{i,j,k,\xi,\vecl,I,\diamond} \sigma_{ \divH t_{i,j,k,\xi,\vecl,I,\diamond}^{q+\half+1}}^{\pm} \, . 
\end{align}
Then \eqref{eq:oooooldies} gives that \eqref{eq:o.p.6} is satisfied for $m'=q+\half+1$. 
From \eqref{eq:desert:cowboy:1}, \eqref{eq:desert:cowboy:sum}, \eqref{eq:inductive:partition}, and Corollary~\ref{lem:agg.pt} with 
\begin{align}
&H=\divH t_{i,j,k,\xi,\vecl,I,\diamond}^{q+\half+1} \, , \qquad \varpi=\left[\left(\sigma_{ \divH t_{i,j,k,\xi,\vecl,I,\diamond}^{q+\half+1}}^+\right)^{\sfrac 32}r_{q+\half+1}^{-1}+\delta_{q+3\bn}^2\right] \mathbf{1}_{\supp a_{\pxi,\diamond}(\rhob_\pxi^\diamond\zetab_\xi^I)\circ\Phiik }  \,  , \qquad p=1 \, , \notag 
\end{align}
we have that for $N,M\leq \sfrac{\Nfin}{7}$, 
\begin{align}\label{eq:desert:cowboy:2}
    \left| \psi_{i,q} D^N \Dtq^M \sum_{i',j,k,\xi,\vecl,I,\diamond} \divH t_{i',j,k,\xi,\vecl,I,\diamond}^{q+\half+1}  \right| &\lesssim \left( \left( \sigma_{\ov\phi_O^{q+\half+1}}^+ \right)^{\sfrac 32} r_{q+\half+1}^{-1} + \delta_{q+3\bn}^2 \right) \notag\\
    &\qquad \times (\lambda_{q+\half+1}\Ga_q)^N \MM{M,\Nindt,\tau_q^{-1}\Ga_q^{i+15},\Tau_q^{-1}\Ga_q^{9}} \, .
\end{align}
In combination with the bound in \eqref{eq:cur:osc:ugh:2:redux}, we have that \eqref{eq:co.p.1} is satisfied for $m=q+\half+1$. From \eqref{est.S.prbypr.pt:c}, \eqref{condi.Nfin0}, and \eqref{condi.Nindt}, we have that for $N,M\leq \sfrac{\Nfin}{7}$,
\begin{align}
     \left| D^N \Dtq^M \sigma_{ \divH t_{i,j,k,\xi,\vecl,I,\diamond}^{q+\half+1}}^+ \right| \lesssim \left( \sigma_{ \divH t_{i,j,k,\xi,\vecl,I,\diamond}^{q+\half+1}}^+ + \delta_{q+3\bn}^2 \right) (\lambda_{q+\half+1}\Ga_q)^N \MM{M,\Nindt,\tau_q^{-1}\Ga_q^{i+15},\Tau_q^{-1}\Ga_q^{9}} \, . \label{eq:desert:cowboy:4}
\end{align}
From \eqref{eq:desert:cowboy:4}, \eqref{eq:desert:cowboy:sum}, \eqref{eq:inductive:partition}, and Corollary~\ref{lem:agg.pt} with 
\begin{align}
&H=\sigma^+_{\divH t_{i,j,k,\xi,\vecl,I,\diamond}^{q+\half+1}} \, , \qquad \varpi= \left[ H + \delta_{q+3\bn}^2 \right] \mathbf{1}_{\supp a_{\pxi,\diamond}(\rhob_\pxi^\diamond\zetab_\xi^I)\circ\Phiik}  \,  , \qquad p=1 \, , \notag 
\end{align}
we have that \eqref{eq:co.p.2} is satisfied for $m=q+\half+1$.

Next, from \eqref{est.S.pr.p:c}, we have that
\begin{align}
    \left\| \sigma^{\pm}_{\divH t_{i,j,k,\xi,\vecl,I,\diamond}^{q+\half+1}} \right\|_{\sfrac 32} &\les \left(\delta_\qbn r_q^{-\sfrac 23} \Ga_q^{2j+30} \Lambda_q^{\sfrac 23} \left| \supp \left( \eta_{i,j,k,\xi,\vecl,\diamond} \zetab_\xi^{I,\diamond} \right) \right|^{\sfrac 23} 
     {+\la_{q+\bn}^{-10}}\right)
    (\lambda_{q+\half}\Ga_q)^{-\sfrac 23} r_{q+\half+1}^{\sfrac 23} \notag \, .
\end{align}
Now from \eqref{eq:desert:pressure:def}, \eqref{eq:desert:ineq}, and Corollary~\ref{rem:summing:partition} with $\theta=2$, $\theta_1=0$, $\theta_2=2$, $H=\sigma^{\pm}_{\divH t_{i,j,k,\xi,\vecl,I,\diamond}^{q+\half+1}}$, and $p=\sfrac 32$, we have that
\begin{align}
    \left\| \psi_{i,q} \sigma^{\pm}_{\ov \phi_O^{q+\half+1}} \right\|_{\sfrac 32} &\les \delta_\qbn r_q^{-\sfrac 23} \Ga_q^{33} \Lambda_q^{\sfrac 23}  (\lambda_{q+\half}\Ga_q)^{-\sfrac 23} r_{q+\half+1}^{\sfrac 23} \notag\\
    &\leq \delta_{q+\bn+\half+1} \Ga_{q+\half+1}^{-10} \, . \notag 
\end{align}
Combined with \eqref{eq:co.p.2}, this verifies \eqref{eq:co.p.3} at level $q+\half+1$.  Arguing now for $p=\infty$ from \eqref{est.S.pr.p:c} and using \eqref{ineq:jmax:use}, we have that
\begin{align}
    \left\| \sigma^{\pm}_{\divH t_{i,j,k,\xi,\vecl,I,\diamond}^{q+\half+1}} \right\|_{\infty} &\les \delta_\qbn r_q^{-\sfrac 23}\Ga_q^{2j+30}\La_q^{\sfrac 23} \left( \frac{\la_{q+\half+1}}{\la_{q+\half}\Ga_q} \right)^{\sfrac 43} (\la_{q+\half}\Ga_q)^{-\sfrac 23} r_{q+\half+1}^{\sfrac 23} \\
    &\les \Ga_q^{36+\badshaq} \left( \frac{\la_{q+\half+1}}{\la_{q+\half}\Ga_q} \right)^{\sfrac 43} \Lambda_q^{\sfrac 23}  (\lambda_{q+\half}\Ga_q)^{-\sfrac 23}
    \leq \Ga_{q+\half+1}^{\badshaq-11}
    \notag \, .
\end{align}
Now from \eqref{eq:desert:pressure:def} and Corollary~\ref{lem:agg.pt} with $H=\sigma^{\pm}_{\divH t_{i,j,k,\xi,\vecl,I,\diamond}^{q+\half+1}}$, $\varpi=\Ga_{q+\half+1}^{\badshaq-11}\mathbf{1}_{\supp a_{\pxi,\diamond}(\rhob_\pxi^\diamond\zetab_\xi^I)\circ\Phiik}$ and $p=1$, we have that
\begin{align}
    \left\| \psi_{i,q} \sigma^{\pm}_{\ov \phi_O^{q+\half+1}} \right\|_\infty 
    &\leq \Ga_{q+\half+1}^{\badshaq-10} \, . \notag 
\end{align}
Combined again with \eqref{eq:co.p.2}, this verifies \eqref{eq:co.p.3.5} at level $q+\half+1$. 

Finally, from \eqref{est.S.prminus.pt:c}, \eqref{ind:pi:lower}, \eqref{condi.Nindt}, \eqref{ineq:r's:eat:Gammas}, and \eqref{condi.Nfin0}, we have that for $N,M\leq \sfrac{\Nfin}{7}$, 
\begin{align}
    \left| D^N \Dtq^M \sigma^{-}_{\divH t_{i,j,k,\xi,\vecl,I,\diamond}^{q+\half+1}} \right| &\les \left(\frac{r_{q+\half+1}}{r_q} \right)^{\sfrac 23} \Ga_q^{28} \pi_q^q \La_q^{\sfrac 23} \la_{q+\half}^{-\sfrac 23} (\la_{q+\half}\Ga_q)^N \MM{M,\Nindt,\tau_q^{-1}\Ga_q^{i+15},\Tau_q^{-1}\Ga_q^9} \notag\\
    &\leq \Ga_q^{-10}\left( \frac{\la_q}{\la_{q+\half}} \right)^{\sfrac 23} \pi_q^q (\la_{q+\half}\Ga_q)^N \MM{M,\Nindt,\tau_q^{-1}\Ga_q^{i+15},\Tau_q^{-1}\Ga_q^9} \, . \notag 
\end{align}
Applying \eqref{eq:desert:pressure:def} and Corollary~\ref{lem:agg.pt} with $H=\sigma^{-}_{\divH t_{i,j,k,\xi,\vecl,I,\diamond}^{q+\half+1}}$, $\varpi=\left( \frac{\la_q}{\la_{q+\half}} \right)^{\sfrac 23} \pi_q^q\Ga_q^{-1} \mathbf{1}_{\supp a_{\pxi,\diamond}(\rhob_\pxi^\diamond \zetab_\xi^I)\circ\Phiik}$ and $p=1$, we have that \eqref{eq:co.p.4} is verified at level $m=q+\half+1$.
\smallskip

\noindent\texttt{Case 3:} Analysis for  \eqref{osc.current.med1} and \eqref{osc.current.high}.
Fix $\xi$, $i$, $j$, $k$, $\vecl$, $I$, and $\diamond$. In order to check the low-frequency, preliminary assumptions in Part 1 of Proposition~\ref{lem.pr.invdiv2.c}, we may use the exact same choices as in \eqref{eq:desert:choices}.  In order to check the high-frequency, preliminary assumptions in Part 1 of Proposition~\ref{lem.pr.invdiv2.c}, we set
\begin{subequations}
\begin{align}
    &\varrho_R = \left( \tP_{(q+\half+1, q+\half+{\sfrac32}]}^\xi +\tP_{(q+\half+{\sfrac32}, q+\half+2]}^\xi \right) (\varrho^I_{\pxi,R})^3 r_q \quad \textnormal{if} \quad m=q+\half+2 \, , \notag\\
    &\varrho_\vp = \left( \tP_{(q+\half+1, q+\half+{\sfrac32}]}^\xi +\tP_{(q+\half+{\sfrac32}, q+\half+2]}^\xi \right) (\varrho^I_{\pxi,\vp})^3 \quad \textnormal{if} \quad m=q+\half+2 \notag\\
    &\varrho_R = \tilde{\mathbb{P}}_{(m-1,m]}^\xi  (\varrho^I_{\pxi,R})^3 r_q \, , \quad \varrho_\vp = \tilde{\mathbb{P}}_{(m-1,m]}^\xi  (\varrho^I_{\pxi,\vp})^3 \quad \textnormal{ if } \quad m\neq q+\half+2 \notag\\
    &\vartheta_\diamond^{i_1i_2\dots i_{\dpot-1}i_\dpot} \text{ given by Lemma~\ref{lem:LP.est}}, , \quad \const_{\ast,1} = 1 \, , \quad \const_{\ast,\infty} = \left( \frac{\min(\lambda_m,\lambda_\qbn)}{\lambda_{q+\half}\Ga_q} \right)^{2} \, , \quad \mu =\la_{q+\half}\Ga_q \, , \notag\\
    &\Upsilon = \lambda_{q+\half+1} \, , \quad \Upsilon'=\La=\la_{q+\half+\sfrac 32} \quad \textnormal{for the first projector if } \quad m=q+\half+2 \, ,  \notag\\
    &\Upsilon = \lambda_{q+\half+\sfrac 32} \, , \quad \Upsilon'=\La=\la_{q+\half+2} \quad \textnormal{for the second projector if } \quad m=q+\half+2 \, ,  \notag\\
    &\Upsilon = \lambda_{m-1} \, , \quad \Upsilon'=\La=\min\left(\la_{m},\lambda_{q+\bn}\right) \quad \textnormal{if} \quad m\neq q+\half+2 \, ,  \notag\\
    &\Ndec \textnormal{ as in \eqref{i:par:9}/\eqref{condi.Ndec0}} \, , \quad \dpot \textnormal{ as in \eqref{i:par:10}/\eqref{ineq:dpot:1}} , . \notag
\end{align}
\end{subequations}
Now \eqref{item:inverse:i}--\eqref{item:inverse:ii} of the high-frequency assumptions of Proposition~\ref{prop:intermittent:inverse:div} hold by definition and from \eqref{lem:LP.est} and \eqref{int:pipe:bundle:short} as before, \eqref{eq:DN:Mikado:density} holds due to Propositions~\ref{prop:pipeconstruction} and \ref{prop:pipe.flow.current} and estimate \eqref{eq:LP:div:estimates} from Lemma~\ref{lem:LP.est} applied with the obvious choices, \eqref{eq:inverse:div:parameters:0} holds by definition, by \eqref{condi.Nfin0}, and by our extra splitting in the case $m=q+\half+2$, and \eqref{eq:inverse:div:parameters:1} and \eqref{eq:inv:div:wut}--\eqref{eq:riots:4} hold after appealing to the same parameter inequalities as the previous case. In order to check the additional assumptions in Part 2 of Proposition~\ref{lem.pr.invdiv2.c}, we set
\begin{align}
    &N_{**} \textnormal{ as in \eqref{i:par:10}/\eqref{ineq:Nstarz:1}} \,, \quad \NcutLarge, \NcutSmall \textnormal{ as in \eqref{i:par:6}/\eqref{condi.Ncut0}} \, , \quad \Ga=\Ga_q^{\frac 1{10}} \, , \quad \delta_{\rm tiny} = \delta^2_{q+3\bn} \, , \quad r_\phi = r_{\min(m,\qbn)} \, , \notag\\
    &\delta_{\phi,p}^{\sfrac 32} = \const_{G_\diamond,p}\const_{*,p}\Upsilon'\Upsilon^{-2}r_{\min(m,\qbn)} \, , \quad \mu_0 = \lambda_{q+\half+1} \, , \quad \mu_1 = \lambda_{q+\half+\sfrac 32}\Ga_q^2 \, , \notag\\
    &\mu_{m'} = \lambda_{q+\half+{m'}}\Ga_q^2 \quad \textnormal{if} \quad 2\leq m' \leq \half+1 \, , \notag\\
    &\bm = {1} \quad \textnormal{for the first projector if} \quad  m=q+\half+2 \, , \notag \\
    &\bm = {2} \quad \textnormal{for the second projector if} \quad  m=q+\half+2 \, , \notag \\
    &\bm = {m-q-\half} \quad \textnormal{if} \quad  m>q+\half+2 \, . \label{eq:more:desert:choices}
\end{align}
Then \eqref{i:st:sample:wut:c}--\eqref{eq:sample:prop:decoup:c} hold after appealing to the same inequalities as in the previous case, \eqref{eq:sample:prop:par:00:c} holds by definition, \eqref{eq:sample:prop:parameters:0:c} holds by definition and immediate computation, and \eqref{eq:sample:riots:4:c}--\eqref{eq:sample:riot:4:4:c} hold as in the previous case.

First, we have that \eqref{eq:osc.current:estimate:1} at level $m'$ for $q+\half+2\leq m' \leq q+\bn$ is satisfied by an argument essentially identical to that of the previous case.  Next, from \eqref{d:press:stress:sample:c}--\eqref{est.S.by.pr.final2:c} and \eqref{condi.Nfin0}, we have that for $q+\half+2\leq m \leq q+\bn+1$, there exists a pressure increment $\sigma_{ \divH t_{i,j,k,\xi,\vecl,I,\diamond}^{m}}^+$ such that for $N,M\leq \sfrac{\Nfin}{7}$, 
\begin{align}
    \left| D^N \Dtq^M \divH t_{i,j,k,\xi,\vecl,I,\diamond}^{m}  \right| &\lesssim \left( \left( \sigma_{ \divH t_{i,j,k,\xi,\vecl,I,\diamond}^{m}}^+ \right)^{\sfrac 32} r_{\min(m,\qbn)}^{-1} + \delta_{q+3\bn}^2 \right)  \notag\\
    &\qquad \qquad \times (\min(\lambda_m,\lambda_{q+\bn})\Ga_q)^N \MM{M,\Nindt,\tau_q^{-1}\Ga_q^{i+14},\Tau_q^{-1}\Ga_q^{9}} \, .  \label{eq:desert:cowboy:1:redux}
\end{align}
From \eqref{eq:inverse:div:linear}, \eqref{est.S.pr.p.support:1:c}, and \eqref{eq:LP:div:support}, we have that 
\begin{align}
    \supp \left( \sigma_{ \divH t_{i,j,k,\xi,\vecl,I,\diamond}^{m}}^+ \right) \subseteq \supp \left( \divH t_{i,j,k,\xi,\vecl,I,\diamond}^{m} \right) \subseteq \supp \left( a_{\pxi,\diamond} \left(\rhob_\pxi^\diamond \zetab_\xi^I \right)\circ\Phiik \right) \cap B\left( \supp \varrho_{\pxi,\diamond}^I, \lambda_{m-1}^{-1} \right) \, . \label{eq:desert:support} 
\end{align}

Now define
\begin{align}
    \sigma_{\ov\phi_O^{m}}^{\pm} = \sum_{i,j,k,\xi,\vecl,I,\diamond} \sigma_{ \divH t_{i,j,k,\xi,\vecl,I,\diamond}^{m}}^{\pm} \quad \textnormal{if} \quad m\neq q+\bn \, , \label{eq:desert:pressure:def:redux} \\
    \sigma_{\ov\phi_O^{q+\bn}}^{\pm} = \sum_{m'=q+\bn}^{q+\bn+1} \sum_{i,j,k,\xi,\vecl,I,\diamond} \sigma_{ \divH t_{i,j,k,\xi,\vecl,I,\diamond}^{m'}}^{\pm} \quad \textnormal{if} \quad m=q+\bn \, , \label{eq:desert:pressure:def:redux:redux} 
\end{align}
Then \eqref{eq:oooooldies}--\eqref{eq:dodging:oldies} and \eqref{eq:desert:support} give that \eqref{eq:o.p.6} is satisfied for $q+\half+2\leq m'\leq q+\bn$. 
From \eqref{eq:desert:cowboy:1:redux}, \eqref{eq:desert:cowboy:sum}, \eqref{eq:inductive:partition}, and Corollary~\ref{lem:agg.pt} with 
\begin{align}
&H=\divH t_{i,j,k,\xi,\vecl,I,\diamond}^{m} \, , \qquad \varpi=\left[\left(\sigma_{ \divH t_{i,j,k,\xi,\vecl,I,\diamond}^{m}}^+\right)^{\sfrac 32}r_{\min(m,\qbn)}^{-1}+\delta_{q+3\bn}^2\right] \mathbf{1}_{\supp a_{\pxi,\diamond}(\rhob_\pxi^\diamond\zetab_\xi^I))\circ\Phiik}  \,  , \qquad p=1 \, , \notag 
\end{align}
we have that for $N,M\leq \sfrac{\Nfin}{7}$ and $q+\half+2\leq m <q+\bn$, 
\begin{align}\label{eq:desert:cowboy:2:redux}
    \left| \psi_{i,q} D^N \Dtq^M \sum_{i',j,k,\xi,\vecl,I,\diamond} \divH t_{i',j,k,\xi,\vecl,I,\diamond}^{m}  \right| &\lesssim \left( \left( \sigma_{\ov\phi_O^{m}}^+ \right)^{\sfrac 32} r_{\min(m,\qbn)}^{-1} + \delta_{q+3\bn}^2 \right) \notag\\
    &\qquad \times (\lambda_{m}\Ga_q)^N \MM{M,\Nindt,\tau_q^{-1}\Ga_q^{i+15},\Tau_q^{-1}\Ga_q^{9}} \, .
\end{align}
An analogous statement holds if $m=q+\bn$, with the only change being the extra summation needed on the left-hand side, which leads to \eqref{eq:co.p.1} for $q+\half+2\leq m \leq \qbn$. From \eqref{est.S.prbypr.pt:c}, \eqref{condi.Nfin0}, and \eqref{condi.Nindt}, we have that for $N,M\leq \sfrac{\Nfin}{7}$,
\begin{align}
     \left| D^N \Dtq^M \sigma_{ \divH t_{i,j,k,\xi,\vecl,I,\diamond}^{m}}^+ \right| \lesssim \left( \sigma_{ \divH t_{i,j,k,\xi,\vecl,I,\diamond}^{m}}^+ + \delta_{q+3\bn}^2 \right) (\min(\lambda_{m},\la_\qbn)\Ga_q)^N \MM{M,\Nindt,\tau_q^{-1}\Ga_q^{i+15},\Tau_q^{-1}\Ga_q^{9}} \, . \label{eq:desert:cowboy:4:redux}
\end{align}
From \eqref{eq:desert:cowboy:4:redux}, \eqref{eq:desert:cowboy:sum}, \eqref{eq:inductive:partition}, and Corollary~\ref{lem:agg.pt} with 
\begin{align}
&H=\sigma^+_{\divH t_{i,j,k,\xi,\vecl,I,\diamond}^{m}} \, , \qquad \varpi= \left[ H + \delta_{q+3\bn}^2 \right] \mathbf{1}_{\supp a_{\pxi,\diamond}(\rhob_\pxi^\diamond\zetab_\xi^I)\circ\Phiik}  \,  , \qquad p=1 \, , \notag 
\end{align}
we have that \eqref{eq:co.p.2} is satisfied for $q+\half+2\leq m \leq q+\bn$.

Next, from \eqref{est.S.pr.p:c}, we have that
\begin{align}
    \left\| \sigma^{\pm}_{\divH t_{i,j,k,\xi,\vecl,I,\diamond}^{m}} \right\|_{\sfrac 32} &\les \left(\delta_\qbn r_q^{-\sfrac 23} \Ga_q^{2j+30} \Lambda_q^{\sfrac 23} \left| \supp \left( \eta_{i,j,k,\xi,\vecl,\diamond} \zetab_\xi^{I,\diamond} \right) \right|^{\sfrac 23} + \la_{q+\bn}^{-10}
    \right)
    \left(\lambda_{m-1}^2\lambda_m^{-1}\right)^{-\sfrac 23} r_{\min(m,\qbn)}^{\sfrac 23} \notag \, .
\end{align}
Now from \eqref{eq:desert:pressure:def:redux}--\eqref{eq:desert:pressure:def:redux:redux}, \eqref{eq:desert:ineq}, and Corollary~\ref{rem:summing:partition} with $\theta=2$, $\theta_1=0$, $\theta_2=2$, $H=\sigma^{\pm}_{\divH t_{i,j,k,\xi,\vecl,I,\diamond}^{m}}$, and $p=\sfrac 32$, we have that
\begin{align}
    \left\| \psi_{i,q} \sigma^{\pm}_{\ov \phi_O^{m}} \right\|_{\sfrac 32} &\les \delta_\qbn r_q^{-\sfrac 23} \Ga_q^{33} \Lambda_q^{\sfrac 23}  \left(\lambda_{m-1}^2\lambda_m^{-1}\right)^{-\sfrac 23} r_{\min(m,\qbn)}^{\sfrac 23} \notag\\
    &\leq \delta_{m+\bn} \Ga_{m}^{-10} \, . \notag 
\end{align}
Combined with \eqref{eq:co.p.2}, this verifies \eqref{eq:co.p.3} for $q+\half+2\leq m' \leq q+\bn$.  Arguing now for $p=\infty$ from \eqref{est.S.pr.p:c}, we have that
\begin{align}
    \left\| \sigma^{\pm}_{\divH t_{i,j,k,\xi,\vecl,I,\diamond}^{m}} \right\|_{\infty} &\les \delta_\qbn r_q^{-\sfrac 23}\Ga_q^{2j+30}\La_q^{\sfrac 23} \left( \frac{\min(\la_m,\la_\qbn)}{\la_{q+\half}\Ga_q} \right)^{\sfrac 43} \left(\lambda_{m-1}^2\lambda_m^{-1}\right)^{-\sfrac 23} r_{\min(m,\qbn)}^{\sfrac 23}  \notag \, .
\end{align}
Now from \eqref{eq:desert:pressure:def:redux}--\eqref{eq:desert:pressure:def:redux:redux},  \eqref{ineq:jmax:use}, \eqref{eq:par:div:2}, and Corollary~\ref{lem:agg.pt} with $H=\sigma^{\pm}_{\divH t_{i,j,k,\xi,\vecl,I,\diamond}^{m}}$, $\varpi=\mathbf{1}_{\supp a_{\pxi,\diamond}(\rhob_\pxi^\diamond\zetab_\xi^I)\circ\Phiik}$ and $p=1$, we have that
\begin{align}
    \left\| \psi_{i,q} \sigma^{\pm}_{\ov \phi_O^{m}} \right\|_\infty &\les  \Ga_q^{36+\badshaq} \left( \frac{\min(\la_{m},\la_\qbn)}{\la_{q+\half}\Ga_q} \right)^{\sfrac 43} \Lambda_q^{\sfrac 23}  \left(\lambda_{m-1}^2\lambda_m^{-1}\right)^{-\sfrac 23} r_{\min(m,\qbn)}^{\sfrac 23}r_q^{-\sfrac23} \notag\\
    &\leq \Ga_{q+\half+1}^{\badshaq-10} \, . \notag 
\end{align}
Combined again with \eqref{eq:co.p.2}, this verifies \eqref{eq:co.p.3.5} at level $q+\half+2\leq m' \leq q+\bn$. 

Finally, from \eqref{est.S.prminus.pt:c}, \eqref{ind:pi:lower}, \eqref{condi.Nindt}, \eqref{ineq:r's:eat:Gammas}, and \eqref{condi.Nfin0}, we have that for $N,M\leq \sfrac{\Nfin}{7}$ and $q+\half+3\leq m \leq q+\bn+1$, 
\begin{align}
    \left| D^N \Dtq^M \sigma^{-}_{\divH t_{i,j,k,\xi,\vecl,I,\diamond}^{m}} \right| &\les \left(\frac{r_{\min(m,\qbn)}}{r_q} \right)^{\sfrac 23} \Ga_q^{40} \pi_q^q \La_q^{\sfrac 23} \left(\lambda_{m-1}^2\lambda_m^{-1}\right)^{-\sfrac 23} (\la_{q+\half}\Ga_q)^N \MM{M,\Nindt,\tau_q^{-1}\Ga_q^{i+15},\Tau_q^{-1}\Ga_q^9} \notag\\
    &\leq \Ga_q^{-10}\left( \frac{\la_q}{\la_{q+\half}} \right)^{\sfrac 23} \pi_q^q (\la_{q+\half}\Ga_q)^N \MM{M,\Nindt,\tau_q^{-1}\Ga_q^{i+15},\Tau_q^{-1}\Ga_q^9} \, . \notag 
\end{align}
A similar inequality holds for $m=q+\half+2$ after using the extra splitting of the Littlewood-Paley projector to mitigate the loss from $\lambda_{m-1}^{-1}\lambda_m$.  Then applying \eqref{eq:desert:pressure:def:redux}--\eqref{eq:desert:pressure:def:redux:redux} and Corollary~\ref{lem:agg.pt} with $H=\sigma^{-}_{\divH t_{i,j,k,\xi,\vecl,I,\diamond}^{m}}$, $\varpi=\Ga_q^{-10}\left( \frac{\la_q}{\la_{q+\half}} \right)^{\sfrac 23} \pi_q^q \mathbf{1}_{\supp a_{\pxi,\diamond}(\rhob_\pxi^\diamond \zetab_\xi^I)\circ\Phiik}$ and $p=1$, we have that \eqref{eq:co.p.4} is verified.
\smallskip

\noindent\texttt{Case 4:} Analysis for \eqref{osc.current.high2}. We expect that the error term in \eqref{osc.current.high2} is vanishingly small due to the Littlewood-Paley projector on the cubed pipe density.  Therefore no pressure increment will be necessary, and we do not even need a local portion of the inverse divergence.  We thus apply Proposition~\ref{prop:intermittent:inverse:div} with $p=\infty$ and the following choices.  The low-frequency assumptions in Part 1 are exactly the same as the $L^\infty$ low-frequency assumptions in the previous two steps.  For the high-frequency assumptions, we recall the choice of $N_{**}$ from \eqref{i:par:10}/\eqref{ineq:Nstarz:1} and set
\begin{align}
    &\varrho_\ph = (\Id - \tilde{\mathbb{P}}_{q+\bn+1}^\xi ) \mathbb{P}_{\neq 0} \left( \varrho_{\pxi,\ph}^I \right)^3 \, , \quad 
    \varrho_R = (\Id - \tilde{\mathbb{P}}_{q+\bn+1}^\xi ) \mathbb{P}_{\neq 0} \left( \varrho_{\pxi,R}^I \right)^3 r_q\, , \quad
    \vartheta^{i_1i_2\dots i_{\dpot-1}i_\dpot}_\diamond = \delta^{i_1i_2\dots i_{\dpot-1}i_\dpot} \Delta^{-\sfrac \dpot 2}\varrho_\diamond \, , \notag\\
    &\mu = \Upsilon=\Upsilon' = \lambda_{q+\half}\Ga_q \, , \quad \Lambda=\lambda_{q+\bn} \, , \quad \const_{*,\infty} = \left( \frac{\lambda_\qbn}{\lambda_{q+\bn+1}} \right)^{N_{**}} \lambda_\qbn^3 \, , \quad \Ndec \textnormal{ as in \eqref{i:par:9}/\eqref{condi.Ndec0}} \, , \quad \dpot=0 \, . \notag
\end{align}
Then we have that item~\ref{item:inverse:i} is satisfied by definition, item~\ref{item:inverse:ii} is satisfied as in the previous steps, \eqref{eq:DN:Mikado:density} is satisfied using Propositions~\ref{prop:pipeconstruction} and \ref{prop:pipe.flow.current} and \eqref{eq:remainder:inverse} from Lemma~\ref{lem:special:cases}, \eqref{eq:inverse:div:parameters:0} is satisfied by definition and as in the previous steps, and \eqref{eq:inverse:div:parameters:1} is satisfied by \eqref{condi.Ndec0}.  For the nonlocal assumptions, we choose $M_\circ,N_\circ=2\Nind$ so that \eqref{eq:inv:div:wut}--\eqref{eq:inverse:div:v:global:parameters} are satisfied as in Case 1, and \eqref{eq:riots:4} is satisfied from \eqref{ineq:Nstarz:1}. We have thus satisfied all the requisite assumptions, and we therefore obtain nonlocal bounds very similar to those from the previous steps, which are consistent with \eqref{eq:osc.current:estimate:1} at level $q+\bn$.  We omit further details.
\end{proof}

\begin{lemma}[\bf Pressure current]\label{lem:currentoscillation:pressure:current}
For every $m'\in\{q+\half+1,\dots,q+\bn\}$, there exist a current error $\phi_{{\ov\phi_O^{m'}}}$ associated to the pressure increment $\si_{\ov\phi_O^{m'}}$ defined by Lemma \ref{lem:oscillation.current:general:estimate} and a function of time $\bmu_{\phi_{{\ov\phi_O^{m'}}}}$ which satisfy the following properties.  
\begin{enumerate}[(i)]
    \item\label{i:pc:2} We have the decompositions and equalities
    \begin{subequations}
    \begin{align}\label{eq:desert:decomp}
        \phi_{{\ov\phi_O^{m'}}} &= \phi_{{\ov\phi_O^{m'}}}^* + \sum_{m=q+\half+1}^{{m'}} \phi_{{\ov\phi_O^{m'}}}^{m}, \quad 
         \phi_{\ov \phi_O^{m'}}^{m} = \phi_{\ov \phi_O^{m'}}^{m,l} + \phi_{\ov \phi_O^{m'}}^{m,*} \, \\
         \div \phi_{{\ov\phi_O^{m'}}} &+ \bmu_{\si_{\phi_{{\ov\phi_O^{m'}}}}}'\, 
          = D_{t,q}\si_{{\ov\phi_O^{m'}}}\, . \label{eq:desert:decomp:_}
    \end{align}
    
    \end{subequations}
    \item\label{i:pc:3} For $q+\half+1 \leq m \leq m'$ and $N,M\leq  2\Nind$,
    \begin{subequations}
    \begin{align}
        &\left|\psi_{i,q} D^N \Dtq^M \phi_{\ov \phi_O^{m'}}^{m,l} \right| < \Ga_{m}^{-100} \left(\pi_q^m\right)^{\sfrac 32} r_m^{-1} (\la_m \Ga_m^2)^M \MM{M,\Nindt,\tau_q^{-1}\Ga_q^{i+17},\Tau_q^{-1}\Ga_q^9} \label{eq:desert:estimate:1} \\
        &\left\| D^N \Dtq^M \phi_{\ov \phi_O^{m'}}^{m,*} \right\|_\infty + \left\| D^N\Dtq^M \phi_{\ov\phi_O^{m'}}^{*} \right\|_\infty < \Tau_\qbn^{2\Nindt} \delta_{q+3\bn}^{\sfrac 32} (\la_{m'}\Ga_{m'}^2)^N \tau_q^{-M} \label{eq:desert:estimate:2} \, .
    \end{align}
    \end{subequations}
    \item\label{i:pc:4} For all $q+\half+1\leq m \leq m'$ and all $q+1\leq q' \leq m-1$, 
    \begin{align}
        B\left( \supp \hat w_{q'}, \sfrac 12 \lambda_{q'}^{-1} \Ga_{q'+1} \right) \cap \supp \left( \phi^{m,l}_{\ov \phi_O^{m'}} \right) = \emptyset \label{eq:desert:dodging} \, .
    \end{align}
    \item \label{i:pc:5} 
    The function of time $\bmu_{\si_{\phi_{{\ov\phi_O^{m'}}}}}$
    satisfies that for $M\leq 2\Nind$, 
    \begin{align}\label{eq:desert:mean}
        \bmu_{\sigma_{\ov\phi_O^{m'}}}(t) = \int_0^T \left\langle D_{t,q}\si_{{\ov\phi_O^{m'}}}\right\rangle (s) \, ds \, , \quad
        \left|\frac{d^{M+1}}{dt^{M+1}} \bmu_{\si_{\phi_{{\ov\phi_O^{m'}}}}} \right| 
        \leq (\max(1, T))^{-1}\delta_{q+3\bn} \tau_q^{-M} \, .
    \end{align}
\end{enumerate}
\end{lemma}

\begin{proof}
We follow the case numbering from Lemma~\ref{lem:oscillation.current:general:estimate}.  Note that the only cases which require a pressure increments are \texttt{Cases 2} and \texttt{3}, which correspond to the analysis of \eqref{osc.current.med0}--\eqref{osc.current.high}. 

\noindent\texttt{Case 2: }In this case, we recall from \eqref{eq:desert:choices:1} that we have chosen $\bm=1$ in item~\ref{i:st:sample:8:c}, $\mu_0=\lambda_{q+\half+1}\Ga_q^{-1}$, and $\mu_{\bm}=\mu_1=\lambda_{q+\half+1}\Ga_q^2$. We therefore have from \eqref{S:pr:current:dec} that
\begin{align}
    \phi_{ \divH t_{i,j,k,\xi,\vecl,I,\diamond}^{q+\half+1}} 
    = \phi_{ \divH t_{i,j,k,\xi,\vecl,I,\diamond}^{q+\half+1}}^* + \phi_{ \divH t_{i,j,k,\xi,\vecl,I,\diamond}^{q+\half+1}}^{0} + \phi_{ \divH t_{i,j,k,\xi,\vecl,I,\diamond}^{q+\half+1}}^{1} \, . \notag
\end{align}
and define the current error
$\phi_{\ov\phi_O^{q+\half+1}}:= \sum_{i,j,k,\xi,\vecl,I,\diamond} \phi_{\ov\phi_{i,j,k,\xi,\vecl,I,\diamond}^{q+\half+1}}$ which has a decomposition into
\begin{subequations}
\begin{align}
      \phi_{\ov\phi_O^{q+\half+1}}^* = \sum_{i,j,k,\xi,\vecl,I,\diamond} \phi_{\ov\phi_{i,j,k,\xi,\vecl,I,\diamond}^{q+\half+1}}^* \, ,  \quad   \phi_{\ov\phi_O^{q+\half+1}}^{q+\half+1}=\sum_{\substack{i,j,k,\xi,\vecl,I,\diamond \\ \iota=0,1}} \phi_{\ov\phi_{i,j,k,\xi,\vecl,I,\diamond}^{q+\half+1}}^{\iota} 
\end{align}
\end{subequations}
which satisfies \eqref{eq:desert:decomp:_} from \eqref{d:cur:error:stress:sample:c}. 
We make a further decomposition into the local and nonlocal parts, $\phi_{\ov \phi_O^{q+\half+1}}^{q+\half+1} = \phi_{\ov \phi_O^{q+\half+1}}^{q+\half+1,l} + \phi_{\ov \phi_O^{q+\half+1}}^{q+\half+1,*}$ from item~\eqref{sample2.item3:c}.

In order to check \eqref{eq:desert:estimate:1}, we recall the parameter choices from \texttt{Case 2} of the previous lemma and apply Part 4 of Proposition~\ref{lem.pr.invdiv2.c}, specifically \eqref{est.S.by.pr.final3:c}.  We then have from \eqref{condi.Nfin0} and \eqref{ind:pi:lower} that for each $i,j,k,\xi,\vecl,I,\diamond,\iota$ and $M,N\leq 2\Nind$, (after appending a superscript $l$ to refer to the local portion)
\begin{align}
    \left| D^N \Dtq^M \phi_{\ov\phi_{i,j,k,\xi,\vecl,I,\diamond}^{q+\half+1}}^{\iota,l} \right| &\leq {\tau_q^{-1} \Ga_q^{i+60}} {\pi_q^q} {\La_q^{\sfrac 23} \la_{q+\half}^{-\sfrac 23}} {\left( \frac{r_{q+\half+1}}{r_q} \right)^{\sfrac 23}} \left( \frac{\la_{q+\half+1}\Ga_q}{\la_{q+\half}} \right)^{\sfrac 43} \la_{q+\half}^{-1} \notag\\
    &\qquad \qquad \times (\la_{q+\half+1}\Ga_q^{  2})^N \MM{M,\Nindt-\NcutSmall-1,\tau_{q}^{-1}\Ga_q^{i+14},\Tau_q^{-1}\Ga_q^9} \, . \label{eq:desert:parsing}
\end{align}
Next, from \eqref{est.S.pr.p.support:2:c} and \eqref{eq:ocdc:support}, we have that
\begin{align}
    \supp \left( \phi_{\ov\phi_{i,j,k,\xi,\vec,I,\diamond}^{q+\half+1}}^{\iota,l} \right) &\subseteq B\left( \divH t^{q+\half+1}_{i,j,k,\xi,\vecl,I,\diamond}, 2\lambda_{q+\half+1}\Ga_q^{-1} \right) \notag\\
    &\subseteq B\left( \supp \left( a_{\pxi,\diamond}(\varrho_\pxi^\diamond\zetab_\xi^I)\circ\Phiik \right), 2\lambda_{q+\half+1}\Ga_q^{-1} \right) \,. \notag
\end{align}
Then applying \eqref{eq:oooooldies}, we have that \eqref{eq:desert:dodging} is verified for $m=q+\half+1$. Returning to the proof of \eqref{eq:desert:estimate:1}, we can now apply Corollary~\ref{lem:agg.Dtq} with
\begin{align*}
    H= \phi_{\ov\phi_{i,j,k,\xi,\vecl,I,\diamond}^{q+\half+1}}^{\iota,l} \, , \qquad \varpi = \Ga_q^{50} \pi_\ell \La_q^{\sfrac 23} \left( \frac{r_{q+\half+1}}{r_q} \right)^{\sfrac 23} \la_{q+\half}^{-\sfrac 23} \left( \frac{\la_{q+\half+1}\Ga_q}{\la_{q+\half}} \right)^{\sfrac 43} \la_{q+\half}^{-1}\mathbf{1}_{\supp a_{\pxi,\diamond}(\rhob_\pxi^\diamond \zetab_\xi^I)\circ\Phiik} \, .
\end{align*}
From \eqref{eq:aggDtq:conc:1}, \eqref{condi.Nindt}, \eqref{ind:pi:lower}, \eqref{eq:ind.pr.anticipated}, \eqref{condi.Nfin0}, and \eqref{ineq:in:the:morning}, we have that
\begin{align}
    &\left| \psi_{i,q} \sum_{i',j,k,\xi,\vecl,I,\diamond,\iota} \divH \left( \Dtq \sigma_{ \divH t_{i,j,k,\xi,\vecl,I,\diamond}^{q+\half+1}}^{\iota} \right) \right|\notag\\
    &\quad \lesssim \underbrace{\Ga_q r_q^{-1} \la_q \left( \pi_q^q \right)^{\sfrac 12}}_{\textnormal{cost of $\Dtq$}} \underbrace{\pi_q^q}_{\substack{\textnormal{dominates $\sfrac 23$ power} \\ \textnormal{of low-freq. coeff's}}} \underbrace{\La_q^{\sfrac 23}\la_{q+\half}^{-\sfrac 23} }_{\textnormal{$\sfrac 23$ power of freq. gain}} \underbrace{\Ga_q^{51} \left( \frac{r_{q+\half+1}}{r_q} \right)^{\sfrac 23}}_{\textnormal{lower order}} \underbrace{\left( \frac{\la_{q+\half+1}\Ga_q}{\la_{q+\half}} \right)^{\sfrac 43}}_{\textnormal{intermittency loss}} \underbrace{\la_{q+\half}^{-1}}_{\textnormal{inv. div. gain}} \notag\\
    &\quad \qquad \times (\la_{q+\half+1}\Ga_q^{ 2})^N \MM{M,\Nindt-\NcutSmall-1,\tau_{q}^{-1}\Ga_q^{i+15},\Tau_q^{-1}\Ga_q^9} \notag\\
    &\quad \les \Ga_q r_q^{-1} \la_q \left(\pi_q^{q+\half+1} \frac{\delta_{q+\bn}}{\delta_{q+\half+1+\bn}}\right)^{\sfrac 32} \la_q^{\sfrac 23} 
     \la_{q+\half}^{-\sfrac 23} \left( \frac{\la_{q+\half+1}\Ga_q}{\la_{q+\half}} \right)^{\sfrac 43} \la_{q+\half}^{-1} \notag\\
    &\qquad \qquad \times (\la_{q+\half+1}\Ga_q^{ 2})^N \MM{M,\Nindt-\NcutSmall-1,\tau_{q}^{-1}\Ga_q^{i+15},\Tau_q^{-1}\Ga_q^9} \notag\\
    &\quad \leq \Ga_q^{-150} r_q^{-1} 
    \left( \pi_q^{q+\half+1} \right)^{\sfrac 32} (\la_{q+\half+1}\Ga_q^{ 2})^N \MM{M,\Nindt,\tau_{q}^{-1}\Ga_q^{i+16},\Tau_q^{-1}\Ga_q^9} \label{eq:desert:parsing:2} 
\end{align}
for all $N,M\leq 2\Nind$, which verifies \eqref{eq:desert:estimate:1} at level $q+\half+1$. In order to achieve \eqref{eq:desert:estimate:2}, we appeal to \eqref{est.S.by.pr.final4:c}--\eqref{est.S.by.pr.final.star:c}, the choice of $K_\circ$ in item~\eqref{i:par:9.5}, \eqref{ineq:dpot:1}, and \eqref{condi.Nfin0}. Finally, the proof of \eqref{eq:desert:mean} follows from \eqref{est:mean.Dtsiph} in a very similar way, the only difference being that we need a large choice of $a_*$ in item~\eqref{i:choice:of:a} in order to have the advantageous prefactor of $\max(1,T)^{-1}$.
\smallskip

\noindent\texttt{Case 3: }In this case we consider the higher shells from the oscillation error.  The general principle is that the estimate will only be sharp in the $m=m'=\qbn$ double endpoint case, for which the intermittency loss is most severe.  We now explain why this is the case by parsing estimates \eqref{eq:desert:parsing} and \eqref{eq:desert:parsing:2}. We incur a material derivative cost of $\tau_q^{-1}\Ga_q^{i+60}$, which is converted into $r_q^{-1}\lambda_q (\pi_q^q)^{\sfrac 12}$ using \eqref{eq:psi:q:q'} and the rough definition of $\tau_q^{-1}=\delta_q^{\sfrac 12}\lambda_q r_q^{-\sfrac 13}$, or equivalently Corollary~\ref{lem:agg.Dtq}. The rescaled size of the high-frequency coefficients from the oscillation error is always $1$ (see the choices of $\const_{*,1}$ from the last lemma), and remains so upon being raised to the $\sfrac 23$ power in~\eqref{heatsie:2}. The low-frequency coefficient function from a trilinear oscillation error incurs a derivative cost of $\lambda_q$ (which we have grouped with ``frequency gain") and is dominated by $(\pi_\ell)^{\sfrac 32}r_q^{-1}$, at which point the $r_q^{-1}$ is scaled out due to the $L^1-L^{\sfrac 32}$ scaling balance between current and stress errors (see \eqref{est.S.by.pr.final2:c}, \eqref{est.S.prminus.pt:c}).  The negative power in the frequency gain is determined by which shell of the oscillation error is being considered.  The lower order terms may essentially be ignored.  Next, we have an $L^{\sfrac 32}\rightarrow L^\infty$ intermittency loss, which is used to pointwise dominate the high-frequency portion of the pressure increment using the $L^{\sfrac 32}$ norm and prevent a loop of new current error and new pressure creation. Finally, we have an inverse divergence gain depending on which synthetic Littlewood-Paley shell of the pressure increment we are considering. The net effect is that the ``frequency gain" upgrades the $\pi_\ell$ to $\pi_q^{m}$ since $m\leq m'$, the half power of $\pi_q^q$ is upgraded using $\la_q^{\sfrac 13}$ from the cost of $\Dtq$ and $\la_{m}^{-\sfrac 13}$ from the inverse divergence gain, and the remaining $\la_q^{\sfrac 23}\la_{m}^{-\sfrac 23}$ is strong enough to absorb the intermittency loss, with a perfect balance in the case
$$  m=m'=\qbn \qquad \implies \qquad \left( \frac{\la_{\qbn}}{\la_{q+\half}}  \right)^{\sfrac 43} \la_q^{\sfrac 23} \la_{\qbn}^{-\sfrac 23} \approx 1 \, . $$

In order to fill in the details, we now recall the choices of $\bm$ and $\mu_m$ from \eqref{eq:more:desert:choices}.  For the sake of brevity we ignore the slight variation in the case of the first projector for $m'=q+\half+2$ and focus on the second projector for $m'=q+\half+2$ and the other cases $q+\half+2<m'\leq \qbn$. We have from \eqref{d:cur:error:stress:sample:c} that
\begin{align}
    \phi_{ \divH t_{i,j,k,\xi,\vecl,I,\diamond}^{m'}} 
    = \phi_{ \divH t_{i,j,k,\xi,\vecl,I,\diamond}^{m'}}^* + \sum_{\iota=0}^{{m'-q-\half}} \phi_{ \divH t_{i,j,k,\xi,\vecl,I,\diamond}^{m'}}^{\iota} \, . \notag
\end{align}
and define the current error $\phi_{\ov\phi_O^{m'}} := \sum_{i,j,k,\xi,\vecl,I,\diamond} \phi_{ \divH t_{i,j,k,\xi,\vecl,I,\diamond}^{m'}}$ which has a decomposition into
\begin{align}
   \phi_{\ov\phi_O^{m'}}^* &= \sum_{i,j,k,\xi,\vecl,I,\diamond} \phi_{ \divH t_{i,j,k,\xi,\vecl,I,\diamond}^{m'}}^* \, , \qquad   \phi_{\ov\phi_O^{m'}}^{q+\half+1} = \sum_{{i,j,k,\xi,\vecl,I,\diamond}} \phi_{ \divH t_{i,j,k,\xi,\vecl,I,\diamond}^{m'}}^{0} \, , \notag \\
   \phi_{\ov\phi_O^{m'}}^{q+\half+2} &= \sum_{\substack{i,j,k,\xi,\vecl,I,\diamond \\ \iota=1,2}} \phi_{ \divH t_{i,j,k,\xi,\vecl,I,\diamond}^{m'}}^{\iota} \, , \notag\\
   \sigma_{\ov\phi_O^{m'}}^{q+\half+m} &= \sum_{\substack{i,j,k,\xi,\vecl,I,\diamond \\ \iota=m}} \phi_{ \divH t_{i,j,k,\xi,\vecl,I,\diamond}^{m'}}^{\iota} \quad \textnormal{if $q+\half+m = q+\half+\iota {\leq} m'$} \, . \notag 
\end{align}
As in the previous case, we make further decomposition into the local and nonlocal parts, $\phi_{\ov\phi_O^{m'}}^{q+\half+m} = \phi_{\ov\phi_O^{m'}}^{q+\half+m,l} + \phi_{\ov\phi_O^{m'}}^{q+\half+m, *}$ using \eqref{sample2.item3:c}.
We have thus verified \eqref{eq:desert:decomp} and \eqref{eq:desert:decomp:_} immediately from these definitions and from \eqref{d:cur:error:stress:sample:c} and item~\ref{sample2.item3:c}. In order to check \eqref{eq:desert:estimate:1}, we define the temporary notation $m(\iota)$ to make a correspondence between the value of $\iota$ above and the superscript on the left-hand side, which determines which bin the current errors go into.  Specifically, we set $m(0)=1$, $m(1)=m(2)=2$, $m(\iota)=\iota$ if $q+\half+\iota {\leq} m'$.
Then from Part 4 of Proposition~\ref{lem.pr.invdiv2.c}, specifically \eqref{est.S.by.pr.final3:c}, and \eqref{condi.Nfin0}, we have that for each $i,j,k,\xi,\vecl,I,\diamond,\iota$ and $M,N\leq 2\Nind$,
\begin{align}
    &\left| D^N \Dtq^M \divH  \left( \Dtq \sigma^\iota_{\divH t^{m'}_{i,j,k,\xi,\vecl,I,\diamond}} \right) \right| \notag\\
    &\quad \lec \tau_q^{-1} \Ga_q^{i+60} \pi_q^q \La_q^{\sfrac 23} \left( \frac{r_{q+\half+m(\iota)}}{r_q} \right)^{\sfrac 23} \left(\la_{m'-1}^{-2} \la_{m'}\right)^{\sfrac 23} \left( \frac{\min(\la_{q+\half+m(\iota)},\la_\qbn)\Ga_q}{\la_{q+\half}} \right)^{\sfrac 43}  \notag\\
    &\qquad \times \la_{q+\half+m(\iota)-1}^{-2} \la_{q+\half+m(\iota)} \left(\min(\la_{q+\half+m(\iota)},\la_{m'})\Ga_q^{ 2}\right)^N \MM{M,\Nindt-\NcutSmall-1,\tau_{q}^{-1}\Ga_q^{i+14},\Tau_q^{-1}\Ga_q^9} \, . \notag 
\end{align}
Next, from \eqref{est.S.pr.p.support:2:c} and \eqref{eq:desert:support}, we have that
\begin{align}
    \supp \left( \divH \left( \Dtq \sigma^\iota_{\divH t^{m'}_{i,j,k,\xi,\vecl,I,\diamond}} \right) \right) &\subseteq B\left( \divH t^{m'}_{i,j,k,\xi,\vecl,I,\diamond}, 2\lambda_{q+\half+m(\iota)-1}\Ga_q^{-2} \right) \notag\\
    &\subseteq B\left( \supp \left( a_{\pxi,\diamond}(\varrho_\pxi^\diamond\zetab_\xi^I)\circ\Phiik \rho_{\pxi,\diamond}^I  \right), \lambda_{m-1}^{-1}+2\lambda_{q+\half+m(\iota)-1}\Ga_q^{-2} \right) \,. \notag
\end{align}
Then applying \eqref{eq:oooooldies}, we have that \eqref{eq:desert:dodging} is verified for $m=q+\half+m(\iota)$. Returning to the proof of \eqref{eq:desert:estimate:1}, we can now apply Corollary~\ref{lem:agg.Dtq} with
\begin{align*}
    H&= \divH \left( \Dtq \sigma^\iota_{\divH t^{m'}_{i,j,k,\xi,\vecl,I,\diamond}} \right) \, , \notag\\
    \varpi &= \Ga_q^{60} \pi_q^q \left(
    \frac{\La_q\la_{m'}}{ \la_{m'-1}^{2}}\cdot \frac{r_{q+\half+m(\iota)}}{r_q} 
    \right)^{\sfrac 23} \left( \frac{\min(\la_{q+\half+m(\iota)},\la_\qbn)\Ga_q}{\la_{q+\half}} \right)^{\sfrac 43}  \frac{\la_{q+\half+m(\iota)}}
    {\la_{q+\half+m(\iota)-1}^{2}}
    \mathbf{1}_{\supp a_{\pxi,\diamond}(\rhob_\pxi^\diamond \zetab_\xi^I)\circ\Phiik}
    \, .
\end{align*}
From \eqref{ineq:r's:eat:Gammas}, \eqref{eq:aggDtq:conc:1}, \eqref{condi.Nindt}, \eqref{ind:pi:lower}, \eqref{eq:ind.pr.anticipated}, \eqref{condi.Nfin0}, and \eqref{ineq:in:the:afternoon}, we have that
\begin{align}
    &\left| \psi_{i,q} \sum_{i',j,k,\xi,\vecl,I,\diamond} \divH \left( \Dtq \sigma_{ \divH t^{m'}_{i,j,k,\xi,\vecl,I,\diamond}}^{\iota} \right) \right|\notag\\
    &\quad \lesssim \Ga_q r_q^{-1} \la_q \left( \pi_q^q \right)^{\sfrac 12} \Ga_q^{60} \pi_q^q \La_q^{\sfrac 23}  \left( \frac{r_{q+\half+m(\iota)}}{r_q} \right)^{\sfrac 23} \left(\la_{m'-1}^{-2} \la_{m'}\right)^{\sfrac 23} \left( \frac{\min(\la_{q+\half+m(\iota)},\la_\qbn)\Ga_q}{\la_{q+\half}} \right)^{\sfrac 43} \notag\\
    &\qquad \times \la_{q+\half+m(\iota)-1}^{-2} \la_{q+\half+m(\iota)} \left(\min(\la_{q+\half+m(\iota)},\la_{m'})\Ga_q^{ 2}\right)^N \MM{M,\Nindt,\tau_{q}^{-1}\Ga_q^{i+16},\Tau_q^{-1}\Ga_q^9} \notag\\
    &\quad \les \Ga_q r_q^{-1} \la_q \left(\pi_q^{q+\half+m(\iota)} \frac{\delta_{q+\bn}}{\delta_{q+\half+m(\iota)+\bn}} \right)^{\sfrac 32} \La_q^{\sfrac 23} \left(\la_{m'-1}^{-2} \la_{m'}\right)^{\sfrac 23} \left( \frac{\min(\la_{q+\half+m(\iota)},\la_\qbn)\Ga_q}{\la_{q+\half}} \right)^{\sfrac 43} \notag\\
    &\qquad \times \la_{q+\half+m(\iota)-1}^{-2} \la_{q+\half+m(\iota)} \left(\min(\la_{q+\half+m(\iota)},\la_{m'})\Ga_q^{ 2}\right)^N \MM{M,\Nindt,\tau_{q}^{-1}\Ga_q^{i+16},\Tau_q^{-1}\Ga_q^9} \notag\\
    &\quad \leq \Ga_q^{-150} r_q^{-1} \left( \pi_q^{q+\half+m(\iota)} \right)^{\sfrac 32} \left(\min(\la_{q+\half+m(\iota)},\la_{m'})\Ga_q^{ 2}\right)^N \MM{M,\Nindt,\tau_{q}^{-1}\Ga_q^{i+16},\Tau_q^{-1}\Ga_q^9} \label{eq:desert:parsing:3}
\end{align}
for all $N,M\leq 2\Nind$, which verifies \eqref{eq:desert:estimate:1} at level $m>q+\half+1$. In order to achieve \eqref{eq:desert:estimate:2} and \eqref{eq:desert:mean}, we appeal to \eqref{est.S.by.pr.final4:c}--\eqref{est.S.by.pr.final.star:c}, \eqref{est:mean.Dtsiph}, the choice of $K_\circ$ in item~\eqref{i:par:9.5}, \eqref{ineq:dpot:1}, and \eqref{condi.Nfin0}.

\end{proof}

\subsection{Linear current error}

\begin{lemma}[\bf Definition and basic estimates]\label{lem:lin:current:error}
There exists a current error $\ov\phi_{L}=\ov\phi_{L}^{q+\bn}$ and a function of time $\bmu_L$ such that the following hold.
\begin{enumerate}[(i)]
    \item We have the equality and decomposition
    \begin{align}
        \div \ov\phi_L^{q+\bn} + \bmu_L' &=  w_{q+1} \cdot \left( \partial_t u_q + u_q \cdot \nabla u_q + \nabla p_q \right)  \, , \qquad 
        \ov\phi_L^{q+\bn} = \ov\phi_L^{q+\bn,l} + \ov\phi_L^{q+\bn, *} \, . \label{eckel:equalities}
    \end{align}
    \item For all $N+M\leq \sfrac{\Nind}{4}$, we have that
    \begin{subequations}
    \begin{align}
        \left| \psi_{i,q+\bn-1} D^N D_{t,\qbn-1}^M \ov\phi_L^{q+\bn,l} \right| &\leq \Ga_{q+\bn}^{-100} \left( \left(\pi_q^{q+\bn}\right)^{\sfrac 32} + \left(\sigma_{\upsilon}^+\right)^{\sfrac 32} \right) r_{q+\bn}^{-1} \notag\\
     &\quad \times (\la_\qbn\Ga_\qbn)^N \MM{M,\Nindt,\tau_{q+\bn-1}\Ga_{\qbn-1}^{i{+4}},\Tau_{\qbn-1}^{-1}\Ga_{\qbn-1}^2} \label{eckel:local:est} \\
    \left\| D^N D_{t,\qbn-1}^M \ov\phi_L^{q+\bn,*} \right\|_\infty &\leq \delta_{q+2\bn}^{\sfrac 32} (\la_\qbn\Ga_\qbn)^N \MM{M,\Nindt,\tau^{-1}_{\qbn-1},\Tau_{\qbn-1}^{-1}\Ga_{\qbn-1}^2} \, .  \label{eckel:nonlocal:est}
    \end{align}
    \end{subequations}
     \item For all $q+1\leq q'\leq q+\qbn-1$, we have that
     \begin{align}
         \supp \left( \phi_L^{q+\bn, l} \right) \cap B\left( \hat w_{q'} , \Ga_{q'-1} \la_{q'}^{-1} \right) = \emptyset \, . \label{eckel:support}
     \end{align}

    \item The time function $\bmu_L$ satisfies $\bmu_L' = \langle  w_{q+1} \cdot \left( \partial_t u_q + u_q \cdot \nabla u_q + \nabla p_q \right)\rangle$ and
    \begin{align}\label{eq:sat:evening}
        \left| \frac{d^{M+1}}{dt^{M+1}}\bmu_L \right|
        \leq (\max(1,T))^{-1} \de_{q+3\bn} \MM{M,\Nindt,\tau_q^{-1},\Tau_{q+1}^{-1}} \quad \text{for }M\leq \sfrac{\Nind}4 \, .
    \end{align}
    \end{enumerate}
\end{lemma}
\begin{proof}
\texttt{Step 0: Splitting the error term and upgrading material derivatives. }We use the Euler-Reynolds system \eqref{eqn:ER}, the mollified stresses and pressures from  \eqref{eq:inductive:pointwise:upgraded:1:higher} and \eqref{eq:pressure:upgraded:higher}, respectively, and the formula for the premollified velocity increment potential from Remark~\ref{rem:rep:vel:inc:potential} to split the error
\begin{align}
    w_{q+1} \cdot &\left( \partial_t u_q + u_q \cdot \nabla u_q + \nabla p_q \right) 
    = w_{q+1} \cdot \div \left( R_q - \pi_q \Id \right) \notag\\
    &= \underbrace{(w_{q+1}-e_{q+1}) \cdot \div \left( \sum_{m=q}^{q+\bn-1} R_\ell^m - \sum_{m=q}^{q+\Npr-1}\pi_\ell^m \Id \right)}_{=\div\ov\phi_{L1}^{q+\bn}+\bmu_{L1}'}
    + \underbrace{e_{q+1} \cdot \div \left( \sum_{m=q}^{q+\bn-1} R_\ell^m - \sum_{m=q}^{q+\Npr-1}\pi_\ell^m \Id \right)}_{=\div\ov\phi_{L2}^{q+\bn}+\bmu_{L2}'} \notag\\
    &\qquad 
    + \underbrace{ w_{q+1} \cdot \div \left( \sum_{m=q}^{q+\bn-1} (R_q^m - R_\ell^m) - \sum_{m=q}^{q+\Npr-1}(\pi_q^m - \pi_\ell^m) \Id \right) }_{=\div\ov\phi_{L3}^{q+\bn}+\bmu_{L3}'} \, . \notag 
\end{align}
Notice that the tail of the sum $\sum_{q+\Npr}^\infty \pi_q^k$ of $\pi_q$ does not appear because of \eqref{defn:pikq.large.k}. The term $\ov\phi_{L1}^{q+\bn}$ is the main term and requires sharp estimates from the inverse divergence operator from Lemma~\ref{rem:no:decoup:inverse:div2}, while $\ov\phi_{L2}^{q+\bn}$ and $\ov\phi_{L3}^{q+\bn}$ may be estimated much more brutally. 
\smallskip

\noindent\texttt{Step 1: Estimating the main local term.} In order to estimate $\ov\phi_{L1}^{q+\bn}$, we need to upgrade the material derivative estimates on $w_{q+1}-e_{q+1}=\div^\dpot\upsilon_{q+1}$.  Towards this end, we claim that on $\supp \psi_{i,m}$ and for all $N\leq \sfrac{\Nfin}{4}-2\dpot^2$, $M\leq \sfrac{\Nfin}{5}$, and $q+1\leq m \leq \qbn-1$,
\begin{align}
\la_\qbn^{\dpot-k}\left|
D^N D_{t,m}^M \partial_{i_1}\cdots\partial_{i_k} \upsilon_{q+1}\right|
    \lec (\si_{\upsilon}^+  + \de_{q+3\bn})^{\sfrac12}r_{q}^{-1} (\la_{q+\bn}\Ga_{q+\bn})^N\MM{M,\Nindt,\tau_m^{-1}\Ga_m^{i-5},\Tau_q^{-1}\Ga_{q}^9} \, .
    \label{est.vel.inc.pot.by.pr.upgraded:new}
\end{align}
This estimate follows from dodging; more precisely, we appeal to \eqref{item:eckel:1} and \eqref{supp.upsilon.e.lem} from Lemma~\ref{lem:rep:vel:inc:potential}, 
and \eqref{item:dodging:1} from Lemma~\ref{lem:dodging} 
to assert that for all $m=q+1,\dots,q+\bn-1$, $\hat w_{m} \cdot \nabla \upsilon_{q+1}\equiv 0$. Then using \eqref{est.vel.inc.pot.by.pr} from Lemma~\ref{lem:pr.inc.vel.inc.pot} and applying \eqref{eq:inductive:partition} and \eqref{eq:inductive:timescales} concludes the proof. 

We now fix $i\leq \imax$ and $m =q+1, q+2,\dots,q+\bn-1$ (the cases $m=q$ and $m\geq q+\bn$ will require minor modifications) and apply Lemma~\ref{rem:no:decoup:inverse:div2} with the following choices:
\begin{align*}
    &G= \div \left( R_\ell^m - \pi_\ell^m \Id \right)^\bullet \, , \quad \varrho = (w_{q+1}-e_{q+1})^\bullet \, , \quad \vartheta = \upsilon_{q+1}^\bullet \, , \quad v=\hat u_{m-1} \, , \quad \lambda' = \Lambda_m \Ga_m \, ,  \\
    &\nu' = \Tau_{m-1}^{-1}\Ga_{m-1}^{12} \, , \quad \nu = \tau_{m-1}^{-1}\Ga_{m-1}^{i+23} \, , \quad N_* = \sfrac{\Nfin}{4}-2\dpot^2 \, , \quad M_* = \sfrac{\Nfin}{5} \, , \quad \dpot \textnormal{ as in \eqref{i:par:10}/\eqref{ineq:dpot:1}} \, ,  \\
    &\pi' = \left( \sigma_{\upsilon}^+ + \delta_{q+3\bn} \right)^{\sfrac 12}r_q^{-1} \, ,\quad \Omega= \supp \psi_{i,m-1} \, , \quad \pi = 2\Ga_m^3 \pi_\ell^m \La_m\Ga_m \, , \quad M_t = \Nindt \, , \\
    &\la = \La_m \Ga_m \, , \quad \Upsilon = \La = \la_{q+\bn}\Ga_\qbn \, , \quad M_\circ = N_\circ = 3\Nind \, , \quad K_\circ \textnormal{ as in \eqref{i:par:9.5}/\eqref{ineq:K_0}} \, . 
\end{align*}
Then we have that \eqref{eq:DDv} is satisfied by \eqref{eq:nasty:D:vq:old}, \eqref{eq:inv:div:extra:pointwise:noflow} is satisfied from \eqref{eq:pressure:inductive:dtq:pointwise:higher} and \eqref{eq:inductive:pointwise:upgraded:1:higher}, \eqref{eq:inv:div:extra:pointwise2:noflow} is satisfied from \eqref{est.vel.inc.pot.by.pr} and \eqref{est.vel.inc.pot.by.pr.upgraded:new}, and \eqref{parameter:noflow} is satisfied by definition and by \eqref{condi.Nfin0}. 

We then conclude from~\eqref{eq:inv:div:pointwise:local} that for all $N\leq \sfrac{\Nfin}{4}-2\dpot^2-\dpot$ and $M\leq \sfrac{\Nfin}{5}$,
\begin{align}
    &\left|\mathbf{1}_{\supp \psi_{i,m-1}} D^N D_{t,m-1}^M \divH\left( (w_{q+1}-e_{q+1})^\bullet \div \left( R_\ell^m - \pi_\ell^m \Id \right)^\bullet \right) \right| \notag\\
    & \qquad \les \Ga_m^3 \pi_\ell^m \La_m \Ga_m \left( \sigma_{\upsilon}^+ + \delta_{q+2\bn} \right)^{\sfrac 12}r_q^{-1} \la_{q+\bn}^{-1} (\la_\qbn \Ga_\qbn)^N \MM{M,\Nindt,\tau_{m-1}^{-1}\Ga_{m-1}^{i+4},\Tau_{m}^{-1}\Ga_{m-1}^2} \, . \label{eq:eckel:almost:ready}
\end{align}
From \eqref{eq:div:no:flow:support} and \eqref{supp.upsilon.e.lem}, we have that 
\begin{align}\label{eq:eckel:almost:ready:support}
    \supp \left( \divH\left( (w_{q+1}-e_{q+1})^\bullet \div \left( R_\ell^m - \pi_\ell^m \Id \right)^\bullet \right) \right) \subseteq \supp \upsilon_{q+1} \, ,
\end{align}
which leads to \eqref{eckel:support} from \eqref{supp.upsilon.e.lem} and \eqref{eq:dodging:oldies:prep}. Indeed, from \eqref{supp.upsilon.e.lem} and Lemma~\ref{lem:axis:control}, we have
\begin{align*}
\supp(\upsilon_{q+1}) 
&\subseteq \bigcup_{\xi,i,j,k,\vecl,I,\diamond} \supp \left(\chi_{i,k,q} \zeta_{q,\diamond,i,k,\xi,\vecl} \left(\rhob_{(\xi)}^\diamond \zetab_{\xi}^{I,\diamond} \right)\circ \Phiik \right) \cap
    B\left( \supp \varrho^I_{(\xi),\diamond}\circ\Phiik , {3} \lambda_\qbn^{-1} \right) \, .  
\end{align*}
In order to have effective dodging with $\hat w_{q'}$, $q+1\leq q'\leq q+\half$, we appeal to \eqref{eq:dodging:oldies:prep}.  Also, using \eqref{eq:eckel:almost:ready}--\eqref{eq:eckel:almost:ready:support} and appealing to a similar dodging and upgrading argument which produced the bound \eqref{est.vel.inc.pot.by.pr.upgraded:new}, we have that 
for all $N\leq \sfrac{\Nfin}{4}-2\dpot^2-\dpot$ and $M\leq \sfrac{\Nfin}{5}$,
\begin{align}
    &\left|\mathbf{1}_{\supp \psi_{i,\qbn-1}} D^N D_{t,\qbn-1}^M \divH\left( (w_{q+1}-e_{q+1})^\bullet \div \left( R_\ell^m - \pi_\ell^m \Id \right)^\bullet \right) \right| \notag\\
    &\qquad \les \Ga_m^3 \pi_\ell^m \La_m \Ga_m \left( \sigma_{\upsilon}^+ + \delta_{q+2\bn} \right)^{\sfrac 12}r_q^{-1} \la_{q+\bn}^{-1} (\la_\qbn\Ga_\qbn)^{N} \MM{M,\Nindt,\tau_{m-1}^{-1}\Ga_{m-1}^{i+4},\Tau_{m-1}^{-1}\Ga_{m-1}^2} \notag\\
    & \qquad \leq \Ga_\qbn^{-101} \pi_q^{q+\bn} \left( \sigma_{\upsilon}^+ + \delta_{q+2\bn} \right)^{\sfrac 12}r_q^{-1} (\la_\qbn\Ga_\qbn)^{N} \MM{M,\Nindt,\tau_{\qbn-1}^{-1}\Ga_{\qbn-1}^{i-5},\Tau_{\qbn-1}^{-1}\Ga_{\qbn-1}^2} \, , \label{eq:eckel:ready}
\end{align}
where we have used \eqref{eq:ind.pr.anticipated}, \eqref{ind:pi:lower}, and \eqref{la.beats.de} to conclude the last line. Finally, the estimates in \eqref{eq:sat:evening} follow from Remark~\ref{rem:est.mean}, and we omit the details now and in the rest of this proof.

In the case $m=q$, we make slight changes in the choices of $v$, $\pi$, $\Omega$, $\nu$, and $\nu'$ based on \eqref{eq:pressure:inductive:dtq:pointwise} and \eqref{eq:nasty:D:vq:old}. Then, applying the same reasoning as for the case $m=q+1$ for example, we find that in fact \eqref{eq:eckel:ready} holds for $m=q$.  In the cases $q+\bn\leq m<q+\Npr$, we set $G= \div \left(\pi_\ell^m\Id\right)^\bullet$, $v = \hat u_{q+\bn-1}$, and make suitable changes to the parameters and functions based on \eqref{eq:pressure:inductive:dtq:pointwise:higher:much} and \eqref{eq:nasty:D:vq:old} for $q'= q+\bn-1$.  Concluding as before, we find that
\begin{align}
    &\left|\mathbf{1}_{\supp \psi_{i,\qbn-1}} D^N D_{t,\qbn-1}^M \divH\left( (w_{q+1}-e_{q+1})\cdot \div \left(  \pi_\ell^m\Id \right) \right) \right| \notag\\
     &\qquad \les \Ga_m^3 \pi_\ell^m \La_{q+\bn-1} \Ga_{q+\bn-1}^2 \left( \sigma_{\upsilon}^+ + \delta_{q+2\bn} \right)^{\sfrac 12}r_q^{-1} \la_{q+\bn}^{-1} \notag\\ 
     &\qquad\qquad \times
     (\la_\qbn\Ga_\qbn)^{N} \MM{M,\Nindt,\tau_{q+\bn-1}^{-1}\Ga_{q+\bn-1}^{i+4},\Tau_{q+\bn-1}^{-1}\Ga_{q+\bn-1}^2} \notag\\
     &\qquad \leq \Ga_\qbn^{-101} \pi_q^{q+\bn} \left( \sigma_{\upsilon}^+ + \delta_{q+2\bn} \right)^{\sfrac 12}r_q^{-1}  
     (\la_\qbn\Ga_\qbn)^{N} \MM{M,\Nindt,\tau_{q+\bn-1}^{-1}\Ga_{q+\bn-1}^{i+4},\Tau_{q+\bn-1}^{-1}\Ga_{q+\bn-1}^2}
     \label{eq:eckel:ready2}
\end{align}
for all $N\leq \sfrac{\Nfin}{4}-2\dpot^2-\dpot$ and $M\leq \sfrac{\Nfin}{5}$. In the last inequality, we have used \eqref{ind:pi:lower}, \eqref{ind:pi:upper}, and \eqref{eq:ind.pr.anticipated} to write that $\pi_\ell^m \Lambda_{q+\bn-1}\la_{q+\bn}^{-1} \leq 2\pi_q^{q+\bn-1} \Lambda_{q+\bn-1}\la_{q+\bn}^{-1}\leq \pi_q^{q+\bn} \Gamma_{q+\bn}^{-150}$.

We can now set
$$ \ov\phi_L^{q+\bn,l} =  \divH\left( (w_{q+1}-e_{q+1})\cdot \div \left(\sum_{m=q}^{\qbn-1} R_\ell^m - \sum_{m=q}^{\qbn+\Npr-1}\pi_\ell^m \Id \right)^\bullet \right) \, , $$
which is well-defined over various values of $i$ since the algorithm used to define $\divH$ is independent of the value of $i$. Summing the estimate in \eqref{eq:eckel:ready}--\eqref{eq:eckel:ready2} over the various values of $m$ and using Cauchy-Schwarz, \eqref{low.bdd.pi}, and \eqref{condi.Nfin0} gives \eqref{eckel:local:est}.
\smallskip

\noindent\texttt{Step 2: Estimating the main nonlocal term and remainder terms.}  Due to the negligible quality of these error terms, we shall omit most of the details in this portion of the proof.  We first finish the application of Lemma~\ref{rem:no:decoup:inverse:div2} to the main terms by setting up the nonlocal assumptions and output in Part 3. In the case of $q+1\leq m\leq q+\bn-1$, we 
have that for $N,M\leq 3\Nind$, 
\begin{subequations}\label{eckel:nonlocal:prep:1}
\begin{align}\label{eckel:nonlocal:prep:1.1}
   \left\| D^N D_{t,m-1}^M \divR \left( (w_{q+1}-e_{q+1})^\bullet \div (R_\ell^m - \pi_\ell^m \Id)^\bullet \right) \right\|_{\infty} &\leq 
   \Tau_{\qbn}^{2\Nindt} \delta_{q+3\bn}^{\sfrac 32} 
   (\la_\qbn\Ga_\qbn)^N \Tau_{m-1}^{-M} \, .
\end{align}
In the remaining cases, we have that for $N,M\leq 3\Nind$,
\begin{align}\label{eckel:nonlocal:prep:1.25}
   \left\| D^N D_{t,q}^M \divR \left( (w_{q+1}-e_{q+1})^\bullet \div (R_\ell - \pi_\ell \Id)^\bullet \right) \right\|_{\infty} &\leq 
   \Tau_{\qbn}^{2\Nindt} \delta_{q+3\bn}^{\sfrac 32} 
   (\la_\qbn\Ga_\qbn)^N \Tau_{q}^{-M} \\
\label{eckel:nonlocal:prep:1.5}
   \left\| D^N D_{t,q+\bn-1}^M \divR \left( (w_{q+1}-e_{q+1})^\bullet \div (\pi_\ell^m \Id)^\bullet \right) \right\|_{\infty} &\leq 
   \Tau_{\qbn}^{2\Nindt} \delta_{q+3\bn}^{\sfrac 32} 
   (\la_\qbn\Ga_\qbn)^N
   \Tau_{q+\bn-1}^{-M}
\end{align}
\end{subequations}
for $q+\bn\leq m <q+\Npr$. We will upgrade the material derivatives at the end of \texttt{Step 2}. 

Next, we must treat the second error identified in \texttt{Step 0}, namely the remainder term which includes $e_{q+1}$. This is easily handled using Lemma~\ref{rem:no:decoup:inverse:div2}. In the cases $q+1\leq m \leq q+\bn-1$, we 
\eqref{eq:inverse:div:error:stress:bound:no:flow} gives that
\begin{subequations}\label{eckel:nonlocal:prep:2}
\begin{align}
    \left\| D^N D_{t,m-1}^M \divR \left( \div \left( R_\ell^m - \pi_\ell^m \Id \right)^\bullet e_{q+1}^\bullet \right) \right\|_\infty &\les \Tau_\qbn^{5\Nindt} \delta_{q+3\bn}^3 (\Tau_{m-1}^{-1}\Ga_{m-1}^2)^{M} (\la_\qbn\Ga_\qbn)^N 
    \label{eckel:nonlocal:prep:2.1}
\end{align}
for $N,M\leq 3\Nind$. Similarly, we have that for $N,M\leq 3\Nind$ and $q+\bn\leq m <q+\Npr$, 
\begin{align}
    &\left\| D^N D_{t,q}^M \divR \left( \div \left( R_\ell - \pi_\ell \Id \right)^\bullet e_{q+1}^\bullet \right) \right\|_\infty 
    \les \Tau_\qbn^{5\Nindt} \delta_{q+3\bn}^3 (\la_\qbn\Ga_\qbn)^N
    \MM{M, \Nindt, \tau_q^{-1},\Tau_{q}^{-1}\Ga_{q}^2 }
     \label{eckel:nonlocal:prep:2.25}\\
    & \left\| D^N D_{t,q+\bn-1}^M \divR \left( \div \left(  \pi_\ell^m \Id \right)^\bullet e_{q+1}^\bullet \right) \right\|_\infty \notag\\
    &\qquad\les \Tau_\qbn^{4\Nindt} \delta_{q+3\bn}^3 (\la_\qbn\Ga_\qbn)^N
     \MM{M, \Nindt, \tau_{q+\bn-1}^{-1}, \Tau_{q+\bn-1}^{-1}\Ga_{q+\bn-1}^2}
    \label{eckel:nonlocal:prep:2.5}
\end{align}    
\end{subequations}

Finally, we must treat the third error identified in \texttt{Step 0}, namely the remainder term which includes the differences between mollified and inductive stresses and pressures. This again is a simple application of Lemma~\ref{rem:no:decoup:inverse:div2}. In the case of $q\leq m\leq q+\bn-1$, we
have from \eqref{eq:inverse:div:error:stress:bound:no:flow} that for $N,M\leq 3\Nind$,
\begin{subequations}\label{eckel:nonlocal:prep:3}
\begin{align}
    &\left\| D^N D_{t,m-1}^M \divR \left( \div \left( (R_q^m - R_\ell^m) - (\pi_q^m-\pi_\ell^m) \Id \right)^\bullet w_{q+1}^\bullet \right) \right\|_\infty \notag\\ 
    &\qquad\leq \Tau_\qbn^{3\Nindt+1} \delta_{q+3\bn}^3 (\Tau_{m-1}^{-1}\Ga_{m-1}^2)^{M} (\la_\qbn\Ga_\qbn)^N 
    \label{eckel:nonlocal:prep:3.1}
\end{align}
In the remaining cases, we have
\begin{align}
    &\left\| D^N D_{t,q}^M \divR \left( \div \left( (R_q^q - R_\ell) - (\pi_q^q-\pi_\ell) \Id \right)\cdot w_{q+1} \right) \right\|_\infty \notag\\ 
    &\qquad\leq \Tau_\qbn^{2\Nindt+1} \delta_{q+3\bn}^3  (\la_\qbn\Ga_\qbn)^N \MM{M,\Nindt,\tau_{q}^{-1},\Tau_{q}^{-1}\Ga_{q}^2}\label{eckel:nonlocal:prep:3.25} \\
    &\left\| D^N D_{t,q+\bn-1}^M \divR \left( \div \left( (\pi_q^m-\pi_\ell^m) \Id \right) w_{q+1} \right) \right\|_\infty \notag\\ 
    &\qquad\leq \Tau_\qbn^{2\Nindt+1} \delta_{q+3\bn}^3 (\la_\qbn\Ga_\qbn)^N\MM{M,\Nindt,\tau_{q+\bn-1}^{-1},\Tau_{q+\bn-1}^{-1}\Ga_{q+\bn-1}^2}  
    \label{eckel:nonlocal:prep:3.5}
\end{align}
for $N, M\leq 3\Nind$ and $q+\bn\leq m< q+\Npr$.
\end{subequations}

We can now set
\begin{align*}
    \ov\phi_L^{q+\bn,*} &= \divR \left[ (w_{q+1}-e_{q+1})\cdot \div \left[\sum_{m=q}^{\qbn-1}  R_\ell^m - \sum_{m=q}^{\qbn+\Npr-1}  \pi_\ell^m \Id\right] \right] \\
    &\qquad \qquad + \divR \left[\left[\sum_{m=q}^{\qbn-1}  R_\ell^m - \sum_{m=q}^{\qbn+\Npr-1}  \pi_\ell^m \Id\right]\cdot e_{q+1} \right] \\
    &\qquad \qquad \qquad  + \divR \left[ \div \left[ \sum_{m=q}^{\qbn-1}(R_q^m - R_\ell^m) - \sum_{m=q}^{\qbn+\Npr-1}(\pi_q^m-\pi_\ell^m) \Id \right]\cdot w_{q+1} \right] \, .
\end{align*}
We must now upgrade the material derivatives in the estimates \eqref{eckel:nonlocal:prep:1.1}, \eqref{eckel:nonlocal:prep:2.1}, \eqref{eckel:nonlocal:prep:3.1} in order to match the bound in \eqref{eckel:nonlocal:est}.  Specifically, we apply Lemma~\ref{rem:upgrade.material.derivative.end} with $p=\infty$, $N_x=N_t=\infty$, $N_*=\sfrac{\Nind}{4}$, $\Omega=\T^3\times\mathbb{R}$, $v=\hat u_{m-1}$, $w=\hat u_{\qbn-1} - \hat u_{m-1}$, and parameter choices according to \eqref{eq:nasty:D:wq:old}, which verifies \eqref{eq:cooper:w}, parameter choices according to \eqref{eq:nasty:D:vq:old}, which verifies \eqref{eq:cooper:2:v:0}, and parameter choices according to \eqref{eckel:nonlocal:prep:1}--\eqref{eckel:nonlocal:prep:3}, which verify \eqref{eq:cooper:2:f:0}.  We however emphasize the choice of $\const_f=\Tau_\qbn^{\Nindt+1}\delta_{q+3\bn}^3$, which can be used to absorb lossy material derivative estimates. We then have from \eqref{eq:cooper:f:mat} that \eqref{eckel:nonlocal:est} holds, concluding the proof.
\end{proof}

\subsection{Stress current error}\label{sec:stress:current}
From subsection~\ref{ss:rle:new}, the stress current error is given by 
$$ -\ov\phi_R =(\hat u_{q+1}- \hat u_q) \overline \ka_{q+1}
+ \left( \ov R_{q+1} - (\pi_q-\pi_q^q)\Id\right) (\hat u_{q+1}- \hat u_q) \, .$$ 
Next, we recall from \eqref{eq:hat:no:hat} that $\hat u_{q+1}- \hat u_q = \hat w_{q+1}$, and that from \eqref{eq:def:ov:kappa},
$$\overline \ka_{q+1} = \frac12 \tr (-\pi_q \Id + \pi_q^q \Id + \ov R_{q+1}) = -\frac32 (\pi_q - \pi_q^q) + \frac12 \tr \ov R_{q+1}\, .$$

\begin{lemma}[\bf Definition and basic estimates]\label{lem:current:stressss}
The current error $\ov\phi_R$ satisfies the following.
\begin{enumerate}[(i)]
    \item We have the decomposition
    \begin{equation}\label{eq:stress:current:decomp}
       \ov\phi_R = \ov\phi_{R}^{q+1,l} + \sum_{m=q+1}^{\qbn} \ov\phi_{R}^{m,*} \, . 
    \end{equation}
    \item  For all $N,M$ such that $N+M \leq \sfrac{\Nind}{4}$ and $q+1 \leq m \leq  \qbn$, we have that
    \begin{subequations}
    \begin{align}
     \left|\psi_{i,q} D^N D_{t,q}^M \ov\phi_R^{q+1,l}\right| &\leq \Ga_{q+1}^{-99} \left(\pi_q^{q+1} \right)^{\sfrac32} r_{q+1}^{-1} \Lambda_{q+1}^N \MM{M,\Nindt,\Gamma_{q}^{i+20} \tau_q^{-1}, \Gamma_{q}^{10}\Tau_q^{-1} } \label{eq:sc:loc:est} \\
    \norm{ D^N D_{t,m-1}^M \ov\phi_R^{m,*} }_{L^\infty} &\leq \delta_{q+3\bn}^2 \La_m^N \MM{M,\Nindt,\tau_{m-1}^{-1},\Tau_{m-1}^{-1}} \, . \label{eq:sc:nonloc:est}
    \end{align}
    \end{subequations}
    
    \item The local part $\ov\phi_R^{q+1,l}$ has the support property, 
    \begin{align}\label{eq:sc:loc:supp}
        B\left( \supp \hat w_{q}, \lambda_{q}^{-1} \Gamma_{q+1} \right) \cap \supp \ov \phi^{q+1,l}_O &= \emptyset\, . 
    \end{align}
\end{enumerate}
\end{lemma}
\begin{proof}
Recalling \eqref{eq:ER:decomp:basic}, \eqref{eq:Rnnl:inductive:dtq}, \eqref{defn:newstress}, and \eqref{defn:primitive.stress}, we define
\begin{align}
    -\ov\phi_R^{q+1,l} &= \hat w_{q+1} \left[ -\frac 32 (\pi_q-\pi_q^q) + \sum_{m=q+1}^\qbn \frac 12 \tr \ov R_{q+1}^{m,l} \right] + \left[ -(\pi_q - \pi_q^q)\Id + \sum_{m=q+1}^\qbn \ov R_{q+1}^{m,l} \right] \hat w_{q+1} \notag\\
    &= \hat w_{q+1} \left[ -\frac 32 (\pi_q-\pi_q^q) + \sum_{m=q+1}^\qbn \frac 12 \tr \left( R_q^{m,l} + S_{q+1}^{m,l} \right) \right] + \left[ -(\pi_q-\pi_q^q)\Id + \sum_{m=q+1}^\qbn R_q^{m,l} + S_{q+1}^{m,l} \right] \hat w_{q+1} \label{eq:local:sc:def} \\
    -\ov\phi_R^{m,*} &= \hat w_{q+1} \frac12 \tr \ov R_{q+1}^{m,*} + \ov R_{q+1}^{m,*} \hat w_{q+1}
    \qquad 
    \textnormal{ for $q+1\leq m \leq \qbn$} \, .  \label{eq:nonlocal:sc:def}
\end{align}
In order to prove \eqref{eq:sc:loc:est} for $\ov\phi_R^{q+1,l}$, it suffices to prove the estimate for the second term from the second line of \eqref{eq:local:sc:def}, as it is clear that the first term will obey identical estimates. We first consider the term with the stresses, before handling the term with the pressures next. The crucial first step is to employ \emph{dodging} to eliminate most of the terms from \eqref{eq:local:sc:def}. Specifically, we have from \eqref{eq:ER:decomp:basic} (which gives that $R_q^{\qbn,l}\equiv 0$) and \eqref{eq:ind:stressdodging:equiv} that 
$$  \left( \sum_{m=q+1}^\qbn R_q^{m,l} \right) \hat w_{q+1} = R_q^{q+1,l} \hat w_{q+1} \, . $$
Therefore, we have from \eqref{eq:ind:stress:by:pi}, \eqref{eq:Rnnl:inductive:dtq} and \eqref{eq:ind:velocity:by:pi} that for $N+M\leq 2\Nind$, 
\begin{align*}
     \left|\psi_{i,q} D^N D_{t,q}^M R_{q}^{q+1,l} \hat w_{q+1}\right| &\leq 
     \sum_{N_1=0}^N \sum_{M_1=0}^M 
     \sum_{i' = i-1}^{i+1}
     \left|\psi_{i,q} D^{N_1} D_{t,q}^{M_1}R_{q+1}^{q+1,l}\right| \left|\psi_{i',q} D^{N-N_1} D_{t,q}^{M-M_1}\hat w_{q+1}\right|\\
     &\leq \pi_q^{q+1} (\pi_q^{q+1})^{\sfrac12} r_{q-\bn+1}^{-1} \Lambda_{q+1}^N \MM{M,\Nindt,\Gamma_{q}^{i+20} \tau_q^{-1}, \Tau_q^{-1}\Ga_q^{10}}\\
     &\leq \Ga_{q+1}^{-101} \left(\pi_q^{q+1}\right)^{\sfrac32} r_{q+1}^{-1} \Lambda_{q+1}^N \MM{M,\Nindt,\Gamma_{q}^{i+20} \tau_q^{-1}, \Tau_q^{-1}\Ga_q^{10}} \, ,
\end{align*}
where we have used \eqref{ineq:r's:eat:Gammas} to achieve the final inequality.

In order to prove a similar estimate for the term with pressures, we appeal to \eqref{eq:pi:decomp:basic} and \eqref{eq:pinl:inductive:dtq} to write that for $N+M \leq 2\Nind$,
\begin{align}
    \left| \psi_{i,q} D^N \Dtq^M \left[ (\pi-\pi_q^q) \hat w_{q+1} \right] \right| &= \left|\psi_{i,q} D^N \Dtq^M \left[ \left(\sum_{k=q+1}^{\infty } \pi_q^k \right) \hat w_{q+1} \right] \right| \notag\\
    &\les \Ga_q\Ga_{q+1}^{-100} (\pi_q^{q+1})^{\sfrac 32} r_{q+1}^{-1} \La_{q+1}^N \MM{M,\Nindt, \Ga_q^{i+20}\tau_q^{-1},\Ga_q^{10}\Tau_q^{-1}} \, . \notag 
\end{align}
Combined with the previous estimate, this concludes the proof of \eqref{eq:sc:loc:est} for terms from \eqref{eq:local:sc:def} which involve stresses $R_q^{m,l}$ and pressure. 

In order to prove \eqref{eq:sc:loc:est} for terms from \eqref{eq:local:sc:def} which involve stresses $S_{q+1}^{m,l}$ defined in \eqref{defn:newstress}, we again employ the dodging results from \eqref{supp.si.m+.stress} to write that
$$  \sum_{m=q+1}^\qbn S_{q+1}^{m,l} \hat w_{q+1} = S_{q+1}^{q+1,l} \hat w_{q+1} \, . $$
Then from \eqref{eq:lo:upgrade:2} and \eqref{eq:ind:velocity:by:pi}, we have that for $N+M\leq 2\Nind$, 
\begin{align*}
     \left|\psi_{i,q} D^N D_{t,q}^M S_{q+1}^{q+1,l} \hat w_{q+1}\right| &\leq 
     \sum_{N_1=0}^N \sum_{M_1=0}^M \sum_{i'=i-1}^{i+1}
     \left|\psi_{i,q} D^{N_1} D_{t,q}^{M_1} S_{q+1}^{q+1,l} \right| \left|\psi_{i',q} D^{N-N_1} D_{t,q}^{M-M_1}\hat w_{q+1}\right|\\
     &\les \Ga_{q+1}^{-50} \pi_q^{q+1} (\pi_q^{q+1})^{\sfrac12} r_{q-\bn+1}^{-1} \Lambda_{q+1}^N \MM{M,\Nindt,\Gamma_{q}^{i+20} \tau_q^{-1}, \Tau_q^{-1}\Ga_q^{10}}\\
     &\leq \Ga_{q+1}^{-101} \left(\pi_q^{q+1}\right)^{\sfrac32} r_{q+1}^{-1} \Lambda_{q+1}^N \MM{M,\Nindt,\Gamma_{q}^{i+20} \tau_q^{-1}, \Tau_q^{-1}\Ga_q^{10}} \, ,
\end{align*}
concluding the proof of \eqref{eq:sc:loc:est}.

Lastly, the nonlocal estimate \eqref{eq:sc:nonloc:est} follows immediately from \eqref{eq:ind:velocity:by:pi}, \eqref{eq:Rnnl:inductive:dtq}, \eqref{eq:nlstress:upgraded}, and immediate computation. We omit further details. Also, the support property \eqref{eq:sc:loc:supp} can be easily obtained from the definition \eqref{eq:local:sc:def} of $\ov\phi_R^{q+1,l}$ and Hypothesis \ref{hyp:dodging1}. 
\end{proof}

\subsection{Divergence correctors}\label{ss:dce:rle}

\newcommand{\werc}{w_{q+1,R}^{(c)}}
\newcommand{\wercdiamond}{w_{q+1,\diamond}^{(c)}}
\newcommand{\werpdiamond}{w_{q+1,\diamond}^{(p)}}
\newcommand{\werp}{w_{q+1,R}^{(p)}}
\newcommand{\wphic}{w_{q+1,\varphi}^{(c)}}
\newcommand{\wphip}{w_{q+1,\varphi}^{(p)}}

\begin{lemma}[\bf Divergence corrector current error and the associated pressure increment]\label{lem:corrector.current:general:estimate}

There exist current errors $\ov \phi_C^{m}=\ov \phi_C^{m,l}+ \ov \phi_C^{m,*}$ and pressure increments $\sigma_{\ov \phi_C^{m}}=\sigma_{\ov \phi_C^{m}}^+ -\sigma_{\ov \phi_C^{m}}^-$ for $m=q+\half+1,\dots,q+\bn$ such that the following hold.
\begin{enumerate}[(i)]
\item\label{item:cor:c:1} We have the equality
$$ \div\left( \frac{1}{2} |w_{q+1}^{(c)}|^2 w_{q+1}^{(p)} +( w_{q+1}^{(c)} \cdot w_{q+1}^{(p)} ) w_{q+1}^{(p)}  +  \frac{1}{2} |w_{q+1}|^2 w_{q+1}^{(c)}  \right) = \sum_{m=q+\half+1}^{q+\bn} \div \phi_C^{m} \, . $$ 
\item For all $N,M \leq 2\Nind$, we have that
\begin{subequations}\label{eq:cc.p}
\begin{align}
\label{eq:cc.p.1}
\left|\psi_{i,q} D^N \Dtq^M \ov\phi^{m,l}_C\right| &\les(\si_{\ov \phi^{m}_C}^+ + \de_{q+3\bn} )^{\sfrac 32}r_m^{-1}  \left(\lambda_m\Gamma_q\right)^N \MM{M,\Nindt,\tau_q^{-1}\Gamma_{q}^{i+15},\Tau_q^{-1}\Ga_q^9}  \\
\label{eq:cc.p.1.5}
\left|\psi_{i,q} D^N \Dtq^M \ov\phi^{q+\bn}_C\right| &\les
(\si_{\ov \phi^{q+\bn,l}_C}^+
+\si_{\upsilon}^+
+ \de_{q+3\bn} )^{\sfrac 32}r_m^{-1}
 \left(\lambda_{q+\bn}\Gamma_q\right)^N \MM{M,\Nindt,\tau_q^{-1}\Gamma_{q}^{i+15},\Tau_q^{-1}\Ga_q^9}\, ,  
\end{align}
where the first estimate holds for $m=q+\half+1,\dots, q+\bn-1$, and $\si_{\upsilon}^+$ is defined as in Lemma \ref{lem:pr.inc.vel.inc.pot} in the second estimate. In addition, for all $m=q+\half+1,\dots,q+\bn$ and $N,M\leq {\sfrac{\Nfin}{200}}$, we have that
\begin{align}
\label{eq:cc.p.2}
\left|\psi_{i,q} D^N \Dtq^M \si_{\ov\phi^{m}_C}^+\right| &\les  \left(\si_{\ov\phi^{m}_C}^+ +\de^2_{q+3\bn}\right) \left(\lambda_{m}\Gamma_q\right)^N \MM{M,\Nindt,\tau_q^{-1}\Gamma_{q}^{i+16},\Tau_q^{-1}\Ga_q^9}\\
\label{eq:cc.p.3}
\norm{\psi_{i,q} D^N \Dtq^M \si_{\ov\phi^{m}_C}^+}_{\sfrac32} &< \de_{m+\bn} \Gamma_{m}^{-9} \left(\lambda_{m}\Gamma_q\right)^N \MM{M,\Nindt,\tau_q^{-1}\Gamma_{q}^{i+16},\Tau_q^{-1}\Ga_q^9}\\
 \label{eq:cc.p.3.5}
\norm{\psi_{i,q} D^N \Dtq^M \si_{\ov\phi^{m}_C}^+}_{\infty} &< \Ga_{m}^{\badshaq-9} \left(\lambda_{m}\Gamma_q\right)^N \MM{M,\Nindt,\tau_q^{-1}\Gamma_{q}^{i+16},\Tau_q^{-1}\Ga_q^9} \\
\label{eq:cc.p.4}
\left|\psi_{i,q} D^N \Dtq^M \si_{\ov\phi^{m}_C}^-\right| &< 
\pi_{q}^{q+\half}
\left(\lambda_{q+\floor{\bn/2}}\Gamma_q\right)^N \MM{M,\Nindt,\tau_q^{-1}\Gamma_{q}^{i+16},\Tau_q^{-1}\Ga_q^9}
\end{align}
\end{subequations}
Finally, we have that for all $m = q+\half+1,\dots,\qbn$,
\begin{subequations}\label{eq:cc.p.6}
    \begin{align}
    \label{eq:cc.p.6.1}
        B\left( \supp \hat w_{q'}, \lambda_{q'}^{-1} \Gamma_{q'+1} \right) \cap
\supp \ov \phi^{ m,l}_C  &= \emptyset\qquad \forall q+1\leq q' \leq  m-1\\
\label{eq:cc.p.6.2}
B\left( \supp \hat w_{q'}, \lambda_{q'}^{-1} \Gamma_{q'+1} \right) \cap \supp (\si_{\ov \phi^{m}_C}^+) &= \emptyset \qquad \forall q+1\leq q' \leq  m-1\\
\label{eq:cc.p.6.3}
B\left( \supp \hat w_{q'}, \lambda_{q+1}^{-1} \Gamma_{q}^2 \right) \cap \supp (\si_{\ov \phi^{m}_C}^-) &= \emptyset \qquad \forall q+1\leq q' \leq  q+\half \, .
    \end{align}
\end{subequations}
\item For all $m=q+\half + 1,\dots,q+\bn$ and $N,M \leq 2\Nind$, the non-local part $\bar\phi^{m,*}_C$ satisfies
\begin{align}
\left\| D^N D_{t, q}^M \bar\phi^{m,*}_C \right\|_{L^\infty}
    &\leq \Tau_\qbn^{2\Nindt}\delta_{q+3\bn}^{\sfrac32}\la_{m}^{N}\tau_{q}^{-M}\, ,
    \label{eq:cor.current:estimate:1}
\end{align}
\end{enumerate}
\end{lemma}

\begin{lemma}[\bf Pressure current]\label{lem:currentdivergencecorrector:pressure:current}
For every $m'\in\{q+\half+1,\dots,q+\bn\}$, there exists a current error 
$\phi_{\ov \phi_C^{m'}}$ associated to the pressure increment $\si_{\ov \phi_C^{m'}}$ in the sense of
\begin{align}
     \div \phi_{{\ov\phi_C^{m'}}} &= \Dtq \sigma_{\ov\phi_C^{m'}} - \int_{\T^3} \Dtq \sigma_{\ov\phi_C^{m'}}(t,x') \,dx' \, .
\end{align}
The current error $\phi_{\ov\phi_C^{m'}}$ has a decomposition 
$$\phi_{{\ov\phi_C^{m'}}} = \phi_{\ov\phi_C^{m'}}^* +  \sum_{m=q+\half+1}^{{m'}} \phi_{{\ov\phi_C^{m'}}}^{m}
= \phi_{\ov\phi_C^{m'}}^* +  \sum_{m=q+\half+1}^{{m'}} \phi_{{\ov\phi_C^{m'}}}^{m,l}+ \phi_{{\ov\phi_C^{m'}}}^{m,*} \, , $$
where the local parts $\phi_{{\ov\phi_C^{m'}}}^{m,l}$ and the nonlocal parts $\phi_{{\ov\phi_C^{m'}}}^{m,*}$ and $\phi_{\ov\phi_C^{m'}}^*$ satisfy the following properties.
\begin{enumerate}[(i)]
\item\label{i:pc:3:dc} For $q+\half+1 \leq m \leq m'$ and $N,M\leq  2\Nind$,
    \begin{subequations}
    \begin{align}
        &\left|\psi_{i,q} D^N \Dtq^M \phi_{\ov \phi_C^{m'}}^{m,l} \right| < \Ga_{m}^{-100} \left(\pi_q^m\right)^{\sfrac 32} r_m^{-1} (\la_m \Ga_m^2)^M \MM{M,\Nindt,\tau_q^{-1}\Ga_q^{i+17},\Tau_q^{-1}\Ga_q^9} \label{eq:desert:estimate:1:dc} \\
        &\left\| D^N \Dtq^M \phi_{\ov \phi_C^{m'}}^{m,*} \right\|_\infty  \left\| D^N \Dtq^M \phi_{\ov\phi_C^{m'}}^* +  \right\|_{\infty} 
        < \Tau_\qbn^{2\Nindt} \delta_{q+3\bn}^{\sfrac 32} (\la_m\Ga_m^2)^N \tau_q^{-M} \label{eq:desert:estimate:2:dc} \, .
    \end{align}
    \end{subequations}
    \item\label{i:pc:4:dc} For all $q+\half+1\leq m \leq m'$ and all $q+1\leq q' \leq m-1$, 
    \begin{align}
        B\left( \supp \hat w_{q'}, \sfrac 12 \lambda_{q'}^{-1} \Ga_{q'+1} \right) \cap \supp  \phi^{m,l}_{\ov \phi_C^{m'}}  = \emptyset \label{eq:desert:dodging:dc} \, .
    \end{align}
     \item \label{i:pc:5:dc} For $M\leq 2\Nind$, the mean part $\langle D_{t,q}\si_{{\ov\phi_C^{m'}}}\rangle$ satisfies
    \begin{align}\label{eq:desert:mean:dc}
        \left|\frac{d^M}{dt^M}\langle D_{t,q}\si_{{\ov\phi_C^{m'}}}\rangle\right| 
        \leq (\max(1, T))^{-1}\delta_{q+3\bn} \MM{M, \Nindt, \tau_q^{-1}, \Tau_q^{-1}\Ga_q^9} \, .
    \end{align}
\end{enumerate}
\end{lemma}

\begin{proof}[Proof of Lemma \ref{lem:corrector.current:general:estimate}--\ref{lem:currentdivergencecorrector:pressure:current}]

\noindent\texttt{Step 1: Analyze the error}. We first decompose 
\begin{align}
    \frac{1}{2} |w_{q+1}^{(c)}|^2 w_{q+1}^{(p)} &+ ( w_{q+1}^{(c)} \cdot w_{q+1}^{(p)} ) w_{q+1}^{(p)}  +  \frac{1}{2} |w_{q+1}|^2 w_{q+1}^{(c)} \notag\\
    &= \frac{1}{2} |w_{q+1}^{(c)}|^2 w_{q+1}^{(p)} + \frac 12 w_{q+1}^{(c)} |w_{q+1}^{(c)}|^2 + \left(w_{q+1}^{(c)}\cdot w_{q+1}^{(p)}\right) w_{q+1}^{(c)} \label{eq:dc:current:case:1} \\
    &\quad + ( w_{q+1}^{(c)} \cdot w_{q+1}^{(p)} ) w_{q+1}^{(p)} + \frac 12 |w_{q+1}^{(p)}|^2 w_{q+1}^{(c)} \, . \label{eq:dc:current:case:2}
\end{align}
The first set \eqref{eq:dc:current:case:1} of terms is simpler because each term has two divergence correctors and thus will be absorbed directly into $\ov\phi^{q+\bn,l}_C$. The second set \eqref{eq:dc:current:case:2} is more delicate, so we now rewrite this term using a few algebraic identities similar to the divergence corrector error terms in the Euler-Reynolds system ({cf.~\cite[Lemma~8.10]{GKN23}}). 

Taking the divergence operator to the first term in \eqref{eq:dc:current:case:2} and using $\UU_{\pxi,\diamond}^{I,s}$ to denote the $s$ component of the vector field $\UU_{\pxi,\diamond}^I$ (the potential for $\WW_{\pxi,\diamond}^{I,s}$), we have that
\begin{align}
    &\div\left( \werpdiamond (\wercdiamond \cdot \werpdiamond) \right)\nonumber\\
    &=\xi^\ell A_\ell^m \partial_m \biggl{(} \xi^\theta A_\theta^n a_{(\xi),\diamond}^2 \left( \rhob_{(\xi)}^\diamond \zetab_\xi^{I,\diamond} \varrho_{(\xi),\diamond}^I \right)^2\circ\Phiik   \epsilon_{npr} \partial_p \left( a_{(\xi),\diamond} \left( \rhob_{(\xi)}^\diamond \zetab_\xi^{I,\diamond} \right)\circ\Phiik \right) \partial_r \Phiik^s \UU_{(\xi),\diamond}^{I,s}\circ\Phiik
  \biggr{)}  \nonumber\\
  &=: (\mathbf{C}_0)_{(\xi),\diamond}^{I} \left((
  \varrho_{(\xi),\diamond}^I )^2\UU_{(\xi),\diamond}^{I,s}\right)\circ\Phiik \, . \label{dc:curr:mess:0} 
\end{align} 
As we see in the second and the third line, for the time being we omit the summation over the indices $\diamond,i,j,k,\xi,\vecl,I$ for convenience, until we reintroduce them. In the first equality, observe that we have commuted $\xi^\ell A_\ell^m$ with $\partial_m$ so that we see the good differential operator $\xi^\ell A_\ell^m \partial_m$, which can only cost $\Lambda_q\Gamma_q^{13}$ from Lemma~\ref{lem:a_master_est_p}. This is because it can never land on a high-frequency object (any of $\rhob_{(\xi)}^\diamond$, $\zetab_\xi^{I,\diamond},\varrho_{(\xi),\diamond}^I,\UU_{(\xi),\diamond}^I$). In particular, this term can be written in the form appearing in the second equality. We will treat this term similar to the oscillation current error.

We now introduce the following notation:
\begin{align*}
    a_{(\xi),\diamond}^{p, \rm good} &:= \partial_p \Phiik^n \xi^n  \xi^{\ell} A_\ell^j \partial_j \left( a_{(\xi),\diamond} \left( \rhob_{(\xi)}^\diamond \zetab_\xi^{I,\diamond} \right)\circ \Phiik\right) \\     
    a_{(\xi),\diamond}^{p, \rm bad} &:= \partial_p \Phiik^n (\xi^\prime)^n (\xi^\prime)^{\ell} A_\ell^j \partial_j \left( a_{(\xi),\diamond} \left( \rhob_{(\xi)}^\diamond \zetab_\xi^{I,\diamond} \right)\circ \Phiik\right) + \partial_p \Phiik^n (\xi^{\prime \prime})^n (\xi^{\prime \prime})^\ell A_\ell^j \partial_j \left( a_{(\xi),\diamond} \left( \rhob_{(\xi)}^\diamond \zetab_\xi^{I,\diamond} \right)\circ \Phiik\right)
     \,. \notag
\end{align*}

Next, we write the divergence of the second term in \eqref{eq:dc:current:case:2} as
\begin{align*}
    &\div\left( \wercdiamond |\werpdiamond|^2 \right) \notag\\
    &\quad = \left( \epsilon_{m p r}   \left( a_{(\xi),\diamond}^{p,\rm bad} + a_{(\xi),\diamond}^{p,\rm good} \right) \partial_r \Phi^s (\mathbb{U}_{(\xi),\diamond}^I)^s \circ \Phiik \,  a^2_{(\xi),\diamond} (\rhob_{(\xi)}^\diamond\zetab_\xi^{I,\diamond}\varrho_{(\xi),\diamond}^I)^2\circ \Phiik \xi^\ell   A_{\ell}^j \xi^n A_n^j  \right)\\
    &\quad =: \mathbf{V}_3 + 
    \mathbf{V}_4
    \, ,
\end{align*}
The term inside of the divergence in $\mathbf{V}_4$ enjoys properties identical to the terms in \eqref{eq:dc:current:case:1}; indeed, the good differential operator in $a_{(\xi),\diamond}^{p,\rm good}$ only costs $\Lambda_q\Gamma_q^{13}$, 
and so 
we absorb these terms into $\phi_C^{q+\bn,l}$.  
On the other hand, we deal with the term $\mathbf{V}_3$ using \eqref{eq:UU:explicit} to expand
\begin{align}
    \mathbf{V}_3& = \partial_m \biggl{[} \epsilon_{mpr} \biggl{(} \partial_p \Phiik^n (\xi^\prime)^n (\xi^\prime)^{\ell} A_\ell^j \partial_j \left( a_{(\xi),\diamond} \left( \rhob_{(\xi)}^\diamond \zetab_\xi^{I,\diamond} \right)\circ \Phiik\right) \notag\\
    &\qquad\qquad + \partial_p \Phiik^n (\xi^{\prime \prime})^n (\xi^{\prime \prime})^\ell A_\ell^j \partial_j \left( a_{(\xi),\diamond} \left( \rhob_{(\xi)}^\diamond \zetab_\xi^{I,\diamond} \right)\circ \Phiik\right) \biggr{)} \notag\\
    &\qquad \times  \frac 13 \partial_r \Phiik^s \left( -(\xi')^s \varphi_\xi^{''} + (\xi'')^s \varphi_\xi' \right)  \circ \Phiik \,  a^2_{(\xi),\diamond} (\rhob_{(\xi)}^\diamond\zetab_\xi^{I,\diamond}\varrho_{(\xi),\diamond}^I)^2\circ \Phiik \xi^\ell   A_{\ell}^j \xi^n A_n^j \biggr{]} \, . \notag 
\end{align}
Note importantly that this term includes factors of either $\partial_p \Phiik^n (\xi')^n$ or $\partial_p \Phiik^n (\xi'')^n$ from $a_{(\xi),\diamond}^{p,\rm bad}$ and $\partial_r\Phiik^s (\xi')^s$ or $\partial_r\Phiik^s (\xi'')^s$ from $\UU_{(\xi),\diamond}^{I,s}$. We immediately see from the alternating property of the Levi-Civita tensor that the terms including
$$ \epsilon_{mpr} \left( \partial_p\Phiik^n (\xi')^n \partial_r \Phiik^s (\xi')^s + \partial_p\Phiik^n (\xi'')^n \partial_r \Phiik^s (\xi'')^s \right) $$
vanish. Thus we only have to consider the cross terms, for example the term
\begin{align}
&\partial_m \biggl{[} \epsilon_{mpr} \partial_p \Phiik^n (\xi^\prime)^n (\xi^\prime)^{\ell} A_\ell^j \partial_j \left( a_{(\xi),\diamond} \left( \rhob_{(\xi)}^\diamond \zetab_\xi^{I,\diamond} \right)\circ \Phiik\right) \notag\\
&\qquad\qquad \times \frac 13 \partial_r \Phiik^s (\xi'')^s \varphi'_\xi \circ \Phiik \,  a^2_{(\xi),\diamond} (\rhob_{(\xi)}^\diamond\zetab_\xi^{I,\diamond}\varrho_{(\xi),\diamond}^I)^2\circ \Phiik \xi^\ell   A_{\ell}^j \xi^n A_n^j \biggr{]} \, . \label{dc:messy:curr}
\end{align}
Our first claim is that the vector indexed by $m$ inside the parentheses is actually \emph{parallel} to $\xi^\ell A_\ell^m$, which means that we have the good differential operator!  To see this, we write
\begin{align*}
   \epsilon_{mpr} \partial_p \Phi^n (\xi')^n \partial_r \Phi^s (\xi'')^s \qquad &\parallel \qquad \xi^\ell A_\ell^m \\
   \iff \epsilon_{mpr} \partial_p \Phi^n (\xi')^n \partial_r \Phi^s (\xi'')^s \partial_m \Phi^k \qquad &\parallel \qquad \xi^\ell A_\ell^m \partial_m \Phi^k = \xi^k \\
   \iff \epsilon_{mpr} \partial_p \Phi^n (\xi')^n \partial_r \Phi^s (\xi'')^s \partial_m \Phi^k (\xi')^k &= \epsilon_{mpr} \partial_p \Phi^n (\xi')^n \partial_r \Phi^s (\xi'')^s \partial_m \Phi^k (\xi'')^k = 0 \, .
\end{align*}
But the last two expressions are again equal to zero by the alternating property of the Levi-Civita tensor! Thus we have shown that the $\partial_m$ on the outside of the expression in \eqref{dc:messy:curr} will only cost $\Lambda_q\Gamma_q^{13}$, and furthermore that it \emph{cannot} land on $\varrho_{(\xi),\diamond}^I\circ \Phiik$ or $\varphi'_\xi\circ \Phiik$ (which is a component of $\UU_{(\xi),\diamond}^I$). Therefore, we write
\begin{align}
    \mathbf{V}_3
&= (\mathbf{C}_{31})_{(\xi),\diamond}^I
(\varphi'_\xi(\varrho_{(\xi),\diamond}^I)^2) \circ \Phiik
+ (\mathbf{C}_{32})_{(\xi),\diamond}^I
(\varphi''_\xi(\varrho_{(\xi),\diamond}^I)^2) \circ \Phiik
\end{align}
where $(\mathbf{C}_{3r})_{(\xi),\diamond}^I$, $r=1,2$, are defined by
\begin{align*}
    (\mathbf{C}_{31})_{(\xi),\diamond}^I
    &:=\frac 13\xi_\ell A_\ell^m\partial_m \biggl{[}
    \biggl{[} 
    \epsilon_{mpr} \partial_p 
    \Phiik^n (\xi^\prime)^n 
    (\xi'')^s
    \biggr{]}_{\parallel\xi}
    (\xi^\prime)^{\ell} A_\ell^j \partial_j \left( a_{(\xi),\diamond} \left( \rhob_{(\xi)}^\diamond \zetab_\xi^{I,\diamond} \right)\circ \Phiik\right) \notag\\
&\hspace{4cm} 
\times  \partial_r \Phiik^s    a^2_{(\xi),\diamond} (\rhob_{(\xi)}^\diamond\zetab_\xi^{I,\diamond})^2\circ \Phiik \xi^\ell   A_{\ell}^j \xi^n A_n^j \biggr{]} \\
  (\mathbf{C}_{32})_{(\xi),\diamond}^I
    &:=\frac 13\xi_\ell A_\ell^m\partial_m \biggl{[}
    \biggl{[} 
    \epsilon_{mpr} \partial_p 
    \Phiik^n (\xi'')^n 
    (\xi^\prime)^s
    \biggr{]}_{\parallel\xi}
    (\xi'')^{\ell} A_\ell^j \partial_j \left( a_{(\xi),\diamond} \left( \rhob_{(\xi)}^\diamond \zetab_\xi^{I,\diamond} \right)\circ \Phiik\right) \notag\\
&\hspace{4cm} 
\times  \partial_r \Phiik^s    a^2_{(\xi),\diamond} (\rhob_{(\xi)}^\diamond\zetab_\xi^{I,\diamond})^2\circ \Phiik \xi^\ell   A_{\ell}^j \xi^n A_n^j \biggr{]}\, , 
\end{align*}
and $[f_m]_{\parallel \xi}$ denotes the $m^{\rm th}$ component of the projection of $f$ onto the vector $\xi^\ell A_{\ell}^\bullet$. The analysis of $\mathbf{V}_3$ will then mimic exactly the analysis of \eqref{dc:curr:mess:0}, since we have a good differential operator in $\partial_m$ and one costly differential operator $\partial_j$ landing on $\zetab_\xi^{I,\diamond}$. 
\smallskip

\noindent\texttt{Step 2: Define the current error, pressure increment, and current error, and verify their properties.} Based on the analysis above, we now define the current errors as
\begin{subequations}
\begin{align}
    \bar\phi_C^{q+\half+1}
    &:= (\divH + \divR)\left[(\mathbf{C}_{0})_{(\xi),\diamond}^{I,s}
\tilde{\mathbb{P}}_{\lambda_{q+\half+1}}\left((
  \varrho_{(\xi),\diamond}^I )^2\UU_{(\xi),\diamond}^{I,s}\right)\circ\Phiik \right]
  \label{current:c:low1}
  \\
  &\quad+ (\divH + \divR)\left[
  (\mathbf{C}_{31})_{(\xi),\diamond}^I
\tilde{\mathbb{P}}_{\lambda_{q+\half+1}}(\varphi'_\xi(\varrho_{(\xi),\diamond}^I)^2) \circ \Phiik\right]\label{current:c:low2}\\
&\quad+ (\divH + \divR)\left[(\mathbf{C}_{32})_{(\xi),\diamond}^I
\tilde{\mathbb{P}}_{\lambda_{q+\half+1}}(\varphi''_\xi(\varrho_{(\xi),\diamond}^I)^2) \circ \Phiik\right] \label{current:c:low3}
\end{align}
\end{subequations}
for the lowest shell,
\begin{subequations}
\begin{align}
\bar\phi_C^{m}
    &:= (\divH + \divR)\left[(\mathbf{C}_{0})_{(\xi),\diamond}^{I,s}
\tilde{\mathbb{P}}_{(\lambda_{m-1},\lambda_{m}]}
\left((
  \varrho_{(\xi),\diamond}^I )^2\UU_{(\xi),\diamond}^{I,s}\right)\circ\Phiik \right]\label{current:c:mid1}\\
  &\quad+ (\divH + \divR)\left[
  (\mathbf{C}_{31})_{(\xi),\diamond}^I
\tilde{\mathbb{P}}_{(\lambda_{m-1},\lambda_{m}]}(\varphi'_\xi(\varrho_{(\xi),\diamond}^I)^2) \circ \Phiik\right]\label{current:c:mid2}\\
&\quad + (\divH + \divR)\left[(\mathbf{C}_{32})_{(\xi),\diamond}^I
\tilde{\mathbb{P}}_{(\lambda_{m-1},\lambda_{m}]}(\varphi''_\xi(\varrho_{(\xi),\diamond}^I)^2) \circ \Phiik\right]\label{current:c:mid3}
\end{align}
\end{subequations}
for $q+\half +1 <m <q+\bn$, and
\begin{subequations}
\begin{align}
\bar\phi_C^{q+\bn}
    &:= \sum_{m=q+\bn}^{q+\bn+1} (\divH + \divR)\left[(\mathbf{C}_{0})_{(\xi),\diamond}^{I,s}
\left(\tilde{\mathbb{P}}_{(\lambda_{m-1},\lambda_{m}]} 
+\left( \Id - \tilde{\mathbb{P}}_{\lambda_{q+\bn+1}} \right)
\right)
\left((
  \varrho_{(\xi),\diamond}^I )^2\UU_{(\xi),\diamond}^{I,s}\right)\circ\Phiik \right]
  \label{current:c:high1}
  \\
  &\quad+
 \sum_{m=q+\bn}^{q+\bn+1} (\divH + \divR)\left[(\mathbf{C}_{31})_{(\xi),\diamond}^I
\left(\tilde{\mathbb{P}}_{(\lambda_{m-1},\lambda_{m}]} 
+\left( \Id - \tilde{\mathbb{P}}_{\lambda_{q+\bn+1}} \right)
\right)(\varphi'_\xi(\varrho_{(\xi),\diamond}^I)^2) \circ \Phiik\right]\label{current:c:high2}\\
&\quad+ \sum_{m=q+\bn}^{q+\bn+1}
(\divH + \divR)\left[(\mathbf{C}_{32})_{(\xi),\diamond}^I
\left(\tilde{\mathbb{P}}_{(\lambda_{m-1},\lambda_{m}]} 
+\left( \Id - \tilde{\mathbb{P}}_{\lambda_{q+\bn+1}} \right)
\right)(\varphi''_\xi(\varrho_{(\xi),\diamond}^I)^2) \circ \Phiik\right]\label{current:c:high3}\\
&\quad + \eqref{eq:dc:current:case:1}
+\epsilon_{\bullet p r}    a_{(\xi),\diamond}^{p,\rm good} \partial_r \Phi^s (\mathbb{U}_{(\xi),\diamond}^I)^s \circ \Phiik \,  a^2_{(\xi),\diamond} (\rhob_{(\xi)}^\diamond\zetab_\xi^{I,\diamond}\varrho_{(\xi),\diamond}^I)^2\circ \Phiik \xi^\ell   A_{\ell}^j \xi^n A_n^j  \, . \label{current:c:high4}
\end{align}
\end{subequations}
The terms involved with $\divR$ or $\Id - \tilde{\mathbb{P}}_{\lambda_{q+\bn+1}}$ go into the non-local parts while the rest goes into the local parts. Indeed, in the case of \eqref{current:c:low1}, \eqref{current:c:mid1}, \eqref{current:c:high1} for example, fix indices $\xi, i, j, k, \vecl, I$, and set
when $\diamond = \ph$
\begin{align*}
G_\ph  = \frac{(\mathbf{C}_{0})_{(\xi),\ph}^{I,s}}{\la_\qbn} \,, \quad 
    \frac{\varrho_\ph}{\la_\qbn} = 
    \begin{cases}
    \tilde{\mathbb{P}}_{\lambda_{q+\half+1}}\left((
  \varrho_{(\xi),\ph}^I )^2\UU_{(\xi),\ph}^{I,s}\right) &\text{ for } \eqref{current:c:low1}\\
  \tilde{\mathbb{P}}_{(\la_{m-1}, \la_m]}\left((
  \varrho_{(\xi),\ph}^I )^2\UU_{(\xi),\ph}^{I,s}\right) &\text{ for } \eqref{current:c:mid1}, \text{ 
 the first term of }\eqref{current:c:high1}\\
    (\Id -\tilde{\mathbb{P}}_{\lambda_{q+\bn +1}})
    \left((
  \varrho_{(\xi),\ph}^I )^2\UU_{(\xi),\ph}^{I,s}\right) &\text{ for the second term of }\eqref{current:c:high1},
    \end{cases}
\end{align*}
and when $\diamond = R$, 
\begin{align*}
G_R  = \frac{(\mathbf{C}_{0})_{(\xi),R}^{I,s}}{ \la_{q+\bn}r_q}  \, , \quad 
    \frac{\varrho_R}{\la_\qbn} = 
    \begin{cases}
    r_q\tilde{\mathbb{P}}_{\lambda_{q+\half+1}}\left((
  \varrho_{(\xi),R}^I )^2\UU_{(\xi),R}^{I,s}\right) &\text{ for } \eqref{current:c:low1}\\
  r_q\tilde{\mathbb{P}}_{(\la_{m-1}, \la_m]}\left((
  \varrho_{(\xi),R}^I )^2\UU_{(\xi),R}^{I,s}\right) &\text{ for } \eqref{current:c:mid1}, \text{ 
 the first term of }\eqref{current:c:high1}\\
    r_q(\Id -\tilde{\mathbb{P}}_{\lambda_{q+\bn +1}})
    \left((
  \varrho_{(\xi),R}^I )^2\UU_{(\xi),R}^{I,s}\right) &\text{ for the second term of }\eqref{current:c:high1} \, . 
    \end{cases}
\end{align*}
Notice that $\varrho_{(\xi),\diamond}^{I} (\mathbb{U}_{(\xi),\diamond}^{I})^s$ has zero mean from \eqref{item:pipe:means} of Proposition \ref{prop:pipeconstruction} and \eqref{item:pipe:means:current} of Proposition \ref{prop:pipe.flow.current}. 
The rest of parameters and the functions are chosen the same as in \texttt{Case~2}---\texttt{Case~4} of the proof of Lemma \ref{lem:oscillation.current:general:estimate}.  The assumptions in \eqref{eq:inverse:div:DN:G} and \eqref{eq:DN:Mikado:density} of Proposition \ref{prop:intermittent:inverse:div} can be verified using Lemma~\ref{lem:a_master_est_p}, Lemma~\ref{lem:special:cases}, Lemma~\ref{lem:LP.est},
item~\eqref{item:pipe:5} from Proposition~\ref{prop:pipeconstruction} and item~\eqref{item:pipe:5:current} from Proposition~\ref{prop:pipe.flow.current}, and we leave the details to the reader; the rest of the conditions of the inverse divergence are then satisfied exactly as the oscillation error, and we omit further details. We note only that the support properties for both $G_\diamond$ and $\rho_\diamond$ are also the same as in the oscillation error, and so we can expect the same support (and dodging) properties to hold for the output of the inverse divergence in this case.

Thus we can apply the inverse divergence from Proposition~\ref{prop:intermittent:inverse:div}. With these choices, we also apply Proposition~\ref{lem.pr.invdiv2.c} to construct the associated pressure increments and pressure currents. Note that as in \texttt{Case 3} of the proof of Lemma \ref{lem:oscillation.current:general:estimate}, 
when $m= q+\half+2$, we split the synthetic Littlewood-Paley operator $\tilde{\mathbb{P}}_{(\la_{m-1}, \la_{m}]}$ further into $\tilde{\mathbb{P}}_{(\la_{q+\half+1}, \la_{q+\half+\sfrac32}]}+\tilde{\mathbb{P}}_{(\la_{q+\half+\sfrac32}, \la_{q+\half+2}]}$ and apply the propositions to each of them. The analysis of \eqref{current:c:low2}, \eqref{current:c:low3}, \eqref{current:c:mid2}, \eqref{current:c:mid3}, 
\eqref{current:c:high2}, \eqref{current:c:high3} is similar; we replace $\UU_{(\xi),\diamond}^{I,s}$ by $\ph'_\xi$ or $\ph''_\xi$ and $(\mathbf{C}_{0})_{(\xi),\diamond}^{I,s}$ by $(\mathbf{C}_{31})_{(\xi),\diamond}^{I}$ or $(\mathbf{C}_{32})_{(\xi),\diamond}^{I}$. As a result, we get the same conclusion as that for the oscillation current error. More precisely, we can verify \eqref{eq:cc.p.1}--\eqref{eq:cor.current:estimate:1} for $q+\half+1\leq m <q+\bn$, \eqref{eq:desert:estimate:1}--\eqref{eq:desert:dodging} for $q+\half+1\leq m' <q+\bn$, and these properties associated to \eqref{current:c:high1}--\eqref{current:c:high3}.

Lastly, we consider \eqref{current:c:high4}. 
From \eqref{est.vel.inc.p.by.pr} and \eqref{est.vel.inc.c.by.pr}, the error terms in \eqref{eq:dc:current:case:1} satisfy
\begin{align*}
    \left|\psi_{i.q}D^N D_{t,q}^M \eqref{eq:dc:current:case:1}\right| \les
    (\si_{\upsilon}^+ + {\de_{q+3\bn}})^{\sfrac32}r_q^{-1}
(\la_{q+\bn}\Ga_{q+\bn}^{{\sfrac1{10}}})^N\MM{M,\Nindt,\tau_q^{-1}\Ga_q^{i+16},\Tau_q^{-1}\Ga_{q}^9}\, . 
\end{align*}
Therefore, \eqref{eq:dc:current:case:1} does not contribute to the pressure increment $\si_{\ov \phi_{C}^{m,l}}$. Recalling \eqref{eq:dodging:oldies}, one can also verify \eqref{eq:cc.p.6} for \eqref{eq:dc:current:case:1}. On the other hand, the remaining term in \eqref{current:c:high4} generates a new pressure increment. To deal with this, we fix values of $i,j,k,\xi,\vecl,I,\diamond$ and define the functions $\hat\upsilon_{b,\diamond} = \hat\upsilon_{b, i,j,k,\xi,\vecl,I,\diamond}$ as
\begin{align*}
    \hat\upsilon_{1,\diamond}&:= r_q^{\sfrac13}
    a_{(\xi),\diamond} \left( \rhob_{(\xi)}^\diamond \zetab_{\xi}^{I,\diamond} \varrho_{(\xi),\diamond}^{I} \right)\circ \Phiik\\
    \hat\upsilon_{3,\diamond}&:= r_q^{-\sfrac13}\epsilon_{\bullet p r}    a_{(\xi),\diamond}^{p,\rm good} \partial_r \Phi^s  \xi^\ell   A_{\ell}^j \xi^n A_n^j \mathbb{U}_{(\xi),\diamond}^{I,s} \circ \Phiik\, .
\end{align*}
The pressure increments and pressure currents associated to the $\hat\upsilon_{1,\diamond}$ and $\hat\upsilon_{3,\diamond}$ can be constructed very similarly to those for the premollified velocity increment potential in Lemmas~\ref{lem:pr.inc.vel.inc.pot} and \ref{lem:pr.current.vel.inc}. In fact, they satisfy better estimates due to the gains from $r_q$ and $\mathbb{U}_{(\xi),\diamond}^{I,s}$. These can now be collected together to yield the desired pressure increment and pressure current estimates associated to \eqref{current:c:high4}. We leave the details to the interested reader. 

\end{proof}

\subsection{Upgrading material derivatives}\label{upgrade:rle:ss}
\begin{definition}[Definition of $\ov\phi_{q+1}$ and $\ov\varphi_{q+1}$]
Recalling Lemmas~\ref{lem:oscillation.current:general:estimate}, \ref{lem:ct:general:estimate}, \ref{lem:lin:current:error}, \ref{lem:current:stressss}, \ref{lem:corrector.current:general:estimate}, and \ref{lem:moll:curr}, we define $\ov\phi_{q+1}= \sum_{m=q+1}^{\qbn} \ov\phi_{q+1}^m$ and $\ov\phi_{q+1}^m = \ov\phi_{q+1}^{m,l}+\ov\phi_{q+1}^{m,*}$ for $q+1\leq m \leq  \qbn$ by
\begin{subequations}
\begin{align}
    \ov\phi_{q+1}^{m,l} &= \ov\phi_O^{m,l} + \ov\phi_W^{m,l} + \ov\phi^{m,l}_{TNC} + \ov\phi_{S_O^{m,l}}^l + \ov\phi_{S_{TN}^{m,l}}^l + \ov\phi_{S_{C1}^{m,l}}^l + \ov\phi_{S_{M2}^{m,l}}^l + \mathbf{1}_{m=\qbn}\ov\phi_L^{\qbn,l} + \mathbf{1}_{m=q+1}\ov\phi_R^{q+1,l}  + \ov\phi_C^{m,l} \\
    \ov\phi_{q+1}^{m,*} &= \ov\phi_O^{m,*}+ \ov\phi_W^{m,*} + \ov\phi^{m,*}_{TNC} + \ov\phi_{S_O^{m,l}}^* + \ov\phi_{S_{TN}^{m,l}}^* + \ov\phi_{S_{C1}^{m,l}}^* + \ov\phi_{S_{M2}^{m,l}}^* + \ov\phi_{S_O^{m,*}} + \ov\phi_{S_{TN}^{m,*}} + \ov\phi_{S_{C1}^{m,*}} + \ov\phi_{S_{M2}^{m,*}} \notag\\
    &\qquad \qquad + \mathbf{1}_{m=\qbn}\ov\phi_L^{\qbn,*} + \ov\phi_R^{m,*} + \ov\phi_C^{m,*} + \ov\phi_{M}^{m}     \end{align}
\end{subequations}
Here, any undefined terms are taken to be $0$.  We then define the primitive current error $\ov \varphi_{q+1}$ by
\begin{align}\label{defn:primitive:current}
    \ov\varphi_{q+1} := \sum_{m=q+1}^{\qbn} \ov\varphi_{q+1}^m \, , \qquad \ov\varphi_{q+1}^m = \varphi_q^m + \ov\phi_{q+1}^m \, ,
\end{align}
which we note is consistent with \eqref{eq:def:ov:phi}.
\end{definition}

\begin{lemma}[\bf Upgrading material derivatives]\label{lem:oscillation.current:general:estimate2}
The new current errors $\ov\phi_{q+1}^{m}=\ov\phi_{q+1}^{m,l}+\ov\phi_{q+1}^{m,*}$ satisfy the following. For $N+M\leq \sfrac{\Nind}{4}$, we have that
    \begin{align}
    \left| \psi_{i,q+\half-1} D^N D_{t,q+\half-1}^M \ov\phi^{q+\half}_{q+1} \right| 
    &\lesssim \Gamma_{q+\half}^{-50}  \pi_q^{q+\half} r_{q+\half}^{-1} \Lambda_{q+\half}^N \MM{M, \Nindt, \tau_{q+\half-1}^{-1}\Gamma_{q+\half-1}^{i-5},\Tau_{q}^{-1}\Ga_q^{11}}  \, . \label{eq:osc.current.upgraded}
    \end{align}
For the same range of $N+M$, the current error $\ov\phi_{q+1}^{q+1}$ obeys the estimate
    \begin{align}
    \left| \psi_{i,q} D^N D_{t,q}^M \ov\phi^{q+1}_{q+1} \right| 
    &\lesssim \Gamma_{q+1}^{-50} \pi_q^{q+1} r_{q+1}^{-1} \Lambda_{m}^N \MM{M, \Nindt, \tau_{q}^{-1}\Gamma_{q}^{i+20}, \Tau_{q}^{-1}\Ga_q^{{10}}}  \, . \label{eq:osc.current.upgraded.special}
\end{align}
Finally, we have that for $q+\half+1\leq m \leq \qbn$ and the same range of $N+M$,
\begin{align}
    \left| \psi_{i,m-1} D^N D_{t,m-1}^M \ov\phi^{m}_{q+1} \right| 
    &\lesssim \left(\si_{m,q+1}^+ + \mathbf{1}_{m=\qbn}\Ga_\qbn^{-50}\pi_q^{\qbn} +  {\de_{q+3\bn}} \right)^{\sfrac32} r_m^{-1} \notag\\
    &\qquad \qquad \times (\la_{m}\Ga_m)^N 
    \MM{M, \Nindt, \Ga_q^{i+18} \tau_q^{-1}, \Tau_q^{-1}\Ga_q^{11}} \, . \label{eq:osc.current.upgraded.general}
\end{align}
\end{lemma}
\begin{proof}[Proof of Lemma~\ref{lem:oscillation.current:general:estimate2}]
We have that \eqref{eq:osc.current.upgraded.special} follows immediately from \eqref{eq.cur.osc.low.1}, \eqref{eq:ct:lowshell:nopr:1}, \eqref{eq:ctnl:estimate:1}, \eqref{eq:sc:loc:est}, \eqref{eq:sc:nonloc:est}, \eqref{est:curr.mollification1}, \eqref{low.bdd.pi}, \eqref{eq:ind.pr.anticipated}, and \eqref{la.beats.de}.  In order to prove the remaining estimates, we appeal to Lemma~\ref{lem:upgrading.material.derivative}.  The proof is very similar to the proofs of items~\eqref{upgrade:item:2}--\eqref{upgrade:item:4} of Lemma~\ref{l:divergence:stress:upgrading}, and so we omit most of the details. The basic idea is however that nonlocal error terms can be upgraded trivially using the minuscule amplitude, and the local error terms can be upgraded using the dodging conclusions that have been included in  Lemmas~\ref{lem:oscillation.current:general:estimate}, \ref{lem:ct:general:estimate}, \ref{lem:lin:current:error}, \ref{lem:current:stressss}, \ref{lem:corrector.current:general:estimate}, and \ref{lem:moll:curr}.
\end{proof}

\section{New intermittent pressure}\label{sec:intermittent:pressure}
In this section, we first define a new {pressure increment}\index{pressure increment} $\si_{q+1}$, which leads to a new pressure $p_{q+1}$, and a new intermittent pressure\index{intermittent pressure} $\pi_{q+1}$. We then replace the Euler-Reynolds system and the relaxed local energy inequality for $(u_{q+1}, p_q, \overline{R}_{q+1}, \overline\ph_{q+1}, -(\pi_q-\pi_q^q))$ with the Euler-Reynolds system and relaxed local energy inequality for $(u_{q+1}, p_{q+1}, R_{q+1}, \ph_{q+1}, -\pi_{q+1})$ by defining new stress errors $R_{q+1}$ and new current errors $\ph_{q+1}$. In fact adding the pressure increment to the pressure at the $q^{\rm th}$ step creates some current errors, called pressure current errors.  These will be absorbed into the new current error $\ph_{q+1}$. This section also verifies all inductive assumptions related to a new intermittent pressure and new stress/current errors, specifically the assumptions in subsection~\ref{ss:relaxed}, \ref{sec:pi:inductive}, Hypothesis \ref{hyp:dodging5}, and \eqref{est.upsilon.ptwise}. 

\subsection{New pressure increment and new intermittent pressure}\label{sec:new.pressure}

In this subsection, we define a new pressure increment $-\si_{q+1} = p_{q+1}-p_q$ and introduce a new intermittent pressure $\pi_{q+1}$.

We collect the pressure increments generated by new errors and new velocity increment potentials. Recall that Lemma~\ref{lem:pr.inc.vel.inc.pot} defined a pressure increment ($\sigma_\upsilon$) associated to premollified velocity increment potentials, subsection~\ref{sec:new.pressure.stress} recalled pressure increments ($\sigma_{S^m}$) associated to various stress errors, and Lemmas~\ref{lem:oscillation.current:general:estimate}, \ref{lem:ct:general:estimate}, and \ref{lem:corrector.current:general:estimate} defined pressure increments ($\sigma_{\ov\phi_O^m}$, $\sigma_{\ov\phi_{TN}^m}$, and $\sigma_{\ov\phi_C^M}$) associated to various current errors. Then fixing $m$ such that $q+\half+1 \leq m \leq q+\bn$, we define\index{$\si_{m,q+1}$}
\begin{align}
    \si_{m,q+1} &:= \si_{S^{m}} + \si_{\ov \phi^{m}_O} + \si_{\ov\phi^{m}_{TN}} + \si_{\ov\phi^{m}_C} +\mathbf{1}_{\{m=\qbn\}} \sigma_{\upsilon} \,.  \label{defn:si.m}
\end{align}
Recalling that every pressure increment referenced above has a decomposition $\si_\bullet = \si_\bullet^+ - \si_\bullet^-$, we define $\si_{m,q+1}^+$\index{$\si_{m,q+1}^{\pm}$} and $\si_{m,q+1}^-$ in the obvious way. 

Next, associated to each pressure increment $\sigma_\bullet$ listed above is a function of time $\bmu_{\sigma_\bullet}$ which satisfies $\bmu_{\sigma_\bullet}' = \langle \Dtq \sigma_\bullet \rangle$ (see Lemma~\ref{l:divergence:stress:upgrading} and Lemmas~\ref{lem:pr.inc.vel.inc.pot},  \ref{lem:currentoscillation:pressure:current}, \ref{lem:ctn:pressure:current}, and \ref{lem:currentdivergencecorrector:pressure:current}), and so we define\index{$\bmu_{m,q+1}$}
\begin{align}
    \bmu_{m,q+1} &:= \bmu_{\sigma_{S^{m}}} + \bmu_{\sigma_{\ov \phi^{m}_O}} + \bmu_{\sigma_{\ov\phi^{m}_{TN}}}  + \bmu_{\ov\phi^{m}_C} +\mathbf{1}_{\{m=\qbn\}} \bmu_{\sigma_\upsilon} \, . \label{defn:bmu.m}
\end{align} \index{$\bmu_{m,q+1}$}
Furthermore, recall that Lemma~\ref{lem:pr.current.vel.inc} defined a current error associated to velocity pressure increments, subsection~\ref{sec:new.pressure.stress} recalled current errors associated to various stress error pressure increments, and Lemmas~\ref{lem:currentoscillation:pressure:current}, \ref{lem:ctn:pressure:current}, and \ref{lem:currentdivergencecorrector:pressure:current} defined current errors associated to various current error pressure increments. Then fixing $m,m'$ such that $q+\half+1 \leq m' \leq m \leq \qbn$, we define
\begin{subequations}\label{defn:phi.m.m'}
\begin{align}
    \phi^{m',l}_{m,q+1} &:= \phi^{m',l}_{{S^{m}}} + \phi^{m',l}_{{\ov \phi^{m}_O}} + \phi^{m',l}_{{\ov\phi^{m}_{TN}}} + \phi^{m',l}_{{\ov\phi^{m}_C}} +\mathbf{1}_{\{m=\qbn\}}  \phi^{m',l}_{\sigma_\upsilon}\\
    \phi^{m',*}_{m,q+1} &:= \phi^{m',*}_{{S^{m}}} + \phi^{m',*}_{{\ov \phi^{m}_O}} + \phi^{m',*}_{{\ov\phi^{m}_{TN}}} + \phi^{m',*}_{{\ov\phi^{m}_C}} +\mathbf{1}_{\{m=\qbn\}} \phi^{m',*}_{\upsilon}  \notag \\
   &\qquad \qquad + \mathbf{1}_{\{m'=m\}} \bigg{(} \phi^{*}_{{S^{m}}} + \phi^{*}_{{\ov \phi^{m}_O}} + \phi^{*}_{{\ov\phi^{m}_{TN}}} + \phi^{*}_{{\ov\phi^{m}_C}} +\mathbf{1}_{\{m=\qbn\}} \phi^{*}_{\upsilon} \bigg{)} \, .
\end{align}
\end{subequations}
Now we set\index{$\phi_{m,q+1}$}
\begin{equation}\label{def:phi:m:qplus}
    \phi_{m,q+1} := \sum_{m'=q+\half{+1}}^m \phi^{m',l}_{m,q+1} + \phi^{m',*}_{m,q+1} \, ,
\end{equation}
so that the aforementioned lemmas along with \eqref{Sunday:Sunday:Sunday} give the equality 
\begin{align}\label{exp.Dtq.si}
\div \phi_{m,q+1} = \Dtq \si_{m,q+1} - \bmu'_{m,q+1}
=\Dtq \si_{m,q+1}-\langle \Dtq \si_{q+1,m} \rangle \, .
\end{align}
By appealing to the lemmas mentioned above along with Lemma~\ref{lem:prop.si.pre.stress}, we have that the $\si_{m,q+1}$'s satisfy the properties listed in the following lemma; we refer the reader to Sections~ \ref{sec:new.pressure.stress}, and \ref{sec:RLE:errors} for more details.

\begin{lemma}[\bf Collected properties of error terms and pressure increments]\label{lem:prop.si.pre}
For each $q+\half+1 \leq m \leq q+\bn$, $\si_{m,q+1}$ satisfies the following properties.
\begin{enumerate}[(i)]
\item For any $0\leq k \leq \dpot$, we have that
    \begin{subequations}
        \begin{align}
        \left|\psi_{i,q}D^N D_{t,q}^M S^{m,l}_{q+1}\right|
    &\lec \left(\si_{m,q+1}^+ +  {\de_{q+3\bn}}\right)  (\la_{m}\Ga_m)^N  \label{est.S.m.pt.sim}
    \MM{M, \Nindt, \Ga_q^{i+18} \tau_q^{-1}, \Tau_q^{-1}\Ga_q^9}  \\
    \left|\psi_{i,q}D^N D_{t,q}^M \bar \phi^{m,l}_{q+1} \right|
    &\lec \left(\si_{m,q+1}^+ + \mathbf{1}_{m=\qbn}\Ga_\qbn^{-50}\pi_q^{\qbn} 
    +  {\de_{q+3\bn}} \right)^{\sfrac32} r_m^{-1} \notag\\
    &\qquad \qquad \times (\la_{m}\Ga_m)^N 
    \MM{M, \Nindt, \Ga_q^{i+18} \tau_q^{-1}, \Tau_q^{-1}\Ga_q^9} 
    \label{est.phi.m.pt.sim}
    \end{align}
    \end{subequations}
where the first bound holds for $N+M\leq 2\Nind$, and the second bound holds for $N+M\leq \sfrac{\Nind}{4}$. 
\item For $N, M \leq {\sfrac{\Nfin}{200}}$, we have that
\begin{subequations}
    \begin{align}
        \norm{\psi_{i,q}D^N D_{t,q}^M \si_{m,q+1}^+}_{\sfrac32}
    &\lec \Ga_m^{-9} \de_{m+\bn}  (\la_{m}\Ga_m)^N 
    \MM{M, \Nindt, \Ga_q^{i+18} \tau_q^{-1}, \Tau_q^{-1}\Ga_q^9} \label{est:si.m+.32.Dtq}\\
    \norm{\psi_{i,q}D^N D_{t,q}^M \si_{m,q+1}^+}_{\infty}
    &\lec {\Ga_m^{\badshaq-9}}   (\la_{m}\Ga_m)^N 
    \MM{M, \Nindt, \Ga_q^{i+18} \tau_q^{-1}, \Tau_q^{-1}\Ga_q^9} \label{est:si.m+.infty.Dtq}\\
    \left|\psi_{i,q}D^N D_{t,q}^M \si_{m,q+1}^+\right|
    &\lec \left(\si_{m,q+1}^+ +  {\de_{q+3\bn}}\right)  (\la_{m}\Ga_m)^N 
    \MM{M, \Nindt, \Ga_q^{i+18} \tau_q^{-1}, \Tau_q^{-1}\Ga_q^9} \label{est:si.m+.pt}\\
    \left|\psi_{i,q}D^N D_{t,q}^M \si_{m,q+1}^-\right|
    &\lec {\Ga_{q+\half}^{-100}} \pi_q^{q+\half}  (\la_{q+\half}\Ga_{q+\half})^N 
    \MM{M, \Nindt, \Ga_q^{i+18} \tau_q^{-1}, \Tau_q^{-1}\Ga_q^9} \, .
\end{align}
\end{subequations}
\item $\si_{m,q+1}$ and $\si_{m,q+1}^+$ have the support properties
\begin{subequations}
\begin{align}
    B(\supp \hat w_{q'}, \la_{q'}^{-1}{\Ga_{q'+1}})
    \cap \si_{m,q+1} = \emptyset \qquad \forall q+1\leq q'\leq q+\half  \, , \label{supp.si.m}\\
    B(\supp \hat w_{q'}, \la_{q'}^{-1}{\Ga_{q'+1}})
    \cap \si_{m,q+1}^+ = \emptyset \qquad \forall q+1\leq q'\leq m-1\, . \label{supp.si.m+}
\end{align}
\end{subequations}
\end{enumerate}
\end{lemma}
\begin{remark}[\bf Upgrading material derivatives]\label{upgrade.mat.sim.plus}
As a consequence of \eqref{est:si.m+.32.Dtq}, \eqref{est:si.m+.infty.Dtq}, \eqref{est:si.m+.pt}, and \eqref{supp.si.m+},
we may apply Lemma \ref{lem:upgrading.material.derivative}
to $F= F^l = \si_{m,q+1}^{\pm}$ to upgrade the material derivative estimates.  In particular, we obtain that
\begin{subequations}
    \begin{align}
        \norm{\psi_{i,m-1}D^N D_{t,m-1}^M \si_{m,q+1}^{+}}_{\sfrac32}
        &\lec {\Ga_m^{-9}} \de_{m+\bn} (\la_m\Ga_m)^N 
        \MM{M, \Nindt, \Ga_{m-1}^{i-5} \tau_{m-1}^{-1}, \Tau_{m-1}^{-1}\Ga_{m-1}^{-1}} \\
        \norm{\psi_{i,m-1}D^N D_{t,m-1}^M \si_{m,q+1}^{+}}_{\infty}
        &\lec \Ga_m^{\badshaq-9}  (\la_m\Ga_m)^N
        \MM{M, \Nindt, \Ga_{m-1}^{i-5} \tau_{m-1}^{-1}, \Tau_{m-1}^{-1}\Ga_{m-1}^{-1}} \label{est.td.si.m.pl}\\
        \left|\psi_{i,m-1}D^N D_{t,m-1}^M \si_{m,q+1}^+\right|
    &\lec \left(\si_{m,q+1}^+ +  {\de_{q+3\bn}}\right)  (\la_{m}\Ga_m)^N 
    \MM{M, \Nindt, \Ga_{m-1}^{i-5} \tau_{m-1}^{-1}, \Tau_{m-1}^{-1}\Ga_{m-1}^{-1}}\label{est.td.si.m.pl.ptwise} \\
    \left|\psi_{i,q+\half-1}D^N D_{t,q+\half-1}^M \si_{m,q+1}^-\right|
    &\lec {\Ga_{q+\half}^{-100}} \pi_q^{q+\half}  (\la_{q+\half}\Ga_{q+\half})^N  \notag\\
    &\qquad \qquad \times 
    \MM{M, \Nindt, \Ga_{q+\half-1}^{i-5} \tau_{q+\half-1}^{-1}, \Tau_{q+\half-1}^{-1}\Ga_{q+\half-1}^{-1}} \, .
    \label{est.td.si.m.pi.-}
    \end{align}
\end{subequations}
    for ${N, M \leq \sfrac{\Nfin}{200}}$. Then, applying Lemma \ref{lem:cooper:2} to $v= \hat u_{m-1}$, $f = \si_{m,q+1}^{\pm}$, $p=\infty$, and $\Omega = \supp(\psi_{i,m-1})$ (or $\Omega = \mathcal{N}(x)$ where $\mathcal{N}(x)$ is a closed neighborhood of $x$ contained in $\supp(\psi_{i,m-1})$), we have
    \begin{subequations}
         \begin{align}
        \norm{D^{N} D_{t,m-1}^M \na \si_{m,q+1}^{+}}_{L^\infty(\supp \psi_{i,m-1})}
        &\lec \Ga_m^{\badshaq-9}  (\la_m\Ga_m)^{N+1} 
        \MM{M, \Nindt, \Ga_{m-1}^{i-5} \tau_{m-1}^{-1}, \Tau_{m-1}^{-1}\Ga_{m-1}^{-1}} \label{est.td.si.m.pl:comm}\\
    \left| \psi_{i,m-1} D^{N} D_{t,m-1}^M \na \si_{m,q+1}^+ \right|
    &\lec \left(\si_{m,q+1}^+ +  {\de_{q+3\bn}}\right)  (\la_{m}\Ga_m)^{N+1} \notag\\
    &\qquad  \times \MM{M, \Nindt, \Ga_{m-1}^{i-5} \tau_{m-1}^{-1}, \Tau_{m-1}^{-1}\Ga_{m-1}^{-1}}\label{est.td.si.m.pl.ptwise:comm}\\
    \left|\psi_{i,q+\half-1}D^N D_{t,q+\half-1}^M \nabla \si_{m,q+1}^-\right|
    &\lec {\Ga_{q+\half}^{-100}} \pi_q^{q+\half}  (\la_{q+\half}\Ga_{q+\half})^{N+1}  \notag\\
    &\qquad  \times 
    \MM{M, \Nindt, \Ga_{q+\half-1}^{i-5} \tau_{q+\half-1}^{-1}, \Tau_{q+\half-1}^{-1}\Ga_{q+\half-1}^{-1}} \, ,
    \label{ineq:billystrings:1}
    \end{align}
    \end{subequations}
for $N < \sfrac{\Nfin}{200}$ and $M \leq \sfrac{\Nfin}{200}$.
\end{remark}

\begin{definition}[\bf Pressure increment $\sigma_{q+1}$\index{$\sigma_{q+1}$, $\sigma_{q+1}^k$} and decomposition into $\sigma_{q+1}^k$]\label{def:new:pressure}
Define constants
\begin{align}\label{defn:A}
    a_{m,q,k} := 2^{k-q-1}\left(\frac{\de_{k+\bn}}{\de_{m+\bn}}\right){\Ga_m^9}, \quad A_{m,q} = \sum_{k=m}^{{q+\Npr}} a_{m,q,k}\, ,
\end{align}
where $\Npr$ is chosen in item~\eqref{i:par:4.5} in subsection \ref{sec:para.q.ind}. With these constants in hand, we define \index{$\si_{q+1}^\pm$}
\begin{align}
    \label{def.new.si}
    \si_{q+1} &:= \sum_{m=q+\half+1}^{q+\bn} \underbrace{A_{m,q} \si_{m,q+1}}_{=: \td \si_{m,q+1}}, \qquad \si_{q+1}^\pm := \sum_{m=q+\half+1}^{q+\bn} \underbrace{A_{m,q} \si_{m,q+1}^\pm}_{=: \td \si_{m,q+1}^\pm} \, .
\end{align} \index{$\td \si_{m,q+1}$, $\td \si_{m,q+1}^\pm$}
Then reorganizing terms in $\si_{q+1}$ based on amplitude, have that
\begin{align}
    \si_{q+1} = \sum_{k=q+\half}^{q+\Npr} \si_{q+1}^k\, ,  \label{eq:rewriting:new:si}
\end{align}
where
\begin{align}
    \si_{q+1}^{q+{\half}}=-\si_{q+1}^- \, , \qquad \si_{q+1}^k
    &=\sum_{m=q+\half {+1}}^{\min(k, q+\bn)} a_{m,q,k} \si_{m,q+1}^+, \quad\text{for all }q+\half+1 \leq k \leq q+\Npr \, .  \label{def:sigma:qplus:k}
\end{align}
\end{definition}

As a direct consequence of Lemma~\ref{lem:prop.si.pre} and Definition~\ref{def:new:pressure}, we have that $\si_{q+1}^k$ satisfies the following properties.

\begin{lemma}[\bf Properties of $\sigma_{q+1}$ and $\sigma_{q+1}^k$]\label{lem:saturday:one}
For all $q+\half \leq k \leq q+\Npr$, the pressure increment $\si_{q+1}^k$ has the following properties. 
\begin{enumerate}[(i)]
\item \label{item.pr.1} $\si_{q+1}^k$ has the support property
\begin{align}\label{supp.sik}
    B(\supp \hat w_{q'}, \la_{q'}^{-1}{\Ga_{q'+1}})
    \cap \supp(\si_{q+1}^k) = \emptyset \qquad \forall q+1\leq q'\leq q+\half \, . 
\end{align}
\item \label{item.pr.2} 
For all $q+\half+1 \leq k \leq q+\bn$ and $N,M\leq \sfrac{\Nfin}{200}$, we have that $\si_{q+1}^k$ satisfies
\begin{subequations}
\begin{align}
    \norm{\psi_{i,k-1}D^N D_{t,k-1}^M \si_{q+1}^k}_{\sfrac32}
    &\lec \de_{k+\bn}  (\lambda_{k}\Ga_{k})^N 
    \MM{M, \Nindt, \Ga_{k-1}^{i-3} \tau_{k-1}^{-1}, \Tau_{k-1}^{-1}\Ga^{-1}_{k-1}} \, , \label{est:sik.32}\\
    \norm{\psi_{i,k-1}D^N D_{t,k-1}^M \si_{q+1}^k}_{\infty}
    &\lec \Ga_{k}^{\badshaq}   (\lambda_{k}\Ga_{k})^N 
    \MM{M, \Nindt, \Ga_{k-1}^{i-3} \tau_{k-1}^{-1}, \Tau_{k-1}^{-1}\Ga^{-1}_{k-1}} \, . \label{est:sik.infty}
\end{align}
\end{subequations}
For all $q+\bn+1 \leq k \leq q+\Npr$ and $N, M \leq \sfrac{\Nfin}{200}$, we have that $\si_{q+1}^k$ satisfies
\begin{subequations}
\begin{align}
    \norm{\psi_{i,\qbn}D^N D_{t,\qbn}^M \si_{q+1}^k}_{\sfrac32}
    &\lec \de_{k+\bn}  (\lambda_{\qbn}\Ga_{\qbn})^N 
    \MM{M, \Nindt, \Ga_\qbn^{i-3} \tau_\qbn^{-1}, \Tau_\qbn^{-1}\Ga^{-1}_\qbn} \,  \label{est:sik.32.higher} \\
    \norm{\psi_{i,\qbn}D^N D_{t,\qbn}^M \si_{q+1}^k}_{\infty}
    &\lec \Ga_{\qbn}^{\badshaq}  (\lambda_{\qbn}\Ga_{\qbn})^N 
    \MM{M, \Nindt, \Ga_\qbn^{i-3} \tau_\qbn^{-1}, \Tau_\qbn^{-1}\Ga^{-1}_\qbn} \, . \label{est:sik.infty.higher}
\end{align}
\end{subequations}
\item \label{item.pr.3} For $q+\half+1 \leq k \leq q+\bn$ and $0\leq k'\leq \dpot$, we have that
\begin{subequations}
\begin{align}
    \left|\psi_{i,q}D^N D_{t,q}^M S^{k}_{q+1}\right|
    &\leq \Ga_k^{-8} \left(\si_{q+1}^k + \de_{q+2\bn}\right)  (\lambda_k\Ga_k)^N  \MM{M, \Nindt, \Ga_q^{i+18} \tau_q^{-1}, \Tau_q^{-1}\Ga_q^9} \, , \label{est.S.k.pt.sik} \\
    \left|\psi_{i,q}D^N D_{t,q}^M \bar \phi^k_{q+1} \right|
    &\leq \Ga_k^{-13} \left(\si_{q+1}^k + \mathbf{1}_{m=\qbn}\Ga_\qbn^{-50}\pi_q^{\qbn} 
    +\de_{q+2\bn}
    \right)^{\sfrac32} r_k^{-1} \notag\\
    &\qquad \qquad \times (\lambda_k\Ga_k)^N 
    \MM{M, \Nindt, \Ga_q^{i+18} \tau_q^{-1}, \Tau_q^{-1}\Ga_q^9} \, ,
    \label{est.phi.k.pt.sik}
    \end{align}
\end{subequations}
where the first bound holds for $N+M\leq 2\Nind$, and the second bound holds for $N+M\leq \sfrac{\Nind}{4}$. 
\item \label{item.pr.3.5} For $q+\half+1\leq k \leq q+\Npr$ and $N,M\leq \sfrac{\Nfin}{200}$ and $q+1 \leq k' \leq \min(k-1,\qbn)$, we have that
\begin{align}
    \left|\psi_{i,k'} D^N D_{t,k'}^M \si_{q+1}^k\right|
    &\lec  (\si_{q+1}^k  + \Ga_q^{-100} \de_{k+\bn}) (\min(\lambda_k\Ga_k, \la_\qbn \Ga_\qbn))^N   \MM{M, \Nindt, \Ga_{k'}^{i-3} \tau_{k'}^{-1}, \Tau_{k'}^{-1}\Ga^{-1}_{k'}}  \, . \label{est.si.k.pt}
\end{align}
For the same range of $N$ and $M$,
we have that
\begin{align}
    \left|\psi_{i,q+\half-1}D^N D_{t,q+\half-1}^M \si_{q+1}^{q+\half}\right|
    &\leq \Ga_{q+\half}^{-25} \pi_{q}^{q+\half} (\lambda_{q+\half}\Ga_{q+\half})^N\\ &\qquad\times \MM{M, \Nindt, \Ga_{q+\half-1}^{i-3} \tau_{q+\half-1}^{-1}, \Tau_{q+\half-1}^{-1}\Ga_{q+\half-1}^{-1}} \, . \label{est.si-.pt}
\end{align}    
\item \label{item.pr.4}  For all $q+\half+1 \leq k \leq k' \leq q+\Npr$, we have that
\begin{align}\label{eq:ind.pr.inc.anticipated}
    \frac{\de_{k'+\bn}}{\de_{k+\bn}} \si_{q+1}^k \leq 2^{k-k'} \si_{q+1}^{k'} \, .
\end{align}
\item \label{item.pr.5} For all $q+\bn \leq k' \leq k \leq q+\Npr$, we have that
\begin{align}\label{pr.inc.upper}
     \si_{q+1}^k \leq \si_{q+1}^{k'} . 
\end{align}
\end{enumerate}
\end{lemma}
\begin{proof}[Proof of Lemma~\ref{lem:saturday:one}]\texttt{Proof of item~\eqref{item.pr.1}.} The proof of this item is immediate from Definition~\ref{def:new:pressure} and item~\eqref{item:sat:one} from Lemma~\ref{lem:prop.si.pre}.
\smallskip

\noindent\texttt{Proof of item~\eqref{item.pr.2}.} We first consider the estimates for $q+\half \leq k \leq \qbn$. From Remark~\ref{upgrade.mat.sim.plus}, which ensures that every $\sigma_{m,q+1}^+$ has size $\Ga_m^{-9}\delta_{m+\bn}$ in $L^{\sfrac 32}$, and Definition~\ref{def:new:pressure}, which ensures that the term in $\sigma_{q+1}^k$ coming from $\sigma_{m,q+1}$ has been rescaled by a factor of $\delta_{k+\bn}\delta_{m+\bn}^{-1}\Ga_m^{-9}$, we have that \eqref{est:sik.32} holds when $N=M=0$. Similarly, when $N=M=0$, we have that \eqref{est:sik.infty} holds since $\Ga_k^{\badshaq}$ is increasing in $k$. In order to prove the versions of these estimates which involve derivatives, we must use Lemma~\ref{rem:upgrade.material.derivative.end} and \eqref{eq:inductive:timescales} (at level $q$ since we do not require $D_{t,\qbn-1}$) to upgrade the estimates in Remark~\ref{upgrade.mat.sim.plus}, since $\sigma_{q+1}^k$ is comprised of rescaled versions of $\sigma_{m,q+1}$ for $m\leq k$ which came with $D_{t,m-1}$ estimates. We omit further details and simply note that the material derivative cost and the assumptions required in \eqref{eq:cooper:w} follow from \eqref{eq:nasty:D:vq:old} at level $q$ (i.e. we apply \eqref{eq:nasty:D:vq:old} for $q'\leq \qbn-1$), and that the pointwise bounds follow from the usual trick of choosing $\Omega$ to be a neighborhood centered at a point $(x,t)$ and then shrinking the diameter of $\Omega$ to zero and using continuity. Finally, the proofs of the estimates for $q+\bn+1\leq k \leq q+\Npr$ are quite similar, except that we have to use \eqref{eq:inductive:timescales} and \eqref{eq:nasty:D:vq:old} at level $q+1$ (i.e. $q'=\qbn$), both of which have been already verified in Proposition~\ref{prop:verified:vel:cutoff} and Proposition~\ref{prop:inductive:velocity:bdd:verified}.
\smallskip

\noindent\texttt{Proof of item~\eqref{item.pr.3}.}
To obtain \eqref{item.pr.3}, we use \eqref{est.S.m.pt.sim}, \eqref{est.phi.m.pt.sim}, 
\eqref{defn:A}, and \eqref{def.new.si} to give the inequality $\sigma_{q+1}^k \geq \sigma_{k,q+1}^+ 2^{k-q-1}\Ga_k^9$, nonlocal estimates for $S_{q+1}^{k,*}$ from Lemma~\ref{l:divergence:stress:upgrading}, nonlocal estimates for current errors from Section~\ref{sec:RLE:errors}, and spare factors of $\Ga_k^{-\sfrac 12}$ to absorb implicit constants.
\smallskip

\noindent\texttt{Proof of item~\eqref{item.pr.3.5}.}  In order to prove \eqref{est.si.k.pt}, we first prove the estimate when no derivatives have been applied. First note that $\delta_{q+3\bn}\delta_{k+\bn}\delta_{m+\bn}^{-1}\Ga_m^9 \leq \delta_{k+\bn}\Ga_q^{-100}$, since $m\leq \qbn$ so that the definition of $\delta_{q'}$ and \eqref{eq:prepping:badshaq} can absorb $\Ga_q^{-500}$.  Then since both sides are linear in $\sigma_{q+1}^k$, the rescalings involved in the definition of $\sigma_{q+1}^k$, \eqref{est.td.si.m.pl.ptwise}, and the inequality just noted give the proper amplitude bound. At this point we must upgrade material derivative in a manner analogous to that which is required to prove the $L^{\sfrac 32}$ and $L^\infty$ bounds from item~\eqref{item.pr.2} of Lemma \ref{lem:saturday:one}, and so omit further details. In order to prove \eqref{est.si-.pt}, first note that $a_{m,q,k}$ is at most $\Ga_m^{10}$ if $k=m$ and we choose $a_0$ sufficiently large, and $a_{m,k,q}\ll 1$ if $k>m$. Then using \eqref{def:sigma:qplus:k}, \eqref{def.new.si}, and the fact that $\Ga_q^2 > \Ga_{\qbn}$ since $b^{\bn}<2$ from \eqref{ineq:b:second}, the estimate without derivatives follows from \eqref{est.td.si.m.pi.-}. Upgrading material derivatives then follows in the usual way, and we omit further details.
\smallskip

\noindent\texttt{Proof of item~\eqref{item.pr.4}.} Since $k'\geq k$, we have from \eqref{def:sigma:qplus:k} that
\begin{align*}
    \frac{\de_{k'+\bn}}{\de_{k+\bn}} \si_{q+1}^k &=  \sum_{m=q+\half+1}^{\min(k,q+\bn)} \frac{\de_{k'+\bn}}{\de_{k+\bn}} a_{m,q,k} \si_{m,q+1}^+ 
    = \sum_{m=q+\half+1}^{\min(k,q+\bn)} 2^{k-q-1} \frac{\de_{k'+\bn}}{\de_{m+\bn}}\Ga_m^9 \si_{m,q+1}^+\\
    &= 2^{k-k'} \sum_{m=q+\half+1}^{\min(k,q+\bn)} a_{m,q,k'} \si_{m,q+1}^+ 
    \leq 2^{k-k'} \si_{q+1}^{k'}\, .
\end{align*}
\smallskip

\noindent\texttt{Proof of item~\eqref{item.pr.5}.} This estimate follows from the observation in the proof of the previous item that $a_{m,k,q}\ll 1$ if $k>m$, \eqref{eq:rewriting:new:si}, \eqref{def:sigma:qplus:k}, and a large choice of $a_0$ which can be used to absorb implicit constants.
\end{proof}

Finally, we can define the new intermittent pressure $\pi_{q+1}$.

\begin{definition}[\bf New intermittent pressure $\pi_{q+1}$\index{$\pi_{q+1}$, $\pi_{q+1}^k$} and decomposition into $\pi_{q+1}^k$]\label{def:new:presh}
We define $\pi_{q+1}^k$, $k\geq q+1$, by
\begin{align}
    \label{def:new.pr.piece}
    \pi_{q+1}^k &:= \pi_q^k +
    \si_{q+1}^k+2^{k-q-1}\de_{k+\bn} \qquad \textnormal{for  $q+\half\leq k \leq q+\Npr$} \, , \qquad \\
    \pi_{q+1}^k &:=\pi_q^k \qquad \textnormal{for $q+1 \leq k \leq q+\half-1$, $q+\Npr+1\leq k < \infty$} \, . \notag
\end{align}
Then $\pi_{q+1} = \sum_{k=q+1}^\infty \pi_{q+1}^k$ satisfies
\begin{align}\label{defn:new.pr}
    \pi_{q+1} = \pi_q-\pi_q^q+
\sigma_{q+1} + \sum_{k=q+\half}^{q+\Npr} 2^{k-q-1}\de_{k+\bn}\,. 
\end{align}
\end{definition}

\subsection{Inductive assumptions on the new intermittent pressure}\label{sec.pr.ind.verify}

In this section, we verify the inductive assumptions on $\pi_{q+1}^k$ which are required in subsections~\ref{sec:pi:inductive}--\ref{sec:inductive:secondary:velocity}.

\begin{lemma}[\bf $L^{\sfrac 32}$, $L^\infty$, and pointwise bounds on $\pi_{q+1}^k$]\label{lem:verify.ind.pressure1}
The inductive assumptions \eqref{eq:pressure:inductive} and \eqref{eq:pressure:inductive:largek}
are verified at step $q+1$.
\end{lemma}
\begin{proof}
We first consider \eqref{eq:pressure:inductive:dtq}--\eqref{eq:ind:pi:by:pi}. In the case that $q+1\leq k \leq q+\half-1$, we have that $\pi_{q+1}^k = \pi_q^k$ from Definition~\ref{def:new:presh}, so that the desired estimates follow trivially from inductive assumptions \eqref{eq:pressure:inductive:dtq}--\eqref{eq:ind:pi:by:pi} at step $q$. In the case that $q+\half\leq k \leq q+\bn$, we have from Definition~\ref{def:new:presh} that $\pi_{q+1}^k = \pi_q^k + \si_{q+1}^k + 2^{k-q-1}\de_{k+\bn}$. Therefore, the desired estimates follow from the inductive assumptions and Lemma~\ref{lem:saturday:one}.  In order to get \eqref{eq:pressure:inductive:dtq:largek}--\eqref{eq:ind:pi:by:pi:largek}, we have from Definition~\ref{def:new:presh} that $\pi_{q+1}^k=\pi_q^k+\sigma_{q+1}^k+2^{k-q-1}\delta_{k+\bn}$.  Then the desired estimates follow again from the inductive assumptions and Lemma~\ref{lem:saturday:one}.
\end{proof}

\begin{lemma}[\bf Lower and upper bounds for $\pi_{q+1}^k$]\label{lem:lower:upper} Inductive assumptions \eqref{low.bdd.pi}--\eqref{eq:ind.pr.anticipated} are verified at step $q+1$.
\end{lemma}
\begin{proof}
In order to prove \eqref{low.bdd.pi} at level $q+1$, we first consider the cases when $q+1\leq k \leq q+\half-1$.  In these cases the inductive assumption \eqref{low.bdd.pi} and Definition~\ref{def:new:presh} imply that
\begin{align*}
    \pi_{q+1}^k = \pi_{q}^k \geq \de_{k+\bn} \, .
\end{align*}
For the case $k=q+\half$, we use \eqref{est.si-.pt} and \eqref{def:new.pr.piece} to write that
\begin{align*}
    \pi_{q+1}^{q+\half} = \pi_q^{q+\half}
    +\si_{q+1}^{q+\half} +2^{\half-1}\de_{q+\half+\bn}
    \geq \de_{q+\half+\bn} \, ,
\end{align*}
concluding the proof of \eqref{low.bdd.pi} at level $q+1$. For the remaining cases, we use \eqref{eq:ind.pr.anticipated} at the level of $q+1$, so that we postpone the proof to the end. 

Next, from Definition~\ref{def:new:presh}, the inductive assumption \eqref{ind:pi:upper}, and \eqref{pr.inc.upper}, we have that for $\qbn\leq k' < k < q+\Npr$, 
\begin{align*}
    \pi_{q+1}^k &= \pi_q^k + \sigma_{q+1}^k + 2^{k-q-1} \delta_{k+\bn}
    \leq \pi_q^{k'} + \sigma_{q+1}^{k'} + \delta_{k'+\bn} \leq \pi_{q+1}^{k'} \, .
\end{align*}
In the endpoint case when $k=q+\Npr$, we use that $\pi_q^{k+\Npr}\equiv \Ga_{q+\Npr}\delta_{q+\Npr+\bn}$ from \eqref{defn:pikq.large.k}, in which case a similar string of inequalities then concludes the proof that \eqref{ind:pi:upper} is satisfied at level $q+1$.

From Definition~\ref{def:new:presh} and \eqref{defn:pikq.large.k} at level $q$, we have that
\begin{align*}
    \pi_{q+1}^k = \Ga_k \de_{k+\bn}
\end{align*}
for $k\geq q+\Npr +1$, so that the inductive assumption \eqref{defn:pikq.large.k} for $q+1$ holds true.

Finally, we must prove \eqref{eq:ind.pr.anticipated} at level $q+1$. We split into cases depending on the value of $q''$. If $q''\geq q+\Npr+1$, then we have from Definition~\ref{def:new:presh}, the Sobolev inequality applied to \eqref{eq:pressure:inductive:dtq}, \eqref{eq:pressure:inductive:dtq:largek}, and \eqref{est:sik.32}, and \eqref{defn:Npr} that
\begin{align*}
    \frac{\de_{q''+\bn}}{\de_{q'+\bn}} \pi_{q+1}^{q'} &\leq \frac{\de_{q''+\bn}}{\de_{q'+\bn}} \left(\pi_q^{q'} + \mathbf{1}_{\{q+\half \leq q' \leq q+\Npr\}} \max(0,\si_{q+1}^{q'}) +2^{\Npr}\de_{q'+\bn}\right) \\
    &\leq \frac{\de_{q''+\bn}}{\de_{q'+\bn}} \left(\|\pi_q^{q'}\|_\infty + \mathbf{1}_{\{q+\half \leq q' \leq q+\Npr\}} \|\si_{q+1}^{q'}\|_\infty \right)   {+2^{\Npr}\de_{q''+\bn}}\\
    &\leq  \de_{q''+\bn}  \left(\Ga_q\Ga_{q+\Npr}\La_{q+\bn}^3
     + 2^{\Npr} \right) 
    \leq \Ga_q^{-1}\Ga_{q+\Npr+1}\de_{q''+\bn} \\
    &\leq 2^{q'-q''} \Ga_{q''}  \delta_{q''+\bn} 
    = 2^{q'-q''} \pi_{q+1}^{q''} \, .
\end{align*}
Note that in the inequalities above, we have assumed a large choice of $a_0$ to absorb the implicit constant.  Next, in the cases when $ q+1+\half\leq q'' \leq q+\Npr$, we first note that $\si_{q+1}^{q''} \geq 0$ and $\sigma_{q+1}^{q+\half}\leq 0$ since all the minus portions of the pressure have been absorbed into $\sigma_{q+1}^{q+\half}$ in \eqref{eq:rewriting:new:si} and \eqref{def:sigma:qplus:k}.  As a consequence of these facts, Definition~\ref{def:new:presh}, \eqref{eq:ind.pr.anticipated}, and \eqref{eq:ind.pr.inc.anticipated}, we have that
\begin{align*}
    \frac{\de_{q''+\bn}}{\de_{q'+\bn}} \pi_{q+1}^{q'} &\leq \frac{\de_{q''+\bn}}{\de_{q'+\bn}} \left(\pi_q^{q'} + \mathbf{1}_{\{q+\half \leq q' \leq q+\Npr\}}\max\{0,\si_{q+1}^{q'}\}  {+2^{q'-q-1}\de_{q'+\bn}}\right)\\ &\leq 2^{q'-q''} \left(\pi_q^{q''}+\max\{0,\si_{q+1}^{q''}\}  {+2^{q''-q-1}\de_{q''+\bn}}\right) \\
    &= 2^{q'-q''}\pi_{q+1}^{q''} \, .
\end{align*}
Now for the case $q''=q+\half$, we only must consider $q'\leq q+\half-1$, and so from Definitions~\ref{def:new:pressure} and \ref{def:new:presh} and \eqref{est.si-.pt},
\begin{align*}
    \frac{\de_{q+\half+\bn}}{\de_{q'+\bn}} \pi_{q+1}^{q'} &= \frac{\de_{q+\half+\bn}}{\de_{q'+\bn}} \pi_q^{q'}
    \leq 2^{q'-q+\half}\pi_q^{q+\half}\\
    &= \left(\pi_q^{q+\half} + \si_{q+1}^{q+\half}\right) + \left(\left(2^{q'-q+\half}-1\right)\pi_q^{q+\half} - \si_{q+1}^{q+\half}\right)\\
    &\leq \pi_{q+1}^{q+\half} + \left(- \sfrac12 + \Ga_{q+\half}^{-25}\right)\pi_q^{q+\half}\\
    &\leq  \pi_{q+1}^{q+\half} \, .
\end{align*}
In the final cases $q''<q+\half$, we have from Definition~\ref{def:new:presh} and inductive assumption \eqref{eq:ind.pr.anticipated} that
\begin{align*}
    \frac{\de_{q''+\bn}}{\de_{q'+\bn}} \pi_{q+1}^{q'} &= \frac{\de_{q''+\bn}}{\de_{q'+\bn}} \pi_q^{q'} \leq \pi_q^{q''} = \pi_{q+1}^{q''} \, ,
\end{align*}
concluding the proof of \eqref{eq:ind.pr.anticipated} at level $q+1$.

Lastly, we consider \eqref{low.bdd.pi} for
$k>q+\half$. We first note that from \eqref{low.bdd.pi} for $q+1\leq k \leq q+\half$ and \eqref{eq:ind.pr.anticipated} at level $q+1$, we have that for all $q+\half+1 \leq k' < \infty$, 
\begin{equation}\notag
    \pi_{q+1}^{k'} > 2^{k'-q-\half} \frac{\delta_{k'+\bn}}{\delta_{q+\half+\bn}} \pi_{q+1}^{q+\half} > \delta_{k'+\bn} \, ,
\end{equation}
and so
\begin{equation}\notag
    \pi_{q+1}^k \geq \delta_{k+\bn} \qquad \forall q+1 \leq k < \infty \, .
\end{equation}
\end{proof}

\begin{lemma}[\bf Pressure dominates velocity]\label{prop:velocity:domination}
The inductive assumptions in \eqref{eq:ind:velocity:by:pi}, \eqref{eq:psi:q:q'}, and \eqref{est.upsilon.ptwise} are verified at level $q+1$.
\end{lemma}
\begin{proof}
\texttt{Step 1: Verification of \eqref{est.upsilon.ptwise} at level $q+1$.} From \eqref{est.upsilon.ptwise.verify} and the definition of $\sigma_q^\qbn$ in Definition~\ref{def:new:pressure}, which give an extra prefactor of $\Ga_\qbn^9$, we have that
\begin{align*}
\left|\psi_{i,q+\bn-1}D^N D_{t,q+\bn-1}^M  \hat\upsilon_{\qbn,k'} \right|
    &\leq \Ga_{q+\bn}^{-4} \left(\si_{q+1}^{q+\bn} + \de_{q+2\bn}\right)^{\sfrac12} r_{q}^{-1} \\
    &\qquad\times (\lambda_{q+\bn}\Ga_{q+\bn})^N 
    \MM{M, \Nindt, \Ga_{q+\bn-1}^{i} \tau_{q+\bn-1}^{-1}, \Ga_{q+\bn-1}^{2}\Tau_{q+\bn-1}^{-1}}
\end{align*}
for all $N+M \leq \sfrac{3\Nfin}{2}$.  Then using the definition of $\pi_{q+1}^{q+\bn}$ from Definition~\ref{def:new:presh} gives the proof of \eqref{est.upsilon.ptwise} for $q'=\qbn$; in fact we retain the extra smallness prefactor of $\Ga_\qbn^{-4}$, which we shall use in the next step of this proof.  In order to verify \eqref{est.upsilon.ptwise} for $q\leq q' \leq \qbn-1$, we appeal to Definition~\ref{def:new:presh} for the definition of $\pi_{q+1}^{q'}$.  Noticing that $\pi_{q+1}^{q'} \geq \pi_{q}^{q'}$ for all  $q'\geq q+1$ except for $q'= q+\half$, we have that the verification of \eqref{est.upsilon.ptwise} for $q\leq q' \leq \qbn-1$, $q' \neq q+\half$ is trivial.  In the case $q' = q+\half$, we use \eqref{est.si-.pt} and \eqref{def:new.pr.piece} to write that
$$  \pi_{q+1}^{q+\half} \geq \sfrac 12 \pi_{q}^{q+\half}  \, , $$
from which \eqref{est.upsilon.ptwise} follows using the increase from $\Ga_q$ to $\Ga_{q+1}$ in \eqref{est.upsilon.ptwise} at level  $q$ versus level $q+1$, respectively.
\smallskip

\noindent\texttt{Step 2: Verification of \eqref{eq:ind:velocity:by:pi} at level $q+1$.} We first consider the cases $q+1\leq k \leq q+\bn-1$.  From the  same reasoning as above, which showed that $\pi_{q+1}^{k} \geq \sfrac 12 \pi_q^k$ for $q+1\leq k \leq q+\bn-1$, we have that \eqref{eq:ind:velocity:by:pi} trivially holds.  In the case $k=\qbn$, we use \eqref{exp.w.q'} at level $q'=q+\bn$ (verified in Proposition~\ref{prop:inductive:velocity:bdd:verified}) and \eqref{est.upsilon.ptwise} for $q'=q+\bn$, $k=\dpot$ (which we just verified with extra factor gain), \eqref{est.e.inf} at level $q'=\qbn$ (verified in Proposition~\ref{prop:inductive:velocity:bdd:verified}), and \eqref{def:new.pr.piece} to write that for $N+M\leq \sfrac{3\Nfin}{2}$
\begin{align*}
    \left| \psi_{i,\qbn-1} D^N D_{t,\qbn-1}^M \hat w_\qbn \right| &= \left| \psi_{i,\qbn-1} D^N D_{t,\qbn-1}^M \left( \hat \upsilon_{\qbn,\dpot} + \hat e_{\qbn} \right) \right| \\
    &\leq \Ga_\qbn^{-3} \left( (\pi_{q+1}^{\qbn})^{\sfrac 12} r_q^{-1} + \delta_{q+3\bn}^3 \right) (\la_\qbn \Ga_\qbn)^N \\
    &\qquad \qquad \times \MM{M, \Nindt, \Ga_{q+\bn-1}^{i} \tau_{q+\bn-1}^{-1}, \Ga_{q+\bn-1}^{2}\Tau_{q+\bn-1}^{-1}} \\
    &\leq (\pi_{q+1}^{\qbn})^{\sfrac12} r_q^{-1} (\la_\qbn \Ga_\qbn)^N \MM{M, \Nindt, \Ga_{q+\bn-1}^{i} \tau_{q+\bn-1}^{-1}, \Ga_{q+\bn-1}^{2}\Tau_{q+\bn-1}^{-1}} \, ,
\end{align*}
which verifies \eqref{eq:ind:velocity:by:pi} at level $q+1$ with $q'=\qbn$.
\smallskip

\noindent\texttt{Step 3: Verification of \eqref{eq:psi:q:q'} for $q'=q+\bn$.} We will prove that
\begin{align}
    \sum_{i=0}^{\imax} \psi_{i,q+\bn}^2 \delta_{q+\bn} r_{q}^{-\sfrac 23} \Gamma_{q+\bn}^{2i} &\les  {r_{q}^{-2}} \pi_{q+1}^{q+\bn} \,  \label{eq:psi:q:qplusbn:ineq:0}
\end{align}
for a $q$-independent implicit constant, from which \eqref{eq:psi:q:q'} for $q'=q+\bn$ follows by using the extra factor of $\Ga_\qbn$ to absorb the implicit constant and the powers of $2$. 

Note that from \eqref{defn:si.m} and the subsequent sentence, which shows that $\sigma_{\qbn,q+1}^+ \geq \sigma_{\upsilon}^+$, Definition~\ref{def:new:pressure}, which shows that $\sigma_{q+1}^\qbn \geq \sigma_{\qbn,q+1}^+$, and Definition~\ref{def:new:presh}, which shows that $\pi_{q+1}^\qbn \geq \sigma_{q+1}^\qbn + \delta_{q+3\bn}$, we have that $\sigma_\upsilon^+ + \delta_{q+3\bn} \leq \pi_{q+1}^\qbn$. Thus, using \eqref{def:new.pr.piece}, we see that \eqref{eq:psi:q:qplusbn:ineq:0} follows from \eqref{eq:psi:q:qplusbn:ineq:0:recall} of Lemma~\ref{lem:pr.vel.dom.cutoff} and, consequently, we have \eqref{eq:psi:q:q'} at level $q+1$ with $q'=\qbn$.
\smallskip

\noindent\texttt{Step 4: Verification of \eqref{eq:psi:q:q'} for $q+1 \leq q' \leq \qbn-1$.} Recall that in \texttt{Step 1}, we showed that
\begin{equation}\label{eq:billystrings:p}
    \pi_{q+1}^{q'} \geq \sfrac 12 \pi_q^{q'} \qquad  q' \geq q+1 \, .
\end{equation}
Therefore we may use \eqref{eq:psi:q:q'} at level $q$ to write that
\begin{align*}
    \sum_{i=0}^{\imax} \psi_{i,q'}^2 \delta_{q'} r_{q'-\bn}^{-\sfrac23} \Gamma_{q'}^{2i} &\leq 2^{q-q'} \Ga_{q'} r_{q'-\bn}^{-2} \pi_{q}^{q'}\\
    &= 2^{q+1-q'}\Ga_{q'} r_{q'-\bn}^{-2} \left( \sfrac 12\pi_{q}^{q'} \right)\\
    &\leq 2^{q+1-q'}\Ga_{q'} r_{q'-\bn}^{-2} \pi_{q+1}^{q'} \, ,
\end{align*}
concluding the proof of \eqref{eq:psi:q:q'} at level $q+1$.
\end{proof}

\begin{lemma}[\bf Pressure dodging at level $q+1$] \index{pressure dodging}
Hypothesis \ref{hyp:dodging5} is verified at step $q+1$. 
\end{lemma}
\begin{proof} We must show that for all $q+1<k\leq q+\bn$, $k\leq k'$, and $N+M\leq 2\Nind$,
\begin{subequations}
\begin{align}
\notag
    \left|\psi_{i,k-1} D^N D_{t,k-1}^M \left(\hat w_{k}\pi_{q+1}^{k'}\right)\right| &<
    \Ga_{q+1} 
    {\Ga_k^{-100}} \left(\pi_{q+1}^k\right)^{\sfrac32} r_k^{-1} \Lambda_k^{N} \MM{M, \NindRt, \Gamma_{k-1}^{i+{1}} \tau_{k-1}^{-1} , \Gamma_{k}^{-1} \Tau_{k}^{-1} } \, .
\end{align}
\end{subequations}
We divide up the proof into cases based on the value of $k'$.
\smallskip

\noindent\texttt{Case 1: $q+1<k \leq k' < q+\half$.} From \eqref{def:new.pr.piece}, we have that $\pi_{q+1}^{k'} = \pi_q^{k'}$, and so using Hypothesis~\ref{hyp:dodging5} at level $q$ and \eqref{eq:billystrings:p}, we have that for $N+M\leq  2\Nind$,
\begin{align*}
    \left|\psi_{i,k-1} D^N D_{t,k-1}^M \left(\hat w_{k}\pi_{q+1}^{k'}\right)\right| &=     \left|\psi_{i,k-1} D^N D_{t,k-1}^M \left(\hat w_{k}\pi_{q}^{k'}\right)\right| \\
    &< \Ga_q 
    {\Ga_k^{-100}} \left(\pi_q^k\right)^{\sfrac32} r_k^{-1} \Lambda_k^{N} \MM{M, \NindRt, \Gamma_{k-1}^{i+{1}} \tau_{k-1}^{-1} , \Gamma_{k}^{-1} \Tau_{k}^{-1} } \\
    &< \Ga_{q+1} 
    {\Ga_k^{-100}} \left(\pi_{q+1}^k\right)^{\sfrac32} r_k^{-1} \Lambda_k^{N} \MM{M, \NindRt, \Gamma_{k-1}^{i+{1}} \tau_{k-1}^{-1} , \Gamma_{k}^{-1} \Tau_{k}^{-1} } \, .
\end{align*}

\noindent\texttt{Case 2: $q+1 < k \leq k' = q+\half$. } In this case, we have from \eqref{def:new.pr.piece} that $\pi_{q+1}^{q+\half} = \pi_q^{q+\half} + \sigma_{q+1}^{q+\half}+2^{\half-1}\delta_{q+\half+\bn}$. 
Then considering just the contribution $\hat w_k \pi_{q}^{q+\half}$ to $\hat w_k \pi_{q+1}^{q+\half}$ from the first term, we have from Hypothesis~\ref{hyp:dodging5} at level $q$ and \eqref{eq:billystrings:p} that for $N+M\leq 2\Nind$,
\begin{align*}
    \left|\psi_{i,k-1} D^N D_{t,k-1}^M \left(\hat w_{k}\pi_{q}^{q+\half}\right)\right| &< \Ga_q 
    {\Ga_k^{-100}} \left(\pi_q^k\right)^{\sfrac32} r_k^{-1} \Lambda_k^{N} \MM{M, \NindRt, \Gamma_{k-1}^{i+{1}} \tau_{k-1}^{-1} , \Gamma_{k}^{-1} \Tau_{k}^{-1} } \\
    &< \sfrac 12 \Ga_{q+1} 
    {\Ga_k^{-100}} \left(\pi_{q+1}^k\right)^{\sfrac32} r_k^{-1} \Lambda_k^{N} \MM{M, \NindRt, \Gamma_{k-1}^{i+{1}} \tau_{k-1}^{-1} , \Gamma_{k}^{-1} \Tau_{k}^{-1} } \, .
\end{align*}
Next, we have from \eqref{supp.sik} that $\hat w_k \sigma_{q+1}^{q+\half}\equiv 0$ for $q+1<k \leq q+\half$, and so we may ignore the contribution from $\sigma_{q+1}^{q+\half}$.  Finally, in order to bound the contribution coming from the constant term $\delta_{q+\half+\bn}$, we use \eqref{low.bdd.pi} and \eqref{eq:ind:velocity:by:pi} at level $q+1$ and \eqref{ineq:r's:eat:Gammas} to write that
\begin{align*}
    \left|\psi_{i,k-1} D^N D_{t,k-1}^M \left(\hat w_{k} 2^{\sfrac \bn 2 -1} \delta_{q+\half+\bn} \right)\right| &\leq 2^{\sfrac \bn 2 -1} \Ga_{q+1} r_{k-\bn}^{-1} (\pi_{q+1}^k)^{\sfrac 12} \delta_{k+\bn} \\
    &\qquad \qquad \times \Lambda_k^{N} \MM{M, \NindRt, \Gamma_{k-1}^{i+\blue{1}} \tau_{k-1}^{-1} , \Gamma_{k}^{-1} \Tau_{k}^{-1} } \\
    &< \sfrac 12 \Ga_{q+1} 
    {\Ga_k^{-100}} \left(\pi_{q+1}^k\right)^{\sfrac32} r_k^{-1} \Lambda_k^{N} \MM{M, \NindRt, \Gamma_{k-1}^{i+{1}} \tau_{k-1}^{-1} , \Gamma_{k}^{-1} \Tau_{k}^{-1} } \, .
\end{align*}
\smallskip

\noindent\texttt{Case 3: $q+1<k \leq k'$, $q+1 < k \leq \qbn$, and $q+\half+1 \leq k' < \infty$.} From Definition~\ref{def:new:presh}, we have that in these cases, either $\pi_{q+1}^k = \pi_q^k + \sigma_{q+1}^k + 2^{k-q-1}\delta_{k+\bn}$ or $\pi_{q+1}^k = \pi_q^k$.  We therefore first make a few preliminary calculations to help bound the contributions from $\pi_q^k$ and $2^{k-q-1}\delta_{k+\bn}$ before dividing up further into subcases.  We first recall \eqref{low.bdd.pi} at level $q+1$, \begin{equation}\label{ineq:sunday:afternoon:1}
    \pi_{q+1}^k \geq \delta_{k+\bn} \qquad \forall q+1 \leq k < \infty \, .
\end{equation}
Then, we have from \eqref{eq:ind:velocity:by:pi} at level $q+1$ that for all $q+1<k \leq k'\leq\qbn-1$, $k\leq k'$,  
\begin{align}
    \left|\psi_{i,k-1} D^N D_{t,k-1}^M \left(\hat w_{k} \delta_{k'+\bn} \right)\right| &\leq \Ga_{q+1} r_{k-\bn}^{-1} (\pi_{q+1}^k)^{\sfrac 12} \delta_{k+\bn} \Lambda_k^{N} \MM{M, \NindRt, \Gamma_{k-1}^{i+{1}} \tau_{k-1}^{-1} , \Gamma_{k}^{-1} \Tau_{k}^{-1} } \notag \\
    &< \sfrac 13 \Ga_{q+1} 
    {\Ga_k^{-100}} \left(\pi_{q+1}^k\right)^{\sfrac32} r_k^{-1} \Lambda_k^{N} \MM{M, \NindRt, \Gamma_{k-1}^{i+{1}} \tau_{k-1}^{-1} , \Gamma_{k}^{-1} \Tau_{k}^{-1} } \, . \label{billystrings:1}
\end{align}

Next, we have from Hypothesis~\ref{hyp:dodging5} at level $q$ and \eqref{eq:billystrings:p} that for $q+1< k \leq \qbn-1$ and $k \leq k' < \infty$,
\begin{align}
    \left|\psi_{i,k-1} D^N D_{t,k-1}^M \left(\hat w_{k} \pi_{q}^{k'}\right)\right| &< \Ga_q
    {\Ga_k^{-100}} \left(\pi_q^k\right)^{\sfrac32} r_k^{-1} \Lambda_k^{N} \MM{M, \NindRt, \Gamma_{k-1}^{i+{1}} \tau_{k-1}^{-1} , \Gamma_{k}^{-1} \Tau_{k}^{-1} } \notag  \\
    &< \sfrac 13 \Ga_{q+1} 
    {\Ga_k^{-100}} \left(\pi_{q+1}^k\right)^{\sfrac32} r_k^{-1} \Lambda_k^{N} \MM{M, \NindRt, \Gamma_{k-1}^{i+{1}} \tau_{k-1}^{-1} , \Gamma_{k}^{-1} \Tau_{k}^{-1} } \, . \label{billystrings:2}
\end{align}
We claim that the above estimate holds in addition for $k=\qbn$ and $k\leq k' < \infty$. Indeed from \eqref{eq:ind:pi:by:pi:largek} and \eqref{ind:pi:upper} at level $q$, \eqref{eq:ind:velocity:by:pi} at level $q+1$, \eqref{eq:billystrings:p}, and \eqref{ineq:r's:eat:Gammas}, we have that 
\begin{align}
    \left|\psi_{i,q+\bn-1} D^N D_{t,q+\bn-1}^M \left(\hat w_{\qbn} \pi_{q}^{k'}\right)\right| &\les \Ga_{q+1} r_q^{-1} (\pi_{q+1}^{\qbn})^{\sfrac 12} \Ga_q
    \pi_q^{k'}  \Lambda_\qbn^{N} \MM{M, \NindRt, \Gamma_{\qbn-1}^{i+{1}} \tau_{\qbn-1}^{-1} , \Gamma_{\qbn}^{-1} \Tau_{\qbn}^{-1} } \notag  \\
    &< \Ga_{q+1} 
    {\Ga_k^{-101}} \left(\pi_{q+1}^{\qbn}\right)^{\sfrac32} r_\qbn^{-1} \Lambda_\qbn^{N} \MM{M, \NindRt, \Gamma_{\qbn-1}^{i+{1}} \tau_{\qbn-1}^{-1} , \Gamma_{\qbn}^{-1} \Tau_{\qbn}^{-1} } \, . \label{billystrings:3}
\end{align}

Finally, we use Definition~\ref{def:new:pressure}, equations~\eqref{defn:A}--\eqref{def:sigma:qplus:k} and the dodging ensured by \eqref{supp.si.m+} to write that for $q+1< k \leq \qbn$, $k \leq k'$, and $q+\half+1 \leq k' < \infty$, 
\begin{align*}
    \left| \psi_{i,k-1} D^N D_{t,k-1}^M \left( \hat w_k \sigma_{q+1}^{\min(k', q+\bn)} \right) \right| &= \left| \sum_{m=q+\half +1}^{k'}  \psi_{i,k-1} D^N D_{t,k-1}^M  \left(a_{m,q,k'}\hat w_{k} \si_{m,q+1}^+ \right) \right| \\
    &= \left| \sum_{m=q+\half+1}^{k} \psi_{i,k-1} D^N D_{t,k-1}^M  \left(a_{m,q,k'}\hat w_{k} \si_{m,q+1}^+ \right) \right| \, .
\end{align*}
Then using \eqref{eq:ind:velocity:by:pi} and \eqref{eq:inductive:timescales} at level $q+1$, \eqref{est:si.m+.pt}, and \eqref{supp.si.m+}, we have that the quantity above is controlled by
\begin{align}
    &\sum_{N_1=0}^{N} \sum_{M_1=0}^{M} 
    \left| \psi_{i,k-1} D^{N_1} D_{t,k-1}^{M_1}  \hat w_{k}\right| \left| \sum_{m=q+\half +1}^{k} a_{m,q,k'} \mathbf{1}_{\supp \psi_{i,k-1}}  D^{N-N_1} D_{t,k-1}^{M-M_1}   \si_{m,q+1}^+  \right| \notag
    \\
    &\lec \sum_{N_1,M_1}
    \left|\psi_{i,k-1} D^{N_1} D_{t,k-1}^{M_1}  \hat w_{k}\right|
     \sum_{i': \psi_{i',q}\psi_{i,k-1}\not\equiv 0} \sum_{m=q+\half +1}^k a_{m,q,k}
     \mathbf{1}_{\supp \psi_{i',q}} \left|D^{N-N_1} D_{t,q}^{M-M_1}   \si_{m,q+1}^+ \right| \notag \\
    &\lec \Ga_{q+1}  r_{k-\bn}^{-1} (\pi_{q+1}^k)^{\sfrac 12} (\si_{q+1}^k + \Ga_q^{-100}\de_{k+\bn})
    \Lambda_{k}^N \MM{M, \NindRt, \Gamma_{k-1}^{i} \tau_{k-1}^{-1} , \Gamma_{k}^{-1} \Tau_{k}^{-1} } \notag \\
    &\leq \Ga_k^{-101} \left( \pi_{q+1}^k \right)^{\sfrac 32} r_k^{-1}\Lambda_{k}^N \MM{M, \NindRt, \Gamma_{k-1}^{i} \tau_{k-1}^{-1} , \Gamma_{k}^{-1} \Tau_{k}^{-1} } \, . \label{billystring:4}
\end{align}
Here, the second inequality follows from $0\leq a_{m,q,k'} \leq a_{m,q,k}$ because of $k\leq k'$, and the third inequality follows from the proof of \eqref{est.si.k.pt}. 
\smallskip

\noindent\texttt{Case 3a: $q+1 < k \leq \qbn$, $k \leq k'$, and $q+\half+1 \leq k' \leq q+\Npr$.} In these cases, we have from Definition~\ref{def:new:presh} that $\pi_{q+1}^k=\pi_q^k + \sigma_{q+1}^k + 2^{k-q-1}\delta_{k+\bn}$. Then combining \eqref{billystrings:1}, \eqref{billystrings:2}, \eqref{billystrings:3}, and \eqref{billystring:4} concludes the proof.
\smallskip

\noindent\texttt{Case 3b: $q+1 < k \leq \qbn$, $k \leq k'$, and $q+\Npr+1 \leq k' <\infty$.} In these cases, we have from Definition~\ref{def:new:presh} that $\pi_{q+1}^k=\pi_q^k$. Then \eqref{billystrings:2} gives the desired estimate.
\end{proof}

\subsection{The Euler-Reynolds system and the relaxed LEI adapted to new pressure}

In this section, we upgrade the Euler-Reynold system \eqref{ER:new:equation} and the relaxed local energy inequality \eqref{eq:LEI:new} adapted to the new pressure $p_{q+1} = p_q - \si_{q+1}$ and the new intermittent pressure $\pi_{q+1}$ defined in subsections~\ref{sec:new.pressure}. 
\begin{lemma}[Relaxed equations at level $q+1$]\label{lem:relaxed:new}
The inductive assumptions \eqref{eqn:ER}--\eqref{eq:LEI:decomp:basic} are satisfied at level $q+1$.
\end{lemma}
\begin{proof}
We first set a few notations and definitions.  Referring to \eqref{defn:bmu.m}, \eqref{defn:phi.m.m'}, and Definition~\eqref{def:new:presh}, we first define $\phi_{P1}$ and $\bmu'_{P1}$ by
\begin{align}
     \bmu'_{P1} := \frac32\sum_{m=q+\half{+1}}^{q+\bn} A_{m,q} \bmu'_{m,q+1} \, , \qquad
     \phi_{P1} := \frac32\sum_{m=q+\half{+1}}^{q+\bn} A_{m,q} \phi_{m,q+1} \, . \label{eq:monday:morning}
\end{align}
From the definition \eqref{def.new.si} of $\si_{q+1}$ and \eqref{exp.Dtq.si}, we therefore have that
\begin{align}\label{si.phi.bmu}
\frac32 \Dtq \si_{q+1} = \div \phi_{P1} + \bmu'_{P1} \, .
\end{align}
In particular, $\bmu'_{P1}=\frac32\langle \Dtq \si_{q+1} \rangle$. Recalling \eqref{defn:local.stress} and \eqref{eq:meanz:def}, we now define \index{$R_{q+1}$}
\begin{equation}\label{eq:sat:night:7:42}
R^{q+\bn,*}_{q+1} = \ov R^{q+\bn,*}_{q+1} + \frac23 (\bmu_T + \bmu_N + \bmu_L+ \bmu_{M1} + \bmu_{M2} + \bmu_{P1}) \Id \, , 
\end{equation}
and we do not modify the local part of the Reynolds stress that was defined in \eqref{defn:local.stress}, nor do we modify any of the nonlocal portions for $q+1\leq m \leq \qbn-1$.  We then define a new stress error at step $q+1$ by
\begin{align}
    R_{q+1} &= \ov R_{q+1} + \frac23 (\bmu_T + \bmu_N + \bmu_L+ \bmu_{M1} + \bmu_{M2} + \bmu_{P1}) \Id \underset{\eqref{eq:meanz:def}}{=} \ov R_{q+1} + \frac 23 \left( \bmu_{\ov\phi_{q+1}} + \bmu_{P1} \right) \Id \, , \label{billystrings:eq:1}
\end{align}\index{$\bmu_{\ov\phi_{q+1}}$} \index{$\bmu_{P1}$}
which verifies \eqref{eq:ER:decomp:basic}.

Recalling \eqref{defn:new.pr}, \eqref{ER:new:equation} and the new pressure \index{$p_{q+1}$}
\begin{align}\label{new:pressure}
    p_{q+1} &= -\sigma_{q+1} + {p_{q}} \,,
\end{align}
we now have that
$(u_{q+1}, p_{q+1}, R_{q+1}, -\pi_{q+1})$ solves the Euler-Reynolds system
\begin{align}
    \partial_t u_{q+1} 
    + \div \left( u_{q+1} \otimes u_{q+1} \right) + \nabla {p_{q+1}}    
    = \div (  -\pi_{q+1}\Id + R_{q+1})
   \, , \label{ER:new:equation:final}
\end{align}
where we have used that the constant term in \eqref{defn:new.pr} and the terms with functions of time $\bmu_\bullet$ in $R_{q+1}$ vanish inside of the divergence.  Thus we have verified \eqref{eqn:ER}.  Note that we have also verified \eqref{eq:pi:decomp:basic} as well.  Recalling \eqref{def.w.mollified}, we have in addition that \eqref{eq:hat:no:hat} is verified, and so it only remains to check \eqref{eq:LEI:decomp:basic} and \eqref{ineq:relaxed.LEI} at level $q+1$.

Let us set \index{$\ka_{q+1}$}
\begin{align}
    \ka_{q+1} = \frac12\tr(R_{q+1}-\pi_{q+1}\Id)) \, . \label{billys:1} 
\end{align}
Then we can now rewrite \eqref{eq:LEI:new} as the relaxed local energy inequality for $(u_{q+1}, p_{q+1}, R_{q+1}, -\pi_{q+1}, \ph_{q+1})$ adapted to the upgraded stress error $R_{q+1}$ and the new pressure $\pi_{q+1}$. Specifically, we have that \index{$\phi_{P1}$} \index{$\phi_{P2}$}
\begin{align}
    &\pa_t \left( \frac 12 |u_{q+1}|^2 \right)
    + \div\left( \left(\frac 12 |u_{q+1}|^2 + p_{q+1}\right) u_{q+1}\right) \notag\\
    &\underset{\eqref{eq:LEI:new},\eqref{new:pressure}}{=}(\pa_t + \hat u_{q+1} \cdot \na ) (\overline{\ka}_{q+1}+ \bmu_{\ov\phi_{q+1}})
    + \div\left( {(\ov R_{q+1}-(\pi_q-\pi_q^q)\Id)}  \hat u_{q+1}\right) \notag\\
    &\qquad \qquad + \div \overline\ph_{q+1} -E  - \div \left( \sigma_{q+1} u_{q+1} \right) \notag \\
    &\underset{\substack{\eqref{eq:def:ov:kappa}, \\ \div \hat u_{q+1} \equiv 0 \\ \nabla \bmu_{\bullet}\equiv 0}}{=}(\pa_t + \hat u_{q+1} \cdot \na ) \left( \frac 12 \tr \left( \ov R_{q+1} - \left( \pi_q - \pi_q^q + \sigma_{q+1} + \sum_{k=q+\half}^{q+\Npr} 2^{k-q-1} \delta_{k+\bn} - \frac 23 \left( \bmu_{\ov\phi_{q+1}} + \bmu_{P1} \right) \right)\Id \right) \right) \notag\\
    &\qquad\qquad \qquad  + \div\left( \left((\ov R_{q+1}-\left( \pi_q-\pi_q^q + \sigma_{q+1} + \sum_{k=q+\half}^{q+\Npr} 2^{k-q-1} \delta_{k+\bn}  - \frac 23 \left(\bmu_{\ov\phi_{q+1}} + \bmu_{P1} \right) \right)\Id\right)  \hat u_{q+1}\right) \notag\\
    &\qquad \qquad\qquad \qquad  + \div \overline\ph_{q+1} -E  - \div \left( \sigma_{q+1} (u_{q+1} - \hat u_{q+1}) \right) + \frac 32 (\pa_t + \hat u_{q+1}\cdot\nabla) \sigma_{q+1} - \bmu'_{P1} \notag \\
   &\underset{\substack{\eqref{defn:new.pr},\eqref{billystrings:eq:1}, \\ \eqref{billys:1}}}{=} (\pa_t + \hat u_{q+1} \cdot \na ) \ka_{q+1} + \div\left( {(R_{q+1}-\pi_{q+1}\Id)}  \hat u_{q+1}\right) \notag \\
   &\qquad \qquad\qquad \qquad  + \div \overline\ph_{q+1} -E  - \div \left( \sigma_{q+1} (u_{q+1} - \hat u_{q+1}) \right) + \frac 32 \underbrace{(\pa_t + \hat u_{q+1}\cdot\nabla)}_{=\pa_t + (\hat u_q + \hat w_{q+1})\cdot \nabla} \sigma_{q+1} - \bmu'_{P1} \notag \\
    &\underset{\eqref{supp.sik},\eqref{eq:rewriting:new:si}}{=} (\pa_t + \hat u_{q+1} \cdot \na ) \ka_{q+1} + \div\left( {(R_{q+1}-\pi_{q+1}\Id)}  \hat u_{q+1}\right) \notag\\
    &\qquad \qquad \qquad  + \div \ov \varphi_{q+1} - E - \underbrace{\div (\si_{q+1}(u_{q+1}-\hat u_{q+1}))}_{=:\div \phi_{P2}} + \underbrace{\frac32 (\pa_t + \hat u_{q} \cdot \nabla ) \si_{q+1} - \bmu'_{P1}}_{= \div \phi_{P1}} \notag\\
    &\qquad = (\pa_t + \hat u_{q+1} \cdot \na ) \ka_{q+1}
    + \div\left( {(R_{q+1}-\pi_{q+1}\Id)}  \hat u_{q+1}\right) +  \underbrace{ \div \ph_{q+1}}_{= \div(\ov\varphi_{q+1} +  \phi_{P1} + \phi_{P2})} - E
    \, . \label{eq:LEI:semi.final}
\end{align}\index{$\phi_{q+1}$}
Thus we have verified \eqref{ineq:relaxed.LEI}, and recalling \eqref{defn:primitive:current}, we have that \eqref{eq:LEI:decomp:basic} is verified as well.
\end{proof}

\subsection{Pressure current error}\label{sec:pressure:current}
In this subsection, we analyze the new pressure current errors\index{pressure current error} $\phi_{P1}$ and $\phi_{P2}$ defined in \eqref{eq:LEI:semi.final}.

\subsubsection{Pressure current error I}
Recalling the definition of $\phi_{m,q+1}$ from \eqref{defn:phi.m.m'}--\eqref{exp.Dtq.si} and the definition of $\phi_{P1}$ from \eqref{eq:monday:morning}, we decompose $\phi_{P1}$ into \index{$\phi_{P1}$}
\begin{align}
     \phi_{P1} &= \frac32\sum_{m=q+\half{+1}}^{q+\bn} \sum_{m'=q+\half+1}^m \left( A_{m,q} {\phi^{m',l}_{m,q+1}} + A_{m,q}{\phi^{m', *}_{m,q+1}} \right) \nonumber\\
    &= \frac32 \sum_{m'=q+\half{+1}}^{q+\bn} \left( \sum_{m=m'}^{q+\bn} A_{m,q} \phi^{m',l}_{m,q+1} + \sum_{m=m'}^{q+\bn}A_{m,q}\phi^{m', *}_{m,q+1} \right) \notag\\
    &=: \sum_{m'=q+\half{+1}}^{q+\bn} \phi^{m',l}_{P1} + \phi_{P1}^{m',*} \, .
    \label{eq.def.upgrade1}
\end{align}

\begin{lemma}[\bf Properties of $\phi_{P1}$]\label{lem:phi:p:1}
For all $q+\half{+1} \leq m \leq q+\bn$, the terms $\phi_{P1}^{m,l}$ and $\phi_{P1}^{m,*}$ satisfy the following properties.
\begin{enumerate}[(i)]
    \item The local part $\phi_{P1}^{m,l}$ satisfies 
\begin{subequations}
\begin{align}
    \left|\psi_{i,q} D^N \Dtq^M \phi_{P1}^{m,l}\right| &< \Ga_{m}^{-80} (\pi_{q}^{m})^{\sfrac32} {r_{m}^{-1}} \Lambda_m^N \MM{M,\Nindt,\tau_q^{-1}\Ga_q^{i+18}, \Tau_q^{-1}\Ga_q^{11}}
        \label{est:phi.up1.pt}\\
         \supp \phi_{P1}^{m,l} &\cap B(\supp \hat w_{q'}, \la_{q'}^{-1} {\Ga_{q'}}) = \emptyset, \qquad q+1\leq q' \leq m-1
         \label{supp.current.p1}
    \end{align}
    \end{subequations}
    for all $N,M <\sfrac{\Nfin}{200}$. 
    \item For all $N, M \leq 2\Nind$, the non-local part $\phi_{P1}^{m,*}$ satisfies
    \begin{align}
  \norm{\psi_{i,q} D^N \Dtq^M \phi_{P1}^{m,*}}_\infty &\lec
        \delta_{q+3\bn}^{\sfrac 32} \Tau_{q+\bn}^{2\Nindt} \Lambda_m^N \MM{M,\Nindt,\tau_q^{-1},\Tau_q^{-1}\Ga_q^9}
    \end{align}
\end{enumerate}
\end{lemma}
\begin{remark} 
Applying Lemma \ref{lem:upgrading.material.derivative} to $F^\ell = \phi_{P1}^{k,l}$ and $F^*= \phi_{P1}^{k,*}$, for $q+\half{+1} \leq k \leq q+\bn$ we can upgrade the material derivatives in the estimates \eqref{est:phi.up1.pt} to obtain that for $N+M \leq 2\Nind$,
\begin{align}
        \left|\psi_{i,k-1} D^N D_{t,k-1}^M \phi_{P1}^k\right| &< \Ga_{k}^{-60} (\pi_{q}^{k})^{\sfrac32} {r_{k}^{-1}} \Lambda_k^N \MM{M,\Nindt,\tau_{k-1}^{-1}\Ga_{k-1}^i, \Tau_{k-1}^{-1}} \, .
        \label{est:phi.up1.pt:upgraded}
\end{align}
\end{remark}
\begin{proof}[Proof of Lemma~\ref{lem:phi:p:1}]
From the definition \eqref{defn:A} of $A_{m,q}$, we have that $A_{m,q}\leq \Ga_m^{10}$. Recalling the definition \eqref{defn:phi.m.m'} of  $\phi^{m'}_{m,q+1}$, it is immediate that $\phi^{m'}_{m,q+1}$ satisfy the same properties delineated in Lemma~\ref{lem:pr.current.vel.inc} (for the current errors associated to velocity pressure increments), Lemma~\ref{lem:stress:pressure:current.stress} (for the current errors associated to stress error pressure increments), and Lemmas~\ref{lem:currentoscillation:pressure:current}, \ref{lem:ctn:pressure:current}, and \ref{lem:currentdivergencecorrector:pressure:current} (for the  current errors associated to current error pressure increments). Therefore, the lemma follows from the definition \eqref{eq.def.upgrade1}.
\end{proof}

\subsubsection{Pressure current error II}
Here we deal with the error
\begin{align*}
   \div \phi_{P2} &= \div (\sigma_{q+1} (u_{q+1}-\hat u_{q+1})) \, .
\end{align*}
Recall from \eqref{eq:hat:no:hat} at level $q+1$ that $$u_{q+1} - \hat {u}_{q+1} = \sum_{q'=q+2}^{q+\bn} \hat {w}_{q'} \, .$$ 
By the definition of $\si_{q+1} $ in \eqref{def.new.si} and the dodging properties \eqref{supp.si.m}, \eqref{supp.si.m+}, and \eqref{supp.sik}, we first write
\begin{align*}
\sigma_{q+1}\sum_{q'=q+2}^{q+\bn} \hat {w}_{q'}
&= \sum_{q'=q+{\sfrac {\bn}2}+1}^{q+\bn}\sigma_{q+1} \hat {w}_{q'}\\
&=\sum_{m=q+\half+1}^{q+\bn}
\sum_{q'=m+1}^{q+\bn}
\td\sigma_{m,q+1}^{+} \hat {w}_{q'}
+ \sum_{m=q+{\sfrac {\bn}2}+1}^{q+\bn}\td\sigma_{m,q+1}^{+}\hat {w}_{m}
-\sum_{q'=q+{\sfrac {\bn}2}+1}^{q+\bn}\sigma_{q+1}^- \hat {w}_{q'}\, .
\end{align*}
Since $\div \left(\sigma_{q+1} \sum_{q'=q+2}^{q+\bn}\hat {w}_{q'}\right)$ has zero mean, we recall the identity \eqref{exp.w.q'} at level $q+1$ and set
\begin{align}
\phi^{q'}_{P2}
&:= \underbrace{\left[ \sum_{m=q+\half+1}^{q'-1}
(\divH+ \divR)(\na \td\sigma_{m,q+1}^{+} \cdot (\div^\dpot \hat \upsilon_{q'}) ) \right]}_{=:\phi_{P21}^{q'}}
+ \underbrace{\td\sigma_{q',q+1}^{+}\hat {w}_{q'}}_{=:\phi_{P22}^{q'}} -
\underbrace{(\divH+ \divR)(\na \sigma_{q+1}^-\cdot (\div^\dpot \hat \upsilon_{q'}))}_{=:\phi_{P23}^{q'}} \notag\\
&\qquad \qquad + \underbrace{\left[ \sum_{m=q+\half+1}^{q'-1}
\divR(\na \td\sigma_{m,q+1}^{+} \cdot \hat e_{q'}) \right] -
\divR(\na \sigma_{q+1}^-\cdot \hat e_{q'})}_{=:\phi_{P*}^{q'}} \label{with:a:tag} \\
\phi_{P2} &:= \sum_{q'=q+\half+1}^{q+\bn} \phi_{P2}^{q'} \, . \notag 
\end{align}\index{$\phi_{P2}$}
As before, the terms including $\divR$ are non-local and the rest are local. 

\begin{lemma}[\bf Properties of $\sigma_{P2}^m$]\label{lem:nezezzary}
For $q+\half +1 \leq m\leq q+\bn$, the error $\phi^{m}_{P2}$ has no pressure increment; for $N+M \leq{2\Nind}$, it satisfies
\begin{align}
\left|\psi_{i,m-1} D^N D_{t,m-1}^M \phi^{m}_{P2} \right| &< \Ga_{m}^{-80} \left(\pi^{m}_{q+1}\right)^{\sfrac32}r_m^{-1} \La_{m}^N \MM{M,\Nindt, \tau_{m-1}^{-1} \Ga_{m-1}^{i+3}, \Tau_{m-1}^{-1}\Ga_{m-1}^2} \, .
\label{est:phi.up2.pt:upgraded}    
\end{align}
\end{lemma}
\begin{proof} We divide the proof up into cases based on the decomposition in \eqref{with:a:tag}.
\smallskip

\noindent\texttt{Step 1: Estimate of $\phi_{P22}^m$.} By \eqref{est.td.si.m.pl.ptwise}, the definition of $A_{m,q}$ and $\tilde \sigma_{m,q+1}^+$ in Definition~\ref{def:new:pressure}, Definition~\ref{def:new:presh}, which shows that $\pi_{q+1}^m \geq \sigma_{m,q+1}^+ \geq \td \sigma_{m,q+1} \Ga_m^{-25}$, \eqref{eq:ind:velocity:by:pi} at level $q+1$ to bound $\hat w_m$, and \eqref{ineq:r's:eat:Gammas} to absorb errant factors of $\Ga_m$, we have that for $N+M\leq 2\Nind$ and $q+\half+1\leq m \leq q+\bn$,
\begin{align*}
    \left|\psi_{i,m-1}D^N D_{t,m-1}^M (\td\sigma_{m,q+1}^{+}\hat {w}_{m})\right|
    &\lec 
    \sum_{\substack{N_1+N_2=N\\M_1+M_2=M}}
    \left|\psi_{i,m-1} D^{N_1} D_{t,m-1}^{M_1}\td\sigma_{m,q+1}^{+}\right| \left|\psi_{i,m-1}
    D^{N_2} D_{t,m-1}^{M_2}
    \hat {w}_{m}\right|\\
    &\lec (\td\sigma_{m,q+1}^{+}+\de_{q+3\bn}) \left( \pi_{q+1}^m \right)^{\sfrac12} \Ga_{q+1}r_{m-\bn}^{-1}  \Lambda_m^N \\
    &\qquad \times\MM{M, \Nindt, \tau_{m-1}^{-1}\Ga_{m-1}^{i}, \Tau_{m-1}^{-1}\Ga_{m-1}^2}\\
    &\lec \Ga_m^{-101} \left( \pi_{q+1}^m \right)^{\sfrac32} r_{m}^{-1}\Lambda_m^N \MM{M, \Nindt, \tau_{m-1}^{-1}\Ga_{m-1}^{i}, \Tau_{m-1}^{-1}\Ga_{m-1}^2} \, .
\end{align*}
Therefore, $\phi_{P22}^m$ satisfies the desired pointwise estimate.
\smallskip

\noindent\texttt{Step 2: Estimates of $\phi_{P21}^{q'}$, $\phi_{P23}^{q'}$, and $\phi_{P*}^{q'}$.}  We first carry out the preliminary step of upgrading material derivatives on $\nabla\sigma_{m,q+1}^{\pm}$, which will be required for all three terms. We apply Lemma~\ref{rem:upgrade.material.derivative.end} inductively to $v = \hat u_{m-1}$, $f= \na \si_{m,q+1}^\pm$, and  $w = \hat w_{m}, \dots , \hat w_{q'-1}$, and $\Omega = \supp(\psi_{i,q'-1})$. The assumptions in the remark are satisfied due to
\eqref{eq:nasty:D:vq:old}, Remark~\ref{upgrade.mat.sim.plus}, and \eqref{eq:nasty:D:wq:old}.  As a result, we have that for $q+\half+1 \leq m \leq q'-1$,
\begin{subequations}\label{tuesday:afternoon}
\begin{align}
    \left| \psi_{i,q'-1} D^N D_{t,q'-1}^M\na \si_{m,q+1}^+ \right|
    &\lec (\si_{m,q+1}^+ + \delta_{q+3\bn}) \La_{q'-1}^N
    \MM{M, \Nindt, \Ga_{q'-1}^{i+3}\tau_{q'-1}^{-1}, \Ga_{q'-1}\Tau_{q'-1}^{-1} } \\
    \left| \psi_{i,q'-1} D^N D_{t,q'-1}^M\na \si_{m,q+1}^+ \right|
    &\lec \Ga_{q+\half}^{-100} \pi_{q+1}^{q+\half} \La_{q'-1}^N
    \MM{M, \Nindt, \Ga_{q'-1}^{i+3}\tau_{q'-1}^{-1}, \Ga_{q'-1}\Tau_{q'-1}^{-1} }
\end{align}
\end{subequations}
for $N, M <\sfrac{\Nfin}{200}$. In a similar way, we have
\begin{subequations}\label{prep:last:pm}    
\begin{align}
    \norm{D^N D_{t,q'-1}^M\na \td\si_{q+1}^{m,+}}_{\infty}
    &\lec  \la_m\Ga_m^{\badshaq+2} \La_{q'-1}^N \MM{M, \Nindt, \tau_{q'-1}^{-1}\Ga_{q'-1}^{\imax+3}, \Tau_{q'-1}^{-1}\Ga_{q'-1}^{-1}}\\
    \norm{D^N D_{t,q'-1}^M\na \si^-_{q+1}}_{\infty}
    &\lec \la_{q+\half}\Ga_{q+\half}^{\badshaq-50} \La_{q'-1}^N \MM{M, \Nindt, \tau_{q'-1}^{-1}\Ga_{q'-1}^{\imax+3}, \Tau_{q'-1}^{-1}\Ga_{q'-1}^{-1}}
\end{align}
\end{subequations}
for $N,M<\sfrac{\Nfin}{200}$ and $q+\half\leq m \leq q+\bn$. The first inequality holds for $m+1\leq q' \leq q+\bn$, while the second one holds for $ q+\half+1\leq q' \leq q+\bn$.

\smallskip

\noindent\texttt{Step 2a:} We now estimate $\phi_{P21}^{m'}$ by applying Lemma \ref{rem:no:decoup:inverse:div2} with 
\begin{align*}
    &G = \na \td\si_{m,q+1}^+,\quad \vartheta = \hat\upsilon_{q'}, \quad \pi = \la_m\Ga_m^{25} \pi_{q+1}^m, \quad \pi' =\Ga_{q+1} \Ga_{q'} (\pi_{q+1}^{q'})^{\sfrac12}r_{q'-\bn}^{-1} \, , \quad M_t = \Nindt, \quad v=\hat u_{q'-1},\\
    & \Omega = \supp(\hat \upsilon_{q'}\psi_{i,q'-1}) \, , \quad \la =\la' = \la_m\Ga_m,\quad
    \Upsilon=\la_{q'}, \quad  \Lambda = \la_{q'}\Ga_{q'}, \quad
    \nu = \Ga_{q'-1}^{i+3} \tau_{q'-1}^{-1}, \quad
    \nu' = \Tau_{q'-1}^{-1}\Ga_{q'-1}^2,\\
    &\const_{G,\infty} = \la_m\Ga_m^{\badshaq+1}, \quad  \quad \const_{*,\infty} = \Ga_{q+1} \Ga_{q'}(\Ga_{q+1}\Ga_{q'}^{\badshaq+1})^{\sfrac12} r_{q'-\bn}^{-1} \,, \quad \dpot\textnormal{ as in \eqref{i:par:10}/\eqref{exp.w.q+1}} \,,  \\
    & M_\circ = N_\circ = 2\Nind \, , \quad K_\circ \textnormal{ as in \eqref{ineq:K_0}}\, ,\quad N_*=M_*=\sfrac{\Nfin}{300}   \, .
\end{align*}
Then we have that \eqref{eq:DDv} is satisfied due to \eqref{eq:nasty:D:vq:old} at level $q+1$, \eqref{eq:inv:div:extra:pointwise:noflow} is satisfied due to and Definitions~\ref{def:new:presh} and \ref{def:new:pressure} and \eqref{est.td.si.m.pl.ptwise:comm}, \eqref{eq:inv:div:extra:pointwise2:noflow} is satisfied due to \eqref{est.upsilon.ptwise} at level $q+1$, all assumptions from item~\eqref{item:nonlocal:v} in Part 4 of Proposition~\ref{prop:intermittent:inverse:div} are satisfied due to Remark~\ref{rem:lossy:choices}, \eqref{eq:inverse:div:DN:G:noflow} and \eqref{eq:DN:Mikado:density:noflow} are satisfied due to \eqref{eq:pressure:inductive} at level $q+1$, and \eqref{dpot:noflow} is satisfied due to \eqref{ineq:dpot:1}. Then from \eqref{eq:inv:div:pointwise:local}, \eqref{ineq:r's:eat:Gammas}, \eqref{i:par:11}, and \eqref{eq:ind.pr.anticipated} at level $q+1$, we have that for $q+\half + 1 \leq m \leq q+\bn$ and $m+1\leq q'\leq q+\bn$ and $N+M\leq 2\Nind$,
\begin{align*}
&\left|\psi_{i,q'-1} D^ND_{t,q'-1}^{M} \divH(\na \td\sigma_{q+1}^{m,+} \cdot \div^{\dpot}\hat \upsilon_{q'})\right| \\
&\qquad \lec \la_m\Ga_m^{25} \pi_{q+1}^m \cdot \la_{q'}^{-1} \Ga_{q+1} \Ga_{q'}\left(\pi_{q+1}^{q'}\right)^{\sfrac12} r_{q'-\bn}^{-1} (\la_{q'}\Ga_{q'})^N
     \MM{M, \Nindt, \Ga_{q'-1}^{i+3} \tau_{q'-1}^{-1}, \Tau_{q'-1}^{-1}\Ga_{q'-1}^2}\\
&\qquad \lec \Ga_{q'}^{-100} \left(\pi_{q+1}^{q'}\right)^{\sfrac32} r_{q'}^{-1} \left(\la_{q'}\Ga_{q'}\right)^N
     \MM{M, \Nindt, \Ga_{q'-1}^{i+3} \tau_{q'-1}^{-1}, \Tau_{q'-1}^{-1}\Ga_{q'-1}^2} \, .
\end{align*}    
From \eqref{eq:inverse:div:error:stress:bound:no:flow} and for the same range of $N$ and $M$, we also have that
\begin{align*}
\norm{D^ND_{t,q'-1}^{M} \divR(\na \td\sigma_{q+1}^{m,+} \cdot \div^{\dpot}\hat \upsilon_{q'})}_\infty
&\les \Ga_{\qbn}^{-100}\de_{q+3\bn}^\frac32 (\la_{q+\bn}\Ga_\qbn)^N \MM{M, \NindRt, \tau_{{q+\bn}-1}^{-1} ,  \Tau_{{q+\bn-1}}^{-1} \Ga_{\qbn-1}^2 } \, ,
\end{align*}    
concluding the proof of the desired estimates for $\phi_{P21}^{q'}$.
\smallskip

\noindent\texttt{Step 2b:} In the case of $\phi_{P23}^{q'}$, we instead set
\begin{align*}
    &G = \na \si_{q+1}^-,\quad \const_{G,\infty} = \Ga_{q+1}\la_{q+\half}\Ga_{q+\half}^{\badshaq+2}, \quad \la =\la' = \la_{q+\half}\Ga_{q+\half},\quad \pi = \la_{q+\half}\Ga_{q+\half}^{25} \pi_{q+1}^{q+\half} \,,
\end{align*}
while the remaining parameters stay the same. Concluding again as before, we have that
\begin{align*}
&\left|\psi_{i,q'-1} D^ND_{t,q'-1}^{M}\divH(\na \sigma_{q+1}^{-} \cdot \div^\dpot\hat \upsilon_{q'})\right|\\
&\qquad \lec \la_{q+\half}\Ga_{q+\half} \pi_{q+1}^{q+\half} \cdot \la_{q'}^{-1} 
\Ga_{q+1}\Ga_{q'}
\left(\pi_{q+1}^{q'}\right)^{\sfrac12} 
r_{q'-\bn}^{-1} \la_{q'}^N
     \MM{M, \Nindt, \Ga_{q'-1}^{i+3} \tau_{q'-1}^{-1}, \Tau_{q'-1}^{-1}\Ga_{q'-1}^2}\\
&\qquad \leq \Ga_{q'}^{-100} \left(\pi_{q+1}^{q'}\right)^{\sfrac32} r_{q'}^{-1} \left(\la_{q'}\Ga_{q'}\right)^N
     \MM{M, \Nindt, \Ga_{q'-1}^{i+3} \tau_{q'-1}^{-1}, \Tau_{q'-1}^{-1}\Ga_{q'-1}^2}\,, \\
&\norm{ D^ND_{t,q'-1}^{M}\divR(\na \sigma_{q+1}^{-} \cdot \div^\dpot\hat \upsilon_{q'})}_\infty
\leq \Ga_{q'}^{-100}\de_{q'+\bn}^\frac32 r_{q'}^{-1}\la_{q'+\bn}^N \MM{M, \NindRt, \Gamma_{{q'+\bn}-1}^{i+3} \tau_{{q'+\bn}-1}^{-1} ,  \Tau_{{q'+\bn-1}}^{-1}\Ga_{q'+\bn-1}^2 }\,,
\end{align*}
for $q+\half+1\leq q' \leq q+\bn$, where the range of $N$ and $M$ are the same as before.
\smallskip

\noindent\texttt{Step 2c:} Finally, we must estimate $\phi_{P*}^{q'}$. By Remark \ref{rem:inverse.div.spcial} and using \eqref{prep:last:pm}, \eqref{est.e.inf}, and \eqref{est.e.inf} for $q\mapsto q+1$ (verified in Proposition~\ref{prop:inductive:velocity:bdd:verified}), we have that for $N+M\leq 2\Nind$, 
\begin{align*}
\sum_{m=q+\half}^{q+\bn}
\sum_{q'=m+1}^{q+\bn} 
&\norm{D^ND_{t,q'-1}^{M} \divR(\na \td\sigma_{q+1}^{m,+} \cdot \hat e_{q'})}_\infty 
+
\sum_{q'=q+\floor{\sfrac {\bn}2}+1}^{q+\bn}\norm{ D^ND_{t,q'-1}^{M}\divR(\na \sigma_{q+1}^{-} \cdot \hat e_{q'})}_\infty\\
&\hspace{2cm}\leq \Ga_{q+\bn}^{-100}\de_{q+2\bn}^\frac32 r_q^{-1} \La_{q'}^N 
\MM{M, \NindRt,  \tau_{{q'+\bn}-1}^{-1} ,  \Tau_{q'-1}^{-1}\Ga_{q'-1}^2 }\,.
\end{align*}

\end{proof}

\subsection{{Inductive estimates on the new errors}}

\begin{lemma}[\bf Inductive pointwise error estimates]\label{lem:verify.ind.pressure2}
The inductive assumptions \eqref{eq:ind:stress:by:pi}, \eqref{eq:ind:current:by:pi}, and \eqref{eq:Rnnl:inductive:dtq}
are satisfied at level $q+1$.
\end{lemma}

\begin{proof}
We recall from \eqref{defn:primitive.stress} and \eqref{billystrings:eq:1} the definition of the stress error
$$ R_{q+1} = \ov R_{q+1}  + \frac 23 \left( \bmu_{\ov\phi_{q+1}} + \bmu_{P1} \right)\Id = \sum_{m=q+1}^{\qbn} \left( R_q^m + S_{q+1}^m \right)  + \frac 23 \left( \bmu_{\ov\phi_{q+1}} + \bmu_{P1} \right)\Id \, . $$ 
Recall also the definition of $$ \varphi_{q+1}  = \ov \varphi_{q+1} + \phi_{P1} + \phi_{P2} = \sum_{m=q+1}^{\qbn} \left( \varphi_q^m + \ov\phi_{q+1}^m \right) + \phi_{P1} + \phi_{P2}  $$
from \eqref{defn:primitive:current}, \eqref{eq:LEI:semi.final}, \eqref{eq.def.upgrade1}, and \eqref{with:a:tag}.  We therefore define the new current errors $\phi_{q+1}^k$ by
\begin{align*}
    \ph_{q+1}^k = \bar\varphi_{q+1}^k + \phi_{P1}^k + \phi_{P2}^k \, . 
\end{align*}
In order to prove \eqref{eq:ind:stress:by:pi} and \eqref{eq:ind:current:by:pi} at level $q+1$, we first consider the cases $q+1\leq k\leq q+\half$. Recall from Lemma~\ref{l:divergence:stress:upgrading} and Lemma~\ref{lem:oscillation.current:general:estimate2} that
\begin{align*}
    \left|\psi_{i,k-1} D^N D_{t,k-1}^M S_{q+1}^k \right| 
    &\lec \Ga^{-10}_{k} \pi_{q}^{k}
    \La_{k}^N \MM{M,\Nindt, \Ga_{k-1}^{i+19}\tau_{k-1}^{-1} , \Tau_{k-1}^{-1}\Ga_{k-1}^9}\\
    \left|\psi_{i,k-1} D^N D_{t,k-1}^M \ov\phi_{q+1}^k \right| 
    &\lec \Ga^{-10}_{k} (\pi_{q}^{k})^{\sfrac32}r_k^{-1}
    \La_{k}^N \MM{M,\Nindt, \Ga_{k-1}^{i+19}\tau_{k-1}^{-1} , \Tau_{k-1}^{-1}\Ga_{k-1}^9} \, ,
\end{align*}
where the first bound holds for $N+M\leq 2\Nind$, the second holds for $N+M\leq  \sfrac{\Nind}{4}$, and we have used the lower bound on $\pi_q^k$ given in \eqref{low.bdd.pi}. Then \eqref{eq:ind:stress:by:pi} and \eqref{eq:ind:current:by:pi} at level $q+1$ follow from the definitions recalled at the beginning of the proof, \eqref{eq:ind:stress:by:pi} and \eqref{eq:ind:current:by:pi} at level $q$, the estimates just recorded, and \eqref{eq:billystrings:p}.

In the cases when $q+\half+1 \leq k\leq q+\bn$, we upgrade the material derivatives in \eqref{est.S.k.pt.sik} and \eqref{est.si.k.pt} applying Lemma \ref{lem:upgrading.material.derivative} to $F:= S_{q+1}^k = S_{q+1}^{k,l}+S_{q+1}^{k,*}=:F^l + F^*$ and $F:= \bar\phi_{q+1}^k = \bar\phi_{q+1}^{k,l}+\bar\phi_{q+1}^{k,*}=:F^l + F^*$, obtaining that
\begin{subequations}
\begin{align}
    \left|\psi_{i,k-1}D^N D_{t,k-1}^M S^k_{q+1}\right|
    &\leq \Ga_k^{-8} \left(\si_{q+1}^k + \de_{k+\bn}\right)  (\lambda_k\Ga_k)^N  \MM{M, \Nindt, \Ga_{k-1}^{i+18} \tau_{k-1}^{-1}, \Tau_{k-1}^{-1}\Ga_{k-1}^{-1}} \label{est.S.k.pt.sik:upgrade} \\
    \left|\psi_{i,{k-1}}D^N D_{t,{k-1}}^M \ov \phi^k_{q+1} \right|
    &\leq \Ga_k^{-13} \left(\si_{q+1}^k + \mathbf{1}_{m=\qbn}\Ga_\qbn^{-50}\pi_q^{\qbn} + \mathbf{1}_{m=q+1}\Ga_{q+1}^{-50}\pi_q^{q+1} \de_{q+2\bn}\right)^{\sfrac32} r_k^{-1} \notag\\
    &\qquad \qquad \times (\lambda_k\Ga_k)^N 
    \MM{M, \Nindt, \Ga_{k-1}^{i} \tau_{k-1}^{-1}, \Tau_{k-1}^{-1}\Ga_{k-1}^{-1}} \label{est.phi.k.pt.sik:upgrade}
\end{align}
\end{subequations}
for $N + M \leq 2\Nind$. In addition, recalling the definitions in \eqref{defn:bmu.m}, \eqref{eq:monday:morning}, and \eqref{eq:meanz:def}, and the estimates given in Lemmas~\ref{lem:pr.inc.vel.inc.pot}, \ref{lem:prop.si.pre.stress}, \ref{lem:currentoscillation:pressure:current}, \ref{lem:ct:general:estimate}, \ref{lem:ctn:pressure:current}, \ref{lem:lin:current:error}, and \ref{lem:currentdivergencecorrector:pressure:current}, we have that for $M\leq 2\Nind$, 
\begin{equation}
    \left| \frac{d^M}{dt^M} \left(\bmu_{\ov\phi_{q+1}} + \bmu_{P1}\right) \right| \leq \delta_{q+\sfrac{5\bn}{2}} \MM{M,\Nindt,\tau_q^{-1},\Tau_{q+1}^{-1}} \, . \notag
\end{equation} Then, using \eqref{est:phi.up1.pt:upgraded}, \eqref{est:phi.up2.pt:upgraded}, Definitions~\ref{def:new:pressure} and \ref{def:new:presh}, and \eqref{eq:billystrings:p}, we conclude the proofs of \eqref{eq:ind:stress:by:pi} and \eqref{eq:ind:current:by:pi}.

Finally, we note that the nonlocal bounds in \eqref{eq:Rnnl:inductive:dtq} follow from \eqref{eq:sat:night:7:42}, \eqref{eq:nlstress:upgraded} from Lemma~\ref{l:divergence:stress:upgrading}, and the estimate just above.
\end{proof}

\section{Parameters}\label{sec:params}
In this section, we treat various parameters appearing in the proof that were not previously specified in subsection~\ref{sec:not.general}. We also record several associated consequences of parameter choices. For further details including the arithmetic which justifies the validity of various choices, we refer the reader to \cite[Section~11]{GKN23}.

\subsection{Definitions and inequalities}\label{sec:para.q.ind}
\begin{enumerate}[(i)]
    \item\label{i:choice:ep} Recalling the definition of $\Ga_q$ in \eqref{eq:deffy:of:gamma} and using the definitions in section~\ref{sec:not.general}, we may now choose $0<\varepsilon_\Gamma\ll (b-1)^2<1$ sufficiently small so that a number of conditions hold. For the full list of these conditions, we refer to \cite[equation~11.10]{GKN23}.  The ones which will be relevant for this paper are
\begin{subequations}
\begin{align}
    \label{eq:dodging:parameterz}
     \Gamma_{q+\bn}^3 \Gamma_q^{-2} \frac{\lambda_{q'+\half}\lambda_{q+\half}}{\lambda_q \lambda_{q'+\bn}} &\leq 1 \qquad \textnormal{for all $q'$ such that $q+\half +1 - \bn \leq q' \leq q$} \, , \\
    \label{la.beats.de}
    \left(\frac{\la_q}{\la_{q'}}\right)^{\sfrac23} \Ga_{q+\bn}^{2000+10\CLebesgue} &< \left(\frac{\de_q}{\de_{q'}}\right)^{-1} \\
    \left( \frac{r_{q+1}}{r_q} \right) \Ga_\qbn^{1000+10\CLebesgue} &\leq 1 \label{ineq:r's:eat:Gammas} \\
    \Ga_\qbn^{1000} &< \min\left(\la_q \la_\qbn^{-1}  r_q^{-2}, \la_q^{-\sfrac{1}{10}} \la_{q+1}^{\sfrac{1}{10}}, \delta_{q}^{\sfrac{1}{10}} \delta_{q+1}^{-\sfrac{1}{10}} \right) \label{eq:prepping:badshaq} \\
    \left \lceil \frac{(b^{\sfrac \bn 2-1}+\dots+b+1)^2}{\varepsilon_\Gamma(b^{\bn-1}+\dots+b+1)} \right \rceil &\geq 20 \, , \qquad 2000 \varepsilon_\Gamma b^{\bn} < 1 \, . \label{eq:prepping:badshaq:2}
    \end{align}
    \end{subequations}
    \item\label{i:par:5} Choose $\badshaq$ as
    \begin{equation}\label{eq:badshaq:choice}
        \badshaq = 3\left \lceil \frac{(b^{\sfrac \bn 2}-1)^2}{(b-1)^2\varepsilon_\Gamma(b^{\sfrac \bn 2 -1}+\dots+b+1)}  + \frac{2000 b^{\bn}}{b^{\sfrac \bn 2}-1} + \frac{4b^{\bn-1}}{(b-1)\varepsilon_\Gamma(1+\dots+b^{\sfrac \bn 2 -1})} \right \rceil \, .
    \end{equation}
    As a consequence of this definition and \eqref{eq:prepping:badshaq:2}, we have that for all $q+\sfrac \bn 2 \leq k \leq q+\bn$,
    \begin{subequations}     \label{eq:par:div:2}
    \begin{align}
      \Ga_q^{\badshaq} \la_q^2 \la_k^4 \la_{q+\half}^{-4} \la_k^{2} \la_{k-1}^{-4} < \Ga_{q+\half}^{\badshaq} \, , \quad \Gamma_q^{\badshaq} &\leq \Gamma_{q+\half}^{\badshaq} \Ga_\qbn^{-2000} \, , \\ \Gamma_q^{\badshaq+500} \Lambda_q \left( \sfrac{\la_k}{\la_{q+\half}} \right)^2 \lambda_{k-1}^{-2}\la_k &\leq \Gamma_{q+\half}^{\badshaq} \Ga_\qbn^{-200}  \, .
    \end{align}
    \end{subequations}
    \item\label{item:choice:of:alpha} Choose $\alpha=\alpha(q)\in (0,1)$ such that
    \begin{align}
     \lambda_{q+\bn}^\alpha = \Gamma_q^{\sfrac 1{10}} \, . \label{eq:choice:of:alpha}
    \end{align}
    \item\label{i:par:tau} Choose $\Tau_q$ for lossy material derivative costs, according to the formula
    \begin{align}
    \frac12 \Tau_{q-1}^{-1} = \tau_q^{-1} \Gamma_q^{\badshaq+100} \delta_q^{-\sfrac 12} r_q^{-\sfrac 23}  + \Gamma_q^{\badshaq+100}\delta_q^{-\sfrac 12}r_q^{-1} \Lambda_q^3 \, . \label{v:global:par:ineq}
    \end{align}
    \item\label{i:par:4.5} 
    For $k\geq q+\Npr$, the intermittent pressure $\pi_q^k$ becomes homogeneous in the sense that $\norm{\pi_q^k}_{L^p}\sim \norm{\pi_q^k}_{L^\infty}$. Such number $\Npr$ is chosen to satisfy \index{$\Npr$}
    \begin{align}\label{defn:Npr}
    \Ga_{q+\Npr}\La_{q+\bn}^4 \leq \Ga_{q+\Npr+1} \, .
    \end{align}
    \item\label{i:par:6} The number $\NcutSmall
    $ and $\NcutLarge$ 
    are auxiliary parameters used to define cutoff functions and pressure increments. It is large enough to absorb a ``Sobolev loss", roughly speaking. More precisely, we choose them to satisfy
\index{$\NcutSmall$, $\NcutLarge$}
\begin{subequations}\label{condi.Ncut0}
    \begin{align}
     \NcutSmall &\leq \NcutLarge \, , \label{condi.Ncut0.1} \\
     \la_\qbn^{200}\left(\frac{\Ga_{q-1}}{\Ga_q}\right)^{\frac{\NcutSmall}5}
    &\leq \min\left(\la_{q+\bn}^{-4} \de_{q+3\bn}^2, \Ga_{q+\bn}^{-\badshaq-17-\CLebesgue} \delta_{q+3\bn}^2r_q \right) \, , \label{condi.Ncut0.2} \\
    \delta_{q+\bn}^{-\sfrac 12} r_q^{-1} \Ga_{q+\bn}^{\sfrac{\badshaq}{2}+16+\CLebesgue} \left( \frac{\Ga_{q+\bn-1}}{\Ga_{q+\bn}}\right)^{\NcutLarge} &\leq \Ga_{q+\bn}^{-1} \, . \label{condi.Ncut0.3}
    \end{align}
    \end{subequations}
    \item\label{i:par:7} Choose the number $\Nindt$ of sharp material derivative estimates of velocity and errors propagated inductively
    such that
    \begin{align}\label{condi.Nindt}
        \Nindt \geq \NcutSmall, \quad       
        \Ga_q^{-\Nindt} (\tau_q^{-1}\Ga_q^{i+40})^{-\NcutSmall-1} 
        (\Tau_q^{-1}\Ga_q)^{\NcutSmall+1}\leq 1 \, .
    \end{align}
    \item\label{i:par:8} 
    The numbers of spatial/material derivative estimates of errors propagated inductively are described by $\Nind$, chosen such that\footnote{The parameters $N_g$ and $N_c$ are used in the proof of the mollification results from Lemma~\ref{lem:upgrading}.  For details we refer to \cite[Lemma~3.1, Proposition~A.24]{GKN23}.}
    \begin{subequations}
    \begin{align}
    N_g \leq N_c &\leq \frac{\Nind}{40}
    \Nindt &\leq \Nind \, , \label{eq:Nindsy:1} \\
    \left(\Ga_{q-1}^{\Nind} \Ga_q^{-\Nind}\right)^{\sfrac{1}{10}} &\leq \delta_{q+5\bn}^3 \Gamma_q^{-2\badshaq-3} r_q \, . \label{eq:Nind:darnit}
    \end{align}
    \end{subequations}
    \item\label{i:par:9} Choose $\Ndec$ such that
    \begin{align}\label{condi.Ndec0}
        (\la_{q+\bn+2}\Ga_q)^4 
        &\leq \left(
        \frac{\Ga_q^{\sfrac{1}{10}}}{{4\pi}}
        \right)^\Ndec \, , \qquad \qquad 
        \Nind \leq \Ndec \, .
    \end{align}
    \item\label{i:par:9.5} Choose $K_\circ$ large enough so that
\begin{equation}\label{ineq:K_0}
    \lambda_q^{-K_\circ} \leq \delta_{q+3\bn}^{3} \Tau_{q+\bn}^{{5\Nind}} \lambda_{q+\bn+2}^{-100} \, .
\end{equation}
\item\label{i:par:10} Choose $\dpot$ and $N_{**}$ such that
\begin{subequations}
    \begin{align}
     2\dpot + 3 &\leq N_{**} \, , \label{ineq:Nstarstar:dpot} \\
     \lambda_{q+\bn}^{100} \Gamma_q^{-\sfrac{\dpot}{200}} \Lambda_{q+\bn+2}^{5+K_\circ} \left( 1 + \frac{\max(\la_\qbn^2\Tau_q^{-1},\Lambda_q^{\sfrac 12}\Lambda_{q+\bn})}{\tau_q^{-1}} \right)^{20\Nind} &\leq \Tau_\qbn^{200\Nindt} \, , \label{ineq:dpot:1} \\
     \lambda_{q+\bn}^{100} \Gamma_q^{-\sfrac{N_{**}}{20}} \Lambda_{q+\bn+2}^{5+K_\circ} \left( 1 + \frac{\max(\la_\qbn^2\Tau_q^{-1},\Lambda_q^{\sfrac 12}\Lambda_{q+\bn})}{\tau_q^{-1}} \right)^{20\Nind} &\leq \Tau_\qbn^{20\Nindt} \, . \label{ineq:Nstarz:1}
\end{align}
\end{subequations}
\item\label{i:par:11} The number $\Nfin$ quantifies the maximum number of spatial/material derivatives used throughout the scheme. We choose it so that
\begin{subequations}
    \begin{align}
        2\Ndec + 4 + 10\Nind &\leq \sfrac{\Nfin}{40000} - \dpot^2 - 10\NcutLarge - 10\NcutSmall - N_{**} - 300 \, . \label{condi.Nfin0}
        \end{align}
    \end{subequations}
    
\item\label{i:choice:of:a} Having chosen all the parameters mentioned in subsection~\ref{sec:not.general} and  items~\eqref{i:par:5}--\eqref{i:par:11} except for $a$, there exists a sufficiently large parameter $a_*$ such that $a_*^{(b-1)\varepsilon_\Gamma b^{-2\bn}}$ is at least fives times larger than \emph{all} the implicit constants throughout the paper, as well as those which have been suppressed in the computations in this section. Choose $a$ to be any natural number larger than $a_*$.
\end{enumerate}

\subsection{A few more inequalities}
As a consequence of parameter choices, a number of additional inequalities hold.  For the full list, we refer to \cite[section~11.2]{GKN23}.  In this paper, we shall use that all $q+\half-1 \leq m \leq m' \leq  \qbn$,
\begin{subequations}
\begin{align}
 \Ga_q^{500+5\CLebesgue} \la_q \left( \frac{\delta_{q+\bn}}{\delta_{m+\bn}} \right)^{\sfrac 32} \La_q^{\sfrac 23} \left(\la_{m'-1}^{-2} \la_{m'}\right)^{\sfrac 23} \left( \frac{\min(\la_{m},\la_\qbn)\Ga_q}{\la_{q+\half}} \right)^{\sfrac 43}  \la_{m-1}^{-2} \la_{m} &\leq  \Ga_q^{-250} \, , \label{ineq:in:the:afternoon} \\
    \label{ineq:in:the:morning}
    \la_q \Ga_q^{250} \La_q^{\sfrac 23} \left( \frac{r_{q+\half+1}}{r_q} \right)^{\sfrac 23} \la_{q+\half}^{-\sfrac 23} \left( \frac{\la_{q+\half+1}\Ga_q}{\la_{q+\half}} \right)^{\sfrac 43} \la_{q+\half}^{-1} \delta_\qbn^{{\sfrac 32}} &\leq \delta^{{\sfrac 32}}_{q+\bn+\half+1} \, , \\
\label{eq:desert:ineq}
   \delta_\qbn  \Ga_q^{500} \La_q^{\sfrac 23} \left(\la_{m-1}^2\lambda_m^{-1}\right)^{-\sfrac 23} \leq \delta_{m+\bn} \quad \textnormal{for} \quad q+\half-5 \leq m &\leq q+\bn+5 \, .
\end{align}
\end{subequations}

\appendix

\section{Appendix}\label{sec:app}
In this section, we collect a number of useful tools from \cite{BMNV21} and \cite{GKN23}.

\subsection{Commutators and sums and products of operators}

We recall \cite[Lemmas A.6, A.7]{GKN23}; for the original versions we refer to \cite[Lemmas~A.10, A.14]{BMNV21}.
\begin{lemma}
\label{lem:cooper:2}
 Let $p\in [1,\infty]$. 
Fix $N_x,N_t,N_*,M_* \in \N$, let $v$ be a vector field, let $D_t = \partial_t + v\cdot \nabla$, and let $\Omega$ be a space-time domain. Assume that the vector field $v$ obeys 
\begin{align}
\norm{D^N D_t^M D v}_{L^\infty(\Omega)} \les \const_v \MM{N+1,N_x,\lambda_v,\tilde \lambda_v} \MM{M,N_t,\mu_v,\tilde \mu_v}
\label{eq:cooper:2:v:0}
\end{align}
for $N \leq N_*$  and $M \leq M_*$.
Moreover, let $f$ be a function which obeys
\begin{align}
\norm{D^N D_t^M f}_{L^p(\Omega)} \les \const_f \MM{N,N_x,\lambda_f,\tilde \lambda_f} \MM{M,N_t,\mu_f,\tilde \mu_f}
\label{eq:cooper:2:f:0}
\end{align}
for all $N\leq N_*$ and $M \leq M_*$. 
Denote
\begin{align*}
\lambda = \max\{ \lambda_f,\lambda_v\}, \quad \tilde \lambda= \max\{\tilde \lambda_f,\tilde \lambda_v\}, \quad \mu = \max\{\mu_f,\mu_v\}, \quad \tilde \mu = \max\{\tilde \mu_f,\tilde \mu_v\}.
\end{align*}
Let $m,n,\ell \geq 0$ be such that $n+\ell \leq N_*$ and $m\leq M_*$. 
Then the commutator $[D_{t}^m,D^n]$ is bounded by
\begin{align}
\norm{D^\ell \left[ D_t^m,D^n \right] f}_{L^{p}(\Omega)} 
&\les \const_f \MM{\ell+n,N_x,\lambda,\tilde \lambda}   \MM{m ,N_t,\max\{\mu,\const_v \tilde \lambda_v\},\max\{\tilde \mu,\const_v \tilde \lambda_v\}}.
\label{eq:cooper:2:f:1}
\end{align}
\end{lemma}

\begin{lemma}\label{rem:upgrade.material.derivative.end}
Fix $p\in [1,\infty]$, $N_x,N_t,N_*  \in \N$, and a space-time domain $\Omega \in \T^d \times \R$. Assume that for all $N,M \leq N_*$, the vector field $v$ and function $f$ obey \eqref{eq:cooper:2:v:0} and \eqref{eq:cooper:2:f:0}. Define $D_t = \pa_t + (v\cdot \na)$. Next, let $w$ be a vector field such that for $k\geq 1$ and $\alpha,\beta \in \N^k$ with $|\alpha|+ |\beta| \leq N_*$, we have that
\begin{align}
\norm{ \left(\prod_{i=1}^k D^{\alpha_i} {D_t^{\beta_i}} \right) w}_{L^\infty(\Omega)} \les \const_w \MM{|\alpha|,N_x,\lambda_w,\tilde \lambda_w} \MM{|\beta|,N_t,\mu_w,\tilde \mu_w}
\label{eq:cooper:w}
\end{align}
for some $\const_w\geq 0$, $1\leq \lambda_w \leq \tilde \lambda_w$, and $1\leq \mu_w \leq \tilde \mu_w$.
Then, we have that for all $N+M\leq N_*$, 
\begin{align}
\norm{D^N (D_t + (w\cdot \na))^M f}_{ L^p (\Omega)} 
\les \const_f  \MM{N,N_x,\lambda,\tilde \lambda}  \MM{M,N_t, \mu, \tilde \mu } \, ,
\label{eq:cooper:f:mat}
\end{align}
where 
\begin{align*}
\lambda = \max\{ \lambda_f,\lambda_v, \lambda_w\}, \quad \tilde \lambda&= \max\{\tilde \lambda_f,\tilde \lambda_v,  \tilde \lambda_w\}, \quad \mu = \max\{\mu_f,\mu_v, \mu_w, 
\const_v\td\la_v,
\const_w \td\la_w \}, \\
\tilde \mu &= \max\{\tilde \mu_f,\tilde \mu_v, \tilde \mu_w, \const_v\td\la_v, \const_w \td\la_w\} \, .
\end{align*}
\end{lemma}

\subsection{Upgrading material derivatives}
We now recall \cite[Lemma~A.23]{GKN23}, which upgrades material derivatives for functions decomposed into a ``localized'' piece with support assumptions and a ``nonlocal'' remainder.
\begin{lemma}[\bf Upgrading material derivatives]\label{lem:upgrading.material.derivative}
Fix $p\in [1, \infty]$ and a positive integer $N_\star\leq \sfrac{3\Nfin}4$. Assume that a tensor $F$ is given with a decomposition $F = F^l + F^*$ which satisfy
\begin{subequations}
\begin{align} 
\left\| \psi_{i,q} D^N D_{t,q}^M F^l \right\|_{p}
&\lesssim 
\const_{p,F}  \lambda_{F}^N  \MM{M,\Nindt,\Gamma_{q}^{i+c} \tau_{q}^{-1}, \Gamma_{q}^{-1}\Tau_{q}^{-1} }
\label{eq:before.upgraded}\\
\left\| D^N D_{t,q}^M F^* \right\|_{\infty} 
&\lesssim \const_{*,F}\Tau_{q+\bn}^{\Nindt}\lambda_{F}^N\tau_q^{-M} 
\label{eq:before.upgraded.uniform}
\end{align}
\end{subequations}
for all $M+N\leq N_{\star}$, an absolute constant {$c\leq 20$}, and constants $\const_{p,F}$ and $\const_{*,F}$. Assume furthermore that there exists $k$ such that $q+1<k\leq q+\bn$ and
\begin{align}\label{supp.F}
    \supp(\hat w_{q'}, \la_{q'}^{-1}{\Ga_{q'}}) \cap \supp(F^l) =\emptyset \quad \forall q+1\leq q' < k \, .
\end{align}
Finally, assume that
\begin{equation}\label{eq:timescale:upgrading}
\la_F \Ga_{q+\bn}^{\imax+2}\de_{q+\bn}^{\frac12}r_q^{-\frac13}\leq \Tau_{q+\bn}^{-1} \, .
\end{equation}
Then $F$ obeys the following estimate with an upgraded material derivative for all $M+N\leq N_{\star}$;
\begin{align}
&\left\| \psi_{i,k-1} D^N D_{t,k-1}^M F \right\|_{p} \lesssim \
(\const_{p,F} + \const_{*,F}) \max(\lambda_{F}, \La_{k-1})^N
  \MM{M,\Nindt,\Gamma_{k-1}^{i} \tau_{k-1}^{-1}, \Gamma_{k-1}^{-1}\Tau_{k-1}^{-1} } \, .
\label{eq:after.upgraded}
\end{align}
In particular, the nonlocal part $F^*$ obeys better estimate
\begin{align}\label{est:nonlocal:upgrade}
\norm{D^N D_{t,k-1}^M F^*}_\infty
    \lec \const_{*,F}\max(\lambda_{F}, \la_{k-1}\Ga_{k-1})^N 
    \MM{M, \Nindt, \tau_{k-1}^{-1},  \Tau_{k-1}^{-1}\Gamma_{k-1}^{-1}}
\end{align}
for $N+M\leq N_\star$. Similarly, if instead of \eqref{eq:before.upgraded}, $F^l$ satisfies 
\begin{align}
    \left|\psi_{i,q} D^N D_{t,q}^M F^l \right| 
&\lesssim 
\pi_{F}  \lambda_{F}^N  \MM{M,\Nindt,\Gamma_{q}^{i+c} \tau_{q}^{-1}, \Gamma_{q}^{-1}\Tau_{q}^{-1} }
\label{eq:before.upgraded.pt}
\end{align}
for all $M+N\leq N_{\star}$, an absolute constant ${c\leq 24}$, and a positive function $\pi_F$ with $\pi_F \geq \const_{*,F}$, we have
\begin{align}
&\left| \psi_{i,k-1} D^N D_{t,k-1}^M F \right| \lesssim \
\pi_{F}\max(\lambda_{F}, \La_{k-1})^N
  \MM{M,\Nindt,\Gamma_{k-1}^{i} \tau_{k-1}^{-1}, \Gamma_{k-1}^{-1}\Tau_{k-1}^{-1} }
\label{eq:after.upgraded.pt}
\end{align}
for all $M+N\leq N_{\star}$, provided that \eqref{eq:timescale:upgrading} holds.
\end{lemma}

\subsection{Synthetic Littlewood-Paley decomposition}\label{sec:LP}
A synthetic Littlewood-Paley decomposition replaces the convolution kernel from the standard Littlewood-Paley decomposition with a compactly supported kernel, so that the spatial support of the output is controlled. We then replace the standard Littlewood-Paley decomposition of a function into frequency shells with a decomposition using so-called synthetic Littlewood-Paley projectors, notated by
\begin{equation}
    \tilde{\mathbb{P}}_{\lambda_0} (\rho) + \left( \sum_{k=1}^{K}  \tilde{\mathbb{P}}_{(\lambda_{k-1},\lambda_k]} (\rho) \right) + \left( \Id -  \tilde{\mathbb{P}}_{\lambda_K} \right)(\rho) \, . \label{eq:decomp:showing} 
\end{equation}
As we recall in this section, each synthetic projector above obeys similar derivative estimates as the standard frequency projectors, but has controlled spatial support.  We recall from \cite[Section~4.3]{GKN23} the definition and give two lemmas showing that effective ``inverse divergences'' can be defined for functions to which the synthetic Littlewood-Paley decomposition has been applied.

\begin{definition}[\bf Synthetic Littlewood-Paley projector]\label{def:synth:LP} \index{synthetic Littlewood-Paley projector}
Let $\bph\in C_c^\infty(\sfrac{-1}{\sqrt 2}, \sfrac{1}{\sqrt 2})$ have unit mass $\int_{\R} \bph d s =1$ and have vanishing moments $\int_{\R} s^n \bph ds =0$ for $n= 1, \dots, 10\Nfin$. Define $\bph_\la(\cdot) = \la \bph (\la \cdot)$ and set $\ph_\la(x) = \bph_\la(x_1)\bph_\la(x_2)$. For $f \in C^\infty(\T^2)$ (which we identify with a periodic function defined on $\R^2$), we define the \emph{synthetic Littlewood-Paley projectors} by
\begin{align}
\tP_{\la}f(x):= \int_{\R^2} \ph_\la (y) f(x-y)  dy \, , \qquad
\tP_{(\la_1, \la_2]}f 
:= (\tP_{\la_2}- \tP_{\la_1})f \, . \label{eq:synth:LP} 
\end{align} \index{$\tP_{\la}$, $\tP_{(\la_1, \la_2]}$}
\end{definition}

\begin{lemma}[\bf Inverse divergence with synthetic LP projectors: lowest and highest shells]\label{lem:special:cases}
Fix \\ $q\in [1,\infty]$. Let $\Nblank$ a positive integer, $N_{**}\leq \sfrac{\Nblank}{2}$ a positive integer, $r,\lambda$ such that $\lambda r, \lambda \in \mathbb{N}$, and $\rho:(\sfrac{\T}{\lambda r})^2\rightarrow \R$ a smooth function such that there exists a constant $\const_{\rho,q}$ with
\begin{equation}\label{eq:moll:1:as}
 \left\| D^N \rho \right\|_{L^q(\T^2)} \lesssim \const_{\rho,q}  \lambda^N \, .
\end{equation}
for $N\leq \Nblank$.  Let $\lambda_0,\lambda_K$ be given with $\lambda r < \lambda_0 < \lambda < \lambda_K$. If the kernel $\overline{\varphi}$ used in Definition~\ref{def:synth:LP} has $N_{**}$ vanishing moments, then for $p\in[q,\infty]$ we have that
\begin{subequations}
\begin{align}
    \left\| D^N \left( \tilde{\mathbb{P}}_{\lambda_0} \rho \right) \right\|_{L^p} &\lesssim \const_{\rho,q} \left( \frac{\lambda_0}{\lambda r} \right)^{\sfrac 2q -\sfrac 2p} \lambda_0^N \qquad \qquad \qquad \forall N \leq \Nblank \, , \label{eq:lowest:shell:estimates}\\
    \left\| D^N \left( \left( \Id - \tilde{\mathbb{P}}_{\lambda_K} \right) \rho \right) \right\|_{L^\infty} &\lesssim \left(\frac{\lambda}{\lambda_K}\right)^{N_{**}} \const_{\rho,q} \lambda^{N+3} \qquad \forall N\leq \Nblank-N_{**}-3 \, .  \label{eq:remainder:estimates}
\end{align}
\end{subequations}
Furthermore, for any chosen positive even integer $\Dpot$ and any small positive number $\alpha$, there exist rank-$\Dpot$ tensor potentials $\vartheta_0$ and $\vartheta_K$ such that for $0\leq k \leq \Dpot$ and $N$ in the same range as above,
\begin{subequations}
\begin{align}
    \div^\Dpot \vartheta_0 &= \tilde{\mathbb{P}}_{\lambda_0} \mathbb{P}_{\neq 0}\rho \, , \qquad\qquad  \left\| D^N \div^{k} \vartheta_0 \right\|_{L^p} \lesssim \lambda_0^\alpha \const_{\rho,q} \left( \frac{\lambda_0}{\lambda r} \right)^{\sfrac 2q -\sfrac 2p} (\lambda r)^{k-\Dpot} \MM{N,\Dpot-k,\lambda r,\lambda_0} \, , \label{eq:lowest:shell:inverse} \\
    \div^\Dpot \vartheta_K &= (\Id - \tilde{\mathbb{P}}_{\lambda_K}) \rho \, , \qquad \left\| D^N \div^{k} \vartheta_K \right\|_{L^\infty} \lesssim \left(\frac{\lambda}{\lambda_K}\right)^{N_{**}} \const_{\rho,q}\lambda^{3} (\lambda r)^{k-\Dpot} \MM{N,\Dpot-k,\lambda r,\lambda}  \, .  \label{eq:remainder:inverse}
\end{align}
\end{subequations}
The implicit constants above depend on $\alpha$ but do not depend on $\lambda$, $\lambda_0$, $\lambda_K$, or $r$.
\end{lemma}

\begin{lemma}[\bf Inverse divergence with synthetic LP projectors: intermediate shells]\label{lem:LP.est}
Fix \\ $q\in [1,\infty]$. Let $\rho:\T^2\rightarrow \R$ be a smooth function which is $\left( \sfrac{\T}{\lambda r} \right)^2$-periodic and for $N\leq 2\Nfin$ satisfies
\begin{equation}
    \left\| D^N \rho \right\|_{L^q(\T^2)} \lesssim \const_{\rho,q} \lambda^N \, . \label{eq:gen:id:assump}
\end{equation}
Then for any $\lambda r < \lambda_1 < \lambda_2$ and a given $\Dpot\geq 1$, there exists a tensor field
$\Theta_\rho^{\la_1, \la_2}:\T^2 \to \R^{(2^\Dpot)}$ such that for $p\in[q,\infty]$, $0\leq k \leq \Dpot$, $0<\alpha\ll 1$, and $N\leq \Nfin$, we have
\begin{subequations}
\begin{align}
    \left(\lambda_1^{-1}\div\right)^{(\Dpot)}\Theta_\rho^{\la_1, \la_2} &= \tP_{(\la_1, \la_2]}(\rho) = \tP_{(\la_1, \la_2]}(\rho- \langle \rho \rangle)
     \label{eq:LP:equality} \\
    \norm{D^N\pa_{i_1\cdots i_{\Dpot-k}}
    (\lambda_1^{-\Dpot}\Theta_{\rho}^{\la_1, \la_2})^{(i_1, \cdots, i_\Dpot)}}_{L^p(\T^2)} &\lec_{\Dpot,\alpha} \const_{\rho,q} \left(\frac{\min\left(\lambda,\la_2\right)}{\la r}\right)^{\frac 2q -\frac 2p+\alpha}  \la_1^{-k} \min\left(\lambda,\la_2\right)^N  \, , \label{eq:LP:div:estimates} \\
    \supp(\Theta_{\rho}^{\la_1, \la_2}) &\subset B(\supp(\rho) , \la_1^{-1}) \, .  \label{eq:LP:div:support}
\end{align}
\end{subequations}
The implicit constants above depend on $\alpha$ but do not depend on $\lambda$, $\lambda_1$, $\lambda_2$, or $r$.
\end{lemma}

\subsection{Inversion of the divergence}
Our general strategy for solving the divergence equation $\div \varphi = g - \langle g \rangle$ follows the approach of \cite[Propositions~A.17, A.18]{BMNV21}. Specifically, we consider $g = G \varrho\circ\Phi$, where $G$ has a significantly lower effective frequency than the minimum frequency of $\varrho$. The solution includes a local part $\divH g$, and a nonlocal part $\divR g$ that is negligible in size. For the proofs we refer the reader to \cite[Propositions~A.17, A.18]{BMNV21} or \cite[Section~A.3]{GKN23}.

\begin{proposition}[\bf Main inverse divergence operator]\index{inverse divergence}
\label{prop:intermittent:inverse:div}
Let $n\geq 2$ and $p\in[1,\infty]$ be free parameters.
\smallskip

\noindent\textbf{Part 1: Low-frequency assumptions}
\begin{enumerate}[(i)]
\item\label{item:cond.G.inverse.div} Let $G$ be a vector field and assume there exist a constant $\const_{G,p} > 0$ and parameters 
\begin{equation}\label{eq:inv:div:NM}
N_*\geq M_*\geq 1 \, ,
\end{equation}
$M_t$, and $\lambda, \nu,\nu' \geq 1$ such that for all $N \leq N_*$ and $M \leq M_*$,
\begin{align}
\norm{D^N D_{t}^M G}_{L^p}&\lesssim \const_{G,p} \lambda^N\MM{M,M_{t},\nu,\nu'} \, .
\label{eq:inverse:div:DN:G}
\end{align}
\item Fix an incompressible vector field $v(t,x):\R\times\T^n\rightarrow \R^n$ and denote its material derivative by $D_t = \partial_t + v\cdot\nabla$. Let $\Phi$ be a volume preserving diffeomorphism of $\T^n$ such that 
\begin{align}
D_t \Phi = 0 \,
\qquad \mbox{and} \qquad
\norm{\nabla \Phi - \Id}_{L^\infty(\supp G)} \leq \sfrac 12 \,. \label{eq:DDpsi2}
\end{align}
Denote by $\Phi^{-1}$ the inverse of the flow $\Phi$,  which is the identity at a time slice which intersects the support of $G$.
Assume that  the velocity field $v$ and the flow functions $\Phi$ and $\Phi^{-1}$ satisfy the bounds 
\begin{subequations}
\begin{align}
\norm{D^{N+1}   \Phi}_{L^{\infty}(\supp G)} + \norm{D^{N+1}   \Phi^{-1}}_{L^{\infty}(\supp G)} 
&\les \lambda'^{N}
\label{eq:DDpsi}\\
\norm{D^ND_t^M D v}_{L^{\infty}(\supp G)}
&\les \nu \lambda'^{N}\MM{M,M_{t},\nu,\nu'}
\label{eq:DDv}
\,,
\end{align}
\end{subequations}
for all $N \leq N_*$, $M\leq M_*$, and some $\lambda'>0$. 
\end{enumerate}
\smallskip

\noindent\textbf{Part 2: High-frequency assumptions}
\begin{enumerate}[(i)]
\item\label{item:inverse:i} Let $\varrho \colon \T^n \to \R$ be a zero mean scalar function such that there exists a large positive even integer $\dpot \gg 1$ and a smooth, mean-zero, tensor potential\footnote{We use $i_j$ for $1\leq j \leq \dpot$ to denote any number in the set $\{1,\dots,n\}$.} $\vartheta^{(i_1,\dots, i_\dpot)}:\T^n \rightarrow \R^{\left(n^\dpot\right)}$ such that $\varrho(x) = \partial_{i_1}\dots\partial_{i_\dpot} \vartheta^{(i_1 \dots i_\dpot)}(x)$.
\item \label{item:inverse:ii} There exists a parameter $\mu\geq 1$ such that $\varrho$ and $\vartheta$ are $(\sfrac{\T}{\mu})^n$-periodic.
\item \label{item:inverse:iii} There exist parameters $1 \ll \Upsilon \leq \Upsilon' \leq \Lambda$, $\const_{*,p}>0$ such that for all $0\leq N \leq {N_*}$ and all $0\leq k \leq \dpot$,
\begin{align}
\norm{D^N \partial_{i_1}\dots \partial_{i_k} \vartheta^{(i_1,\dots, i_\dpot)}}_{L^p} \les \const_{*,p} \Upsilon^{k-\dpot} \MM{N, \dpot - k , \Upsilon', \Lambda} \, .
\label{eq:DN:Mikado:density}
\end{align} 
\item\label{item:inverse:iv} There exists $\Ndec$ such that the above parameters satisfy
\begin{align}
 \lambda', \lambda \ll \mu \leq \Upsilon \leq \Upsilon' \leq \Lambda  \,, \qquad \max(\lambda,\lambda') \Upsilon^{-2} \Upsilon' \leq 1 \, , \qquad N_*-\dpot \geq 2\Ndec + n+1 \, , 
 \label{eq:inverse:div:parameters:0}
\end{align}
where by in the first inequality in \eqref{eq:inverse:div:parameters:0} we mean that 
\begin{align}
\Lambda^{n+1} \left(\frac{\mu}{2\pi \sqrt{3} \max(\lambda,\lambda')}\right)^{-\Ndec} \leq 1
\,.
\label{eq:inverse:div:parameters:1}
\end{align}
\end{enumerate}
\smallskip

\noindent\textbf{Part 3: Localized output}
\begin{enumerate}[(i)] \index{$\divH$}
\item\label{item:div:local:0} There exists a vector field $R=:\divH( G \varrho \circ \Phi)$ and scalar-valued function $E$ such that
\begin{align}
G \; \varrho\circ \Phi  &=  \div R + E  
=: \div\left( \divH \left( G \varrho \circ \Phi \right) \right) + E \, . \label{eq:inverse:div}
\end{align}
\item\label{item:div:local:i} The support of $R$ is a subset of $\supp G \cap \supp \vartheta$.
\item\label{item:div:local:ii} There exists an explicitly computable positive integer $\const_\divH$, an explicitly computable function $r(j):\{0,1,\dots,\const_\divH\}\rightarrow \mathbb{N}$ and explicitly computable tensors
\begin{align*}
    &\rho^{\beta(j)} \, , \qquad \beta(j)=(\beta_1,\beta_2,\dots,\beta_{r(j)})\in  \{1,\dots,n\}^{r(j)} \, , \\
    &H^{\alpha(j)} \, , \qquad \alpha(j)=(\alpha_1,\alpha_2,\dots,\alpha_{r(j)},k,\ell) \in \{1,\dots,n\}^{r(j)+1} \, 
\end{align*}
of rank $r(j)$ and $r(j)+1$, respectively, all of which depend only on $G$, $\varrho$, $\Phi$, $n$, $\dpot$, such that the following holds.  The error $R$ can be decomposed into a sum of localized errors as
\begin{align}\label{eq:divH:formula}
    \divH^{k} (G \varrho \circ \Phi) = R^{k} = \sum_{j=0}^{\const_\divH}  H^{\alpha(j)} \rho^{\beta(j)} \circ \Phi \, ,
\end{align}
where the contraction is on the first $r(j)$ indices. Furthermore, we have that
\begin{align}
\supp H^{\alpha(j)} \subseteq \supp G  \, , \qquad \supp \rho^{\beta(j)} \subseteq \supp \vartheta \,  . \label{eq:inverse:div:linear}
\end{align}
\item\label{item:div:local:iii} For all $N \leq N_* - \sfrac \dpot 2$, $M\leq M_*$, and $j\leq \const_\divH$, we have the subsidiary estimates
\begin{subequations}\label{eq:inverse:div:sub:main}
\begin{align}
    \left\| D^N \rho^{\beta{(j)}} \right\|_{L^p} &\les \const_{*,p} \Upsilon^{-2} \Upsilon' \MM{N,1,\Upsilon',\Lambda} \label{eq:inverse:div:sub:1} \\
    \left\|D^N D_{t}^M H^{\alpha{(j)}}\right\|_{L^p} &\lesssim \const_{G,p} \left(\max(\lambda,\lambda')\right)^N \MM{M,M_{t},\nu,\nu'} \, . \label{eq:inverse:div:sub:2} 
\end{align}
\end{subequations}
\item\label{item:div:local:iv} For all $N \leq N_* - \sfrac \dpot 2$ and $M\leq M_*$, we have the main estimate
\begin{align}
\norm{D^N D_{t}^M R}_{L^p}
 &\les  \const_{G,p} \const_{*,p}  \Upsilon' \Upsilon^{-2} \MM{N,1,\Upsilon',\Lambda} \MM{M,M_{t},\nu,\nu'} 
\label{eq:inverse:div:stress:1}
\end{align}
\item\label{item:div:nonlocal} For $N \leq N_* - \sfrac \dpot 2 $ and $M\leq M_*$ the error term $E$  in \eqref{eq:inverse:div} satisfies
\begin{align}
\norm{D^N D_{t}^M E}_{L^p}  
\les \const_{G,p} \const_{*,p}   \max(\lambda,\lambda')^{\sfrac \dpot 2} \left( \Upsilon' \Upsilon^{-2}\right)^{\sfrac \dpot 2} \Lambda^{N} \MM{M,M_{t},\nu,\nu'} 
\,.
\label{eq:inverse:div:error:1}
\end{align}
\end{enumerate}
\smallskip

\noindent\textbf{Part 4: Nonlocal assumptions and output}
\begin{enumerate}[(i)]
\item 
\label{item:nonlocal:v}
Let $K_\circ$ be a positive integer, and let $N_\circ, M_\circ$ be integers such that 
\begin{equation}\label{eq:inv:div:wut}
1 \leq M_\circ \leq N_\circ \leq \sfrac{M_*}{2} \, .
\end{equation}
Assume that in addition to the bound \eqref{eq:DDv} we have the following global lossy estimates
\begin{align}
\norm{D^N \partial_t^M v}_{L^\infty}\les  \const_v \lambda'^N \nu'^M
\label{eq:inverse:div:v:global}
\end{align}
for all  $M \leq M_\circ$ and $N+M \leq N_\circ + M_\circ$, where 
\begin{align}
\const_v \lambda' \les \nu' 
\,.
\label{eq:inverse:div:v:global:parameters}
\end{align}
\item Assume that $\dpot $ is large enough so that
\begin{align}
\const_{G,p} \const_{*,p} \max(\lambda,\lambda')^{\sfrac \dpot 4} (\Upsilon' \Upsilon^{-2})^{\sfrac \dpot 4}  \Lambda^{n+2+K_\circ} \left(1 + \frac{\max\{ \nu', \const_v \Lambda \}}{\nu 
}\right)^{M_\circ}
\leq 1
\, .
\label{eq:riots:4}
\end{align}
\end{enumerate}
Then there exists a vector field $\RR_{\rm nonlocal} = \divR(G \varrho \circ \Phi)$\index{$\divR$} such that
\begin{align}
E = \div \RR_{\rm nonlocal} + \fint_{\T^3} G \varrho \circ \Phi \, dx =: \div \left(\divR(G \varrho \circ \Phi)\right) + \fint_{\T^3} G \varrho \circ \Phi \,  dx \, .
\label{eq:inverse:div:error:stress}
\end{align} 
In addition, for $N \leq N_\circ$ and $M\leq M_\circ$, $\RR_{\rm nonlocal}$ satisfies the bounds
\begin{align}
\norm{D^N D_{t}^M \RR_{\rm nonlocal} }_{L^\infty}  
\leq  \frac{1}{\Lambda^{K_\circ}} \max(\lambda,\lambda')^{\sfrac \dpot 4} (\Upsilon' \Upsilon^{-2})^{\sfrac \dpot 4} \Lambda^N \nu^M \, .
\label{eq:inverse:div:error:stress:bound}
\end{align}
\end{proposition}

\begin{remark}[\bf Estimates for $\overline{R}_{\rm nonlocal}$]\label{rem:lossy:choices}
Our choice of parameters implies that
\begin{equation}\label{eq:R:nonlocal:in:practice}
   \left\| D^N D_t^M \overline{R}_{\rm nonlocal} \right\|_{L^\infty} \leq \lambda_{q+\bn}^{-10} \delta_{q+3\bn}^2 {\Tau_{q+\bn}^{4\Nindt}} \la_\qbn^N \tau_q^{-M} 
\end{equation}
for $N,M\leq 2\Nind$. By choosing $N_\circ=M_\circ=2\Nind$ and $K_\circ$ large enough so that $\lambda_q^{-K_\circ} \leq \delta_{q+3\bn}^{2} \Tau_{q+\bn}^{4\Nindt}\la_\qbn^{-100}$, which follows from \eqref{ineq:K_0}. We refer to \cite[Remark~A.14]{GKN23} for further details.
\end{remark}

\begin{remark}[\bf Mean of the error term]\label{rem:est.mean} From \cite[Remark~A.17]{GKN23}, under the same choice of parameters suggested in Remark \ref{rem:lossy:choices}, we have
\begin{align*}
    \bigg|\frac{d^M}{dt^M} \langle  G(\varrho\circ\Phi) \rangle
    \bigg|\leq \la_{q+\bn}^{-10}\de_{q+3\bn}^2 \Tau_{q+\bn}^{4\Nindt}\tau_q^{-M} \qquad \qquad \textnormal{for $M\leq 2\Nind$} \, .
\end{align*}
\end{remark}

\begin{remark}[\bf Inverse divergence with pointwise bounds]\label{rem:pointwise:inverse:div}
From \cite[Remark~A.19]{GKN23}, if in addition there exists a smooth, non-negative function $\pi$ such that
\begin{align}\label{eq:inv:div:extra:pointwise}
\left| D^N D_{t}^M G \right| &\lesssim \pi \lambda^N\MM{M,M_{t},\nu,\nu'}
\end{align}
for $N\leq N_*$ and $M\leq M_*$, then we may additionally conclude that for $N\leq N_*- \lfloor\sfrac \dpot 2 \rfloor$ and $M\leq M_*$,
\begin{align}\label{eq:inv:div:extra:conc}
       \left|D^N D_{t}^M H^{\alpha_{(j)}}\right| &\lesssim \pi \left(\max(\lambda,\lambda')\right)^N \MM{M,M_{t},\nu,\nu'}  \, .
\end{align}
\end{remark}

\begin{remark}[\bf Special case for negligible error terms] \label{rem:inverse.div.spcial}
The inverse divergence operator defined in the proposition can be applied to an input without the structure of low and high frequency parts when $\varrho=1$ and $\const_{G,p}$ are sufficiently small. More precisely, we 
keep the low-frequency assumptions (Part 1), replace the high-frequency assumptions (Part 2) with $\varrho =1$, and set $\Upsilon= \Upsilon'= \Lambda = \max(\la,\la')$, $\const_{*,p}=1$, $\dpot =0$. Then as long as $\const_{G,p}$ satisfies \eqref{eq:riots:4}, the conclusions in Part 4 hold. In particular, $\divR G$ satisfies
\begin{align*}
    G = \div \divR G + \fint_{T^3} G \, dx \, .
\end{align*}
\end{remark}

We require the following variant of the inverse divergence operator; details may be found in \cite[Lemma~A.22]{GKN23}.
\begin{lemma}[\bf Inverse divergence without flow map]\label{rem:no:decoup:inverse:div2}
Fix dimension $n\geq 2$. Let $G$ be a smooth scalar field and let $\dpot$ be a non-negative integer such that there exists a smooth scalar field $\varrho$ and tensor field $\vartheta$ defined on $\R\times \T^n$ and satisfying $\varrho = \partial_{i_1}\dots\partial_{i_\dpot}\vartheta^{(i_1\dots i_\dpot)}(x)$. 
\smallskip

\noindent\textbf{Part 1: Algorithm for inverse divergence}\\
We have a decomposition
\begin{align}\label{decomp:noflow}
    G \varrho =: \div (\divH(G\varrho)) + E
\end{align}
where the vector field $\divH(G\varrho)$ and scalar field $E$ are defined by
\begin{align}\label{defn.divHR:noflow}
    \divH(G\varrho)^\bullet
    &:= \sum_{k=0}^{\dpot-1}
    (-1)^{\dpot-k+1}
     \pa_{i_{k+2}} \dots\pa_{i_{\dpot}} G \, \underbrace{\div^{(k)}}_{{\pa_{i_1},\dots,\pa_{i_k}}}\vartheta^{(i_1,\dots,i_k,\bullet, i_{k+2},\dots, i_\dpot)}, \quad
    E = (-1)^\dpot \na^{\dpot}G: \vartheta\, ,
\end{align}
where we use the convention $\pa_{i_{k+2}} \cdots\pa_{i_{\dpot}} G=G$ and $\vartheta^{(i_1,\dots,i_k,\bullet, i_{k+2},\dots, i_\dpot)}=\vartheta^{(i_1,\dots,i_{\dpot-1},\bullet)}
$ when $k=\dpot-1$.
\smallskip

\noindent\textbf{Part 2: Localized assumptions and output}\\
Fix a set $\Omega\subset \R\times \T^n$.  Let parameters $N_*\geq M_*\geq 1$ be given. Define $v$ and $D_t$ as in Part 1 of Proposition~\ref{prop:intermittent:inverse:div}, where $v$ satisfies \eqref{eq:DDv} with $\lambda',\nu,\nu',N_*,M_*$ and $L^\infty(\supp G)$ replaced with $L^\infty(\Omega)$. Let smooth, non-negative functions $\pi$ and $\pi'$ be given such that
\begin{subequations}
\begin{align}
\left| D^N D_{t}^M G \right| &\lesssim \pi \lambda^N\MM{M,M_{t},\nu,\nu'} \qquad \textnormal{on }\Omega \, 
\label{eq:inv:div:extra:pointwise:noflow}\\
\Upsilon^{\dpot-k} \left| D^N D_{t}^M \partial_{i_1}\dots \partial_{i_k} \vartheta^{(i_1,\dots, i_\dpot)}\right| &\lesssim \pi' \Lambda^N\MM{M,M_{t},\nu,\nu'} \qquad \textnormal{on }\Omega
\label{eq:inv:div:extra:pointwise2:noflow}
\end{align}
\end{subequations}
for $N\leq N_*$ and $M\leq M_*$, where the parameters satisfy
\begin{align}\label{parameter:noflow}
    \la', \la \leq \Upsilon\leq \La, \quad
    \max(\la, \la') \Upsilon^{-1}\leq 1, \quad
    N_* \geq \dpot, \quad \la, \nu, \nu' \geq 1 \, .
\end{align}
Then $\divH(G\varrho)$ satisfies 
\begin{equation}\label{eq:div:no:flow:support}
\supp(\divH(G\varrho))\subseteq \supp(G\vartheta) \, ,
\end{equation}
and for $N\leq N_*-\dpot$ and $M\leq M_*$,
\begin{align}
\left| D^N D_{t}^M \divH(G\varrho) \right| &\lesssim 
\pi \pi' \Upsilon^{-1}\Lambda^N\MM{M,M_{t},\nu,\nu'} \,  \quad \textnormal{on }\Omega\, .
\label{eq:inv:div:pointwise:local}
\end{align}
\smallskip

\noindent\textbf{Part 3: Nonlocal assumptions and output}\\
Finally, we assume that all assumptions from \eqref{item:nonlocal:v} in Part 4 in Proposition~\ref{prop:intermittent:inverse:div} hold.  Next, we assume that
\begin{subequations}
\begin{align}
\norm{D^N D_{t}^M G}_{L^{\infty}}&\lesssim \const_{G,\infty} \lambda^N (\nu')^M \, ,
\label{eq:inverse:div:DN:G:noflow}\\
\norm{D^ND_t^M \partial_{i_1}\dots \partial_{i_k} \vartheta^{(i_1,\dots, i_\dpot)}}_{L^\infty} &\les \const_{*,\infty} \Upsilon^{k-\dpot} \Lambda^N(\nu')^M \, 
\label{eq:DN:Mikado:density:noflow}
\end{align}
\end{subequations}
for $N\leq N_*$ and $M\leq M_*$. Also, we choose $\dpot$ large enough to satisfy
\begin{align}\label{dpot:noflow}
    \const_{G,\infty} \const_{*,\infty} (\max(\lambda,\lambda')\Upsilon^{-1})^{{\halfd}}   \Lambda^{K_\circ} \left(1 + \frac{\max\{ \nu', \const_v \Lambda \}}{\nu
}\right)^{M_\circ}\
\leq 1
\, .
\end{align}
Then we may write  
\begin{align}
E =: \div \left(\divR(G \varrho )\right) + \fint_{\T^3} G \varrho \, dx \, ,
\label{eq:inverse:div:error:stress:no:flow}
\end{align}
where $\divR(G \varrho )$ is a vector field which for $N \leq N_\circ$ and $M\leq M_\circ$ satisfies
\begin{align}
\norm{D^N D_{t}^M \divR(G \varrho ) }_{L^\infty}  
\lec  \frac{1}{\Lambda^{K_\circ}} (\max(\lambda,\lambda')\Upsilon^{-1})^{{\halfd}}\Lambda^N \nu^M \, .
\label{eq:inverse:div:error:stress:bound:no:flow}
\end{align}
\end{lemma}

\subsection{Abstract construction of intermittent pressure}
In this subsection, we recall the abstract construction of an intermittent pressure which can be used to dominate the inverse divergence (see the preceding subsection) of an error term $G \varrho \circ \Phi$. For the proof we refer the reader to \cite[Proposition~7.5]{GKN23}.
\begin{proposition}[\bf Pressure increment and upgrade error from current error]\label{lem.pr.invdiv2.c}
\noindent\textbf{Part 1: Preliminary assumptions}
\begin{enumerate}[(i)]
    \item\label{i:st:sample:1:c} There exists a scalar field $G$, constants $\const_{G,p}$ for $p=1,\infty$, and parameters $M_t,\lambda,\nu,\nu',N_*,M_*$ such that \eqref{eq:inv:div:NM} and \eqref{eq:inverse:div:DN:G} are satisfied. There exists $\pi$ smooth and non-negative and $r_G$ such that 
    \begin{align*}
        \left|D^N D_{t}^M G\right| 
    &\les \pi^{\sfrac 32} r_G^{-1} \la ^N\MM{M,M_{t},\nu,\nu'}\, . 
    \end{align*}
    \item\label{i:st:sample:2:c} There exists an incompressible vector field $v$ with material derivative $D_t=\partial_t + v\cdot \nabla$, a volume preserving diffeomorphism $\Phi$, inverse flow $\Phi^{-1}$, and parameter $\lambda'$ such that \eqref{eq:DDpsi2}--\eqref{eq:DDv} are satisfied.
    \item\label{i:st:sample:3:c} There exists a zero mean scalar function $\varrho$, a mean-zero tensor potential $\vartheta$, constants $\const_{*,p}$ for $p=1,\infty$, and parameters $\mu,{\Upsilon,\Upsilon'},\Lambda,\Ndec,\dpot$ such that \eqref{item:inverse:i}--\eqref{item:inverse:iii} of Part 2 in Proposition~\ref{prop:intermittent:inverse:div} and \eqref{eq:DN:Mikado:density}--\eqref{eq:inverse:div:parameters:1} are satisfied.
    \item\label{i:st:sample:5:c} The current error $\varphi=\divH(G\varrho\circ\Phi)$ and nonlocal error $E$ satisfy the conclusions in \eqref{eq:inverse:div}, \eqref{item:div:local:i}--\eqref{item:div:nonlocal} of Part 3 in Proposition~\ref{prop:intermittent:inverse:div}, as well as \eqref{eq:inv:div:extra:conc} from Remark~\ref{rem:pointwise:inverse:div} with $\pi$ replaced by $\pi^{\sfrac 32} r_G^{-1}$.
    \item\label{i:st:sample:6:c} There exist integers $N_\circ,M_\circ,K_\circ$ such that \eqref{eq:inv:div:wut}--\eqref{eq:riots:4} hold, which in turn imply \eqref{eq:inverse:div:error:stress}--\eqref{eq:inverse:div:error:stress:bound}.
\end{enumerate}
\smallskip

\noindent\textbf{Part 2: Additional assumptions}
\begin{enumerate}[(i)]
\item There exists a large positive integer $N_{**}$ and positive integers $\NcutLarge, \NcutSmall$ such that
\begin{subequations}
\begin{align}
N_*-2\dpot - \NcutLarge - N_{**} - 3 &\geq M_* \, , \label{i:st:sample:wut:c} \\
M_*-\NcutSmall - 1 \geq 2 N_\circ \, , &\qquad {M_t > \NcutSmall}\, ,\label{i:st:sample:wut:wut:c} \\
N_{**} &\geq {2\dpot} + 3
\label{i:st:sample:wut:wut:wut:c}
\end{align}
\end{subequations}
\item\label{i:st:sample:7:c} There exist parameters $\Gamma=\Lambda^\alpha$ for $0<\alpha\ll 1$, $\de_{\rm tiny}, r_\phi$, and $\delta_{\phi,p}$ for $p=1,\infty$ satisfying
\begin{subequations}
\begin{align}
0< r_\phi \leq 1 \, , &\qquad \delta_{\phi,p}^{\sfrac 32} = \const_{G,p} \const_{\ast,p}{\Upsilon'\Upsilon^{-2}} r_\phi \, , \label{eq:sample:prop:de:phi:c} \\
\NcutSmall&\leq \NcutLarge \, , \label{eq:sample:prop:Ncut:1:c} \\
\left(\const_{G,\infty} + 1 \right) \left( {\const_{*,\infty}\Upsilon'\Upsilon^{-2}} + 1 \right) \Gamma^{-\NcutSmall} &\leq {\delta_{\rm tiny}}^{\sfrac 32} \, , \const_{G,1} \, , {\const_{*,1}\Upsilon'\Upsilon^{-2}} \, , \label{eq:sample:prop:Ncut:2:c} \\
\qquad 2 \Ndec + 4 &\leq N_* - N_{**} - \NcutLarge - 3 \dpot - 3  \, , \label{eq:sample:prop:Ncut:3:c} \\
\label{eq:sample:prop:decoup:c}
(\Lambda \Gamma)^{4}  &\leq  \left( \frac{\mu}{2 \pi \sqrt{3} \Gamma \max(\lambda,\lambda')}\right)^{\Ndec} \, .
\end{align}
\end{subequations}
\item\label{i:st:sample:8:c} There exists a parameter $\bar{m}$ and an increasing sequence of frequencies $\{\mu_0, \cdots, \mu_{\bar m}\}$ satisfying
\begin{subequations}
\begin{align}
\mu < \mu_0 &< \cdots< \mu_{\bar m -1} \leq \Lambda <\Lambda\Gamma < \mu_{\bar m} \, , \label{eq:sample:prop:par:00:c} \\
\max(\lambda,&\lambda') \Ga \left(  \mu_{m-1}^{-2} \mu_m + {\mu}^{-1} \right) \leq 1 \, , 
\label{eq:sample:prop:parameters:0:c} \\
\left(\const_{G,1} \const_{*,1} r_\phi\right)^{\sfrac 23}& \nu \Gamma (\max(\lambda,\lambda')\Ga)^{\lfloor \sfrac \dpot {{4}} \rfloor} \left(\max\left({\mu}^{-1},\mu_m \mu_{m-1}^{-2}\right) \right)^{\lfloor \sfrac \dpot {{4}} \rfloor} \notag\\
&\qquad \times (\mu_{\bar m})^{5+K_\circ} \left(1 + \frac{\max\{\nu', \const_v \mu_{\bar m} \}}{\nu
}\right)^{M_\circ} \leq 1 \,, \label{eq:sample:riots:4:c}\\
\left(\const_{G,1}\const_{*,1}r_\phi\right)^{\sfrac 23}  \nu \Gamma &\left( \frac{\Lambda\Gamma}{\mu_{\bar m}} \right)^{N_{**}} (\mu_{\bar m})^{8+K_\circ} \left(1 + \frac{\max\{\nu', \const_v \mu_{\bar{m}} \}}{\nu
}\right)^{M_\circ} \leq 1 \, , \label{eq:sample:riot:4:4:c}
\end{align}
\end{subequations}
for all ${1\leq m \leq \bar m}$.
\end{enumerate}
\smallskip

\noindent\textbf{Part 3: Pressure increment} \index{pressure increment}

Define the pressure increment $\si_\phi$ by
\begin{subequations}\label{heatsie:2}
\begin{align}
    \si_{\phi} &= r_\phi^{\sfrac23} \pr(H) \left(\pr(\rho)\circ\Phi-\langle 
    \pr(\rho)\rangle\right) \, , \qquad \qquad 
    \sigma_\phi^+ := r_\phi^{\sfrac 23} \pr(H) \pr(\rho)\circ \Phi \, , \label{eq:cu:d:2}\\
\pr(H) &:= \left( (\const_{G,\infty}\Ga^{-\NcutSmall})^2 + \sum_{N=0}^{\NcutLarge}\sum_{M=0}^{\NcutSmall} (\la \Ga)^{-2N}  (\nu\Ga) ^{-2M} |D^N D_{t}^M H|^2 \right)^\frac13 - (\const_{G,\infty}\Ga^{-\NcutSmall})^{\sfrac23} \, , \label{eq:cu:d:3} \\
\pr(\rho) &:= \left( (\const_{\rho,\infty}\Ga^{-\NcutSmall})^2 + \sum_{N=0}^{\NcutLarge} (\Lambda\Gamma )^{-2N}   |D^N \rho|^2 \right)^\frac13 -  (\const_{\rho,\infty}\Ga^{-\NcutSmall})^{\sfrac 23} \, , \label{eq:cu:d:4}
\end{align}
\end{subequations}
\begin{enumerate}[(i)]
\item $\si_\varphi$ has a synthetic Littlewood-Paley decomposition
\begin{equation}\label{d:press:stress:sample:c}
\si_\varphi=\si_\varphi^+ - \si_\vp^- = \sigma_\vp^* + \sum_{m = 0}^{\bar{m}} \si_\vp^{m} \, .
\end{equation}
\item\label{sample2.item2:c} $\si_\vp^{+}$ dominates derivatives of $\varphi$ with suitable weights, so that for all $N \leq N_* - \floor{\sfrac{\dpot}{2}}$, $M\leq M_*$,
\begin{align}\label{est.S.by.pr.final2:c}
    \left|D^N D_{t}^M \varphi\right|
    \lec \left((\si_\varphi^{+})^{\sfrac 32} r_\phi^{-1}  + \de_{\rm tiny}\right) \left(\Lambda\Ga\right)^N\MM{M,M_{t},\nu\Gamma,\nu'\Gamma} \, .
\end{align}
\item\label{sample2.item4:c} $\si_\vp^+$ dominates derivatives of itself with suitable weights, so that for all $N \leq N_* - \floor{\sfrac{\dpot}{2}} - \NcutLarge $, $M\leq M_* - \NcutSmall$,
\begin{align}
    \label{est.S.prbypr.pt:c}
    \left|D^N D_{t}^M \si_\vp^{+}\right|
    &\lec (\si_\vp^{+}+ \de_{\rm tiny}) \left(\Lambda\Ga\right)^N\MM{M,M_{t}-\NcutSmall,\nu\Ga,\nu'\Ga} \, .
\end{align}
\item\label{sample2.item1:c} $\si_\vp^+$ and $\si_\vp^-$ have size comparable to $\varphi$, so that
    \begin{align}
        \label{est.S.pr.p:c}
        \norm{\si_\vp^{+}}_{\sfrac 32},\, \norm{\si_\vp^{-}}_{\sfrac 32}&\lec \delta_{\phi,1} \, , \qquad \norm{\si_\vp^{+}}_{\infty},\, \norm{\si_\vp^{-}}_{\infty}\lec \delta_{\phi,\infty} \, .
    \end{align}
\item\label{sample2.item5:c} $\pi$ dominates derivatives of $\si_\varphi^-$ with suitable weights, so that for all $N \leq N_* - \floor{\sfrac{\dpot}{2}} - \NcutLarge $, $M\leq M_* - \NcutSmall$,
\begin{align}
    \label{est.S.prminus.pt:c}
    \left|D^N D_{t}^M \si_\varphi^{-}\right|
    &\lec \left( \frac{r_\phi}{r_G} \right)^{\sfrac 23} \left({\const_{*,1} {\Upsilon^{-2}} \Upsilon'}\right)^{\sfrac 23} \pi
    (\max (\la,\la')\Ga)^N\MM{M,M_{t}-\NcutSmall,\nu\Gamma,\nu'\Gamma} \, .
\end{align}
\item\label{sample2.item7:c} 
We have the support properties 
\begin{align}
    \supp(\si_\vp^+) &\subseteq \supp(\varphi) \, , \quad
    \supp(\si_\vp^-) \subseteq \supp(G) \label{est.S.pr.p.support:1:c} \, .
\end{align}
\end{enumerate}
\smallskip

\noindent\textbf{Part 4: Current error}
\begin{enumerate}[(i)]
\item There exists a current error $\phi_{\ph}$, where we have the decomposition and equalities
\begin{subequations}\label{d:cur:error:stress:sample:c}
\begin{align}
\phi_{\ph} &= \phi_\vp^* + \sum_{m=0}^{\bar m} \phi_\vp^m 
\label{S:pr:current:dec}\\
\div \phi^\bullet_\vp(t,x) &= D_t \sigma^{\bullet}_\vp(t,x) - \int_{\T^3} D_t \sigma_\vp^\bullet(t,x') \,dx' \, ,  \qquad \bullet= m, * \, .
\end{align}
\end{subequations}
\item\label{sample2.item3:c} $\phi^m_{\vp}$ can be written as $\phi^m_\vp = \phi^{m,l}_\vp + \phi^{m,*}_\vp$ and {for $1\leq m \leq \bar m$} these satisfy
\begin{subequations}
\begin{align}
\label{est.S.phi.1:c}
\norm{D^N D_{t}^M \phi_{\vp}^m}_{\sfrac 32}
&\lec {\nu\Ga^2 \left(\const_{G,1} \const_{*,1}  {\Upsilon'\Upsilon^{-2}} r_\phi \right)^{\sfrac 23} \mu_{m-1}^{-2}\mu_m}  \left(\min(\mu_m,\Lambda\Gamma)\right)^N\MM{M,M_{t}-\NcutSmall-1,\nu\Ga,\nu'\Ga}\, ,\\
\norm{D^N D_{t}^M \phi_{\vp}^m}_{\infty}
&\lec {\nu\Ga^2 \left(\const_{G,\infty} \const_{*,1} {\Upsilon'\Upsilon^{-2}} r_\phi\right)^{\sfrac 23} \left(\frac{\min(\mu_m,\Lambda\Gamma)}{ \mu}\right)^{\sfrac 43} \mu_{m-1}^{-2}\mu_m} \notag\\
&\qquad \qquad \qquad \times \left(\min(\mu_m,\Lambda\Gamma)\right)^N\MM{M,M_{t}-\NcutSmall-1,\nu\Ga,\nu'\Ga}\, , \label{est.S.phi.infty:c} \\
\left|D^N D_{t}^M \phi_{\vp}^{m,l}\right|
&\lec {\nu\Ga^2 \pi \left( \frac{r_\phi}{r_G}\right)^{\sfrac 23} \left(\const_{*,1} {\Upsilon'\Upsilon^{-2}} \right)^{\sfrac 23} \left(\frac{\min(\mu_m,\Lambda\Gamma)}{ \mu}\right)^{\sfrac 43} \mu_{m-1}^{-2}\mu_m}  \notag\\
&\qquad \qquad \qquad \times \left(\min(\mu_m,\Lambda\Gamma)\right)^N\MM{M,M_{t}-\NcutSmall-1,\nu\Ga,\nu'\Ga}\, , \label{est.S.by.pr.final3:c}
\end{align}
\end{subequations}
for all $N \leq N_* - 2\dpot-{\NcutLarge}$, $M\leq M_* - {\NcutSmall}-1$. For {$m=0$} and the same range of $N$ and $M$, $\phi_\varphi^{m}$ and $\phi_\varphi^{m,l}$ satisfy identical bounds but with $\mu_{m-1}^2\mu_m$ replaced with ${\Gamma\mu^{-1}}$ and $\min(\mu_m,\Lambda\Gamma)$ replaced with {$\mu_0$} in all three bounds. Furthermore, for all $N \leq N_\circ, M\leq M_\circ$, the nonlocal and remainder portions satisfy the improved estimates
\begin{align}
\label{est.S.by.pr.final4:c}
    \norm{D^N D_{t}^M \phi_{\vp}^{m,*}}_\infty
    &\lec  \left(\min(\mu_m,\Lambda\Gamma)\right)^{N-K_\circ} 
    {(\max(\lambda,\lambda')\Ga)^{\lfloor \sfrac \dpot {{4}} \rfloor} \left(\max\left({\mu}^{-1},\mu_m \mu_{m-1}^{-2}\right) \right)^{\lfloor \sfrac \dpot {{4}} \rfloor}}
    (\nu\Gamma)^M \, , \\
\label{est.S.by.pr.final.star:c}
    \norm{D^N D_{t}^M \phi_{\vp}^{*}}_\infty
    &\lec (\Lambda\Gamma)^{-K_\circ}  
    {(\max(\lambda,\lambda')\Ga)^{\lfloor \sfrac \dpot {{4}} \rfloor} \left(\max\left({\mu}^{-1},\mu_m \mu_{m-1}^{-2}\right) \right)^{\lfloor \sfrac \dpot {{4}} \rfloor}}
    \left(\Lambda\Ga\right)^N (\nu\Gamma)^M \, .
\end{align}
\item\label{sample2.item4:c:redux} We have the support properties
\begin{align}
    \supp(\phi^{m,l}_\varphi)
    &\subseteq \supp G \cap B\left( \supp \vartheta, 2\mu_{m-1}^{-1} \right) \circ \Phi \,  \textnormal{ for } \, 1\leq m \leq \bar{m} \, , \qquad \supp\left(\phi_\varphi^{0,l} \right) \subseteq \supp G \,  . \label{est.S.pr.p.support:2:c}
\end{align}

\item\label{sample2.item5:c:redux} 
For all $M\leq M_*-\NcutSmall-1$, we have that the mean $\langle D_t \si_S \rangle$ satisfies
\begin{equation}\label{est:mean.Dtsiph}
    \left|\frac{d^M}{dt^M}\langle D_t \si_\ph \rangle\right|\les
    (\La\Ga)^{-K_\circ}
    {(\max(\lambda,\lambda')\Ga)^{\lfloor \sfrac \dpot {{4}} \rfloor} \mu^{-\lfloor \sfrac \dpot {{4}} \rfloor}}
    \MM{M,M_t - \NcutSmall-1, \nu\Ga, \nu'\Ga}
\end{equation}
\end{enumerate}
\end{proposition}

\printindex

\medskip

\noindent\textsc{Department of Mathematics, ETH Z\"urich, Z\"urich, Switzerland.}
\vspace{.03in}
\newline\noindent\textit{Email address}: \href{mailto:vikramaditya.giri@math.ethz.ch}{vikramaditya.giri@math.ethz.ch}.
\smallskip

\noindent\textsc{Department of Mathematics, ETH Z\"urich, Z\"urich, Switzerland.}
\vspace{.03in}
\newline\noindent\textit{Email address}: \href{hyunju.kwon@math.ethz.ch}{hyunju.kwon@math.ethz.ch}.
\smallskip

\noindent\textsc{Department of Mathematics, Purdue University, West Lafayette, IN, USA.}
\vspace{.03in}
\newline\noindent\textit{Email address}: \href{mailto:mdnovack@purdue.edu}{mdnovack@purdue.edu}.

\end{document}